 \documentclass[11pt,letterpaper, reqno]{amsart} 
 \usepackage[margin=1in]{geometry}
 \usepackage{style}

\usepackage{amssymb,amsmath,amsthm
	, stackengine}
\usepackage{amscd}
\usepackage{amsfonts}
\usepackage{amssymb,mathrsfs}
\usepackage{latexsym}
\usepackage{color}
\usepackage{esint}
\usepackage{kotex}

\usepackage{titletoc}
\usepackage{etoolbox}

\usepackage[makeroom]{cancel}
\usetikzlibrary{decorations.pathreplacing}

\usepackage[english]{babel}

\addto\captionsenglish{
	\renewcommand{\contentsname}%
	{Contents}%
}

 \newtheorem{corollary}{Corollary}
\newtheorem{proposition}{Proposition}

\let\hide\iffalse
\let\unhide\fi

\let\e=\varepsilon

\let\p=\partial

\let\O=\Omega

\let\f=\varphi

\renewcommand{\b}{\mathbf{b}}
\renewcommand{\f}{\mathbf{f}}

\newcommand{\be}{\begin{equation}}
	\newcommand{\bm}{\begin{multline}}
		\newcommand{\ee}{\end{equation}}
	\newcommand{\dd}{\mathrm{d}}

	\newcommand{\xb}{x_{\mathbf{b}}}
	
	\newcommand{\tb}{t_{\mathbf{b}}}
	\newcommand{\vb}{v_{\mathbf{b}}}
		\newcommand{\tB}{t_{\mathbf{B}}}
	\newcommand{\vB}{v_{\mathbf{B}}}
			\newcommand{\xB}{x_{\mathbf{B}}}

	\newcommand{\tF}{{t}_{\mathbf{F}}}

	\newcommand{\tf}{t_{\mathbf{f}}}

	\newcommand{\X}{\mathcal{X}}
	\newcommand{\V}{\mathcal{V}}
 	\newcommand{\Zz}{\mathcal{Z}}

		\newcommand{\w}{\mathfrak{w}} 
	 
	\usepackage{amsthm}
	
	\newcommand{\Bes}{\begin{eqnarray*}}
		\newcommand{\Ees}{\end{eqnarray*}}
	\newcommand{\Be}{\begin{equation}}
		\newcommand{\Ee}{\end{equation}}
	
	\numberwithin{equation}{section}

	\def\p{\partial}

	\def\O{\Omega}
	\def\R{\mathbb{R}}
	
	\def\B{\begin{equation}}
		\def\E{\end{equation}}
	\def\BN{\begin{eqnarray*}}
		\def\EN{\end{eqnarray*}}
	
	\usepackage{color}

	\renewcommand\contentsname{}

\title{Nonlinear asymptotic stability of inhomogeneous steady solutions to boundary problems of Vlasov-Poisson equation}
 
\author{Chanwoo Kim}
\address{Department of Mathematics, University of Wisconsin-Madison, Madison, WI 53706, USA} 
\email{chanwookim.math@gmail.com}
\date{\today}

\begin{document}
	
	\begin{abstract}
	We consider an ensemble of mass collisionless particles, which interact mutually either by an \textit{attraction} of Newton’s law of gravitation or by an electrostatic \textit{repulsion} of Coulomb's law, under a background downward gravity in a horizontally-periodic 3D half-space, whose inflow distribution at the boundary is prescribed. We investigate a \textit{nonlinear asymptotic stability} of its generic steady states in the \textit{dynamical} kinetic PDE theory of the \textit{Vlasov-Poisson} equations. We construct Lipschitz continuous \textit{space-inhomogeneous steady states} and establish exponentially fast asymptotic stability of these steady states with respect to \textit{small perturbation in a weighted Sobolev topology}. In this proof, we crucially use the Lipschitz continuity in the velocity of the steady states. Moreover, we establish well-posedness and regularity estimates for both steady and dynamic problems.  \hide

	First, for given Lipschitz continuous boundary data, we construct steady-states which is \textit{spatial inhomogeneous} in general. Then we construct a global-in-time solution to the dynamic problem in $L^\infty$ around the steady-states. In particular, we prove nonlinear asymptotic stability of such steady-states at an exponential speed as $t \rightarrow \infty$, in which our regularity study of steady-states is crucial.\unhide
		\end{abstract}

	\maketitle






 \hide
 \bigskip
 {
 	\begin{center}
 		\large{I}\footnotesize{NTORDUCTION AND} \large{M}\footnotesize{AIN} \footnotesize{RESULTS}
 	\end{center}
 }
 
 \bigskip
 
 \unhide
 
 \section{Introduction}
 We consider the Vlasov-Poisson equations (\cite{glassey,Rein}) subjected to a vertical downward gravity of a fixed gravitational constant $g>0$ in a 3D half space $(x_1, x_2, x_3) \in \O : =  \T^2 \times (0, \infty )$ with a periodic cube $\T^2 = \{(x_1, x_2) : - \frac{1}{2} \leq x_i < \frac{1}{2},  i=1,2\}$: 
 \begin{align}
 	\p_t F
 	+ v\cdot \nabla_x F
 	-
 	\nabla_x (
 	\phi_F  + g 
 	x_3) \cdot \nabla_v F
 	=0 \ \ &\text{in} \ \R_+ \times  \O \times \R^3,\label{VP_F} \\
 	F(0,x,v)=F_0(x,v)  \ \ &\text{in} \   \O \times \R^3.\label{VP_0}
 \end{align} 
 Here, the potential of an intermolecular force solves Poisson equations as
 \Be\label{Poisson_F}
 \begin{split}
 	\Delta_x \phi_F  = \eta  
 	\int_{\R^3}   F  \dd v \  \ \text{in} \ \O,
 \end{split}
 \Ee
 for either $\eta=+1$ (an attractive potential of the Newton's law of gravitation) or $\eta=-1$ (a repulsive potential of the Coulomb's law). In our paper, 
 all results hold for $\eta=1$ and $\eta=-1$. For the sake of simplicity, we have set physical constants such as masses and size of charges of identical particles to be $1$. 
 At the incoming boundary $\gamma_- := \{(x,v): x_3=0 \ \text{and } v_3>0\}$, the distribution $F$ satisfies an in-flow boundary condition: For a given function $G (x,v) \geq 0$ on $\gamma_-$,
 \Be\label{bdry:F}
 F    = G   \ \ \text{on} \ \  \gamma_-
 := \{(x,v): x_3=0 \ \text{and } v_3>0\}
 ;
 \Ee
 while the potential satisfies the zero Dirichlet boundary condition at the boundary:
 \Be\label{Dbc:F}
 \phi_F  |_{x_3=0} = 0. 
 \Ee
 About the application of this problem, we refer to \cite{CK_VM, CK_KRM, JK1, JK2, JK3, JK4}. For example, see an application in the stellar atmosphere such as the solar wind theory of the Pannekoek-Rosseland condition in \cite{CK_VM} and the references therein.


 In the contents of nonlinear Vlasov systems, constructing steady states (\cite{GuoR, Rein_bdry, EHS,GR}) and studying their stability (\cite{GuoLin, EH}) or instability (\cite{GuoS, Lin}) have been important subjects. Several boundary problems have been studied in \cite{EH,EHS,GR,Guo94,Guo95,HV,Rein}. Among others, we discuss some literature concerning the asymptotic stability of the Vlasov-Poisson system in a confining setting. In \cite{Landau}, Landau looked into analytical solutions of the linearized Vlasov-Poisson system around the Maxwellian and observed that the self-consistent field $\nabla_x \phi_F$ is subject to temporal decay even in the absence of collisions (cf. Boltzmann equation \cite{CKL_VPB}). A rigorous justification of the Landau damping in a nonlinear dynamical sense has been a long-standing major open problem. In \cite{HV,CM}, it was shown that there exist certain analytical perturbations for which the fields decay exponentially at the nonlinear level. Recently, Mouhot-Villani settled the nonlinear Landau damping affirmatively for general real-analytical perturbations of stable space-homogeneous equilibria with exponential decay in \cite{MV}. Bedrossian-Masmoudi-Mouhot establishes the theory in the Gevrey regular perturbations in \cite{BMM}. We also refer to \cite{GNR} for a very recent result in this direction.


 In this celebrated justification of nonlinear Landau damping, the high regularity such as \textit{real}-\textit{analyticity} or \textit{Gevrey regularity} of perturbation seems crucial as some counterexamples are constructed for Sobolev regular perturbations (\cite{LinZeng,Bed}). Moreover, the theories of \cite{MV,BMM,GNR} strictly apply to \textit{space-homogeneous} equilibria but not space-inhomogeneous states. However, in many physical cases, the boundary problems do not allow these two constraints in general. Any steady solution to \eqref{VP_F}-\eqref{Poisson_F}, if exists, is space-inhomogeneous unless the boundary datum $G$ in \eqref{bdry:F} is space-homogeneous. Moreover, derivatives of any solution $F$ to \eqref{VP_F} are singular in general (\cite{Guo95,CK_VM}).

 In this paper, we establish a different stabilizing effect of downward gravity and the boundary in the content of the nonlinear Vlasov-Poisson system. Namely, we construct \textit{space-inhomogeneous steady states} $h(x,v)$, which are Lipschitz continuous, and establish exponentially fast asymptotic stability of these steady states with respect to \textit{small (in a weighted $L^\infty$ topology) perturbation} $f$:
 \Be F (t,x,v) = h (x,v) + f (t,x,v).\label{def:F}
 \Ee
 For the initial datum in \eqref{VP_F}, we set $F_0 (x,v) = h(x,v) + f_0 (x,v)$. \hide
 
 around certain category of space-inhomogeneous steady states $h(x,v)$ in nonlinear dynamic problems:\footnote{
 	We illustrate the stability mechanism of gravity and some Gaussian upper bound ($\nabla_v h$, in particular) in Section \ref{stab_mech}.
 }\unhide

 \hide
 
 \medskip
 
 \newpage
 \noindent\textbf{A. Illustration of Stabilizing effect.} 
 \unhide

 \subsection{Illustration of the Bootstrap argument}\label{sec:SE} We shall illustrate how a strong gravity may stabilize the Vlasov system for a certain class of steady solutions. As far as the author knows, this stability mechanism is \textit{new in the nonlinear contents}. For simplicity, we pick a simplified toy PDE (the real PDE in \eqref{eqtn:f}): for given $E=E(t,x)$, 
 \Be\notag\label{simple:h}
 \begin{split} 
 	\p_t f + v\cdot \nabla_x f - g \p_{v_3} f = E\cdot \nabla_v h  \ \ \text{in} \ \R_+ \times  \O \times \R^3    \ \ &\text{in} \ \R_+ \times  \O \times \R^3.
 \end{split}
 \Ee
 We add $E\cdot \nabla_v h$ to count the nonlinear contribution of the Vlasov-Poisson equation (cf. \cite{JK2}). The natural boundary condition of $f$ is the absorption boundary condition $f=0$ on $\gamma_-$ as in \eqref{bdry:f}.
 

 The characteristics of \eqref{simple:h} is explicitly given by $\dot{\X}= \V$ and $\dot{\V} = - g \mathbf{e}_3$; or $\X_i(s;t,x,v) = x_i- (t-s) v_i - g \delta_{i3} \frac{(t-s)^2}{2} $ and $\V_i(s;t,x,v) = v_i+ g\delta_{i3}  (t-s)$. The unique non-negative time lapse $\tB(t,x,v)\geq 0$, satisfying $\X_3(t-\tB(t,x,v);t,x,v)=0$, is given by $\tB (t,x,v)= \frac{1}{g} \big( \sqrt{|v_3|^2 + 2 g x_3} - v_3\big)$. Therefore, due to a crucial effect of gravity, we can control $\tB$ by the total energy of the particle:
 \Be\label{intro:tb}
 \tB(t,x,v) \leq  \frac{2}{g} \sqrt{|v|^2 + 2g x_3}.
 \Ee
 
 Now we can bound a local density $\varrho (t,x) = \int_{\R^3} f(t,x,v) \dd v$ (see \eqref{def:varrho}) by
 \Be\label{intro:rho}
 \begin{split}
 	|\varrho(t,x)| &\leq  \int_{\R^3}\int^t_{
 		t-\tB(t,x,v)
 	}   | E(s, \X(s;t,x,v))  |
 	\underbrace{|  \nabla_v h(s, \X(s;t,x,v), \V(s;t,x,v))|} \dd s \dd v + \cdots .
 \end{split}
 \Ee
 A major challenge is to achieve an $E\cdot \nabla_v h$-control, which corresponds to the nonlinear contribution for the real problem \eqref{eqtn:f}. Let us impose a crucial condition, namely $\nabla_v h(x,v)$ has some Gaussian upper bound with respect to the total energy: for a universal positive constant C>0, 
 \Be
 |\nabla_v h(x,v)| \leq C e^{-\beta (|v|^2 + 2g x_3)}.\label{intro:Dvh}
 \Ee  
 From the fact that the energy is conserved along the characteristics, we can derive that $|\varrho(t,x)|$ is bounded above by 
 \Be\notag
 \begin{split}
 	e^{-\lambda t }\left( \int_{\R^3}  \frac{ \tB(t,x,v) e^{\lambda  \tB(t,x,v)}  }{e^{\beta(|v|^2 + g x_3) }} \dd v \right)
 	\sup_{(s,x)  
 	}e^{\lambda s}  |  E(s,x)|
 	\sup_{(x,v)
 	}e^{\beta (|v|^2 + 2 gx_3)}| \nabla_v h(x,v)|.
 \end{split}\Ee 
 Now using the crucial bound \eqref{intro:tb} of $\tB$ with respect to the total energy, we can control the term in the parenthesis above by $
 O\Big(\frac{1}{g  \beta^2} e^{C \frac{\lambda^2}{ g^2 \beta  }}  \Big).$ Therefore we might hope that
 \Be\notag
 \begin{split}
 	\sup_{t} e^{\lambda t} \|\varrho(t)\|_{L^\infty(\O)} &\leq \frac{C}{g  \beta^2} e^{C \frac{\lambda^2}{ g^2 \beta  }}   
 	\| e^{\beta (|v|^2 + 2 gx_3)} \nabla_v h \|_{L^\infty(\O \times \R^3)} \sup_{t} e^{\lambda t} \| E(t) \|_ {L^\infty(\O)} + \cdots .
 \end{split}
 \Ee 
 
 On the other hand, for the real nonlinear problem \eqref{eqtn:f}, the Poisson equation \eqref{Poisson_f} of $E= \nabla_x \Psi$ might suggest a control of $ E(t ,x)$ pointwisely (at least locally) by some weighted pointwise bound of $\varrho(t,x)$ mainly. Therefore, as far as we have chosen $g \beta^2$ large enough, depending on our possible control of $\| e^{\beta (|v|^2 + 2 gx_3)} \nabla_v h \|_{L^\infty(\O \times \R^3)} $ and $\lambda$, we find ``a small factor'' in the nonlinear contribution. 
 
 Applying this idea to a real nonlinear problem is challenging as the steady and dynamic characteristics are governed by the different self-contained fields. Moreover, as the total energy is not conserved along the dynamic characteristics, we cannot simply deduce a crucial Gaussian upper bound of the underbraced term in \eqref{intro:rho}, even \eqref{intro:Dvh} is granted. Indeed, we need a fine control of the nonlinear characteristics for both steady and dynamic problems. To realize the idea in the real nonlinear problem, we ought to overcome two major difficulties: nonlinear regularity estimate of steady solutions with a weight of the total energy as \eqref{intro:Dvh}; nonlinear control of characteristics, involving some elliptic estimates in $\T^2 \times (0,\infty)$. These issues are nontrivial as even linear transport equations have a singularity at the grazing set $\gamma_0 = \{ x_3=0 \ \text{and} \ v_3=0\}$ in general (\cite{CK_VM,Guo95,Kim_FPS}).

 \hide
 \medskip
 \newpage
 
 \noindent\textbf{A. Construction of Steady Solutions.} 
 
 \unhide

 \textbf{Notations:} Throughout this paper, we often use the following notations: $C>0$ is a universal positive constant unless it is specified; $x \lesssim y$ means $x \leq C y$.

 \section{Main Results}

 Consider the steady problem (for $h=h(x,v)$): 
 \begin{align}
 	v\cdot \nabla_x h -    \nabla_x (  \Phi + g   x_3 ) \cdot \nabla_v h  =0 \ \  &\text{in} \ \O \times \R^3,\label{VP_h} \\
 	h   = G 
 	\ \  &\text{on} \  \gamma_-
 	:= \{(x,v): x_3=0 \ \text{and } v_3>0\}
 	. \label{bdry:h}
 \end{align}
 We define a steady local density (whenever $h(x, \cdot) \in L^1 (\R^3)$)
 \Be
 \rho (x) = \int_{\R^3} h(x,v) \dd v. \label{def:rho}
 \Ee
 Then the potential of a steady distribution solves (let $\eta \Delta_0^{-1} \rho$ denote $\Phi $)
 \Be\begin{split}
 	\Delta_x \Phi (x)=  \eta \rho(x)
 	\ \   \text{in} \ \O ,\ \    \ \text{and} \  \ 
 	\Phi  =0 
 	\ \   \text{on} \  \p\O .
 	\label{eqtn:Dphi}
 \end{split}\Ee
 
 %
 
 Next we consider the dynamical problem \eqref{VP_F}-\eqref{Dbc:F} as a perturbation $ f (t,x,v)$ in \eqref{def:F} around the steady solution $(h, \Phi)$ to \eqref{VP_h}-\eqref{eqtn:Dphi}:
 \hide
 \Be\begin{split}\notag
 	0=&  \ \big[	\p_t + v\cdot \nabla_x + \frac{1}{m_\pm} \nabla_x (\Phi + \Psi - g m_\pm x_3) \cdot \nabla_v \big] h_\pm\\
 	&+   \big[	\p_t + v\cdot \nabla_x + \frac{1}{m_\pm}  \nabla_x (\phi_F - g m_\pm x_3) \cdot \nabla_v  \big] f_\pm\\
 	= &  \  \frac{1}{m_\pm} \nabla_x \phi_{f} \cdot \nabla_v h_\pm +   \big[	\p_t + v\cdot \nabla_x + \frac{1}{m_\pm} \nabla_x (\phi_F - g m_\pm x_3) \cdot \nabla_v  \big] f_\pm.
 \end{split}\Ee
 Hence we derive an equation for $f_\pm (t,x,v)$ in \eqref{F_pert}:as \unhide
 \begin{align}
 	\p_t  f  + v\cdot \nabla_x   f  - \nabla_x (\Psi + \Phi +g  x_3) \cdot \nabla_v   f   =   \nabla_x \Psi \cdot \nabla_v h   \ \ &\text{in} \ \R_+ \times  \O \times \R^3,  \label{eqtn:f} \\
 	f(0,x,v) = f_0 (x,v)   \ \ &\text{in} \  \O \times \R^3,  \label{f_0} \\
 	f   = 0  
 	\ \ &\text{on}  \  \gamma_-.\label{bdry:f}  
 \end{align}
 We define a local density of the dynamical fluctuation 
 \Be\label{def:varrho}
 \varrho(t,x) = \int_{\R^3} f(t,x,v) \dd v.
 \Ee
 Then the electrostatic potential $\Psi = \eta \Delta_0^{-1} \varrho$ solves
 \begin{align}
 	\Delta_x	\Psi (t,x)
 	= \eta\varrho (t,x)
 	\ \  \text{in} \ \R_+ \times  \O 
 	, \ \text{and }
 	\  
 	\Psi(t,x) =0  \ \   \text{on} \ \R_+ \times \p\O. \label{Poisson_f} 
 \end{align}  
 Often we let $\eta \Delta_0^{-1} \varrho$ denote $\Psi $. The evolution of $\varrho $ is determined by a continuity equation 
 \Be\label{cont_eqtn}
 \p_t\varrho  + \nabla_x \cdot b  =0 \ \ \text{in} \ \R_+ \times \O,
 \Ee 
 where the flux is defined by 
 \Be\label{def:flux}
 b (t,x) := \int_{\R^3} v   f (t,x,v)  \dd v .
 \Ee
 
 \hide \begin{definition}
 	We define the flux 
 	\Be\label{def:flux}
 	b (t,x) := \int_{\R^3} v   f (t,x,v)  \dd v .
 	\Ee
 	The equation \eqref{eqtn:f} yields

 	Note that we should understand   \end{definition}
 \unhide

 \subsection{Lagrangian approach}\label{sec:char}
 
 Consider the characteristics $Z (s;x,v) = (X (s;x,v), V (s;x,v))$ for the steady problem \eqref{VP_h}: 
 \Be
 \begin{split}\label{ODE_h}
 	\frac{dX (s;x,v) }{ds} =V (s;x,v)  ,\ \
 	\frac{dV  (s;x,v) }{ds}  =   - \nabla_x  \Phi(X (s;x,v)) -g   \mathbf{e}_3,
 \end{split}
 \Ee
 where $ \Phi \in C^1 (\bar{\O}) \cap  C^2 (\O)$ solves \eqref{eqtn:Dphi} and $\mathbf{e}_3= (0,0,1)^T$. The data at $s=0$ is given by $Z (0;x,v) = (X (0;x,v) , V (0;x,v)) = (x,v)=z$.
 
 We also define the characteristics $\Zz(s;t,x,v) = (\X(s;t,x,v), \V(s;t,x,v))$ for the dynamical problem \eqref{eqtn:f} solving
 \Be\label{ODE_F}
 \begin{split}
 	\frac{d \X  (s;t,x,v) }{ d s} &= \V(s;t,x,v), \\
 	\frac{d \V (s;t,x,v)}{d s} &= - \nabla_x \Psi  (s, \X(s;t,x,v)) - \nabla_x \Phi (\X(s;t,x,v)) - g  \mathbf{e}_3,
 \end{split}
 \Ee 
 and satisfying $\Zz (t;t,x,v) = (\X (t;t,x,v), \V (t;t,x,v))  = (x,v)=z.$ Here, $\Psi (t, \cdot )\in C^1 (\bar{\O}) \cap  C^2 (\O)$ and $\Phi\in C^1 (\bar{\O}) \cap  C^2 (\O)$ solve \eqref{Poisson_f} and \eqref{eqtn:Dphi}, respectively. Note that the Picard theorem ensures that the unique solutions $Z$ and $\Zz$ to ODEs \eqref{ODE_h} and \eqref{ODE_F} exist, respectively.
 \hide\begin{remark}
 	When $\nabla_x \Phi \in W^{1,p}(\O)$ and $\nabla_x \Psi (t) \in W^{1,p} (\O)$ for $p <\infty$, we should understand our solutions to \eqref{ODE_h} and \eqref{ODE_F} as a regular Lagrangian flow (\cite{DiL_ODE}). In other words, they are absolutely continuous (with respect to $s$) integral solutions, and measure preserving.  
 	
 	On the other hand, when the data (a boundary datum $G$ on $\gamma_-$ or an initial datum $F_0$ on $\O \times \R^3$) are regular as in Theorem \ref{theo:RS} and Theorem \ref{theo:RD} then $\nabla_x \Phi \in W^{1,\infty}(\O)$ and $\nabla_x \Psi (t) \in W^{1,\infty} (\O)$. Therefore the classical Picard theorem ensures that the Lagrangian flow is actually the classical solution of ODEs. 
 \end{remark}\unhide

 \begin{definition}\label{def:tb}(1) Suppose $\nabla_x \Phi \in C^1(\O)$ for $p>2$. Then $Z (s;x,v)$ is well-defined as long as $X (s;x,v) \in \O$. There exists a backward/forward exit time 
 	\Be\label{tb^h}
 	\begin{split}
 		\tb(x,v)& : = \sup\{ s  \in [0,\infty) :	X _{ 3} (-\tau; x,v )>0  \ \text{for all }  \tau \in (0, s)\}  \geq 0,\\ 
 		\tf(x,v)& : = \sup\{ s  \in [0,\infty) :	X _{ 3} ( +\tau; x,v )>0  \ \text{for all }  \tau \in (0, s)\}  \geq 0.
 	\end{split}	\Ee
 	In particular, $X _{ 3} ( -\tb    (x,v); x,v )=0$. Moreover, $Z (s;x,v)$ is continuously extended in a closed interval of $s \in [-\tb(x,v),0]$.  
 	
 	We also define backward exit position and velocity:
 	\Be\label{def:zb^h}
 	\xb(x,v) = X   ( -\tb (x,v); x,v ) \in \p\O, \ \ \vb (x,v) = V  ( -\tb(x,v); x,v ).
 	\Ee
 	\hide
 	Similarly, define $t_{\mathbf{f}} ^{h, \pm}(x,v)\geq 0$ to be the smallest non-negative number satisfying
 	\Be
 	X^{h}_{\pm,3} (t_{\mathbf{f}} ^{h, \pm}(x,v); x,v )=0
 	\Ee
 	and also define 
 	\Be
 	x^{h, \pm}_{\mathbf{f} }(x,v) = X_\pm^h  (  t^{h, \pm }_{\mathbf{f}} (x,v);  x,v ), \ \  v^{h, \pm}_{\mathbf{f}} (x,v) = V^h_\pm  (  t^{h, \pm}_{\mathbf{f}} (x,v);  x,v ).
 	\Ee\unhide
 	
 	(2)	Suppose $\nabla_x \Phi, \nabla_x \Psi(t, \cdot) \in C^1(\O)$ for $p>1$. Then $\Zz(s;t,x,v)$ is well-defined as long as $\X(s;t,x,v) \in \O$. There exists a backward/forward exit time 
 	\Be\label{tb}
 	\begin{split}
 		\tB (t,x,v)
 		&: = \sup \{
 		s \in [0,\infty): \X_3 (t-\tau;t,x,v)>0 \ \text{for all } \tau \in (0,s)
 		\}
 		\geq 0, \\
 		\tF (t,x,v)
 		&: = \sup \{
 		s \in [0,\infty): \X_3 (t+\tau;t,x,v)>0 \ \text{for all } \tau \in (0,s)
 		\}
 		\geq 0,
 	\end{split}
 	\Ee
 	such that $X _{ 3} (t-\tB (t,x,v); t,x,v )=0$ and backward exit position and velocity are defined
 	\Be\label{def:zb}
 	\xB (t,x,v) = \X   ( t-\tB (t,x,v); t,x,v ) \in \p\O, \ \  \vB (t,x,v) = \V (t -\tB (t,x,v); t,x,v ).
 	\Ee
 	Then $\Zz(s;t,x,v)$ is continuously extended in a closed interval of $s \in [t-\tB (t,x,v),t]$. 
 \end{definition}

 \begin{definition}[Mild solution]\label{def:mild}
 	For given $C^2$ potentials and their characteristics, it is well-known that any weak solutions are the Lagrangian solution. For the steady problem \eqref{VP_h}-\eqref{eqtn:Dphi} and dynamic problem \eqref{eqtn:f}-\eqref{Poisson_f}, they take the form of
 	\Be
 	\begin{split}\label{Lform:h}
 		h(x,v)  =  \mathbf{1}_{t \leq t_\b   (x,v)} h(X (-t; x,v),V (-t; x,v) ) 
 		+  \mathbf{1}_{t > t_\b   (x,v)} G(\xb  (x,v), \vb (x,v))) ;
 	\end{split}\Ee
 	and
 	\Be
 	\begin{split}\label{Lform:f}
 		f  (t,x,v) & = \mathbf{1}_{t \leq \tB  (t,x,v)} f  (0,\X (0;t,x,v),\V  (0;t,x,v))
 		\\&+  \int^t_{ \max\{0,t-\tB  (t,x,v)\}} \nabla_x \Psi (s, \X (s;t,x,v)) \cdot \nabla_v h  (  \X (s;t,x,v), \V (s;t,x,v))  \dd s .
 	\end{split}\Ee
 \end{definition}

 As we have described across \eqref{intro:Dvh} in the introduction, Gaussian weight functions have a crucial role in our analysis. 
 \begin{definition}\label{def:w_dyn}
 	For an arbitrary $\beta>0$, we set weight functions for the steady problem \eqref{VP_h}
 	and for dynamic problem \eqref{VP_F} and \eqref{eqtn:f}
 	\begin{align}
 		w (x,v)   &=	w _\beta (x,v) = e^{ \beta \big( |v|^2  + 2 \Phi(x) + 2 g   x_3\big)},\label{w^h}\\
 		\w (t,x,v) & =		\w_\beta  (t,x,v) = e^{ \beta \big( |v|^2 + 2\Phi(x) + 2\Psi(t,x)  + 2g   x_3\big)}.\label{w^F}
 	\end{align}
 \end{definition}

 \begin{remark}Few basic properties: (i) the steady weight $w   (x,v)$ in \eqref{w^h} is invariant along the steady characteristics \eqref{ODE_h}:
 	\Be\label{w:invar}
 	w(X(s;x,v),V(s;x,v)) = w(x,v) \ \ \text{for all} \ s \in [-\tb (x,v),0].
 	\Ee 
 	
 	(ii) The dynamic weight $\w   (t,x,v)$ is not invariant along the dynamic characteristics, as the dynamic total energy is not invariant : 
 	\Be \label{dDTE}
 	\begin{split}
 		&\frac{d}{ds} \big(
 		|\V  (s;t,x,v)|^2 
 		+  2\Phi(\X (s;t,x,v)) 
 		+ 2 \Psi (s, \X (s;t,x,v))
 		+ 2 g  \X _{  3} (s;t,x,v)
 		\big)\\
 		& = 
 		2 \p_t \Psi (s,\X  (s;t,x,v)).  
 	\end{split}
 	\Ee
 	
 	(iii) If the Dirichlet boundary conditions \eqref{eqtn:Dphi} and \eqref{Poisson_f} hold then we have that at the boundary
 	\Be\label{w:bdry}
 	w _\beta (x,v) \equiv \w_\beta  (t,x,v) \equiv e^{\beta |v|^2} \ \ \text{at}  \ \ x_3=0.
 	\Ee
 	
 	(iv) At the initial plan $t=0$, 
 	\Be\label{W_t=0}
 	\w_\beta (0,x,v) = e^{ \beta \big( |v|^2 + 2\Phi(x) + 2\Psi(0,x)  + 2g   x_3\big)}
 	= \w_{\beta,0} (x,v)
 	:= w_{\beta}(x,v) e^{2 \eta \Delta_0^{-1} \int_{\R^3} f_0 (x,v) \dd v }.
 	\Ee
 	Here, $\Delta_0^{-1} \int_{\R^3} f_0 (x,v) \dd v=\int_{\O} G(x,y)\int_{\R^3} f_0 (y,v) \dd v \dd y $ as in \eqref{phi_rho}. 
 \end{remark}
 \hide
 {\color{red}Where should I put this:
 	\begin{lemma}
 		Later we need to prove the convergence of $w^{F^\ell}(t,x,v)  f^\ell(t,x,v)$ where 
 		\Be
 		w^{F^\ell}_{\pm}(x,v) = e^{ \beta \big(\frac{|v|^2}{2} + \phi_{f^\ell}(t,x)  + \phi_h(x) + g m_\pm x_3\big)}
 		\Ee
 		provided 
 		\Be
 		\sup_\ell \|w^{F^\ell}(t,x,v)  f^\ell(t,x,v)\|_{L^\infty_{t,x,v}}< \infty
 		\Ee
 		Then 
 		\Be
 		w^{F^\ell}(t,x,v)  f^\ell(t,x,v) \rightarrow w^{F}(t,x,v)  f(t,x,v)   \ \ L^\infty \  weak \  star
 		\Ee
 	\end{lemma}
 	\begin{proof}
 		
 		[[LATER]]
 	\end{proof}
 }\unhide

 As we have described across \eqref{intro:tb} in the introduction, the next lemma is crucial in our analysis.
 \hide
 For a given function $E(t,x)$, let $Z (s;t,x,v) = (X  (s;t,x,v), V (s;t,x,v))$ be a flow of 
 \Be\label{ODE_XV}
 \begin{split}
 	\frac{d X (s;t,x,v)}{ds} = V (s;t,x,v), \ \ \ 
 	m\frac{d V (s;t,x,v)}{ds} = E(s, X (s;t,x,v)) - g m  \mathbf{e}_3,\\
 	Z (s;t,x,v) |_{s=t} = (x,v).
 \end{split}
 \Ee\unhide

 \begin{lemma}\label{lem:tb}
 	(i) Recall the steady characteristics \eqref{ODE_h} and its self-consistent potential $\Phi\in C^1 (\bar{\O}) \cap  C^2 (\O)$ in \eqref{eqtn:Dphi}. Suppose the condition \eqref{Uest:DPhi} holds.  
 	Then the backward exit time \eqref{tb^h} is bounded above as  \hide	
 	Then we have that, for all $(t,x,v) \in \R_+ \times \bar \O \times \R^3$, \Be
 	\Ee\unhide
 	\Be \label{est:tb^h}
 	\tb  (x,v) \leq  
 	\frac{2}{g} \min \Big\{\sqrt{|v_3|^2 + g x_3} -  v_3,  \sqrt{|v_{\mathbf{b}, 3} (x,v)|^2 - g x_3 } +  v_{\mathbf{b}, 3} (x,v)  \Big\}
 	. 
 	\Ee
 	(ii) Recall the dynamic characteristics \eqref{ODE_F} and its self-consistent potentials $\Phi  \in C^1 (\bar{\O}) \cap  C^2 (\O)$ and $\Psi(t,\cdot) \in C^1 (\bar{\O}) \cap  C^2 (\O)$ in \eqref{eqtn:Dphi} and \eqref{Poisson_f}, respectively. Suppose the  condition \eqref{Bootstrap} holds.
 	
 	Then the backward/forward exit time \eqref{tb} is bounded above as, for all $(t,x,v) \in [0,T] \times  \bar\O \times \R^3$, 
 	\Be\label{est:tB}
 	\begin{split}
 		&\tB(t,x,v)   
 		\leq \frac{2}{g} \min \Big\{\sqrt{|v_3|^2 + g x_3} -  v_3,  \sqrt{|v_{\mathbf{B}, 3} (t,x,v)|^2 - g x_3 }+ v_{\mathbf{B}, 3} (t,x,v)  \Big\} ,\\
 		& \tB(t,x,v) + \tF (t,x,v) \leq \frac{4}{g} \sqrt{|v_3|^2 + g x_3}.
 	\end{split}
 	\Ee

 	
 \end{lemma}\vspace{-10pt}
 \begin{proof}We only prove the dynamical part \eqref{est:tB} as the steady part \eqref{est:tb^h} can be proved similarly. From the bootstrap assumption \eqref{Bootstrap}, the vertical acceleration is bounded from above as
 	\Be\label{upper_dotV_3}
 	\frac{d}{ds}    V_ { 3} (s;t, x,v)   \leq - {g } / {2}  .
 	\Ee
 	\vspace{-4pt}
 	Note that \Be
 	\begin{split}\label{exp:X_3}
 		X_3(s;t,x,v) 
 		= x_3 +  \int^s_t  \Big( v_3 + \int^\tau_{t} \frac{d V_ { 3} (\tau^\prime;t, x,v)}{d\tau^\prime} \dd \tau^\prime \Big)  \dd \tau \leq  x_3 -v_3 (t-s)     - \frac{g}{4}|t-s|^2.
 	\end{split}
 	\Ee
 	The zeros of the above quadratic form are $
 	\{-2 v_3 \pm 2 \sqrt{|v_3|^2 + g x_3}\}/{g}.$
 	Then, from the definition of $\tB$ at \eqref{tb}, we can prove that 
 	\Be\label{est1:tb}
 	\tB(t,x,v)   
 	\leq 2 \big( \sqrt{|v_3|^2 + g x_3} -  v_3 \big)/g.
 	\Ee

 	By expanding \eqref{exp:X_3} at $ t-\tB(t,x,v)$ and using \eqref{upper_dotV_3}, we derive that 
 	\Be\label{zero_quad2}
 	\begin{split}
 		&	X_3 (t-\tB(t,x,v);t,x,v)  \\
 		& 	= x_3+   \int^ {t-\tB(t,x,v)}_{t}  \Big( v_{\mathbf{B}, 3} (t,x,v)  + \int^\tau_{t-\tB(t,x,v)} 
 		\frac{dV_3 (\tau^\prime;t,x,v)}{d \tau^\prime}
 		\dd \tau^\prime \Big) \dd \tau\\
 		& \leq  x_3 -  v_{\mathbf{B}, 3} (t,x,v)  \tB    + \frac{g}{4} |  \tB |^2.
 	\end{split}
 	\Ee
 	The zeros of the above quadratic form (of $\tB$) are 
 	$	\left\{ 2v_{\mathbf{B}, 3} (t,x,v) \pm 2\sqrt{|v_{\mathbf{B}, 3} (t,x,v)|^2 - g x_3 } \right\} / {g}.$
 	Hence we conclude\vspace{-10pt}
 	\Be\label{est2:tb}
 	\tB(t,x,v)   
 	\leq \frac{2}{g}   \Big( \sqrt{|v_{\mathbf{B}, 3} (t,x,v)|^2 - g x_3 }+ v_{\mathbf{B}, 3} (t,x,v)  \Big).
 	\Ee
 	
 	Combining \eqref{est1:tb} and \eqref{est2:tb} together, we conclude that the first bound of \eqref{est:tB}. Following the same argument we can have the bound for $\tF(t,x,v)$. 
 	\hide \eqref{est:tb^h}.
 	For $v_3\geq 0$, \eqref{zero_quad} implies that

 	For $v_3 \leq 0$, 
 	\Be
 	\tb(t,x,v)   \leq  \frac{ 2 |v_3| + 2 \sqrt{|v_3|^2 + g x_3}}{g} \leq \frac{2}{g} (\sqrt{|v_3|^2 + g x_3} + |v_3|)
 	\Ee

 	Hence
 	\Be\notag
 	\begin{split}
 		0&=	X_{\pm, 3} (t- \tb^{ \pm} (t,x,v);t,x,v)  = x_3 -\int_{t- \tb^{  \pm}(t,x,v)}^t V_ {\pm, 3} (\tau; t, x,v) \dd \tau  \\
 		& = x_3 - \tb^{  \pm}(t,x,v) v_{\b,3}^{  \pm} (t,x,v)  + \int_{t- \tb^{  \pm}(t,x,v)}^t \int^\tau _{t- t_\b^{  \pm}(t,x,v)}   (-1)	\frac{d V_ {\pm, 3}(s ;t,x,v)}{d s} \dd  s \dd \tau \\
 		& \geq  x_3- \tb^{  \pm} (t,x,v) v_{\b,3}^{ \pm} (t,x,v)   +\frac{|\tb^{  \pm}(t,x,v)  |^2}{2} \frac{g }{2}.
 	\end{split}
 	\Ee
 	This implies that there exists $\tb^\pm (t,x,v)<\infty$ for any $(t,x,v) \in \R_+ \times \bar \O \times \R^3$. From the quadratic formula, 
 	\Be\notag
 	\tb^\pm (t,x,v) \leq \frac{ v_{\b, 3}^\pm (t,x,v) + \sqrt{ |v_{\b, 3}^\pm (t,x,v)| ^2 - 4\frac{g}{4} x_3 }}{2 \frac{g}{4}} \leq \frac{4}{g}v_{\b, 3}^\pm (t,x,v) .
 	\Ee
 	
 	Similarly, we have 
 	\Be\notag
 	\begin{split}
 		0&=	X_{\pm, 3} (t- \tb^{ \pm} (t,x,v);t,x,v) \\
 		& = x_3 +\int^{t- \tb^{  \pm}(t,x,v)}_t V_ {\pm, 3} (\tau; t, x,v) \dd \tau  \\
 		& 
 		= x_3 +\int^{t- \tb^{  \pm}(t,x,v)}_t 
 		\left(
 		v_3 + \int^\tau_t 	\frac{d V_ {\pm, 3}(\tau^\prime ;t,x,v)}{d\tau^\prime} \dd  \tau^\prime
 		\right)
 		\dd \tau
 		\\
 		& = x_3 - \tb^{  \pm}(t,x,v) v_3 + \int^{t- \tb^{  \pm}(t,x,v)}_t \int_t^\tau   	\frac{d V_ {\pm, 3}(\tau^\prime ;t,x,v)}{d\tau^\prime} \dd  \tau^\prime\dd \tau \\
 		& \leq  x_3- \tb^{  \pm} (t,x,v) v_3  -\frac{|\tb^{  \pm}(t,x,v)  |^2}{2} \frac{g }{2}.
 	\end{split}
 	\Ee
 	Moreover, from the quadratic formula, we have that 
 	\Be\notag
 	\tb^{  \pm} (t,x,v) \leq  \frac{  v_3  + \sqrt{| v_3 |^2 + 4  \frac{ g}{4} x_3 } }{2 \frac{g}{4}} \leq \frac{4}{g}\sqrt{| v_3 |^2 + g x_3 }  .
 	\Ee \unhide \end{proof}

 \subsection{Asymptotic Stability Criterion}As the main purpose of this paper, we establish a bootstrap machinery of starting with linear decay due to gravity effect to prove nonlinear decay. \hide Weight functions play an important role in the analysis. 
 \begin{definition}
 	For an arbitrary $\beta>0$, we define a weight function for the steady problem \eqref{VP_h}
 	\Be
 	w (x,v)   =	w _\beta (x,v) = e^{ \beta \big( |v|^2  + 2 \Phi(x) + 2 g   x_3\big)}.\label{w^h}
 	\Ee
 \end{definition}\unhide


 \begin{theorem}[Asymptotic Stability Criterion]\label{theo:AS}
 	Suppose $(h, \Phi)$ solves \eqref{VP_h}-\eqref{eqtn:Dphi}, and $(f, \varrho,   \Psi)$ solves \eqref{eqtn:f}-\eqref{Poisson_f} globally-in-time in the sense of Definition \ref{def:mild}. Suppose $\nabla_x \cdot b \in L^\infty_{{loc}} (\R_+ \times \O)$.
 	
 	Assume that the following three conditions hold, for $g, \beta>0$ 
 	\begin{align}
 		\| w_\beta \nabla_v h \|_{L^\infty(\O)}	 &< \infty ,\label{condition:Dvh} \\
 		\sup_{t \geq 0} \|
 		e^{ \frac{\beta}{2} (|v|^2 + g x_3)}
 		f(t) \|_{L^\infty(\O \times \R^3)}
 		&\leq \frac{ ( \ln 2 )^{\frac{1}{2}}g^{\frac{1}{2}} \beta^{\frac{3}{2}}}{64 \pi  (1+ \frac{1}{\beta g})},\label{Bootstrap_f} \\
 		\| \nabla_x \Phi \| _{L^\infty(\O)}   +	\sup_{t \geq 0 }  \|  \nabla_x \Psi (t) \|_{L^\infty(\O)}   &\leq \frac{g}{2} .\label{Bootstrap} 
 	\end{align}
 	


 	Then there exists a computable number $\lambda_\infty= \lambda_\infty ( g, \beta,\| w_\beta \nabla_v h \|_{L^\infty(\O)}	 )>0$ such that $(f(t), \varrho(t))$ decays exponentially fast as $t \rightarrow \infty$: 
 	\begin{align}
 		\sup_{ t \geq 0 } e^{ \lambda_\infty t} \| \varrho (t)\|_{L^\infty(\O)} &\lesssim 
 		\|\w_{\beta, 0 }   f_{0  } \|_{L^\infty (\O \times \R^3)} ,\label{decay:varrho}\\
 		\sup_{ t \geq 0 }e^{\lambda_\infty t}	\| e^{ \frac{\beta}{8}  (|v|^2+  g x_3)} f(t) \|_{L^\infty (\O \times \R^3)}  
 		&\lesssim 
 		\left(
 		1+ 
 		\| w_\beta \nabla_v h \|_{L^\infty (\O  )}
 		\right)
 		\| \w_{\beta,0} f_0 \|_{L^\infty (\O \times \R^3)}  .\label{decay_f}
 	\end{align} 
 	
 \end{theorem}

 \begin{remark}
 	An exponent, which we will derive in \eqref{lambda_infty}, depends on $g$ and $\beta$ as $\lambda_\infty \sim g^2 \beta$ roughly. This is somewhat intuitive: larger $\beta$ implies lesser particles of high momentum while a large gravity $g$ would trap the particles rapidly. 
 \end{remark}

 \subsection{Construction of a steady solution}{\color{black}
 	\hide	To carry out the idea of stabilizing effect in Section \ref{sec:SE}, it is important to construct steady solutions satisfying the same in-flow boundary condition \eqref{bdry:F} so that the perturbation exactly satisfies the zero in-flow boundary condition. Despite some previous construction in bounded domains (\cite{Rein_bdry, GR}), there seems no result in the half-space, which relates to the solar wind model (e.g. corona-heating problem). In general, uniqueness theorem plays an important role in the asymptotic stability. We prove the uniqueness of solution to the nonlinear problem by establishing the regularity theorem. Moreover, in a proof of asymptotic stability, it is crucial to establish some Gaussian upper bound of derivatives of the steady solutions (see an explanation across \eqref{intro:Dvh} in the introduction). Generally speaking regularity estimate is difficult, as the derivatives blow up at the grazing set $\gamma_0 = \{x_3=0 \ \text{and} \ v_3 =0\}$. To overcome this difficulty, we introduce a kinetic distance (\cite{CK_VM, ChK}).\unhide
 	
 	To carry out the idea of stabilizing effect in Section \ref{sec:SE}, it is important to construct steady solutions that satisfy the same in-flow boundary condition \eqref{bdry:F} as the perturbation, so that the zero in-flow boundary condition is exactly satisfied. Although some previous constructions have been made in bounded domains (\cite{Rein_bdry, GR}), there seems to be no result in the half-space, which is relevant to the solar wind model (e.g., corona-heating problem). In general, the uniqueness theorem plays an important role in asymptotic stability. We prove the uniqueness of the solution to the nonlinear problem by establishing the regularity theorem. Moreover, in a proof of asymptotic stability, it is crucial to establish some Gaussian upper bound of the derivatives of the steady solutions (see an explanation across \eqref{intro:Dvh} in the introduction). Generally speaking, the regularity estimate is difficult, as the derivatives blow up at the grazing set $\gamma_0 = \{x_3=0 \ \text{and} \ v_3 =0\}$. 

 	\hide
 	We specify a definition of weak solutions to the steady problem that we are about to construct. 
 	\begin{definition}\label{weak_sol}We say $(h, \nabla_x  \Phi) \in L^2_{loc}(\O \times \R^3) \cap L^2_{loc}(\p\O \times \R^3; \dd \gamma) \times L^2_{loc}(\O \times \R^3)$ is a weak solution if all terms below are bounded and they satisfy the following equality, for any $\psi \in C^\infty_c (\bar \O \times \R^3)$, 
 		\Be\label{weak_form}
 		\begin{split}
 			&	\iint_{\O \times \R^3}	 	h(x,v) v\cdot\nabla_x \psi(x,v)  \dd v \dd x \\
 			&
 			- \iint_{\O \times \R^3}	  h(x,v)  \nabla_x (\Phi (x)  + g x_3) \cdot \nabla_v  \psi (x,v)  \dd v \dd x \\
 			&
 			=\int_{ \p\O \times \{v_3<0\}} h (x,v) \psi(x,v) \dd \gamma  - \int_{ \p\O \times \{v_3>0\}} G(x,v) \psi(x,v)  \dd \gamma;
 		\end{split}
 		\Ee
 		and $(h, \rho) \in L^2_{loc}(\O \times \R^3) \times L^2_{loc}(\O  )$ satisfies \eqref{def:rho}; and $(\rho, \Phi) \in L^2_{loc}(\O  ) \times W^{1,2}_{loc} (\O)$ is a weak solution of \eqref{eqtn:Dphi}. Here, the boundary measure is denoted by $\dd \gamma := |n(x) \cdot v| \dd S_x \dd v$.
 	\end{definition}
 	\unhide

 	\hide
 	\begin{theorem}[Construction of Steady Solutions]\label{theo:CS}Given $G  (x,v) \geq 0$
 		\Be
 		\| w(|v|^2) G_\pm  \|_\infty \lesssim gm_\pm,
 		\Ee
 		there exists at least one solution $h(x,v ) \geq 0$ to \eqref{VP_h}-\eqref{eqtn:Dphi} in the sense of Definition \ref{weak_sol}. 
 		
 		Moreover, it satisfies 
 		\Be
 		\| \p_{x_3} \phi_h \|_{L^\infty (\O)} \leq \frac{g  \min_{i=\pm} m_i}{2}
 		\Ee
 		Moreover, 
 	\end{theorem}\unhide
 	
 	\begin{theorem}[Construction of Steady Solutions]\label{theo:CS}
 		Suppose the inflow boundary data satisfy 
 		$$\| e^{\beta |v|^2} G \|_{L^\infty (\gamma_-)}+ \|   e^{{\tilde \beta } |v|^2}   \nabla_{x_\parallel, v} G(x,v)\|_{L^\infty (\gamma_-)}< \infty,$$
 		for $\beta ,\tilde{\beta}>0$. For $g>0$, assume that $\beta> \tilde{\beta}> \max\{1,\frac{4}{g}\}$. We also assume that  
 		\Be\label{condition:beta}
 		\mathfrak{C} \frac{\pi^{3/2}}{\beta^{3/2}} \Big(1 + \frac{1}{\beta g}\Big)
 		\| e^{\beta |v|^2} G \|_{L^\infty (\gamma_-)}
 		\leq \frac{g}{2},
 		\Ee
 		\vspace{-10pt}
 		\Be\label{condition:G}
 		\frac{\mathfrak{C}_1}{\beta^{3/2}} 
 		\| e^{\beta |v|^2} G \|_{L^\infty (\gamma_-)} \bigg\{
 		\frac{1}{g \beta }  +
 		\log \bigg(
 		e+ 	\frac{1}{\tilde{\beta}}\Big(1+ \frac{1}{\tilde{\beta}^{1/2}}+ \frac{1}{ g \tilde{\beta}}\Big) 
 		\| e^{\tilde \beta |v|^2} \nabla_{x_\parallel, v}  G \|_{L^\infty (\gamma_-)}
 		\bigg)\bigg\} \leq \frac{\tilde{\beta}g^2}{16}.
 		\Ee
 		Here, $\mathfrak{C}, \mathfrak{C}_1>0$ are the computable constants, which appeared in \eqref{est:phi_C1} and \eqref{est:phi_C2}. For sufficiently small $\e_1>0$, suppose the following bound also hold:
 		\Be\label{cond:stability}
 		\frac{   \mathfrak{C}
 		}{g  \tilde{\beta}^{2} } 
 		\left\{ 1+ \frac{1}{g \tilde\beta }	\right\} 	\left(1+	\frac{1}{g {\tilde\beta}^{1/2}} 
 		\right)
 		\|   e^{\tilde \beta |v|^2}  \nabla_{x_\parallel,v} G    \|_{L^\infty(\gamma_-)} \leq \e_1.
 		\Ee 
 		\hide
 		\Be\label{cond:stability}
 		\frac{ 2^{7/2} \pi^{3/2} \mathfrak{C}
 		}{g  \tilde{\beta}^{2} } 
 		\left\{ 1+ \frac{8}{\tilde\beta g}	\right\} 	\left(1+	\frac{1}{g {\tilde\beta}^{1/2}} 
 		\right)
 		\|   e^{\tilde \beta |v|^2}  \nabla_{x_\parallel,v} G    \|_{L^\infty(\gamma_-)} \leq \e_1.
 		\Ee

 		and $g, \beta>0$ satisfy

 		Suppose all assumptions of Theorem \ref{theo:CS} hold. Moreover we assume that for $\tilde \beta >0$. Furthermore, let $\beta,\tilde \beta >0$ satisfy  
 		Assume that

 		\hide	\Be\label{condition:beta}
 		\beta \geq 
 		(8 \pi^{3/2} \mathfrak{C})^{2/5} 
 		\| e^{\beta |v|^2}  G \|_{L^\infty (\gamma_-)}^{2/5}
 		g^{-4/5}.
 		\Ee
 		\unhide
 		Here, a constant $\mathfrak{C}>0$ appears in \eqref{est:nabla_phi}, which is related to a Green function of the Laplacian in our domain $\O= \T^2 \times (0,\infty)$. 
 		
 		\unhide	
 		
 		Then there exists a unique strong solution $(h, \rho, \Phi)$ to \eqref{VP_h}-\eqref{eqtn:Dphi}. Moreover, we have 
 		\begin{align}
 			\|    h    \|_{L^\infty ( \bar \O \times \R^3)   } &\leq \|    G \|_{L^\infty (\gamma_-)},
 			\label{Uest:h}
 			\\
 			\| w_\beta  h    \|_{L^\infty (  \bar \O \times \R^3)} &\leq \|  e^{\beta |v|^2} G \|_{L^\infty (\gamma_-)},
 			\label{Uest:wh}
 			\\
 			\|  e^{  \beta g x_3 } \rho  \|_{L^\infty (\bar \O)} &
 			\leq   \frac{\pi ^{3/2}}{   \beta^{ 3/2} }	\| w_\beta   h    \|_{L^\infty ( \bar  \O \times \R^3)}
 			\leq \frac{\pi ^{3/2}}{   \beta^{ 3/2} } \|  e^{\beta |v|^2} G \|_{L^\infty(\gamma_-)},
 			\label{Uest:rho}
 			\\
 			\| \nabla_x \Phi \|_{L^\infty (\bar{\O})} &\leq  \frac{g}{2} . \label{Uest:DPhi}
 		\end{align}
 		Furthermore, 
 	\begin{align}
 		e^{  \frac{\tilde \beta g}{2} x_3} | \p_{x_i} \rho   (x) |
 		& 	\lesssim	_{g  , \tilde \beta }
 		\Big(
 		1+  	\delta_{i3}    \mathbf{1}_{|x_3| \leq 1}  |\ln    ( |x_3|^2 + g x_3 )|   
 		\Big) 
 		\|  e^{\tilde \beta |v|^2}\nabla_{x_\parallel, v} G \|_{L^\infty ({\gamma_-})}   ,\label{theo:rho_x} \\
 		\| \nabla_x^2 \Phi  \|_{L^\infty (\bar \O)}
 		&\lesssim _{g, \tilde \beta }  \Big(
 		1+ \ln \big(e + \|  e^{\tilde \beta |v|^2}  \nabla_{x_\parallel, v}  G \|_{L^\infty (\gamma_-)}  \big)
 		\Big) \| e^{ \beta |v|^2}  G \|_{L^\infty (\gamma_-)} , 
 		\label{theo:phi_C2}
 		\\ 
 		\| e^{ \frac{\tilde \beta}{2}|v|^2} e^{  \frac{\tilde \beta g}{2} x_3 }	\nabla_v h  \|_{L^\infty (\O \times \R^3)} &  \lesssim_{g  , \tilde \beta }
 		\big(1+  \| \nabla_x \Phi \|_\infty \big)  \|   e^{\tilde \beta |v|^2}  \nabla_{x_\parallel,v} G    \|_{L^\infty(\gamma_-)}
 		,	 \label{theo:hk_v}\\ 
 		e^{ \frac{\tilde \beta}{2}|v|^2} e^{  \frac{\tilde \beta g}{2} x_3} |	\p_{x_i}h (x,v) | &	\lesssim _{g  , \tilde \beta }   \Big(1+ \frac{\delta_{i3}}{\alpha(x,v)}\Big) 
 		\| e^{\tilde \beta |v|^2} \nabla_{x_\parallel,v}  G \|_{L^\infty(\gamma_-)}
 		.\label{theo:hk_x}
 	\end{align} 
 
 	\hide	\Be\label{Uest:DPhi}
 	\| \nabla_x \Phi \|_{L^\infty( \bar \O)} \leq \frac{g}{2}.
 	\Ee
 	\unhide
 	Here, a kinetic distance for a steady problem is defined as 
 	\Be\label{alpha_k}
 	\alpha
 	(x,v) =\sqrt{ |v_3|^2 +    |x_3|^2  + 2\p_{x_3} \Phi
 		(x_\parallel , 0) x_3 + 2g 
 		x_3 }.
 	\Ee
 	In particular, $\alpha (x,v) =  |v_3|$ when $x \in \p\O$ (i.e. $x_3=0$). 
\end{theorem}


\hide

\medskip

\noindent\textbf{B. Regularity and Uniqueness.}

\subsection{Regularity and Uniqueness}

\begin{theorem}[Regularity Estimate]\label{theo:RS}
{\color{black}Suppose all assumptions of Theorem \ref{theo:CS} hold.} Moreover we assume that $\|   e^{{\tilde \beta } |v|^2}   \nabla_{x_\parallel, v} G(x,v)\|_{L^\infty (\gamma_-)}< \infty$ for $\tilde \beta >0$. Furthermore, let $\beta,\tilde \beta >0$ satisfy  
\Be\label{condition:G}
\begin{split}
	& \| e^{ \beta |v|^2}  G \|_{L^\infty (\gamma_-)}
	\\ &\times  \log \bigg(
	e+ \Big( \frac{ 1+ g^{-1} \| \nabla_x\Phi \|_\infty + g^{-2} \tilde{\beta}^{-1}}{g \tilde{\beta}^2} 
	+  
	\frac{1+ \tilde{\beta}^{1/2} 
		\| \nabla_x \Phi \|_\infty
	}{\tilde{\beta}^{3/2}} 
	\Big)\|  e^{\tilde \beta |v|^2}  \nabla_{x_\parallel, v}  G \|_{L^\infty (\gamma_-)} 
	\bigg) \\
	& \leq  \frac{1}{16} g^2 \tilde \beta  \beta^{3/2}.
\end{split}	\Ee 



\hide

\Be
\frac{16}{g^2} \| \nabla_x^2 \Phi  \|_\infty
\leq 
\tilde \beta   < \beta.
\Ee
\unhide
{\color{black}Then, a solution $h,\rho,$ and $\nabla_x \Phi$ constructed in Theorem \ref{theo:CS} are locally Lipschitz continuous and the following bounds hold: }
\begin{align}
	e^{  \frac{\tilde \beta g}{2} x_3} | \p_{x_i} \rho   (x) |
	& 	\lesssim	_{g  , \tilde \beta }
	\Big(
	1+  	\delta_{i3}    \mathbf{1}_{|x_3| \leq 1}  |\ln    ( |x_3|^2 + g x_3 )|   
	\Big) 
	\|  e^{\tilde \beta |v|^2}\nabla_{x_\parallel, v} G \|_{L^\infty ({\gamma_-})}   ,\label{theo:rho_x} \\
	\| \nabla_x^2 \Phi  \|_{L^\infty (\bar \O)}
	&\lesssim _{g, \tilde \beta }  \Big(
	1+ \ln \big(e + \|  e^{\tilde \beta |v|^2}  \nabla_{x_\parallel, v}  G \|_{L^\infty (\gamma_-)}  \big)
	\Big) \| e^{ \beta |v|^2}  G \|_{L^\infty (\gamma_-)} , 
	\label{theo:phi_C2}
	\\ 
	\| e^{ \frac{\tilde \beta}{2}|v|^2} e^{  \frac{\tilde \beta g}{2} x_3 }	\nabla_v h  \|_{L^\infty (\O \times \R^3)} &  \lesssim_{g  , \tilde \beta }
	\big(1+  \| \nabla_x \Phi \|_\infty \big)  \|   e^{\tilde \beta |v|^2}  \nabla_{x_\parallel,v} G    \|_{L^\infty(\gamma_-)}
	,	 \label{theo:hk_v}\\ 
	e^{ \frac{\tilde \beta}{2}|v|^2} e^{  \frac{\tilde \beta g}{2} x_3} |	\p_{x_i}h (x,v) | &	\lesssim _{g  , \tilde \beta }   \Big(1+ \frac{\delta_{i3}}{\alpha(x,v)}\Big) 
	\| e^{\tilde \beta |v|^2} \nabla_{x_\parallel,v}  G \|_{L^\infty(\gamma_-)}
	.\label{theo:hk_x}
\end{align} 

\end{theorem}
We give a proof at the end of Section \ref{sec:RS}. 

\hide
\begin{corollary}Assume \eqref{condition:G} holds. Then a solution $(h, \rho,  \Phi)$ constructed in Theorem \ref{theo:CS} satisfies \eqref{theo:rho_x}, \eqref{theo:phi_C2}, \eqref{theo:hk_v}, and \eqref{theo:hk_x}.
\end{corollary}
\unhide
\unhide

\begin{remark}
An exponential decay-in-$(x,v)$ result of \eqref{theo:hk_v} is crucially important in our later proof of an asymptotic stability of a dynamical perturbation. 
\end{remark}



\hide

{\color{black}
\begin{theorem}\label{theo:US1}Suppose all assumptions of Theorem \ref{theo:CS} and Theorem \ref{theo:RS} hold. Then the solution $(h, \rho,   \Phi)$ constructed in Theorem \ref{theo:CS} is unique.
\end{theorem}}
We prove this theorem in Section \ref{sec:US}.

\unhide

\hide
\medskip

\noindent\textbf{C. Asymptotic Stability.} 

\unhide



\hide

\Be\label{form:f}
\begin{split}
\int_{\R^3}|	f  (t,x,v)|   \dd v
&\leq \mathcal{I} (t,x,v) + \mathcal{N}  (t,x,v), 
\end{split}
\Ee where 
\begin{align} 
\mathcal{I} (t,x,v) 	& :=   \mathbf{1}_{t <  \tB (t,x,v)} |	f _{  0} ( \Zz   (0;t,x,v) )| ,\label{form:I} \\
\mathcal{N} (t,x,v) 	&  := \int^t_
{\max\{0, t- \tB (t,x,v) \}} 
|	\nabla	\Psi
(s, \X (s;t,x,v)) || \nabla_v h 
( \Zz(s;t,x,v) )|
\dd s.\label{form:N}
\end{align}


\medskip

\newpage

\noindent\textbf{D. Dynamic Solutions: Construction, Regularity, Stability, and Uniqueness.}

\unhide

\subsection{
Construction of a global-in-time dynamical solution and Asymptotic stability
}  \hide For data, we set ($\beta, \tilde{\beta}>0$)
\begin{align}
M _{\beta} (F_0, G )&:=  
\| \w_{\beta, 0 } F_{0} \|_{L^\infty (\bar\O \times \R^3)}
+  \| e^{\beta |v|^2}
G  \|_{L^\infty  (\gamma_-)} ,\label{set:M} \\
L_{\tilde{}}(F_0, G )&:=  \| \w_{\tilde \beta, 0}   \nabla_{x,v} F_0  \|_{L^\infty (\bar\O \times \R^3)}+ \|  e^{\tilde{\beta} | v|^2}   \nabla_{x_\parallel,v} G   \|_{L^\infty (\gamma_-)}  .\label{set:L}
\end{align} \unhide
\begin{theorem}[Construction of Dynamic Solutions]\label{theo:CD}
Assume a compatibility condition:
\Be\label{CC}
F_0 (x,v) = G(x,v) \   \ \text{on} \  \ (x,v) \in \gamma_-.
\Ee
For $\beta ,\tilde{\beta}, g>0$, assume that $\beta> \tilde{\beta}> \max\{1,\frac{4}{g}\}$. \hide

For $\beta, g>0$, suppose 
\Be\label{choice:g}
M\leq  \frac{\sqrt{\ln 2}}{2^{17/2}  \pi  } g^{3/2} \beta^{5/2}.
\Ee
\hide	Let
\Be\label{choice:M}
M = 4\max\big\{  \| w_\beta  h \|_{L^\infty (\O)}, 
\| \w_{\beta, 0 } F_{0} \|_{L^\infty (\O \times \R^3)}
+  \| e^{\beta |v|^2}
G  \|_{L^\infty  (\O \times \R^3)} \big\}.
\Ee\unhide
In addition, suppose all assumptions in Theorem \ref{theo:CS} hold. \unhide
Suppose $\e>0$ is sufficiently small such that 
\Be \label{bootstrapC1_dyn:thrm}
\frac{1}{\beta^{3/2}} 
\big\{   \| \w_{\beta, 0 } F_{0} \|_{L^\infty (\O \times \R^3)}
+  \| 
e^{\beta|v|^2}
G  \|_{L^\infty (\gamma_-)}\big\}  \leq \e g 
,  
\Ee
\Be
\label{bootstrapC2_dyn:thrm} 
\frac{1}{\tilde{\beta}^3}	\| \w_{\tilde \beta, 0}   \nabla_{x,v} F_0  \|_{L^\infty (\O \times \R^3)} 
+
\frac{1}{\tilde{\beta}^{5/2}} 
\|  e^{\tilde{\beta} | v|^2}   \nabla_{x_\parallel,v} G   \|_{L^\infty (\gamma_-)}  	\leq   \e {g^{1/2}   }  ,
\Ee
\Be	\label{bootstrapC3_dyn:thrm} 
\begin{split}
&	\frac{1
}{\beta^{2}} 
\big\{   \| \w_{\beta, 0 } F_{0} \|_{L^\infty (\O \times \R^3)}
+  \| 
e^{\beta|v|^2}
G  \|_{L^\infty (\gamma_-)}\big\}  \\
& \times 
\log \bigg(
e+ 
\frac{1}{\tilde{\beta}^{3/2}} \| \w_{\tilde \beta, 0}   \nabla_{x,v} F_0  \|_{L^\infty (\O \times \R^3)}
+	\frac{1}{\tilde{\beta}}
\| e^{\tilde \beta |v|^2} \nabla_{x_\parallel, v}  G \|_{L^\infty (\gamma_-)}
\bigg)
\leq \e  g  .
\end{split}
\Ee 
Then there exists a unique global-in-time strong solution $(f, \varrho, \Psi)$ to \eqref{eqtn:f}-\eqref{Poisson_f}. Moreover, for all $(t,x,v) \in \R_+ \times \bar{\O} \times \R^3$,
\Be\label{Uest:wf:dyn}
\begin{split} 
e^{  \frac{
		\beta}{2} ( |v|^2 + g x_3 ) }
|	f  (t,x,v) |
+\frac{1}{\beta^{ 3/2}}     e^{  \frac{
		\beta}{2} g x_3}
|	\varrho    (t,x) | 
\lesssim 
\hide	e^{ \frac{
		C	\beta}{ (\tilde{\beta}g)^{1/2}}
} \unhide
\| \w_{\beta, 0 } F_{0} \|_{L^\infty (\O \times \R^3)}
+  \| 
e^{\beta|v|^2}
G  \|_{L^\infty (\gamma_-)}
,
\end{split}	\Ee
\Be
\begin{split}\label{est_final:F_v:dyn1}
e^{\frac{\tilde \beta}{4} (|v|^2 + g x_3)} |  \nabla_v F(t,x,v) |   
\lesssim  
\| \w_{\tilde \beta, 0}   \nabla_{x,v} F_0  \|_{L^\infty (\O \times \R^3)}  
+\Big(1+ \frac{1}{g\tilde \beta ^{1/2}} 
\Big)  
\|  e^{\tilde{\beta} | v|^2}   \nabla_{x_\parallel,v} G   \|_{L^\infty (\gamma_-)},
\end{split}\Ee 
\Be
\label{est_final:F_x:dyn} 
\begin{split}
&  e^{\frac{\tilde \beta}{4} (|v|^2 + g x_3)} |\p_{x_i} F(t,x,v)|  
\lesssim 
\| \w_{\tilde \beta, 0}   \nabla_{x,v} F_0  \|_{L^\infty (\O \times \R^3)} \\
& \ \  \ \ \ \ \ \ \ \ \ \ \ \ \ \ \ \ \ \ \ \  \  \ 
+ \left[ 
\left(1+ \frac{1}{g \tilde{\beta}^{1/2}}\right)
+ \left(1+ \frac{1}{\tilde{\beta}^{1/2}}\right) \frac{\delta_{i3}}{\alpha_{F} (t,x,v)}
\right]
\|   e^{\tilde \beta |v|^2}  \nabla_{x_\parallel,v} G    \|_{L^\infty(\gamma_-)} ,
\end{split}
\Ee 
\Be	\label{est_final:phi_F1}
|\nabla_x \phi_{F^{\ell+1}}(t,x) |  \lesssim   \frac{1}{\beta^{3/2}} \Big(1 + \frac{1}{\beta g}\Big)
\big\{   \| \w_{\beta, 0 } F_{0} \|_{L^\infty (\O \times \R^3)}
+  \| 
e^{\beta|v|^2}
G  \|_{L^\infty (\gamma_-)}\big\}   \leq \frac{g}{2}
,  
\Ee
and 
{\small\Be
\begin{split}
	\label{est_final:phi_F}
	&   \sup_{t \geq 0 }\|\nabla_x^2  \phi_{F }(t) \|_{L^\infty(\O)}
	\leq 
	\frac{\mathfrak{C}_1}{\beta^{3/2}} 
	\big\{   \| \w_{\beta, 0 } F_{0} \|_{L^\infty (\O \times \R^3)}
	+  \| 
	e^{\beta|v|^2}
	G  \|_{L^\infty (\gamma_-)}\big\}  \\
	& \times \bigg\{
	\frac{1}{g \beta }  +
	\log \bigg(
	e+ 
	\frac{1}{\tilde{\beta}^{3/2}} \| \w_{\tilde \beta, 0}   \nabla_{x,v} F_0  \|_{L^\infty (\O \times \R^3)}
	+	\frac{1}{\tilde{\beta}}\Big(1+ \frac{1}{\tilde{\beta}^{1/2}}+ \frac{1}{ g \tilde{\beta}}\Big) 
	\| e^{\tilde \beta |v|^2} \nabla_{x_\parallel, v}  G \|_{L^\infty (\gamma_-)}
	\bigg)\bigg\}.
\end{split}	\Ee}\hide

\hide\Be
\sup_{t\geq 0}	\|	e^{  \frac{
	\beta}{2}|v|^2}
e^{  \frac{
	\beta}{2} g x_3} f  (t)\|_{L^\infty ( \O \times \R^3)}  
\leq 4M
,\label{Uest:wf:dyn}
\Ee\unhide
\begin{align} 
\sup_{ t \geq 0}	\| \nabla_x \Psi (t ) \|_{L^\infty(\O  )}  + \| \nabla_x \Phi \|_{\infty}
&\leq   \frac{g }{2},\label{Bootstrap:dyn}\\
e^{\frac{64
		\beta}{g}
	\sup_{ t \geq 0}	\| \Delta_0^{-1} (\nabla\cdot b )(t, \cdot ) \|_{L^\infty (  \O)}^2
}  &\leq 2.\label{est:e^b:dyn} 
\end{align}\hide
Here, $b(t,x)$ is defined in \eqref{def:flux}.

Then the solution constructed in Theorem \ref{theo:CD} satisfies that, for any $T>0$ and $0 \leq t \leq T$, 
\Be
\begin{split}\label{est_final:F_v:dyn1}
\| e^{\frac{\tilde \beta}{4} (|v|^2 + g x_3)} \nabla_v F (t)\|_{L^\infty(\O \times \R^3)}  
\lesssim  
\| \w_{\tilde \beta, 0}   \nabla_{x,v} F_0  \|_{L^\infty (\O \times \R^3)} 
+  \|  e^{\tilde{\beta} | v|^2}   \nabla_{x_\parallel,v} G   \|_{L^\infty (\gamma_-)},
\end{split}\Ee 
\Be
\label{est_final:F_x:dyn} 
\begin{split}
e^{\frac{\tilde \beta}{4} (|v|^2 + g x_3)} |\nabla_x F (t,x,v)|  
& \lesssim  
\| \w_{\tilde \beta, 0}   \nabla_{x,v} F_0  \|_{L^\infty (\O \times \R^3)} \\
&
+ \Big( 1+\frac{1}{\alpha_F (t,x,v)} \Big) \|  e^{\tilde{\beta} | v|^2}   \nabla_{x_\parallel,v} G   \|_{L^\infty (\gamma_-)} ,
\end{split}
\Ee  
\Be
\begin{split}\label{est_final:phi_F}
\| \nabla_x^2 \phi_F (t) \|_{L^\infty (\bar \O)} + 
\| \p_t  \nabla_x  \phi_F (t) \|_{L^\infty (\bar \O)} 
\lesssim 	\| \w_{\tilde \beta, 0}   \nabla_{x,v} F_0  \|_{L^\infty (\O \times \R^3)} 
+  \|  e^{\tilde{\beta} | v|^2}   \nabla_{x_\parallel,v} G   \|_{L^\infty (\gamma_-)}.
\end{split}
\Ee\unhide
\unhide
Here, $\alpha_F(t,x,v)= \big[ |v_3|^2 +    |x_3|^2  + 2\p_{x_3} \phi_F
(t,x_\parallel , 0) x_3 + 2g 
x_3 \big]^{1/2} $ is a kinetic distance for a dynamical problem.

\end{theorem}
As a direct consequence of Theorem \ref{theo:AS}, Theorem \ref{theo:CS} and Theorem \ref{theo:CD}, we conclude the following dynamical asymptotic stability.
\begin{theorem}[Asymptotic Stability]\label{cor_stability}
Suppose all conditions in Theorem \ref{theo:CS} and Theorem \ref{theo:CD} hold. Then the dynamical solution $(F(t), \phi_F(t)) \rightarrow (h, \Phi)$ converges exponentially fast to the steady solution of Theorem \ref{theo:CD} as in \eqref{decay:varrho}-\eqref{decay_f}.
\end{theorem}
\hide
Next we establish the uniqueness theory, by proving the nonlinear stability estimate (Theorem \ref{theo:UD}). For that we prove a regularity estimate with regular data: for $\tilde \beta \geq 2 \beta >0$,
\begin{align}
L&:= \max\big\{
\| \w_{\tilde \beta, 0}   \nabla_{x,v} F_0  \|_{L^\infty (\O \times \R^3)} , 
\|  e^{\tilde{\beta} | v|^2}   \nabla_{x_\parallel,v} G   \|_{L^\infty (\gamma_-)} 
\big\}.\label{set:L}
\end{align}

\begin{theorem}[Regularity Estimate]
Assume all the conditions in Theorem \ref{theo:CD}.

Choose $g, \beta ,\tilde \beta$ large enough so that 
\Be\label{choice_ML}
M \ll 
\min \{g \tilde \beta^{1/2} \beta^{3/2}, g \beta^{5/2}\} \  \ \text{and} \ \ 
L \ll  \min \{\tilde\beta^{3/2} , g^2 \tilde \beta^{5/2}\}.
\Ee

\end{theorem}
We will prove the theorem at the end of Section \ref{sec:RD}.

Now we move on the stability. The difference of solutions $(F_1, \nabla_x \phi_{F_1})$ and $(F_2, \nabla_x \phi_{F_2})$ to  \eqref{VP_F}, \eqref{Poisson_F}, \eqref{bdry:F}, \eqref{Dbc:F}  actually solves
\Be\label{eqtn:F1-F2}
\begin{split}
&	\big[
\p_t + v\cdot \nabla_x - \nabla_x (\phi_{F_1} + g x_3) \cdot \nabla_v 
\big] (\w_1 (F_1-F_2))\\
&= -2 \beta  \Delta_0^{-1} (\nabla_x \cdot b_1) \w_1 (F_1-F_2)
+ \nabla_x (\phi_{F_1} - \phi_{F_2}) \cdot \w_1 \nabla_v F_2,
\end{split}
\Ee
where we used the continuity equation \eqref{cont_eqtn} for $\varrho_1(t,x) = \int_{\R^3} (F_1(t,x,v) - h(x,v) ) \dd v$, $b_1 (t,x) = \int_{\R^3} v(F_1(t,x,v) - h(x,v)) \dd v $, and the evolution of total energy:
\Be\label{dEC1}
\begin{split}
& 		\big[
\p_t + v\cdot \nabla_x - \nabla_x (\phi_{F_1} + g x_3) \cdot \nabla_v 
\big]  \left(
|v|^2 + 2 \phi_{F_1} (t,x) + 2 g x_3
\right)   
= 
2 \p_t \phi_{F_1} (t,x) \\
& = 2  \Delta_0^{-1} \p_t \varrho_1 (t,x)  = -2  \Delta_0^{-1} ( \nabla_x \cdot b_1)  (t,x).
\end{split}
\Ee
We will give a full proof at the end of Section \ref{Sec:SU}.

Equipped with \eqref{Bootstrap:dyn} and Theorem \ref{theo:RD}, it is straightforward to see the bound of the second line of \eqref{est:F1-2}.
\begin{corollary}[Uniqueness Theorem] Assume all conditions in Theorem \ref{theo:CD} and Theorem \ref{theo:RD}. Then the solution constructed in Theorem \ref{theo:CD} is unique in a class of Theorem \ref{theo:RD}. 
\end{corollary}
\unhide

\hide
Note that 
\Be
[ v\cdot \nabla_x - \nabla_x (\phi_h + g m_{\pm } x_3 ) \cdot \nabla_v] (w^h_{\pm} h_\pm) =0.
\Ee
If $t_{\b,\pm}(x,v)<\infty$ for each $(x,v) \in \O \times \R^3$ (no need to have a uniform upper bound), then 
\Be
w_\pm h_\pm (x,v) = w_\pm G_\pm (x_{\mathbf{b},\pm}, v_{\mathbf{b},\pm})
\Ee
As long as $\| w_\pm G_{\pm} \|_{L^\infty (\gamma_-)}<\infty $, we can derive that 
\Be
\| w_\pm h_\pm \|_{L^\infty (\O \times \R^3)} \leq \| w_\pm G_{\pm} \|_{L^\infty (\gamma_-)}.
\Ee\unhide

\hide

\bigskip

\noindent\textbf{E. Key Stability mechanism of Gravity and Gaussian bound\footnote{
The notations used here (including $f,h,\varrho, \tB,\X, \V $) are all temporal. These temporal definitions should be nullified outside the introduction.}.}

\medskip

We shall illustrate how a strong gravity may stabilize the Vlasov system for a certain class of steady solutions. For simplicity, we pick a simplified toy model of \eqref{eqtn:f}: for given $E=E(t,x)$, 
\Be\label{simple:h}
\begin{split} 
\p_t f + v\cdot \nabla_x f - g \p_{v_3} f = E\cdot \nabla_v h  \ \ \text{in} \ \R_+ \times  \O \times \R^3    \ \ &\text{in} \ \R_+ \times  \O \times \R^3,
\end{split}
\Ee
with the boundary condition \eqref{bdry:f}.

The characteristics of \eqref{simple:h} is explicitly given by $\dot{\X}= \V$ and $\dot{\V} = - g \mathbf{e}_3$; or $\X_i(s;t,x,v) = x_i- (t-s) v_i - g \delta_{i3} \frac{(t-s)^2}{2} $ and $\V_i(s;t,x,v) = v_i+ g\delta_{i3}  (t-s)$. The unique non-negative time lapse $\tB(t,x,v)\geq 0$ satisfying $\X_3(t-\tB(t,x,v);t,x,v)=0$ is given by $\tB (t,x,v)= \frac{1}{g} \big( \sqrt{|v_3|^2 + 2 g x_3} - v_3\big)$. 

Then the local density \eqref{def:varrho} is bounded as 
\Be\notag
|\varrho(t,x)| \leq  \int_{\R^3}\int^t_{\max \{0, t-\tB(t,x,v) \}}   | E(s, \X(s;t,x,v))  ||  \nabla_v h(s, \X(s;t,x,v), \V(s;t,x,v))| \dd s \dd v .
\Ee

Let us try to impose a crucial condition, namely $\nabla_v h(x,v)$ has some Gaussian upper bound with respect to the total energy: $|\nabla_v h(x,v)| \lesssim e^{\beta (|v|^2 + 2g x_3)}$. Using the fact that the energy is conserved along the characteristics, we can derive that $|\varrho(t,x)|$ is bounded above by 
\Be\notag
\begin{split}
e^{-\lambda t }\left( \int_{\R^3}  \frac{ \tB(t,x,v) e^{\lambda  \tB(t,x,v)}  }{e^{\beta(|v|^2 + g x_3) }} \dd v \right)
\sup_{(s,x)
}e^{\lambda s}  |  E(s,x)|
\sup_{(x,v)
}e^{\beta (|v|^2 + 2 gx_3)}| \nabla_v h(x,v)|.
\end{split}\Ee 
Using the form of $\tB$, we can bound the term in the parenthesis above by $
O\Big(\frac{1}{g \sqrt{\beta}}\Big)
\int_{\R^3} e^{- \frac{\beta}{2} (|v|^2 + 2g x_3)} \dd v \leq O\Big(\frac{1}{g  \beta^2}\Big).$ Therefore we might be able to conclude 
\Be\notag
\sup_{t} e^{\lambda t} \|\varrho(t)\|_{L^\infty(\O)} \leq   O\Big(\frac{1}{g  \beta^2}\Big)
\| e^{\beta (|v|^2 + 2 gx_3)} \nabla_v h \|_{L^\infty(\O \times \R^3)} \sup_{t} e^{\lambda t} \| E(t) \|_ {L^\infty(\O)}.
\Ee 

On the other hand, for the nonlinear problem \eqref{eqtn:f}, the Poisson equation \eqref{Poisson_f} of $E= \nabla_x \Psi$ suggests a control of $ E(t ,x)$ pointwisely (at least locally) by $\|\varrho(t)\|_{L^\infty(\O)}$. In the end, as far as we choose $g \beta^2$ large enough, depending on our possible control of $\| e^{\beta (|v|^2 + 2 gx_3)} \nabla_v h \|_{L^\infty(\O \times \R^3)} $, we find ``a small factor'' in the nonlinear contribution.

\medskip

\noindent\textbf{E. About this work and related results.}

\newpage

\noindent\textbf{F. Structure of the rest of paper.} 

\vspace{-0pt}

\renewcommand\contentsname{}

\begingroup
\let\clearpage\relax
\vspace{-1cm} 
\tableofcontents
\endgroup

\vspace{-10pt}

\unhide


\textbf{Structure of the paper:} In Section 3, we construct a Green function of the Poisson equation \eqref{eqtn:Dphi} in $\O$; in Section 4, we prove the asymptotic stability criterion (Theorem \ref{theo:AS}) using the Lagrangian proof; in Section 5, we establish the steady theorem (Theorem \ref{theo:CS}); and in Section 6, we prove the dynamic theorem (Theorem \ref{theo:CD}). 

\vspace{-15pt}

\Be\label{notation}
\text{\textbf{Notation:} We will use an abbreviation of $\|  g \|_{L^\infty_{t,x}}= \sup_{\tau \in [ 0, t]} \| g(t) \|_{L^\infty (\O)}$ for some $t \geq 0$.}
\Ee

\section{Green function}\label{Sec:Green} In this section, we construct and study the Green's function $G(x,y)$ of the following Poisson equation in the horizontally-periodic 3D half-space:
\Be
\begin{split}\label{Dphi}
\Delta \phi (x) = \rho(x)  \ \ &\text{in} \ x \in \O : = \T^2 \times [0, \infty),\\
\phi (x) = 0  \ \  &\text{on} \ x \in \p \O : =  \T^2 \times \{0\},
\end{split}
\Ee
such that $\phi$ solving \eqref{Dphi} takes the form of  
\Be\label{phi_rho}
\phi (x)  = \int_{\T^2 \times [0, \infty)} G(x,y) \rho (y ) \dd y .
\Ee
We will construct $G(x,y)$ and prove its properties in Proposition \ref{lemma:G}.

The 2D Green's function in $\T \times [0,\infty)$ has an explicit formula. It is so-called the Green's function for the one-dimensional grating in $\R^2$ (see \cite{AmmariL}). However, there seems no known explicit form in 3D.  In this section, we utilize a classical argument of multiple Fourier series (e.g. Theorem 2.17 in \cite{SW}) to study the 3D problem.

\begin{theorem}\label{lemma:G} 
The Green's function for \eqref{Dphi} takes a form of  
\Be\label{GreenF}
G(x,y) =
\frac{|x_3- y_3|}{2} - \frac{|x_3+ y_3|}{2} 
- \mathcal{G}  (x,y), \ \ \text{for} \ x,y \in\O: = \T^2 \times [0, \infty).
\Ee

When $|x_3 - y_3| \geq 1$ and $x, y \in \O$, $\mathcal{G}    (x,y)$ satisfies  
\Be	\label{b0:1}
| \nabla_{x_\parallel, y_\parallel}^k\p_{x_3}^i \p_{y_3}^j  \mathcal{G}    ( x_\parallel , x_3, y_\parallel , y_3)|
\lesssim e^{-|x_3 - y_3|} \ \ \text{for $0 \leq k, i+j \leq 2$ and $|x_3 - y_3| \geq 1$}.
\Ee
When $|x_3 - y_3| \leq 1$ and $x, y \in \O$, 
\begin{align}
\mathcal{G}  (x,y) : = \frac{c_2}{|x-y|} - \frac{c_2}{|\tilde{x} - y|} + \mathcal{G}   _0(x,y)   \ \ \text{for $|x_3 - y_3| \leq 1$,} \label{b0:3}  \\
\text{$\mathcal{G}  _0 (\cdot, x_3, \cdot, y_3 ), \partial_{x_3} \mathcal{G}  _0  (\cdot, x_3, \cdot, y_3 ),  \partial_{y_3} \mathcal{G}  _0  (\cdot, x_3, \cdot, y_3 ) \in C^\infty (\T^2 \times \T^2)$ for $|x_3 - y_3| \leq 1$.}\label{b0:4}
\end{align}
Here, $\tilde x  = (x_1, x_2, - x_3)$, and $c_2 := \frac{1}{2\pi^{3/2}} \Gamma( 3/2)$ with the Gamma function $\Gamma$. Moreover, 
\Be
\begin{split}
\p_{x_3}^2 \mathcal{G}  _0 (x,y ) = \delta_0 (x_3- y_3)  \mathcal{G}  _1 (x,y) + \mathcal{G}  _2 (x,y)
\ \ \text{  for $|x_3 - y_3| \leq 1$}
,
\label{b0:2}
\\
\text{$\mathcal{G}  _1 (x_3, \cdot, y_3, \cdot ), \mathcal{G}  _2 (x_3, \cdot, y_3, \cdot ) \in C^\infty (\T^2 \times \T^2)$ for $|x_3 - y_3| \leq 1$}.
\end{split}
\Ee
\hide	where .

\begin{align} 
\bullet& \  \  | \nabla_{x_\parallel, y_\parallel}^k\p_{x_3}^i \p_{y_3}^j  b ( x_\parallel , x_3, y_\parallel , y_3)|
\lesssim e^{-|x_3 - y_3|} \ \ \text{for $0 \leq k, i+j \leq 2$ and $|x_3 - y_3| \geq 1$}
;
\label{b0:1}\\
\bullet &  \ \ b(x,y) : = \frac{c_2}{|x-y|} - \frac{c_2}{|\tilde{x} - y|} + d(x,y) + r(x,y) \ \ \text{for $|x_3 - y_3| \leq 1$}
\\
& 
\\
\bullet& \   \   \p_{x_3}^2 b (x,y) = \delta_0(x_3- y_3) b_2 (x-y) + b_3 (x,y)\notag \\ 
& 	 \ \ 
\text{where}  \  b_2 (\cdot , x_3 ) 	\in C^\infty (\T^2  ), \  b_3 (\cdot , x_3, \cdot, y_3)
\in C^\infty (\T^2 \times \T^2);
\label{b0:2}\\
\bullet& \ \ 
\| b_0(\cdot, x_3, \cdot, y_3)\|_{C^2(\T^2 \times \T^2)} \lesssim e^{-|x_3 -y_3|}  , \   \| \p_{x_3}b_0(\cdot, x_3, \cdot, y_3)\|_{C^2(\T^2 \times \T^2)} \lesssim e^{-|x_3 -y_3|}  , \notag\\
& \ \   \| \p_{x_3}^2b_0(\cdot, x_3, \cdot, y_3)\|_{C^2(\T^2 \times \T^2)}  \lesssim e^{-|x_3 -y_3|}  
\text{ for } \   | x_3 - y_3|\geq 1 
;\label{b0:3}\\
\bullet& \ \ 
\| b_0(\cdot, x_3, \cdot , y_3)\| _{C^2(\T^2 \times \T^2)}  \lesssim 1  ,    \  \| \p_{x_3}b_0(\cdot, x_3, \cdot , y_3)\|_{C^2(\T^2 \times \T^2)}  \lesssim 1, \notag\\
& \ \   \| \p_{x_3}^2b_2(\cdot, x_3)\|_{C^2(\T^2)}+ \| \p_{x_3}^2b_3(\cdot, x_3, \cdot , y_3)\|_{C^2(\T^2 \times \T^2)} \lesssim 1.
\label{b0:4}
\end{align}
\unhide

\end{theorem}

Once we have the following lemma, the proof of Theorem \ref{lemma:G} is straightforward. 
\begin{lemma}\label{lemma:tildeG} 
We construct a function in $ \T^2 \times \R $, which solves the following equation 
\be \label{g1}
\Delta \tilde G (x) = \sum\limits_{m  \in \Z^2} \delta_0 (x + (m,  0))
= \sum_{m_1, m_2 \in \Z} \delta_0 (x_\parallel+ m, x_3) 
\ \ \text{in $ \T^2 \times \R $.}
\ee
This Green's function takes the form of 
\be \label{G2}
\begin{split}
\tilde  G (x_\parallel, x_3)  
= \frac{1}{2} |x_3| + c - \tilde{\mathcal{G}  }(x_\parallel, x_3) .
\end{split}
\ee
Here, $\tilde{\mathcal{G}  }$ is defined in \eqref{G1} and $c$ is an arbitrary constant. When $|x_3| \geq 1$, $\tilde{\mathcal{G}  }$ satisfies that 
\Be\label{est1:tildeB}
| \nabla_{x_\parallel}^n  \p_{x_3}^i \tilde{\mathcal{G}  } (x_\parallel, x_3) | \lesssim e^{- |x_3|}
\ \ \text{for $|x_3| \geq 1$, and $n \in \mathbb{N}$, $i=0,1,2.$}
\Ee
When $|x_3| \leq 1$, the function can be decomposed as 
\Be\begin{split} \label{dec1:tildeb}
\tilde{\mathcal{G}  }(x )  = \frac{c_2}{(|x_\parallel|^2 + |x_3|^2)^{1/2}} + \tilde{d}  (x ) + \tilde{r}(x ), \ \ \text{for} \  |x_3| \leq 1.
\end{split}\Ee
Then $\tilde{d}$ and $\tilde{r}$ satisfy 
\begin{align}
\text{$ \tilde{d}(\cdot , x_3), \p_{x_3}  \tilde{d} (\cdot, x_3) \in C^\infty(\R^2)$;}\label{prop1:b1} \\
\text{$\p_{x_3}^2  \tilde{d} ( \cdot, x_3) = \delta_0 (x_3)  \tilde{d}_1 (\cdot, x_3)+  \tilde{d}_2 (\cdot, x_3)$, $ \tilde{d}_1(\cdot , x_3),   \tilde{d}_2(\cdot, x_3) \in C^\infty(\R^2)$
	;}\label{prop2:b1}\\
\text{$ \sum_{j=0,1} | \nabla_{x_\parallel}^n \p_{x_3}^j \tilde{d}(x_\parallel, x_3)| 
	+  \sum_{i=1,2} | \nabla_{x_\parallel}^n   \tilde{d}_i(x_\parallel, x_3)|  \lesssim_{n,N} (1+  |x_\parallel|)^{-N} $ for $|x_3| \leq 1$,}	\label{prop3:b1}
\end{align}
and 
\Be \label{est:tilder}
\begin{split}
\nabla_{x_\parallel}^n \p_{x_3}^2\tilde{r} (x_\parallel, x_3) = \delta_0 (x_3) \tilde{r}_1 (x_\parallel, x_3)  + \tilde{r}_2 (x_\parallel, x_3) ,&\\
\sum_{j=0,1}| \nabla_{x_\parallel}^n \p_{x_3}^j\tilde{r} (x_\parallel, x_3)| + 
\sum_{i=1,2}| \nabla_{x_\parallel}^n  \tilde{r}_i (x_\parallel, x_3)|  \lesssim_{n} e^{-|x_3|}.&
\end{split}
\Ee

\end{lemma}

\begin{proof}[\textbf{Proof of Theorem \ref{lemma:G}}]
With the Green's function $ \tilde G (x) $ to \eqref{g1} in our hand, it is straightforward to construct the Green's function of \eqref{Dphi} by setting
\Be\label{GtildeG}
G(x,y) =   \tilde G (x-y) -   \tilde G (\tilde x-y) , 
\Ee
where $\tilde x  = (x_1, x_2, - x_3)$. We can easily show \eqref{b0:1}-\eqref{b0:2} from Lemma \ref{lemma:tildeG}.
\end{proof}

We postpone the proof of Theorem \ref{lemma:G} and first study some elliptic estimates.

\begin{lemma}\label{lem:rho_to_phi}
Suppose $| \rho  (x) |  \leq 
A e^{-B x_3}
$ for $A>0$ and $B>1$. 

Then $\phi(x)$ in \eqref{phi_rho} satisfies that, for some $\mathfrak{C}>0$, 
\begin{equation}\label{est:nabla_phi}
|\p_{x_j}   \phi (x)|   \leq \mathfrak{C} A \left\{
e^{- B \min\{0, x_3 -1 \}}  +  \frac{ e^{- {B}  x_3}}{B}   +  \min\left\{
B^{-1},   \frac{ \mathbf{1}_{B>1} }{B-1} e^{- x_3} 
\right\} 
\right\}
\ \ \text{for  } \  x \in \T^2 \times [0, \infty).
\end{equation}

Moreover, for any $\delta>0$, $\phi(x)$ in \eqref{phi_rho} satisfies that
\Be
\begin{split}
\| \nabla_x^2 \phi \|_{L^\infty(\O)} &\lesssim_\delta  \|\rho \|_{L^\infty(\O)} \log (e+ [\rho ]_{C^{0,\delta}(\O)})
+ AB^{-1}
.\label{est:nabla^2phi}
\end{split}
\Ee  \hide

Moreover, for $1<p < \infty$, 
\Be\label{CZ_infty}
\| \nabla_x^2 \phi \|_{L^p (\O)} \lesssim _p \frac{A}{B ^{1/p}}.
\Ee

\unhide
\end{lemma}

\begin{proof}\textit{Proof of \eqref{est:nabla_phi}.}
From \eqref{GreenF} and \eqref{phi_rho}, we have 
\Be\label{est:p3phi}
\begin{split}
|	\p_{x_j} \phi  (x)| &= | (\p_{x_j} G* \rho )(x)| \leq I_1+ I_2 + I_3 + I_4  \\
&=  \delta_{j3}  \int_0^\infty \int_{\T^2} \left| \p_{x_3} \frac{|x_3 - y_3| }{2}- \p_{x_3} \frac{|x_3 + y_3|}{2}  \right| 
A e^{- B y_3}   \dd y_\parallel \dd y_3\\
& \ \ + \int_0^ {\infty}\int_{\T^2}  
\mathbf{1}_{|x_3 - y_3 | \leq 1 }
\left|  	\p_{x_j}\frac{c_2}{|x-y|}   - 	\p_{x_j}  \frac{c_2}{|\tilde{x}-y|} 
\right|
A e^{- B y_3}
\dd y_\parallel \dd y_3\\
& \ \ 
+ \int_0^\infty  \int_{\T^2}  
\mathbf{1}_{|x_3- y_3| \geq 1} |\p_{x_j} b (x, y) | 	A e^{- B y_3} 
\dd y_\parallel \dd y_3
\\
& \ \ + \int_0^ {\infty}\int_{\T^2}  
\mathbf{1}_{|x_3 - y_3 | \leq 1 }| \p_{x_j} b_0 (x,y)|
A e^{- B y_3}
\dd y_\parallel \dd y_3
.\end{split}
\Ee
The first term can be easily bounded as 
\Be
\begin{split}\label{est:I_1}
I_1   
& =\frac{ \delta_{j3} }{2} \int_{\T^2} \int_0^\infty 
\Big|
\mathbf{1}_{x_3> y_3} - \mathbf{1}_{x_3 < y_3}- \mathbf{1}_{x_3> - y_3} + \mathbf{1}_{x_3< - y_3}
\Big|
A e^{- B y_3}   \dd y_3 \dd y_\parallel \\
& =  \delta_{j3}  \int_{\T^2} \int^\infty_{x_3} 	A e^{- B y_3}   \dd y_3 \dd y_\parallel
\leq\delta_{j3} AB^{-1} e^{- {B}  x_3}.
\end{split}
\Ee

\hide
From \eqref{elliptic_est:C1}, 
\Be\begin{split}
&| \Delta_0^{-1} \rho (x_\parallel, x_3)|\\
&= \Big| \Delta_0^{-1} \rho (x_\parallel, x_3=0)
+ \int_0^{x_3}   \p_{x_3} \Delta_0^{-1} \rho (x_\parallel, y_3) \dd y_3 \Big|\\
& \leq 0+ 10 \frac{g}{20}x_3  = \frac{g}{2} x_3.
\end{split}	\Ee
Then we 
\Be
w_h (x,v) \geq e^{\frac{\beta |v|^2}{2}} e^{\beta (g - \frac{g}{2} ) x_3}
= e^{\frac{\beta |v|^2}{2}} e^{\beta   \frac{g}{2}  x_3}\label{lower:wh}
\Ee
Hence we deduce that 
\Be\begin{split}
| \rho_h(x)| &\leq  \int_{\R^3} \frac{1}{w_h (x,v)} |w_h (x,v)[ h_+ (x,v)- h_- (x,v)] |\dd v \\
& \leq  \frac{e^{- \frac{\beta g}{2} x_3  }}{\beta} \| w_h [h_+ - h_-] \|_{L^\infty}  .
\end{split}\Ee
\unhide

For the second term, using $|\tilde{x} - y| \geq |x-y|$ for $x_3, y_3 \geq 0$, we derive that
\Be\begin{split} \label{est:I_2}
I_2 
&\leq 
2 \int_{x_3-1}^ {x_3 + 1}\int_{\T^2}  \frac{1}{|x-y|^2} A e^{- B y_3}
\dd y_\parallel \dd y_3
\leq  
2 \int_{x_3-1}^ {x_3 + 1} A e^{- B y_3}  \int_0^{\frac{1}{\sqrt{2}}}  \frac{2 \pi r \dd r}{r^2 + |x_3 - y_3|^2} \dd y_3  
\\
&= 2 \pi  A  \int_{x_3-1}^ {x_3 + 1} e^{- B y_3}	   \ln \left(1+ \frac{1}{2 |x_3 - y_3|^2}\right)
\dd y_3 
\lesssim A e^{- B \min\{0, x_3 -1 \}}.
\\
\hide
& = 2 \pi  A \left\{
\int_{\max\{0, x_3-1\}}^{ x_3+1}  + \int_{0}^{\max\{0, x_3-1\}}
\right\},
\\
&
\| w [G_+ -   G_-] \|_{L^\infty_{\gamma_-}}
\int_{0}^\infty   \frac{(2\pi)^{3/2}}{\beta^{3/2}} e^{- \frac{\beta g}{2} y_3 } 	\int_{\T^2}  \frac{1}{|x_\parallel-y_\parallel|^2 + |x_3- y_3|^2}\dd y_\parallel \dd y_3 \\
& \leq
\| w [G_+ -   G_-] \|_{L^\infty_{\gamma_-}}
\int_{0}^\infty   \frac{(2\pi)^{3/2}}{\beta^{3/2}} e^{- \frac{\beta g}{2} y_3 }  	\int_{0}^2  \frac{1}{r + |x_3- y_3|^2}\dd r\dd y_3\\
&=  \frac{(2\pi)^{3/2}}{\beta^{3/2}}  \| w [G_+ -   G_-] \|_{L^\infty_{\gamma_-}}   \int_{0}^\infty  e^{- \frac{\beta g}{2} y_3 }   
\ln \Big( 1+ \frac{2}{|x_3- y_3|^2} \Big) \dd y_3 
\\ & 
=  \frac{(2\pi)^{3/2}}{\beta^{3/2}}  \| w [G_+ -   G_-] \|_{L^\infty_{\gamma_-}}   \left\{  \int_{0}^{\frac{x_3}{2}}  +  \int_{ 2 x_3}^{\infty}  +  \int^{ 2 x_3}_{\frac{x_3}{2}}
\right\}
\unhide
\end{split}\Ee

For the third term, using \eqref{b0:1}, we derive that 
\Be
\begin{split}\label{est:I_3}
I_3 & \lesssim  \int_0^\infty  \int_{\T^2}  
\mathbf{1}_{|x_3- y_3| \geq 1}  e^{-|x_3- y_3|}	A e^{- B y_3} 
\dd y_\parallel \dd y_3\\
& \lesssim \int_{0}^{x_3-1 }  e^{-(x_3- y_3)}	A e^{- B y_3}   \dd y_3 
+  \int^{\infty}_{x_3+1 }   e^{-(y_3-x_3)}  	A e^{- B y_3}   \dd y_3 
\\
& 
\lesssim \min\left\{
AB^{-1},  A \frac{ \mathbf{1}_{B>1} }{B-1} e^{- x_3} 
\right\}
+ \frac{A}{(B+1)  e^{( B+1)}} e^{- B x_3 }
. 
\end{split}
\Ee

Lastly, using \eqref{b0:4}, we derive that 
\Be
\begin{split}\label{est:I_4}
I_4 
\lesssim   \int_{x_3-1}^{x_3+1}
A e^{-B y_3} \dd y_3
\lesssim  A e^{- B \min \{0,x_3-1\}}.
\end{split}
\Ee
Combining \eqref{est:I_1}-\eqref{est:I_4}, we conclude \eqref{est:nabla_phi}.

\textit{Proof of \eqref{est:nabla^2phi}. }First we note that 
\Be
\begin{split}\label{ext:Grho}
\int_{\T^2 \times [0, \infty)}G (x,y) \rho (y) \dd y   =  	\int_{-\infty}^\infty \int_{\T^2} 
G(x,y) \tilde{\rho} (y) 
\dd y_\parallel  \dd y_3 \ \ \text{in $x \in \T^2 \times [0, \infty),$} 
\end{split}
\Ee
where $\tilde{\rho}$ is defined in $y\in \T^2 \times \R$ as 
\Be\label{tilde_rho}
\tilde{\rho} (y_\parallel, y_3) := \frac{1}{2} \left\{
[	\rho (y_\parallel, y_3 ) - \rho(y_\parallel, 0)] - [	\rho (y_\parallel, -y_3 )- \rho(y_\parallel, 0 )
]	\right\} .
\Ee		
Then we apply a standard result of the potential theory (\cite{Rein} or Section 4.2.5 of \cite{glassey}). The last term in \eqref{est:nabla^2phi} comes from $\|\rho\|_{L^1 (\O)} \lesssim AB^{-1}$. \hide

For \eqref{est:nabla^2phi}, we follow 

Note that $\nabla_x^2 \left( \frac{C_2}{4} \frac{1}{|x-y|} - \frac{C_2}{4} \frac{1}{ |\tilde x - y|}\right)$ belongs to the Calder\'on-Zygmund kernel. Using the form of $G(x,v)$ in \eqref{GreenF} and \eqref{b0:2}, we get that, for $1< p <\infty$, 
\Be\label{CZ}
\| \nabla_x^2 \phi \|_{L^p(\O)} 
\lesssim_p \| \rho \|_{L^p(\O)} .
\Ee 

Next we prove \eqref{CZ_infty}. This estimate is a standard result of the potential theory when the kernel in \eqref{phi_rho} is $\frac{C_2}{4} \frac{1}{|x-y| } - \frac{C_2}{4} \frac{1}{|\tilde x - y|}$. A contribution of $b_0 (x,y)$ in \eqref{phi_rho} is also bounded as \eqref{CZ_infty} due to \eqref{b0:1}- \eqref{b0:4}. Finally note that 
\Be\notag
\p_{x_3}^2 \left( \frac{|x_3 - y_3|}{2} - \frac{|x_3+ y_3|}{2}\right) = \delta_0 (x_3- y_3) \ \ \text{for} \ (x,y) \in \O \times \O.
\Ee
Hence this contribution in  \eqref{phi_rho} is also bounded as \eqref{CZ_infty}.\unhide\end{proof}

\hide

Then 
\Be
\begin{split}
&|\phi_h| = |\Delta^{-1}_0 \int_{\R^3} (e_+ h_+ + e_- h_-) \dd v|\\
&
\leq \max \{e_+, e_-\} \Delta^{-1}_0  \int_{\R^3} \frac{  \dd v }{w(\frac{|v|^2}{2} + \frac{g}{2} m_\pm x_3)}
\end{split}
\Ee 
This should verify the bootstrap hypothesis \eqref{Bootstrap}. For example, if $w(s) = e^{\sqrt{s}}$ then 
\Be
\frac{1}{w(\frac{|v|^2}{2} + \frac{g}{2} m_\pm x_3)} \lesssim e^{-|v|/\sqrt{2}} e^{-C \sqrt{x_3}}
\Ee
Then this should give us $ \int_{\R^3} \frac{  \dd v }{w(\frac{|v|^2}{2} + \frac{g}{2} m_\pm x_3)} \lesssim e^{-C \sqrt{x_3}}$. Solving the Poisson equation with Dirichlet BC (Using Lemma of $G(x,y)= \cdots$), we might get
\Be
\begin{split}
|\nabla_x \phi(x) | \lesssim \int_{\T^2 \times \R_+} \frac{1}{|x-y|^2} e^{- \sqrt{y_3}} \dd y
+ \int_{\T^2} \int_0^\infty \nabla_x   G_0 (x,y)e^{- \sqrt{y_3}}  dd y_3 \dd y_\parallel\\
{\color{red} \lesssim  \int_{\T^2 \times \R_+} \frac{1}{|x-y|^2} e^{- \sqrt{y_3- x_3}}  e^{- \sqrt{x_3}} \dd y
+ \| \nabla_x G_0 \|_\infty \int_{\T^2} \int_0^\infty   e^{- \sqrt{y_3 -x _3}} e^{- \sqrt{x_3}}  dd y_3 \dd y_\parallel  \ \ \ \ \  (Wrong!!)  }\\
\lesssim  \langle  x_3 \rangle^{-2}
\end{split}
\Ee

This should give us very good estimate when $x_3 \gg 1$. For $x_3 \lesssim 1$, we have, for $p>3$,  
\Be
\| \nabla_x \phi \|_\infty \lesssim \| \nabla_x \phi \|_{W^{1,p}} \lesssim   \| \rho \|_{L^p} \ll g m_{\pm} \ \ (small \ or \ nearly \ neutral). 
\Ee
Here, 
\Be\begin{split}
\rho (x) &= \int_{\R^3}  (h_+ (x,v) - h_- (x,v)) \dd v \\
& =  \int_{\R^3}  (G_+ (x^\b,v^\b) - G_- (x^\b,v^\b)) \dd v 
\end{split}
\Ee
For simplicity assume $m_+ = m_-$ (e.g. proton and electron) then $x^\b_+ = x^\b_-$ and $v^\b_+ = v^\b_-$. Then 
\Be
\begin{split}
\| \rho \|_{L^p (\O)} &\leq \Big\| \int_{\R^3}  (wG_+ (x^\b(x,v),v^\b(x,v)) - wG_- (x^\b(x,v),v^\b(x,v))) \frac{1}{w(x^\b(x,v),v^\b(x,v)) } \dd v \Big\|_{L^p(\O)}\\
& \leq  \| w (G_+- G_-) \|_{L^p(\gamma_-)}
\end{split} \Ee
where we have used 
\Be
\begin{split}
&\| f( x^\b (x,v), v^\b (x,v)) \frac{1}{w ( x^\b (x,v), v^\b (x,v)} \|_{L^1 (\O \times \R^3)}\\
&=\| f (x^\b, v^\b) \frac{t^\b}{w(x^\b, v^\b)}  |n(x^\b) \cdot v^\b| \| _{L^1  (t^\b, x^\b , v^\b) }
\end{split}
\Ee
from 
\Be
\frac{\p (t^\b, x^\b , v^\b)}{\p (x,v)} = \frac{1}{n(x^\b, v^\b)}
\Ee

{\color{red} (WRONG!!)Of course, this argument has a plenty of room by choosing a gravitation and $w$: For example if $w(\tau) = \tau^{\frac{3}{2}+}$ then 
\Be\begin{split}
\int_{\R^3} \frac{1}{w (\frac{|v|^2}{2} + \frac{g}{2} m_\pm x_3)}\dd v \lesssim 
\int_{\R^3} \frac{1}{  (\frac{|v|^2}{2} + \frac{g}{2} m_\pm x_3)^{\frac{3}{2}+}}\dd v\\
\lesssim \int \frac{dr}{
	r^{1+} + |x_3|^{\frac{3}{2}+}} \lesssim \frac{1}{1+ |x_3|^{\frac{3}{2}+}} (??)
\end{split}
\Ee
This may deduce that 
\Be
\nabla_x \phi (x) \lesssim  \frac{1}{1+ |x_3|^{\frac{3}{2}+}}. (??)
\Ee}

\unhide

\begin{proof}[\textbf{Proof of Lemma \ref{lemma:tildeG}}]
\textit{Step 1. }We claim that $\tilde{G}$ takes the following form: for some constant $c$, 
\be \label{G1}
\begin{split}
\tilde  G (x_\parallel, x_3)  
= \frac{1}{2} |x_3| + c - \tilde{\mathcal{G}  }(x_\parallel, x_3) \ \ \text{with  } \ \  \tilde{\mathcal{G}  }(x_\parallel, x_3) := \sum_{ 
	\substack{
		m \in \mathbb{Z}^2
		\\
		|m|>0}} \frac{ e^{- 2 \pi |m| |x_3|}}{4 \pi |m|} e^{i 2\pi m \cdot  x_\parallel}. 
\end{split}
\ee


For any $x = (x_1, x_2, x_3) \in \T^2 \times\R$, we have
\be
\begin{split} \notag 
\sum\limits_{m_1, m_2 \in \Z} \delta_0 (x + (m_1, m_2, 0))
= \sum\limits_{m_1 \in \Z} \sum\limits_{m_2 \in \Z}  \delta_0 (x_1 + m_1) \delta_0 (x_2 + m_2) \delta_0 (x_3).
\end{split}
\ee
Recall the Poisson summation formula $ \sum\limits_{n \in \Z}  \delta_0 (y + n) 
= \sum\limits_{n \in \Z} e^{i 2\pi n y}$ for $y \in \R$. Thus, we have
\be
\begin{split} \label{g1 2}
\sum\limits_{m_1 \in \Z} \sum\limits_{m_2 \in \Z}  \delta_0 (x_1 + m_1) \delta_0 (x_2 + m_2) \delta_0 (x_3) 
= \sum\limits_{m_1, m_2 \in \Z} \delta_0 (x_3) e^{i 2\pi m_1 x_1} e^{i 2\pi m_2 x_2}.
\end{split}
\ee

Now we try the following Ansatz to solve \eqref{g1}: With unknown functions $w_m = w_{m_1,m_2} : \R \mapsto \R$, 
\be \label{g1 3}
\tilde G (x) = \sum\limits_{m_1, m_2 \in \Z} w_{m_1,m_2} (x_3) e^{i 2\pi m_1 x_1} e^{i 2\pi m_2 x_2}.
\ee
By inserting \eqref{g1 3} in \eqref{g1}, we compute that 
\be
\begin{split}\notag
\Delta \tilde  G (x) 
& = \sum\limits_{m_1, m_2 \in \Z} 
\Big( w_{m_1,m_2}^{\prime\prime} (x_3) + \big( (i 2\pi m_1)^2 + (i 2\pi m_2)^2 \big) w_{m_1,m_2} (x_3) \Big)
e^{i 2\pi m_1 x_1} e^{i 2\pi m_2 x_2}
\\& = \sum\limits_{m_1, m_2 \in \Z} 
\Big( w_{m_1,m_2}^{\prime\prime} (x_3) + (i 2\pi)^2 (m_1^2 + m_2^2)w_{m_1,m_2} (x_3)  \Big) 
e^{i 2\pi m_1 x_1} e^{i 2\pi m_2 x_2}.
\end{split}
\ee
To solve \eqref{g1}, we ought to solve a second order linear ODE with the Dirac delta source term:
\be\notag
w_{m_1,m_2}^{\prime\prime} (x_3) + (i 2\pi)^2 (m_1^2 + m_2^2)  w_{m_1,m_2} (x_3) = \delta_0 (x_3).
\ee
Explicit solutions are given by
\be \label{wmn}
w_m (x_3) =w_{m_1,m_2} (x_3) :=
\begin{cases}
\frac{1}{2} |x_3| + c,  & \text{if } m_1 = m_2 = 0, \\
\frac{-e^{- 2 \pi |m| |x_3|}}{4 \pi |m|} =	\frac{-e^{- 2 \pi \sqrt{m_1^2 + m_2^2} |x_3|}}{4 \pi \sqrt{m_1^2 + m_2^2}} , & \text{otherwise},
\end{cases}
\ee
where $c$ is a constant. 
Finally, inserting \eqref{wmn} in \eqref{g1 3}, we complete the proof of \eqref{G1}.

\smallskip


\textit{Step 2. }
Define $\iota= \iota(x_\parallel) = \iota( |x_\parallel| ) \in C^\infty (\R^2)$ such that 
\Be\label{iota}
\iota (x_\parallel) = \begin{cases}
0 & \text{for } |x_\parallel| \leq 1/2,\\
\text{an increasing function of $|x_\parallel|$}
& \text{for } 1/2 \leq  |x_\parallel| \leq 1,\\
1& \text{for } |x_\parallel| \geq 1.
\end{cases}
\Ee  
Define $Q$ and its inverse horizontal Fourier transform $q$: for $(x_\parallel, x_3) \in \R^2 \times \R,$
\begin{align}
Q(x_\parallel, x_3)  : = \iota(x_\parallel)  \frac{  e^{- 2 \pi |x_\parallel| |x_3|}}{4 
	\pi |x_\parallel|} ,  \ \ \ 
q(x_\parallel, x_3)   := \int_{\R^2} Q(\xi, x_3) e^{2 \pi i \xi \cdot x_\parallel} \dd \xi. \label{def:Q}
\end{align}
From the Poisson summation formula and \eqref{G1}, \eqref{def:Q}, we obtain that 
\Be\label{q-Q}
\tilde{\mathcal{G}  }(x_\parallel, x_3)
=  \sum_{m \in \Z^2} Q(m, x_3) e^{2 \pi i x_\parallel \cdot m} 
= \sum_{m \in \Z^2} q(x_\parallel+ m, x_3)  .
\Ee

\smallskip


\textit{Step 3. }We claim that 
\begin{align}
| \nabla^n_{x_\parallel} q(x_\parallel, x_3)| & \lesssim_{n, N }  |x_\parallel|^{-N}e^{- |x_3|}  \ 
\text{ for 
	all $n , N \in \N^2$ such that $N \geq |n|+2$,}\label{est1:q}\\
| \nabla^n_{x_\parallel} \p_{x_3} q(x_\parallel, x_3)| & \lesssim_{n, N }  |x_\parallel|^{-N}e^{- |x_3|}  \ 
\text{ for 
	all $n , N \in \N^2$ such that $N \geq |n|+3$.}\label{est2:q}
\end{align}
As $Q$ vanishes for $|\xi | \leq 1/2$, we take derivatives to \eqref{def:Q} and derive that, for $n^\prime = (n^\prime_1, n^\prime_2) \in \mathbb{N}^2,$ 
\Be \label{DTq}
(- 2 \pi i x_\parallel)^{n^\prime}  \nabla_{x_\parallel}^{n} q(x_\parallel, x_3)
= \int_{\R^2}  \underline{ \p_{\xi_1}^{n_1^\prime} \p_{\xi_2}^{n_2^\prime} \left(
\iota (\xi) \frac{ e^{-2 \pi |\xi| |x_3|}}{ 4 
	\pi |\xi|}
\right)  }
\nabla_{x_\parallel}^{n } 	e^{2 \pi i \xi \cdot x_\parallel}
\dd \xi.
\Ee
Using the fact that $\iota (\xi )=0$ if $|\xi | \leq 1/2$ and $\nabla_{\xi}\iota (\xi )=0$ if $|\xi| \geq 1$ from \eqref{iota}, we bound the above underlined term in \eqref{DTq} by
\Be
\begin{split}\notag
&C (n^\prime)\|\iota\|_{C^{|n^\prime|}(\R^2)} 	\mathbf{1}_{|\xi| \geq \frac{1}{2}}  \left\{
\frac{1    }{|\xi|^{ |n^\prime| }}  + \frac{ (|\xi||x_3|)^{ |n^\prime|} }{|\xi|^{ |n^\prime|}}  + \mathbf{1}_{|\xi| \leq  1}
\sum_{m=0}^{|n^\prime|-1}  \Big( 
\frac{1}{|\xi|^m} + \frac{ (|\xi||x_3|)^{ m} }{|\xi|^{ m}} 
\Big)  
\right\} \frac{e^{- 2\pi |\xi||x_3|}}{|\xi|} \\
&	\lesssim_{n^\prime, \iota}    
\mathbf{1}_{|\xi| \geq \frac{1}{2}} \left(	\frac{ 1 }{|\xi|^{|n^\prime|+1}}  +
\frac{ \mathbf{1}_{|\xi| \leq  1} }{|\xi|}
\right)e^{- |x_3|}.
\end{split}
\Ee
Suppose $|n^\prime| \geq |n|+2$, then 
\Be\label{bound:q}
|x_\parallel ^{n^\prime}| |\nabla_{x_\parallel}^{n } q(x_\parallel, x_3)| \lesssim_{n, n^\prime, \iota} 
{e^{- |x_3|}}  
\left(
\int_{|\xi| \geq \frac{1}{2}}  \frac{ \dd \xi }{|\xi|^{|n^\prime| - |n|+1}}  
+ \int_{ \frac{1}{2} \leq |\xi| \leq 1}  \frac{|\xi|^{|n|}}{|\xi| }    \dd \xi  
\right)
\lesssim_{n, n^\prime, \iota}  {e^{- |x_3|}}  
.
\Ee
Summing \eqref{bound:q} over all possible $n^\prime \in \Z^2$ such that $|n^\prime| = \max\{N, |n|+2\}$, we conclude \eqref{est1:q}. 

From \eqref{DTq} we compute that  \Be \label{DTtq}
\begin{split}
(- 2 \pi i x_\parallel)^{n^\prime}  \nabla_{x_\parallel}^{n}  \p_{x_3} q(x_\parallel, x_3)
=&-     \int_{\R^2} 
\underline{
	\p_{\xi_1}^{n_1^\prime} \p_{\xi_2}^{n_2^\prime} \left(
	\iota (\xi) 
	\frac{x_3}{|x_3|}
	\frac{ e^{ -2 \pi |\xi| |x_3|}}{ 2 
	}
	\right)  
	\nabla_{x_\parallel}^{n } 	e^{2 \pi i \xi \cdot x_\parallel}}
\dd \xi.
\end{split}
\Ee
We bound the underlined terms of \eqref{DTtq} respectively by 
\Be
\begin{split}\label{est:DTtq}
& {\text{\small	$C_{n^\prime, n, \iota} \mathbf{1}_{|\xi| \geq \frac{1}{2}} \left\{
		\frac{ (|\xi|	|x_3|)^{|n^\prime|}}{|\xi|^{|n^\prime| -| n|}}
		+ \mathbf{1}_{|\xi| \leq 1 } \sum_{m=0}^{|n^\prime|-1} \frac{ (|\xi|	|x_3|)^{m}}{|\xi|^{|m| -| n|}}
		\right\}e^{-2 \pi |\xi| |x_3|}
		\lesssim  \left( \frac{
			\mathbf{1}_{|\xi| \geq \frac{1}{2}}
		}{|\xi|^{|n^\prime| - |n|}}
		+ \mathbf{1}_{\frac{1}{2} \leq |\xi| \leq 1 }\right)			e^{- |x_3|}
		.$}}
\end{split}
\Ee
Choose $|n^\prime| = \max \{N, |n|+3\}$. Then the above upper bounds are integrable-in-$\xi$ in $\R^2$. This allows us to prove \eqref{est2:q}.

\smallskip


\textit{Step 4. }We claim that 
\Be\begin{split}
\nabla^n_{x_\parallel} \p_{x_3}^2 q(x_\parallel, x_3) &= \delta_0 (x_3)
q_1 (x_\parallel, x_3)  + q_2 (x_\parallel, x_3) ,
\\
| \nabla^n_{x_\parallel}   q_i(x_\parallel, x_3)| &\lesssim_{n, N }  |x_\parallel|^{-N}e^{- |x_3|}  \ 
\text{ for $i=1,2,$ and 
	all $n , N \in \N^2$, $N \geq |n|+4$.}\label{est3:q}
\end{split}\Ee

From \eqref{DTtq} we compute that  \Be \label{DTt2q}
\begin{split} 
(- 2 \pi i x_\parallel)^{n^\prime}  \nabla_{x_\parallel}^{n}  \p_{x_3}^2 q(x_\parallel, x_3)
=&-    \delta_0(x_3)   \int_{\R^2} \underline{ \p_{\xi_1}^{n_1^\prime} \p_{\xi_2}^{n_2^\prime}  
	\left(\iota (\xi)  
	e^{ - 2 \pi |\xi| |x_3|} \right)
	\nabla_{x_\parallel}^{n } 	e^{2 \pi i \xi \cdot x_\parallel}}
\dd \xi\\
& + \pi    \int_{\R^2} \underline{ \p_{\xi_1}^{n_1^\prime} \p_{\xi_2}^{n_2^\prime} \left(
	\iota (\xi)  |\xi| 
		e^{ -2 \pi |\xi| |x_3|}
	\right)  
	\nabla_{x_\parallel}^{n } 	e^{2 \pi i \xi \cdot x_\parallel}}
\dd \xi.
\end{split}
\Ee
Following the argument of the previous step, we bound first underlined term by \eqref{est:DTtq}; and bound the second underlined term of \eqref{DTt2q} by 
\Be
\begin{split}\notag
&	{\text{\small	$C_{n^\prime, n, \iota} \mathbf{1}_{|\xi| \geq \frac{1}{2}} \left\{
		\frac{ (|\xi|	|x_3|)^{|n^\prime|}}{|\xi|^{|n^\prime| -| n|-1}}
		+ \mathbf{1}_{|\xi| \leq 1 } \sum_{m=0}^{|n^\prime|-1} \frac{ (|\xi|	|x_3|)^{m}}{|\xi|^{|m| -| n|-1}}
		\right\}e^{-2 \pi |\xi| |x_3|} \lesssim  \left( \frac{
			\mathbf{1}_{|\xi| \geq \frac{1}{2}}
		}{|\xi|^{|n^\prime| - |n| - 1}}
		+ \mathbf{1}_{\frac{1}{2} \leq |\xi| \leq 1 }\right)			e^{- |x_3|}.$}}
\end{split}
\Ee
Choose $|n^\prime| = \max \{N, |n|+4\}$. Then the above upper bounds are integrable-in-$\xi$ in $\R^2$. This allows us to prove \eqref{est3:q}. 

\smallskip

\textit{Step 5. }Define \vspace{-10pt}
\Be\tilde{r} (x_\parallel, x_3) := \sum_{|m|>0} q(x_\parallel + m , x_3) . \label{def:tilder}
\Ee
From \eqref{est1:q}, \eqref{est2:q}, and \eqref{est3:q}, we conclude that the series \eqref{def:tilder} is absolutely convergent and hence \eqref{est:tilder} holds.

\smallskip

\textit{Step 6. }We claim that when $|x_3| \leq 1$, we can decompose $	\tilde{b}$ as \eqref{dec1:tildeb} where $\tilde{r}$ satisfies \eqref{def:tilder}-\eqref{est:tilder}, and \eqref{prop1:b1}-\eqref{prop3:b1} hold. 
Recall the following horizontal Fourier transform: 
\hide
\Be\label{Fourier_e}
\int_{\R^2} 
\frac{e^{- 2 \pi a |x_\parallel|}}{4 \pi |x_\parallel|} 
e^{- 2\pi i x_\parallel \cdot \xi } \dd x_\parallel 
= \frac{c_2  }{\left( a^2 + |\xi|^2  \right)^{1/2}}
\ \ \text{with} \ \ 
c_2 = \frac{1}{2\pi^{3/2}} \Gamma\Big(\frac{3}{2}\Big) \ \ \text{for} \ \ \text{Re} (a)>0,
\Ee  
where the Gamma function at $\frac{3}{2}$ is given by $\Gamma\big(\frac{3}{2}\big)=\int^\infty_0 t^{\frac{3}{2}-1} e^{-t} \dd t $. We set $a = |x_3|$ in \eqref{Fourier_e}. From the duality of the Fourier transform
, we derive \unhide
\Be\label{Fourier_inv}
\frac{e^{- 2\pi |x_3||\xi|}}{ 4 
\pi |\xi|} = 
\int_{\R^2} 
\frac{c_2 }{\left( |x_\parallel|^2 + |x_3|^2\right)^{1/2}} e^{- 2\pi i x_\parallel \cdot \xi}
\dd x_\parallel .
\Ee
Here, $c_2 = \frac{1}{2\pi^{3/2}} \Gamma( 3/2)$ where $\Gamma$ is the Gamma function.

We decompose $q$ and use the duality of Fourier transform with \eqref{Fourier_inv} to get that 
\Be
\begin{split}\label{def:b1}
q(x_\parallel, x_3) 
=   \frac{c_2}{\left( |x_\parallel|^2 + |x_3|^2\right)^{1/2}}  +  \tilde{d}  (x_\parallel, x_3),\ \
\tilde{d}  (x )	
: =  
\int_{\R^2} (\iota(\xi) - 1) \frac{ e^{- 2\pi |\xi| |x_3|}}{4
	\pi |\xi|} e^{2 \pi i \xi \cdot x_\parallel} \dd \xi. 
\end{split}\Ee

Now we only need to prove the properties of $\tilde{d}$, which are \eqref{prop1:b1}-\eqref{prop3:b1}. 
Note that $\tilde{d} (\cdot, x_3)$ is the inverse horizontal Fourier transforms of 
an integrable function with bounded support in $ \R^2$ horizontally. Hence $\tilde{d} (\cdot, x_3) \in C^\infty(\R^2)$. Next, 
we compute its derivatives of $\tilde{d}$. For any $n = (n_1, n_2) \in \N^2$ and $n^\prime = (n^\prime_1, n^\prime_2) \in \N^2$,
\Be
\begin{split}\label{b_1}
(-2 \pi i x_\parallel)^{n^\prime}\nabla_{x_\parallel}^{n}	\p_{x_3} \tilde{d}  (x_\parallel, x_3) & =  \underline{\int_{\R^2}
	\nabla_\xi^{n^\prime} \left((\iota(\xi)-1)  \frac{ x_3}{-2 
		|x_3|} e^{-2 \pi |\xi||x_3|}  \right)
	\nabla_{x_\parallel}^{n } e^{2 \pi i \xi \cdot x_\parallel} \dd \xi},\\
(-2 \pi i x_\parallel)^{n^\prime}\nabla_{x_\parallel}^{n}	\p_{x_3}^2 \tilde{d}  (x_\parallel, x_3) & 
= 	\delta_0 (x_3) (-2 \pi i x_\parallel)^{n^\prime}\nabla_{x_\parallel}^{n} \tilde{d}_1(x_\parallel, x_3)  + (-2 \pi i x_\parallel)^{n^\prime}\nabla_{x_\parallel}^{n} \tilde{d}_2 (x_\parallel, x_3) \\
&= 
- 
\delta_0 (x_3) \underline{\int_{\R^2} 
	\nabla_\xi^{n^\prime} \left(
	(\iota(\xi)-1)   e^{-2 \pi |\xi||x_3|}  \right) \nabla_{x_\parallel}^n  e^{2 \pi i \xi \cdot x_\parallel} \dd \xi } \\
& \ \ \ + \pi  
\underline{  \int_{\R^2}  \nabla_\xi^{n^\prime} \left((\iota(\xi)-1) |\xi|  e^{-2 \pi |\xi||x_3|}  \right) \nabla_{x_\parallel}^n e^{2 \pi i \xi \cdot x_\parallel} \dd \xi}.
\end{split}
\Ee
Here, we have used two functions defined as 
\begin{align}
\tilde{d}_1 (x_\parallel, x_3) &:=- 
\int_{\R^2} 
(\iota(\xi)-1)   e^{-2 \pi |\xi||x_3|}     e^{2 \pi i \xi \cdot x_\parallel} \dd \xi, \label{def:b2}
\\
\tilde{d}_2 (x_\parallel, x_3) &:= 
\pi \int_{\R^2} 
(\iota(\xi)-1) |\xi|  e^{-2 \pi |\xi||x_3|}
e^{2 \pi i \xi \cdot x_\parallel} \dd \xi.
\label{def:b3}
\end{align}
Note that the three underlined integrals in the right hand side of \eqref{b_1} are the inverse horizontal Fourier transform of 
integrable functions with bounded support in $\R^2$ horizontally.  By summing \eqref{b_1} over $|n^\prime| \leq  N$ in $n^\prime \in \Z^2$, we 
conclude \eqref{prop1:b1}-\eqref{prop3:b1}.

\smallskip


\textit{Step 7. }We consider the case of $|x_3| \geq  1$. We claim that \eqref{est1:tildeB} holds.

We first prove that 
\Be\label{est2:q}
\text{
$| \nabla^n_{x_\parallel} \p_{x_3}^i q(x_\parallel, x_3)|\lesssim  (1+|x_\parallel|)^{-N}e^{- |x_3|}  $ for $|x_3| \geq1$ 
and $n, N \in \N$.}
\Ee
We compute that, for $n^\prime = (n^\prime_1, n^\prime_2) \in \mathbb{N}^2$,
\Be\label{Dq}
\begin{split}
(-2 \pi i x_\parallel)^{n^\prime} 
\nabla_{x_\parallel}^{n }	\p_{x_3} q(x) &=\frac{x_3}{- 2 
	|x_3|}  \int_{\R^2}  \underline{  \nabla_\xi^{n^\prime}
	\Big( \iota(\xi)  e^{- 2\pi |\xi| |x_3|}\Big) \nabla_{x_\parallel}^ne^{2 \pi i \xi \cdot x_\parallel} }_{I}\dd \xi,\\
(-2 \pi i x_\parallel)^{n^\prime} 
\nabla_{x_\parallel}^{n }	\p_{x_3}^2 q(x) 
&= \int_{\R^2}  \underline{
	\pi
	\nabla_\xi^{n^\prime}
	\Big( \iota(\xi)  |\xi|  e^{-2 \pi |\xi| |x_3|} \Big) \nabla_{x_\parallel}^n e^{2 \pi i \xi \cdot x_\parallel} }_{II}\dd \xi.
\end{split}
\Ee 
We bound each of underlined terms in \eqref{Dq} as follows: for $|x_3|\geq 1$,
\begin{align}
I & \lesssim_{n, 
	\iota} \mathbf{1}_{|\xi| \geq \frac{1}{2}}
|\xi|^{|n|}  (1+ |x_3|^{|n^\prime|})  e^{- 2\pi |\xi| |x_3|}
\lesssim _{n, 
	\iota}  e^{-|\xi| } e^{- |x_3|},\notag
\\
II & \lesssim _{n, n^\prime, \iota}   \mathbf{1}_{|\xi| \geq \frac{1}{2}}   |\xi|^{|n|  } 
(1+   |\xi||x_3|^{|n^\prime|})
e^{- 2\pi |\xi| |x_3|}\lesssim _{n, n^\prime, \iota}  e^{-|\xi| } e^{- |x_3|}.\notag
\end{align}
Then we follow the argument to prove \eqref{bound:q} and derive that 
\Be\notag
| \nabla_{x_\parallel}^n  \p_{x_3} q(x)| 
+ | \nabla_{x_\parallel}^n  \p_{x_3}^2 q(x)| 
\lesssim_{n, n^\prime, \iota} 
|x_\parallel|^{-|n^\prime|}  e^{- |x_3|} \ \ \text{for  $|x_3|\geq 1$.}
\Ee
Now by choosing $|n^\prime| = N$, we conclude \eqref{est2:q}. 

Using \eqref{est2:q} with $N\geq 3$, we conclude that the summation \eqref{q-Q} is absolutely convergent. Using this together with \eqref{G1} and \eqref{est2:q}, we conclude \eqref{est1:tildeB}.\hide

Finally we prove

Now we claim that  
\Be\label{claim_Fourier} \sum\limits_{ |m|>0} \frac{  e^{- 2 \pi |m| |x_3|}}{
4	{\color{red} \cancel{8}
}
\pi |m|} e^{i 2\pi m \cdot  x_\parallel}
= \frac{c_2}{  {\color{red}\cancel{4}}} \frac{1}{\left( |x_\parallel|^2 + |x_3|^2\right)^{1/2}} + b(x_\parallel, x_3),
\Ee
where $b(x_\parallel, x_3)$, which will be defined in \eqref{def:b} precisely, satisfies 
\begin{align} 
\bullet& \  \    b( \cdot  , x_3),  \p_{x_3} b(  \cdot , x_3) \in C^\infty (\T^2) ;
\label{b:1}\\
\bullet& \   \   \p_{x_3}^2 b(x_\parallel, x_3) = \delta_0(x_3) b_2 (x_\parallel , x_3) + b_3 (x_\parallel , x_3) \ \ \text{where} \ \  b_2 (\cdot, x_3), b_3 (\cdot , x_3)
\in C^\infty (\T^2);
\label{b:2}\\
\bullet& \ \ 
\| b(\cdot, x_3)\|_{C^2(\R^2)} ,  \| \p_{x_3}b(\cdot, x_3)\|_{C^2(\R^2)}, \| \p_{x_3}^2b(\cdot, x_3)\|_{C^2(\R^2)} \lesssim e^{-|x_3|}  \  \text{ for } \   |x_3|\geq 1 
;\label{b:3}\\
\bullet& \ \ 
\| b(\cdot, x_3)\|_{C^2(\R^2)} ,    \| \p_{x_3}b(\cdot, x_3)\|_{C^2(\R^2)},  \| \p_{x_3}^2b_2(\cdot, x_3)\|_{C^2(\R^2)}+ \| \p_{x_3}^2b_3(\cdot, x_3)\|_{C^2(\R^2)} \lesssim 1.
\label{b:4}
\end{align}
For the proof, we adopt a classical argument of multiple Fourier series (e.g. Theorem 2.17 in \cite{SW}).


Define $\iota= \iota(x_\parallel) \in C^\infty (\R^2)$ and $\iota (x_\parallel)=1$ for $|x_\parallel| \geq 1$, and $\iota$ vanishes in a neighborhood of $x_\parallel =0$: $\iota (x_\parallel)=0$ if $|x_\parallel| \leq 1/2$. 
Define 
\Be\label{def:Q}
Q(x_\parallel, x_3) : = \iota(x_\parallel)  \frac{  e^{- 2 \pi |x_\parallel| |x_3|}}{4 {\color{red}\cancel{8}} \pi |x_\parallel|}  \ \ \text{for} \ \ (x_\parallel, x_3) \in \R^2 \times \R. 
\Ee

We claim that there exists a function $q(\cdot, x_3) \in L^1(\R^2)$ such that 
\begin{itemize}
\item[\it{(q1)}] $Q(\xi,x _3) =\hat{q} (\xi, x_3) : = \int_{\R^2} q(x_\parallel, x_3) e^{-2 \pi i \xi \cdot x_\parallel} \dd x_\parallel   ;$
\item[\it{(q2)}] $q(x_\parallel, x_3)= \frac{c_2}{2{\color{red}\cancel{4}}} \frac{1}{\left( |x_\parallel|^2 + |x_3|^2\right)^{1/2}}   + b_1(x_\parallel, x_3)$ where $b_1 (x_\parallel, x_3) :=   \int_{\R^2} (\iota(\xi) - 1) \frac{ e^{- 2\pi |\xi| |x_3|}}{
	4 {\color{red}\cancel
		8}
	\pi |\xi|} e^{2 \pi i \xi \cdot x_\parallel} \dd \xi$;
\item[\it{(q3)}] $b_1(\cdot , x_3), \p_{x_3} b_1 (\cdot, x_3) \in C^\infty(\R^2)$; $\p_{x_3}^2 b_1 ( \cdot, x_3) = \delta_0 (x_3) b_2 (\cdot, x_3)+ b_3 (\cdot, x_3)$ with $b_2(\cdot , x_3),  b_3(\cdot, x_3) \in C^\infty(\R^2)$ defined in \eqref{def:b2} and \eqref{def:b3}; for $|x_3 | \leq 1$ and $i=1,2,3$, $| \nabla_{x_\parallel}^n \p_{x_3} b_i(x_\parallel, x_3)| \lesssim_{n,N} |x_\parallel|^{-N} $ and $| \nabla_{x_\parallel}^n \p^2_{x_3} b_i(x_\parallel, x_3)| \lesssim_{n,N}  |x_\parallel|^{-N} $ for all $N \in \N$ and $n \in \N^2$;
\item[\it{(q4)}] $| \nabla^n_{x_\parallel} q(x_\parallel, x_3)|\lesssim  |x_\parallel|^{-N}e^{- |x_3|}  $ 
for every positive integer $N \in \N$ and $n \in \N^2$;  for $i=1,2$, we have that $| \nabla^n_{x_\parallel} \p_{x_3}^i q(x_\parallel, x_3)|\lesssim  |x_\parallel|^{-N}e^{- |x_3|}  $ for $|x_3| \geq1$ 
for every positive integer $N \in \N$ and $n \in \N^2$.
\end{itemize}

Let $q$ be the inverse horizontal Fourier transform of $Q$ in the sense of tempered distributions. Using the decomposition $Q(x)=   \frac{  e^{- 2 \pi |x_\parallel| |x_3|}}{ 4 \cancel{{\color{red}8}} \pi |x_\parallel|}  
+ \left(  \iota(x_\parallel)- 1 \right)  \frac{  e^{- 2 \pi |x_\parallel| |x_3|}}{4\cancel{{\color{red}8}} \pi |x_\parallel|} $ and the duality of Fourier transform with \eqref{Fourier_inv}, we decompose $q$ and rewrite it as 
\Be
\begin{split}\label{def:Q}
q(x_\parallel, x_3) &:= \int_{\R^2} Q(\xi, x_3) e^{2 \pi i \xi \cdot x_\parallel} \dd \xi\\
& =  \int_{\R^2}  \frac{ e^{- 2\pi |\xi| |x_3|}}{ 4{\color{red}\cancel{8}} \pi |\xi|} e^{2 \pi i \xi \cdot x_\parallel} \dd \xi
+  \int_{\R^2} (\iota(\xi) - 1) \frac{ e^{- 2\pi |\xi| |x_3|}}{4{\color{red}\cancel{8}} \pi |\xi|} e^{2 \pi i \xi \cdot x_\parallel} \dd \xi\\
& = \frac{c_2}{2{\color{red}\cancel{4}} }  \frac{1}{\left( |x_\parallel|^2 + |x_3|^2\right)^{1/2}}  +  b_1 (x_\parallel, x_3).
\end{split}\Ee
This proves (q1) and (q2).


Now we prove (q4). As $Q$ vanishes for $|\xi_\parallel| \leq 1/2$, we take derivatives to \eqref{def:Q} and derive that
\Be\begin{split}\notag
(- 2 \pi i x_\parallel)^{n^\prime}  \nabla_{x_\parallel}^{n} q(x_\parallel, x_3)
= \int_{\R^2}  \underline{\p_{\xi_1}^{n_1^\prime} \p_{\xi_2}^{n_2^\prime} \left(
	\iota (\xi) \frac{ e^{-2 \pi |\xi| |x_3|}}{ 4 {\color{red}\cancel{8}} \pi |\xi|}
	\right) }
\nabla_{x_\parallel}^{n } 	e^{2 \pi i \xi \cdot x_\parallel}
\dd \xi.
\end{split}\Ee
Using the fact that $\iota (\xi_\parallel)=0$ if $|\xi_\parallel| \leq 1/2$, we bound the underlined term by
\Be
\begin{split}\notag
\mathbf{1}_{|\xi| \geq \frac{1}{2}}C (n^\prime) \|\iota\|_{C^{|n^\prime|}(\R^2)}
\Big( \frac{1}{|\xi|^{ |n^\prime|}}+   |x_3|^{|n^\prime|} 
\Big) \frac{e^{- 2\pi |\xi||x_3|}}{|\xi|}
\lesssim_{n^\prime, \iota}    
\frac{  \mathbf{1}_{|\xi| \geq \frac{1}{2}} }{|\xi|^{|n^\prime|+1}} e^{- |x_3|}.
\end{split}
\Ee
Hence, we have that, for $|x_\parallel|\geq 1$,  
\Be\label{bound:q}
|x_\parallel ^{n^\prime}| |\nabla_{x_\parallel}^{n } q(x_\parallel, x_3)| \lesssim_{n, n^\prime, \iota} 
{e^{- |x_3|}}  \int_{\R^2}  \frac{  \mathbf{1}_{|\xi| \geq \frac{1}{2}} }{|\xi|^{|n^\prime| - |n|+1}}  \dd \xi.
\Ee
Summing \eqref{bound:q} over all possible $n^\prime \in \Z^2$ such that $|n^\prime| = \max\{N, |n|+2\}$, we conclude the first statement in (q4) for all $x_3 \in \R$. 

For the next statements in (q4), we compute that, for $|x_3| \geq 1$ 
\Be\label{Dq}
\begin{split}
(-2 \pi i x_\parallel)^{n^\prime} 
\nabla_{x_\parallel}^{n }	\p_{x_3} q(x) &=\frac{x_3}{- 2 {\color{red}\cancel4} |x_3|}  \int_{\R^2}  \underline{  \nabla_\xi^{n^\prime}
	\Big( \iota(\xi)  e^{- 2\pi |\xi| |x_3|}\Big) \nabla_{x_\parallel}^ne^{2 \pi i \xi \cdot x_\parallel} }_{I}\dd \xi,\\
(-2 \pi i x_\parallel)^{n^\prime} 
\nabla_{x_\parallel}^{n }	\p_{x_3}^2 q(x) 
&= \int_{\R^2}  \underline{\frac{ \pi}{{\color{red}\cancel{2}}}  \nabla_\xi^{n^\prime}
	\Big( \iota(\xi)  |\xi|  e^{-2 \pi |\xi| |x_3|} \Big) \nabla_{x_\parallel}^n e^{2 \pi i \xi \cdot x_\parallel} }_{II}\dd \xi.
\end{split}
\Ee 
We bound each of underlined terms in \eqref{Dq} as follows: for $|x_3|\geq 1$,
\begin{align}
I & \lesssim_{n, 
	\iota} \mathbf{1}_{|\xi| \geq \frac{1}{2}}
|\xi|^{|n|}  (1+ |x_3|^{|n^\prime|})  e^{- 2\pi |\xi| |x_3|}
\lesssim _{n, 
	\iota}  e^{-|\xi| } e^{- |x_3|},\notag
\\
II & \lesssim _{n, n^\prime, \iota}   \mathbf{1}_{|\xi| \geq \frac{1}{2}}   |\xi|^{|n|  } 
(1+   |\xi||x_3|^{|n^\prime|})
e^{- 2\pi |\xi| |x_3|}\lesssim _{n, n^\prime, \iota}  e^{-|\xi| } e^{- |x_3|}.\notag
\end{align}
Then we follow the argument to prove \eqref{bound:q} and derive that 
\Be\notag
| \nabla_{x_\parallel}^n  \p_{x_3} q(x)| 
+ | \nabla_{x_\parallel}^n  \p_{x_3}^2 q(x)| 
\lesssim_{n, n^\prime, \iota} 
|x_\parallel|^{-|n^\prime|}  e^{- |x_3|} \ \ \text{for  $|x_3|\geq 1$.}
\Ee
Now by choosing $|n^\prime| = N$, we conclude the second statement in (q4).

\medskip

\textit{Step 3. }We are ready to prove \eqref{claim_Fourier}. From the Poisson summation formula and \eqref{def:Q}, we obtain that 
\Be\label{q-Q}
\sum_{m \in \Z^2} q(x_\parallel+ m, x_3) = \sum_{m \in \Z^2} Q(m, x_3) e^{2 \pi i x_\parallel \cdot m} = \sum_{|m|>0} \frac{  e^{-2 \pi |m||x_3|}}{4 {\color{red}\cancel 8}\pi |m|} e^{2 \pi i x_\parallel \cdot m} .
\Ee 
Note that (q4) guarantees the absolute convergence of the first series above, and hence these identities are valid. 

On the other hand, from \eqref{def:Q}, 
\Be\label{q-b1}
\begin{split}
\sum_{m \in \Z^2}  q(x_\parallel+ m, x_3) &= q(x_\parallel, x_3) +  \sum_{|m |>0}  q(x_\parallel+ m, x_3)\\
&=  \frac{c_2}{2 {\color{red}\cancel 4}}  \frac{1}{\left( |x_\parallel|^2 + |x_3|^2\right)^{1/2}}  +  b_1 (x_\parallel, x_3)+  \sum_{|m |>0}  q(x_\parallel+ m, x_3).
\end{split} \Ee
We define (recall the definition of $q(x)$ and $b_1(x)$ in \eqref{def:Q}) 
\Be\label{def:b}
b(x):= b_1 (x) +  \sum_{|m |>0}  q(x_\parallel+ m, x_3).
\Ee
Recall again that $\sum_{|m |>0}  q(x_\parallel+ m, x_3)$ is absolutely convergent due to (q4). For the sake of convenience, we record our decomposition of $\tilde G$:
\Be\label{dec:tildeG}
\begin{split}
\tilde{G} (x_\parallel, x_3) &  = \frac{|x_3|}{2} - \sum_{m \in \mathbb{Z}^2: |m| >0} Q(m, x_3) e^{i 2 \pi m \cdot x_\parallel} \\
&= \frac{|x_3|}{2} - \sum_{ m \in \mathbb{Z}^2} q(x_\parallel + m, x_3)\\
&=  \frac{|x_3|}{2} -  \frac{c_2}{2 }  \frac{1}{\left( |x_\parallel|^2 + |x_3|^2\right)^{1/2}}  - \Big\{ b_1 (x_\parallel, x_3)  +  \sum_{|m |>0}  q(x_\parallel+ m, x_3) \Big\}\\
& =  \frac{|x_3|}{2} -  \frac{c_2}{2  }  \frac{1}{\left( |x_\parallel|^2 + |x_3|^2\right)^{1/2}} - b(x_\parallel, x_3).
\end{split}
\Ee
Here, we have used \eqref{G1} and \eqref{def:Q} at the first line; used \eqref{q-Q} at the second line; used \eqref{q-b1} at the third line; and have used \eqref{def:b} at the last line.

For all $x_3 \in \R$, the first statement of (q4) implies that {\color{red}[$\| b(\cdot, x_3) \|_{C^2 (\T^2)} \lesssim e^{-|x_3|}$. ALERT!]}

Now we study $\p_{x_3} b$. When $|x_3| \geq 1$, we use (q4). Then we have $\| \p_{x_3}b(\cdot, x_3)\|_{C^2(\R^2)}+ \| \p_{x_3}^2b(\cdot, x_3)\|_{C^2(\R^2)} \lesssim e^{-|x_3|}$ for $|x_3|\geq 1$. If $|x_3| \leq 1$, we should use the second statement of (q3). Hence, we derive that, for $i=1,2,3$, $\| \p_{x_3}b_i(\cdot, x_3)\|_{C^2(\R^2)}+ \| \p_{x_3}^2b_i(\cdot, x_3)\|_{C^2(\R^2)} \lesssim 1$ for $|x_3|\leq 1$. These concludes \eqref{b:2}.
\unhide
%
%
\end{proof}

\hide
From the evenness of $\cos(\cdot)$ and the oddness of $\sin(\cdot)$, we derive the following identities for $I_i$ in \eqref{G1}:
\begin{align}
I_1 &= 4 \sum_{ m,n \geq 1} w_{m,n}\cos (2 \pi m x_1) \cos (2 \pi n x_2)  
+ 2 \sum_{m=1}^\infty w_{m,0} (x_3) e^{i 2\pi m x_1 } 
+ 2 \sum_{n=1}^\infty w_{0,n} (x_3) e^{i 2\pi n x_2 } 
,\label{form:G_cc}\\
I_2 & = - \sum_{m,n \gtrless 0} w_{m,n}  \sin (2 \pi m x_1) \sin (2 \pi n x_2)
+\sum_{m  \gtrless 0 \gtrless n}  w_{m,n} \sin (2 \pi m x_1) \sin (2 \pi (-n) x_2)\notag\\
& =  - \sum_{m,n \gtrless 0} w_{m,n}  \sin (2 \pi m x_1) \sin (2 \pi n x_2)
+\sum_{m,n \gtrless 0}  w_{m,n}  \sin (2 \pi m x_1) \sin (2 \pi n x_2)=0,\label{form:G_ss}\\
I_3 &  = i 2 
\sum_{\substack{m \in \N \\n>0 } } 
w_{m,n}  \cos (2 \pi m x_1) \sin (2 \pi n x_2)
+ i 2\sum_{\substack{m \in \N  \\n<0 } }w_{m,n} \cos (2 \pi m x_1)  (-1)\sin (2 \pi  (-n) x_2)\notag \\
& = i2  \sum_{\substack{m \in \N \\n>0 } }  w_{m,n}  \cos (2 \pi m x_1) \sin (2 \pi n x_2)- i2  \sum_{\substack{m \in \N  \\n>0 } } w_{m,n} \cos (2 \pi m x_1) \sin (2 \pi n x_2)=0,\label{form:G_cs}\\
I_4 &  = 0,\label{form:G_sc}
\end{align}
where at \eqref{form:G_sc} we follow the similar process of \eqref{form:G_cs} but exchanging the role of $m$ and $n$.

Now we assume $|x_3|>1$. Since $m^2 + n^2 \geq 2 mn$ and $m^2 + n^2 \geq \frac{1}{2} (m+n)^2$ for $m,n \in \N$, we have 
\Be\notag
w_{m,n}(x_3) \leq \frac{1}{4 \sqrt{2} \pi mn } e^{-   \pi |x_3|m}e^{-   \pi |x_3| n}
\leq 
\frac{1}{4 \sqrt{2} \pi }e^{-   \pi |x_3|m}e^{-   \pi |x_3| n}
\ \text{for} \ m,n \in \N \backslash \{0\}.
\Ee
Then using the summation formula for a geometric series, we can easily derive that 
\Be\label{bound:I1}
|I_1| \leq \frac{1}{\sqrt{2}  \pi} \left( \frac{e^{- \pi |x_3|}}{1- e^{- \pi |x_3|}}\right)^2+ \frac{1}{\sqrt{2} \pi} \left( \frac{e^{- \pi |x_3|}}{1- e^{- \pi |x_3|}}\right) \leq 10 e^{- \pi |x_3|} \ \ \text{for} \  |x_3|\geq 1. 
\Ee

We summarize the bounds of $\bar{G}(x)$ as 
\Be\label{G:away_0}
\Big| \bar{G}(x) - \frac{|x_3|}{2}\Big| \leq 10 e^{- \pi |x_3|} \ \ \text{for} \ |x_3|\geq 1. 
\Ee

\hide

Now we focus on \eqref{form:G_cc}. Recall a formula (as an anti-derivative of some geometric series)
\Be \notag
\sum_{m=1}^\infty  \frac{1}{2 \pi m}e^{- 2\pi m |x_3|} \cos (2 \pi m x_1  )
= \frac{1}{2}|x_3| - \frac{\ln 2}{2 \pi} - \frac{1}{4 \pi } \ln \big( \sinh^2 | \pi x_3| + \sin^2  (\pi x_1)\big).
\Ee
Then we write \eqref{form:G_cc} as 
\Be
\begin{split}
I_1 &= - \sum_{m=1}^\infty \frac{e^{- 2\pi m |x_3|}}{2\pi m} \cos (2 \pi m x_1) - \sum_{n=1}^\infty  \frac{e^{- 2\pi n |x_3|}}{2\pi n} \cos (2 \pi n x_2)  \\
&  +  4 \sum_{ m,n \geq 1} w_{m,n}\cos (2 \pi m x_1) \cos (2 \pi n x_2) \\
& = -   |x_3|
\end{split}\Ee

From \eqref{G1}, \eqref{form:G_cc}- \eqref{form:G_sc} and a straightforward rearrangement, we derive that 
\Be
\begin{split}
\bar G(x) & = \frac{|x_3|}{2} + c + \sum_{\substack{m\geq 0, n \geq 0,\\ (m,n) \neq (0,0)} }  	\frac{ 1}{  \pi \sqrt{m^2 + n^2}} e^{- 2 \pi \sqrt{m^2 + n^2} |x_3|} \cos (2 \pi m x_1) \cos (2 \pi n x_2)
\end{split}
\Ee

\unhide 

\unhide
\hide
Then by usual reflecting argument, we derive the Green function in $\T^2 \times [0, \infty)$ in Proposition \ref{lemma:G}.

Define $ {G}:  \O   =  \T^2 \times [0, \infty) \mapsto \R$ as the $\T^2$-periodic Green function in the half-space satisfying Dirichlet BC:
\be \label{g2}
\Delta G (x, y) = \sum\limits_{m, n \in \Z} \delta_0 (x - y + (m, n, 0)) - \sum\limits_{m, n \in \Z} \delta_0 (\tilde{x} - y + (m, n, 0)).
\ee

\begin{proposition} \label{lem: g2}
Let  	
\textcolor{red}{  
\be \label{g1 4}    
\begin{split}
	G (x)  
	& = \frac{1}{2} |x_3| + c - \sum\limits_{(m, n) \neq (0, 0)} \frac{1}{4 \pi \sqrt{m^2 + n^2}} e^{- 2 \pi \sqrt{m^2 + n^2} |x_3|} e^{i 2\pi m x_1} e^{i 2\pi n x_2}
	\\& = \frac{1}{2} |x_3| + c - \sum\limits_{(m, n) \neq (0, 0)} \frac{1}{4 \pi \sqrt{m^2 + n^2}} e^{- 2 \pi \sqrt{m^2 + n^2} |x_3|} \cos (2\pi m x_1 + 2\pi n x_2).
\end{split}
\ee
}

For an arbitrary function $\rho$ defined in $\T^2 \times (-\infty, \infty)$ and periodic horizontally in $\T^2$, define 
\Be
\phi_\rho(x) : = \int_{\T^2}G(x-y) \rho(y) \dd y. 
\Ee
Then $\phi_\rho(x)$ is well-defined in $\T^2 \times (-\infty, \infty)$ and periodic horizontally in $\T^2$. Moreover, 
\Be
\Delta \phi_\rho = \rho \ \ \text{in}  \ \  \T^2 \times [0, \infty).
\Ee

\end{proposition}

\begin{proof}

Suppose $x = (x_1, x_2, x_3) \in \T^2 \times [0, \infty)$. From Lemma \ref{lem: g1}, we have
\be
\begin{split} \label{g2 1}
& \sum\limits_{m, n \in \Z} \delta_0 (x - y + (m, n, 0)) - \sum\limits_{m, n \in \Z} \delta_0 (\tilde{x} - y + (m, n, 0))
\\& = \sum\limits_{m, n \in \Z} \delta_0 (x_3 - y_3) e^{i 2\pi m (x_1 - y_1)} e^{i 2\pi n (x_2 - y_2)} - \sum\limits_{m, n \in \Z} \delta_0 (- x_3 - y_3) e^{i 2\pi m (x_1 - y_1)} e^{i 2\pi n (x_2 - y_2)}
\end{split}
\ee

Therefore, we can write $\tilde{G} (x, y)$ as
\be \label{g2 2}
\begin{split}
\tilde{G} (x, y) 
& = \frac{1}{2} |x_3 - y_3| - \sum\limits_{(m, n) \neq (0, 0)} \frac{1}{4 \pi \sqrt{m^2 + n^2}} e^{- 2 \pi \sqrt{m^2 + n^2} |x_3 - y_3|} e^{i 2\pi m (x_1 - y_1)} e^{i 2\pi n (x_2 - y_2)}
\\& \ \ \ \ - \frac{1}{2} |-x_3 - y_3| + \sum\limits_{(m, n) \neq (0, 0)} \frac{1}{4 \pi \sqrt{m^2 + n^2}} e^{- 2 \pi \sqrt{m^2 + n^2} |-x_3 - y_3|} e^{i 2\pi m (x_1 - y_1)} e^{i 2\pi n (x_2 - y_2)}
\\& = \frac{1}{2} |x_3 - y_3| - \frac{1}{2} |-x_3 - y_3|
\\& \ \ \ \ - \sum\limits_{(m, n) \neq (0, 0)} \frac{1}{4 \pi \sqrt{m^2 + n^2}} e^{- 2 \pi \sqrt{m^2 + n^2} |x_3 - y_3|} \cos (2\pi m (x_1 - y_1) + 2\pi n (x_2 - y_2))
\\& \ \ \ \  + \sum\limits_{(m, n) \neq (0, 0)} \frac{1}{4 \pi \sqrt{m^2 + n^2}} e^{- 2 \pi \sqrt{m^2 + n^2} |-x_3 - y_3|} \cos (2\pi m (x_1 - y_1) + 2\pi n (x_2 - y_2)).
\end{split}
\ee
where the last equality follows from the symmetry.

\textcolor{red}{stop here!}

\end{proof}

\subsection{Green function on $\T^2 \times (-\infty, \infty)$}

\begin{lemma} \label{lem: g1}
\be
G(x) = \textcolor{red}{???}.
\ee
\end{lemma}

\begin{proof}\textit{Step 1. } 	The Green function in 2D is derived in \cite{AmmariL}.  
In this step, following the argument of its proof, we claim that
\Be\label{G2}
\begin{split}
G(x) =& \frac{|x_3|}{2} + c + \sum_{m =0}^\infty \cos (2\pi m x_1)  \sum_{n=1}^\infty \frac{e^{- 2 \pi \sqrt{m^2 +  n^2} |x_3|}}{\pi \sqrt{m^2+ n^2}}  \cos (2\pi n x_2)\\
&+  \sum_{n =0}^\infty  \cos (2\pi n x_2) \sum_{m=1}^\infty \frac{e^{- 2 \pi \sqrt{m^2 +  n^2} |x_3|}}{\pi \sqrt{m^2+ n^2}}   \cos (2\pi m x_1).
\end{split}	\Ee

Finally we prove our claim \eqref{G1}, using \eqref{G1} and \eqref{form:G_cc}-\eqref{form:G_sc}.

\medskip

\textit{Step 2. } Using the residue theorem (Chapter 70 in \cite{BC}), we compute that 
\Be
\begin{split}
\sum_{n_1=0}^{+\infty} \cos(2 \pi n_1 x_1 ) \sum_{n_2 =1}^{+\infty} \frac{
	e^{- 2\pi  |x_3| \sqrt{n_1^2 + n_2^2 } }
}{\sqrt{n_1^2 + n_2^2}}\cos (2 \pi n_2 x_2)
\end{split}
\Ee

Using branches of $z^{1/2}$

On the complex plane, 
\Be
(z^2-1)^{1/2} =(z-1)^{1/2}(z+1)^{1/2}
\Ee
By specify the branches for $(z-1)^{1/2}$ and $(z+1)^{1/2}$, we obtain a specific branch of $(z^2-1)^{1/2} $. 

For given $n \in \N$, $x_2 \in \R$, $|x_3|\geq 0$, we are seeking a branch of the following form 
\Be
\frac{e^{-2 \pi |x_3| \sqrt{n^2 + z^2} + i 2 \pi x_2 z} }{\sqrt{n^2+ z^2}}\frac{ 1}{e^{i 2\pi z} -1 } .
\Ee
Note that this expression is not single valued hence not analytic. However, this function is analytic (single-valued, in particular) away from zeros of $e^{i 2\pi z} -1$ and a branch cut of $\sqrt{n^2 + z^2}$ (see Chapter 98 in \cite{BC}). Clearly the zeros of $e^{i 2\pi z} -1$ are all real integer numbers $\Z$. 

For $\sqrt{n^2 + z^2}$, we choose a branch cut to be the line segment connecting $+i |n|$ and $- i |n|$ in $\mathbb{C}$:
\Be\label{br_cut}
BC(|n|):= \{z \in \mathbb{C}: \text{Re} (z) =0 \ \text{and} \ -|n | \leq \text{Im} (z) \leq |n|  \}.
\Ee
Away from this branch cut, we specify a branch of $\sqrt{n^2 + z^2}$ as follows: for $z \in \mathbb{C} \backslash BC(N)$, 
\Be\label{branch}
\sqrt{n^2 + z^2} = \sqrt{r_1 r_2} e^{i\frac{\theta_1+ \theta_2}{2}} \ \ \text{for} \ 
r_1, r_2>0,   r_1 + r_2 >2,  \ \text{and} \ - \frac{\pi}{2} \leq \theta_1, \theta_2< \frac{3 \pi}{2},
\Ee
where 
\Be\label{r&theta}
\begin{split}
r_1 &= \big| z- i |n| \big|>0, \\
r_2  &= \big| z+ i |n| \big|>0,\\
\tan \theta_1 &= \frac{\text{Im} (z- i |n|)}{\text{Re}(z- i |n|)}= \frac{\text{Im} (z)-  |n|}{\text{Re}(z )}  \ \ \text{and} \ \ - \frac{\pi}{2} \leq \theta_1< \frac{3 \pi}{2},\\ 
\tan \theta_2 &= \frac{\text{Im} (z+ i |n|)}{\text{Re}(z+ i |n|)}= \frac{\text{Im} (z)+ |n|}{\text{Re}(z )}  \ \ \text{and} \ \  -\frac{\pi}{2} \leq \theta_2< \frac{3 \pi}{2}.
\end{split}
\Ee

First we define a dumbbell-like open set containing the branch cut $BC(N)$:
\Be\begin{split}\label{BC^e}
BC^\e(|n|) =& \{z \in \mathbb{Z}: \big|z- i |n|\big|<\e\} \cup \{ z \in \mathbb{Z}:  \big|z+ i |n|\big|<\e\} \\
&\cup \Big\{ z \in \mathbb{Z}:\exists \tilde{z} \in BC(N) \text{ such that } |z- \tilde{z}| < \frac{\e}{2} \Big\}.
\end{split}\Ee
Now we define 
\Be\label{U(N)}
U(N, |n|) = \Big\{z \in \mathbb{C}: |z| < N+ \frac{1}{2}\Big\}   \backslash BC^\e(|n|) .
\Ee
\begin{lemma}
Let $n \in \N$, $x_2 \in \R$, and $|x_3|\geq 0$. We define 
\Be\label{GreenF}
\mathfrak{g} (z;n, x_2, |x_3|) := \frac{e^{-2 \pi |x_3| \sqrt{n^2 + z^2} + i 2 \pi x_2 z} }{\sqrt{n^2+ z^2}}\frac{ 1}{e^{i 2\pi z} -1 } \ \ \text{in} \ \  \mathbb{C} \backslash BC^\e (|n|),
\Ee
where a branch of $\sqrt{n^2 + z^2}$ is given in \eqref{branch}. Then $\mathfrak{g} (z;n, x_2, |x_3|)$ is analytic in $\mathbb{C} \backslash BC^\e (|n|)$. Moreover, $\mathfrak{g} (z;n, x_2, |x_3|)$ only has singular points at 
\Be\label{pole}
\{k \in \mathbb{N}: 0<|k|   \}.
\Ee 
All singular points are simple poles and the residues at each pole is given by 
\Be\label{Res}
\text{Res}_{z=k}\mathfrak{g} (z;n, x_2, |x_3|) = \frac{1}{i 2\pi }  \frac{e^{-2 \pi |x_3| \sqrt{n^2 + k^2} + i 2 \pi x_2 k} }{\sqrt{n^2+ k^2}} \ \ \text{for } \ k \in \mathbb{N}, \ 0< |k|  . 
\Ee
\end{lemma}

From the residue theorem (Chapter 70 in \cite{BC}), we derive that 
\Be
\begin{split}
\int_{\mathcal{C}_1(N, |n|, \e) + \cdots + \mathcal{C}_{12}} 
\mathfrak{g} (z;n, x_2, |x_3|) \dd z  &= i 2 \pi  \sum_{ k \in \mathbb{N}:  0<|k| \leq N}\text{Res}_{z= k} \mathfrak{g} (z;n, x_2, |x_3|) \\
& = 2\sum_{k=1}^{N}  \frac{e^{-2 \pi |x_3| \sqrt{n^2 + k^2} + i 2 \pi x_2 k} }{\sqrt{n^2+ k^2}} .
\end{split}
\Ee

Now we compute $ \int_{\mathcal{C}_1 + \cdots + \mathcal{C}_{12}} 
\mathfrak{g} (z;n, x_2, |x_3|) \dd z$.

\Be
\int_{\mathcal{C}_4(N, |n|, \e)} = \int 
\Ee

\Be
\int_{\mathcal{C}_{10}(N, |n|, \e)} =
\Ee

where the second last equality follows from the symmetry.

Here, we first consider $x_3 > 0$,
\be \label{g1 5}
\begin{split}
& \ \ \ \ \sum\limits_{m, n \geq 0} \frac{1}{\sqrt{m^2 + n^2}} e^{- 2 \pi \sqrt{m^2 + n^2} |x_3|} \cos (2\pi m x_1 + 2\pi n x_2)
\\& \leq \sum\limits_{(m, n) \neq (0, 0)} \frac{1}{\sqrt{m^2 + n^2}} e^{- 2 \pi \sqrt{m^2 + n^2} x_3} 
\\& = \sum\limits_{m > 0} \frac{1}{m} e^{- 2 \pi m x_3} + \sum\limits_{n > 0} \frac{1}{n} e^{- 2 \pi n x_3} 
+ \sum\limits_{m, n > 0} \frac{1}{\sqrt{m^2 + n^2}} e^{- 2 \pi \sqrt{m^2 + n^2} x_3}
\\& = 2 \underbrace{\sum\limits_{m > 0} \frac{1}{m} e^{- 2 \pi m x_3}}_{\eqref{g1 5}^*}
+ \underbrace{\sum\limits_{m, n > 0} \frac{1}{\sqrt{m^2 + n^2}} e^{- 2 \pi \sqrt{m^2 + n^2} x_3}}_{\eqref{g1 5}^{**}}.
\end{split}
\ee

Using the summation identity, we deduce
\be
\sum\limits_{m > 0} \frac{1}{m} e^{- 2 \pi m x_3}
= \frac{1}{2} x_3 - \frac{\ln (2)}{2 \pi} - \frac{1}{4 \pi} \ln (\sinh^2 (\pi x_3)).
\ee
For $0 < x_3 \ll 1$, from the Taylor expansion, we have
\be
\eqref{g1 5}^*
\lesssim - \frac{1}{2 \pi} \ln (\pi x_3).
\ee
On the other hand, for $x_3 \gg 1$,
\be
\begin{split}
\sum\limits_{m > 0} \frac{1}{m} e^{- 2 \pi m x_3}
\leq \sum\limits_{m > 0} e^{- 2 \pi m x_3}
= \frac{e^{- 2 \pi x_3}}{1 - e^{- 2 \pi x_3}}.
\end{split}
\ee
Thus, we have
\be
\eqref{g1 5}^{*}
\lesssim - \ln (x_3) \times \mathbf{1}_{0 < x_3 \ll 1} + e^{- 2 \pi x_3} \times \mathbf{1}_{x_3 \gg 1}.
\ee

Since $2 m^2 + 2 n^2 \geq (m + n)^2$, we obtain
\be
\begin{split}
\eqref{g1 5}^{**} 
& \leq \sum\limits_{m, n > 0} \frac{1}{\sqrt{2mn}} e^{- \sqrt{2} \pi (m+n) x_3}
\\& \leq \Big( \sum\limits_{m > 0} \frac{1}{\sqrt{m}} e^{- \sqrt{2} \pi m x_3} \Big) \times
\Big( \sum\limits_{n > 0} \frac{1}{\sqrt{n}} e^{- \sqrt{2} \pi n x_3} \Big)
\\& = \Big( \sum\limits_{m > 0} \frac{1}{\sqrt{m}} e^{- \sqrt{2} \pi m x_3} \Big)^2.
\end{split}
\ee
Under direct computation,
\be
\begin{split}
\sum\limits_{m > 0} \frac{1}{\sqrt{m}} e^{- \sqrt{2} \pi m x_3}
& \leq \int^{\infty}_{0} \frac{1}{\sqrt{m}} e^{- \sqrt{2} \pi m x_3} \dd m
\\& \lesssim \frac{1}{\sqrt{x_3}} \int^{\infty}_{0} \frac{1}{\sqrt{t}} e^{- t} \dd t = \frac{\Gamma (\frac{1}{2})}{\sqrt{x_3}},
\end{split}
\ee
On the other hand, for $x_3 \gg 1$,
\be
\begin{split}
\sum\limits_{m > 0} \frac{1}{\sqrt{m}} e^{- \sqrt{2} \pi m x_3}
& \leq \sum\limits_{m > 0} e^{- \sqrt{2} \pi m x_3}
= \frac{e^{- \sqrt{2} \pi x_3}}{1 - e^{- \sqrt{2} \pi x_3}}.
\end{split}
\ee
Thus, we have
\be
\eqref{g1 5}^{**}
\lesssim \frac{1}{x_3} \times \mathbf{1}_{0 < x_3 \ll 1} + e^{- 2 \pi x_3} \times \mathbf{1}_{x_3 \gg 1}.
\ee

\end{proof}
\unhide

\hide
\subsection{Poisson equation}
Recall that $\frac{-1}{4 \pi} \frac{1}{|x|}$ is the Green function of the Laplacian in $\R^3$ (Section 2.2.1. of Evans). Following Section 2.2.4 of Evans, the Green function of the Laplacian in the half space $\R^3_+= \{(x_1,x_2, x_3) \in \R^3: x_3 >0\}$ is 
\Be
G(x,y) =- \frac{1}{4 \pi} \frac{1}{|x-y|} + \frac{1}{4 \pi} \frac{1}{|\tilde{x}-y|} {\color{red}(CHECK)}
\Ee
where $\tilde x = (x_1, x_2, - x_3)$. 

Now we find the form of $\phi$ when $\rho$ is periodic in $\T^2$:
\Be
\Delta \phi = \rho  \ \ in \ \R^3_+ \ \ \ \phi|_{x_3=0}=0\Ee

\textit{[[Formally 
\Be
\phi  (x) = \int_{\R^3_+} \Big(- \frac{1}{4 \pi} \frac{1}{|x-y|} + \frac{1}{4 \pi} \frac{1}{|\tilde{x}-y|} \Big) \rho (y)  \dd  y .  
\Ee
Since $\rho$ is periodic in $(x_1, x_2) \in \T^2$ and 
\Be
\phi  (x) = \int_{\R^3_+} \Big(- \frac{1}{4 \pi} \frac{1}{|y|} + \frac{1}{4 \pi} \frac{1}{|y  - 2x_3|} \Big) \rho (x-y)  \dd  y ,
\Ee
$\phi  (x)$ is periodic in $(x_1, x_2) \in \T^2$ . However this is not rigorous as the integration is not summable.]]}

Instead we should follow the idea of Batt-Rein 1993. See their Proposition 2.1. I guess using their scheme, we might be able to obtain 
\begin{lemma}{\color{red}(CHECK)}
and 
\Be
\phi(x) = \int_{\T^2 \times (0, \infty)} G(x,y) \rho(y) \dd y 
\Ee
A key is that the integration will be taken in $\T^2 \times (0, \infty)$ not $\R_+^3$. Check that the above integral form is well-defined under some decay condition of $\rho(y)$ as $y_3 \rightarrow \infty$. 
\end{lemma}

Note that when $x,y \in \O = \T^2 \times \R_+$,
\Be
|x-y|  \leq  |\tilde x - y|
\Ee
Hence 
\Be
\frac{1}{|x-y|} \geq \frac{1}{  |\tilde x - y|}
\Ee
Therefore 
\Be
G(x,y) \leq \frac{1}{2 \pi } \frac{1}{|x-y|} + G_0 (x,y)
\Ee

\begin{lemma}\label{lem:elliptic_est:C^1}
\Be\label{elliptic_est:C1}
\|\nabla_x \Delta_0^{-1} \rho_f (s,\cdot ) \|_{L^\infty_x} \lesssim 
\| \rho_f (s,\cdot) \|_{L^\infty_x} ,
\Ee 
\end{lemma}
\begin{proof}
Roughly we need to bound the following term:
\Be
\begin{split}
&	\| \rho_f (s,\cdot) \|_{L^\infty_x} \int_{\T^2} \int_{\R_+}	 \frac{\dd y_3 }{|x-y|^2}\dd y_\parallel  \\
& =	\| \rho_f (s,\cdot) \|_{L^\infty_x} \int_{\T^2} \int_{0}^\infty  \frac{ \dd y_3}{
	|x_3- y_3|^2 + |y_\parallel |^2 
} dy_\parallel \\
& \leq  2\| \rho_f (s,\cdot) \|_{L^\infty_x}  \int_{\T^2}  \frac{1}{|y_\parallel|} \tan^{-1} \frac{y_3}{|y_\parallel|}\Big|_{y_3=0}^{y_3=\infty}   \dd y_\parallel\\
& = \pi\| \rho_f (s,\cdot) \|_{L^\infty_x}   \int_{\T^2}  \frac{ \dd y_\parallel}{|y_\parallel|} \\
&\leq  10 \| \rho_f (s,\cdot) \|_{L^\infty_x} .
\end{split}	\Ee

\end{proof}

\subsection{Bootstrap argument} First we assume that

For simplicity we assume that $m_+ =1= m_-$ and $e_+ =1= -e_-$, so that $Z_+ = Z_-$ and $x_{\b}^+ = x_{\b}^-$.

The key is to estimate $\rho_h$. Note that 
\begin{align}
h_\pm (x,v) &= G_\pm (x_{\b}^{h, \pm} (x,v) , v^{h , \pm}_{\b } (x,v) ),\\
\rho_h(x) & = \int_{\R^3}  [ w^h  h_+ (x,v)  - w^h h_- (x,v)] \frac{1}{w^h (x,v)} \dd v \\
& \leq  
\int_{\R^3}  \Big|  w^h G_+ (x_{\b }^{h, + } (x,v) , v_{\b}^{h, +} (x,v) ) -  w^h G_- (x_{\b}^{h,-} (x,v) , v_{\b }^{h,-} (x,v) )\Big| \frac{1}{w^h (x,v)}  \dd v\\
& \leq \| w^h  G_+ - w^h G_-  \|_{L^\infty(\gamma_-)}
\int_{\R^3}  \frac{1 }{w^h  (x,v)}\dd v.
\end{align}

Recall $w^h  (x,v)$ in \eqref{w^h}. If \eqref{Bootstrap} and \eqref{eqtn:Dphi} hold then 
\Be
\begin{split}
|\phi_h (x_\parallel, x_3)|  =   \Big|\phi_h  (x_\parallel, 0) + \int ^{x_3}_0 \p_{3} \phi_h (x_\parallel, y_3) \dd y_3 \Big|    \leq  \frac{g}{2}x_3. 
\end{split}
\Ee
Therefore we deduce that 
\Be
w^h (x,v)  \geq e^{ \beta \big(\frac{|v|^2}{2} - |\phi_h(x)| + g   x_3\big)} 
\geq e^{ \beta \big(\frac{|v|^2}{2}   + \frac{g}{2}   x_3\big)} ,
\Ee
which implies 
\Be
\int_{\R^3}  \frac{1 }{w^h  (x,v)}\dd v  \leq  \int_{\R^3}  
e^{- \beta \big(\frac{|v|^2}{2}   + \frac{g}{2}   x_3\big)}
\dd v  = \frac{(2\pi)^{3/2}}{\beta^{3/2}} e^{- \frac{\beta g}{2} x_3 }.
\Ee
\hide
\begin{align}
&\int_{\R^3}  \frac{1}{w_h (x_\b (x,v), v_\b (x,v))}\dd v\\
&= \int_{\R^3}  e^{- \frac{\beta}{2} | v_{\b,3} (x,v)|^2 } e^{- \frac{\beta}{2} | v_{\b,\parallel} (x,v)|^2 }  \dd v 
\end{align}
Recall \eqref{est:tB}. Following the estimate below \eqref{est:1/w}, we might have 
\Be\label{est:v_b,||^h}
- |v_{\b, \parallel}^h (x,v)|^2 \leq - \frac{|v_\parallel|^2}{2} + \frac{16 |v_{\b, 3}^h|^2}{g^2} \| \nabla_x \phi_h \|_\infty^2.
\Ee
Now we return to  and derive that 
\begin{align}
&\int_{\R^3}  \frac{1}{w_h (x_\b (x,v), v_\b (x,v))}\dd v\\
&= \int_{\R} \int_{\R^2}  e^{- \frac{\beta}{2}
\big( 1-  \frac{16  \| \nabla_x \phi_h \|_\infty^2}{g^2}  \big)
| v_{\b,3}^h (x,v)|^2 } e^{- \frac{\beta}{4} | v_{\parallel}|^2 }  \dd v_\parallel  \dd v_3\\
& \lesssim \frac{1}{\beta} \int_{\R}  e^{- \frac{\beta}{2}
\big( 1-  \frac{16  \| \nabla_x \phi_h \|_\infty^2}{g^2}  \big)
| v_{3} |^2 }  \dd v_3\\
& \lesssim \frac{1}{\beta^{3/2}} 	\Big( 1-  \frac{16  \| \nabla_x \phi_h \|_\infty^2}{g^2}  \Big)^{-1/2},
\end{align}
where we have used $|v_3| \leq  | v_{\b,3}^h (x,v)|$.\unhide

\unhide

\section{Asymptotic Stability Criterion}\label{sec:AS}


The goal of current section is to give a proof of Theorem \ref{theo:AS}. In this section we always assume all conditions of Theorem \ref{theo:AS} hold. For example, global-in-time self-consistent potentials $\Phi(\cdot), \Psi(t, \cdot )  \in C^{1}(\bar \O) \cap C^2 (\O)$, and $f(t,x,v)$ is a global-in-time Lagrangian solution in the sense of Definition \ref{def:mild} and \eqref{Lform:f}. We also assume that $\nabla_v h \in L^\infty$. Recall the Lagrangian formulation of $f $ solving \eqref{eqtn:f}-\eqref{Poisson_f}:
\Be\label{form:f}
\begin{split}
f   (t,x,v)   
&= \mathcal{I}  (t,x ,v) + \mathcal{N}  (t,x,v ), 
\end{split}
\Ee 
\vspace{-16pt}
where 
\begin{align} 
\mathcal{I}(t,x,v ) 	& :=  
\mathbf{1}_{t <  \tB (t,x,v)}	f _{ 0} ( \Zz   (0;t,x,v) ) 
,\label{form:I} \\
\mathcal{N}  (t,x,v ) 	&  :=
\int^t_
{\max\{0, t- \tB  (t,x,v) \}} 
\nabla_x  \Psi  
(s, \X  (s;t,x,v)) \cdot \nabla_v h 
( \Zz    (s;t,x,v) )
\dd s 
.\label{form:N}
\end{align}

Recall $b (t,x)$ in \eqref{def:flux} and the continuity equation \eqref{cont_eqtn}. Assume that $\nabla_x \cdot b \in L^\infty_{{loc}} (\R_+ \times \O)$. Then a weak solution $\varrho$ of the continuity equation is absolutely continuous in time. Therefore we can take a time derivative to the Poisson equation \eqref{Poisson_f}. This leads  to
\Be
\begin{split}\label{identity:Psi_t}
\p_t \Psi (t,x)  = \eta \Delta_0^{-1} \p_t \varrho (t,x)  = -  \eta \Delta_0^{-1}  (\nabla_x \cdot b ) (t,x).\\
\end{split}
\Ee

Recall the dynamic weight function $\w_\beta$ in \eqref{w^F}. 
Using \eqref{dDTE} and \eqref{identity:Psi_t}, we have that 
\Be\label{dEC}
\begin{split}
&\frac{d}{ds} \big(
|\V  (s;t,x,v)|^2 
+  2\Phi(\X (s;t,x,v)) 
+ 2 \Psi (s, \X  (s;t,x,v))
+ 2 g  \X _{  3} (s;t,x,v)
\big)\\
& = 
2 \p_t \Psi (s,\X  (s;t,x,v))  = 2  \Delta_0^{-1} \p_t \varrho (s,\X  (s;t,x,v))  = -2  \Delta_0^{-1} ( \nabla_x \cdot b )(s,\X  (s;t,x,v)).
\end{split}
\Ee

The forcing term $-2  \Delta_0^{-1} ( \nabla_x \cdot b )(s,\X  (s;t,x,v))$ is bounded pointwisely if a distribution $f$ decays fast with respect to $v$ and $x_3$ in $L^\infty$: 
\begin{lemma}\label{lem:Ddb}Assume that $f$ and $b$ are related as in  \eqref{def:flux}. Suppose $b \in L^\infty_{loc} (\R_+; C^1 (\bar{\O}))$. Then we have that 
\vspace{-10pt}
\Be\label{est:D^-1 Db}
\begin{split}
\|	 \Delta_0^{-1}  (\nabla_x \cdot b  (t,x)) \|_{L^\infty (\O)} \lesssim 
\frac{1+ \frac{1}{  \beta g}}{  \beta^2 }   \|   e^{  \frac{\beta}{2} (|v|^2 + gx_3)}  f (t) \|_{L^\infty(\bar\O)} .
\end{split}
\Ee
\end{lemma}

\begin{proof}
Recall the Green function $G(x,y)$ constructed in Lemma \ref{lemma:G}. By the integration by parts, we derive that 	\vspace{-10pt}	\Be
\begin{split}\label{D^-1 Db}
\Delta_0^{-1}  (\nabla  \cdot b ) (t,x)   &=    \int_{\T^2 \times \R_+} G(x,y) \nabla_y  \cdot b (y) \dd y \\
&=  -\int_{\T^2 \times \R_+} b (y)   \cdot  \nabla_y G(x,y)     \dd y
- \int_{\T^2} G(x, y_\parallel, 0) b_3(y_\parallel, 0) \dd y_\parallel
.
\end{split}
\Ee

From \eqref{GreenF}, we have $G(x, y_\parallel, 0)= - \mathcal{G}(x,y_\parallel, 0)$ at $y_3=0$. From \eqref{b0:3}, when $|y_3| \leq 1$ then $ \mathcal{G}(x,y_\parallel, 0) =  \mathcal{G}_0(x,y_\parallel, 0)$. We also	note that 
\Be\begin{split}\notag
| b (t,x) |  & 
\leq \int_{\R^3} |v| |f (t,x,v)| \dd v 
\leq  \left(e^{-\frac{  \beta  }{2} g x_3} 
\int_{\R^3} |v| e^{- \frac{\beta}{2}|v|^2} \dd v \right)
\|   e^{  \frac{\beta}{2} (|v|^2 + gx_3)}  f (t) \|_{L^\infty (\bar \O)}\\
& \leq  \frac{ 8 \pi e^{-\frac{  \beta  }{2} g x_3}  }{\beta^2}
\|   e^{  \frac{\beta}{2} (|v|^2 + gx_3)}  f (t) \|_{L^\infty (\bar \O)}.
\end{split}\Ee
Now we utilize Lemma \ref{lemma:G}, and follow the proof of Lemma \ref{lem:rho_to_phi} to conclude the lemma.
\end{proof}

In the stability analysis, it is important to compare weight functions along the characteristics.
\begin{lemma}\label{lem:w/w}Suppose the assumption \eqref{Bootstrap} holds. Recall $w_\beta(x,v)$ in \eqref{w^h}, $\w_\beta (t,x,v )$ in \eqref{w^F}, and $(\X  , \V )$ solving \eqref{ODE_F}. Then, for $s , s^\prime\in [t-\tB (t,x,v), t +t_{\mathbf{F} } (t,x,v)
]$,  
\Be
\begin{split}\label{est:1/w_h}
\frac{\w  _\beta  (s^\prime,\Zz  (s^\prime;t,x,v) )}{\w   _\beta (s ,\Zz  (s ;t,x,v))}
&\leq 
e^{ \frac{8\beta}{g}  	\| \Delta_0^{-1} (\nabla_x \cdot b ) \|_{L^\infty_{t,x}}   \sqrt{|v_3|^2 + g x_3}
},
\\
\frac{1}{\w  _\beta  (s,\Zz   (s;t,x,v))}  
&	\leq
e^{  \frac{64 \beta}{g} 	\| \Delta_0^{-1} (\nabla_x \cdot b ) \|_{L^\infty_{t,x}}  ^2 }
e^{-\frac{\beta}{2}|v|^2}  
e^{-\frac{\beta}{2} g x_3}
,
\end{split}\Ee
and \vspace{-5pt}
\Be\label{est:1/w}
\frac{1}{w  _\beta  ( \Zz (s;t,x,v))}   \leq 
e^{\frac{16^2 \beta }{2 g^2}
\| \Delta_0^{-1} (\nabla_x \cdot b ) \|_{L^\infty_{t,x}}^2	
}
e^{- \frac{\beta}{4}|v|^2 } e^{- \frac{\beta g}{4} x_3}.
\Ee
Here, we have used the notation $L^\infty_{t,x}$ defined in \eqref{notation}.
\end{lemma} 
\begin{proof} The proof follows \eqref{dEC} and \eqref{est:D^-1 Db}. For the detail, we refer to the proof of Lemma \ref{lem:w/w_ell}.
\end{proof}

\hide

Taking a $v-$integration to the form of $f_+^{\ell+1} - f_-^{\ell+1}$ in \eqref{form:f^ell}, we get that 
\begin{align}\label{form:rho}
\rho ^{\ell+1} (t,x)=\mathcal{I}^{\ell+1}(t,x) + \mathcal{N}^{\ell+1} (t,x)
,
\end{align}
where
\Be
\begin{split}
\mathcal{I} ^{\ell+1} (t,x) 
=  \int_{\R^3} 
\mathbf{1}_{t \leq t_\b^{\ell+1, \pm} (t,x,v)}
[f_{0,+ } - f_{0,- }]
(  \mathcal{Z}^{\ell+1}(0;t,x,v) ) \dd v, \label{rhof1}
\end{split}\Ee 
\Be
\begin{split}
&\mathcal{N} ^{\ell+1} (t,x)\\
& = \int_{\R^3} \int^t_{{\max\{0, t- \tb^{\ell+1 } (t,x,v) \}}   } \nabla_x\Delta^{-1}_0 \rho ^\ell  (s, \X^{\ell+1}  (s;t,x,v)) \cdot  [ \nabla_v h_+ - \nabla_v h_-]  
( \mathcal{Z} ^{\ell+1}  (s;t,x,v) ) 
\dd s 
\dd v.\label{rhof2}
\end{split}\Ee
\unhide


\begin{lemma}\label{prop:decay}
Suppose \eqref{Bootstrap} and \eqref{Bootstrap_f} hold. Set
\Be\label{lambda_infty}
\lambda_\infty = \frac{g^2 \beta}{2^6} \ln \left(2+ \frac{g \beta^2}{2^{\frac{17}{2} + \frac{2}{g}}
\pi^{\frac{3}{2}} \| w_\beta \nabla_v h \|_{L^\infty(\O)}	
}\right).
\Ee \hide

$\lambda_\infty$ in \eqref{lambda_infty}. and 
\Be\label{Bootstrap_f}
\sup_{0 \leq t \leq T} \|
e^{ \frac{\beta}{2} (|v|^2 + g x_3)}
f(t) \|_{L^\infty(\O \times \R^3)}
\leq \frac{ ( \ln 2 )^{\frac{1}{2}}g^{\frac{1}{2}} \beta^{\frac{3}{2}}}{  64 \pi  (1+ \frac{1}{\beta g})}.
\Ee\unhide
Then   
\begin{align}
\mathcal{I}  (t,x,v)  &\leq 	2 e^{-  \frac{\beta}{4}  ( |v|^2 + g  x_3)} 
e^{- \lambda_\infty t}
e^{ \frac{16 \lambda_\infty^2}{\beta g^2}}	  \|\w_{\beta, 0 }   f_{0  } \|_{L^\infty (\O \times \R^3)}  ,\label{est:I}\\
\mathcal{N}  (t,x,v)  & 
\leq   	 e^{- \frac{\beta}{8} |v|^2} e^{- \frac{\beta g}{8} x_3}
e^{-\lambda_\infty t}  
\frac{8 \cdot 4^{1/g}}{g \beta^{1/2}}  
e^{ \frac{16^2 \lambda_\infty^2}{ 4 g^2 \beta}}  	\sup_{s \in [0, t] } \| e^{\lambda_\infty s} \varrho (s) \|_{L^\infty(\O)} \| w_\beta \nabla_v h \|_{L^\infty (\O \times \R^3)}
.\label{est:N}
\end{align}

\hide
\Be
\mathcal{I}^{\ell+1}  (t,x) \leq e^{- \frac{\beta}{4} \frac{g^2 t^2}{C^2}}  \frac{1}{\beta^{3/2}}
\Ee \unhide
\end{lemma}
\begin{proof} 
Using Lemma \ref{lem:Ddb}, applying \eqref{D^-1 Db} to \eqref{Bootstrap_f}, we derive that 
\Be\label{bound:e^Ddb}
\begin{split}
e^{ \frac{64 \beta}{g}  \sup_{0 \leq t \leq T}	\| \Delta_0^{-1} (\nabla_x \cdot b ) \|_{L^\infty   (\O)}  ^2  }
&  \leq 
e^{\frac{64  (8 \pi )^2 (1+ \frac{1}{\beta g})^2}{g\beta^3} 
	\sup_{0 \leq t \leq T}	 \|  e^{\frac{\beta}{2} (|v|^2+ g x_3) }
	f(t) 
	\|_{L^\infty (  \O  \times \R^3)}  ^2}	
\leq 
2.
\end{split}
\Ee

First we prove \eqref{est:I}. For that, we will apply Lemma \ref{lem:tb}, Lemma \ref{lem:w/w} (since \eqref{Bootstrap} holds), and Lemma \ref{lem:Ddb}. Then we derive that \hide
\Be
\begin{split}
\eqref{form:f^ellI}& \leq \|\w_{\beta, 0 }   f_0 \|_{L^\infty}  \int_{\R^3}  \frac{ \mathbf{1}_{t \leq t_\mathbf{B}^{\ell+1}  (t,x,v)}}{w_{ {\beta} / {2}} (\Zz^{\ell+1} (0;t,x,v) )  }  \dd v\\
& \leq    \|\w_{\beta, 0 }   f_0 \|_{L^\infty}  \int_{\R^3}  
\mathbf{1}_{  
	t \leq \frac{2}{g} \big( \sqrt{|v_3|^2 + g x_3} - v_3\big)
}
\dd v\\
& \leq \frac{1}{\beta^{3/2}} e^{- \frac{\beta g}{4}  x_3 }
\mathbf{1}_{ gt/4 \leq   \sqrt{|v_3|^2 + g x_3}}
e^{- \frac{\beta}{4}  ( |v|^2 + g  x_3)}\\
& \leq  \frac{1}{\beta^{3/2}} e^{- \frac{\beta g}{4}  x_3 } e^{-   \frac{ \beta g^2 }{64} t^2} 
\end{split}	\Ee
\unhide
\Be
\begin{split}\label{bound:tB/w}
&	\mathcal{I} (t,x,v)	/  \|\w_{\beta, 0 }   f_{0  } \|_{L^\infty (\O \times \R^3)}   \leq    \frac{ \mathbf{1}_{t \leq \tB  (t,x,v)}}{\w _\beta (0,\X  (0;t,x,v),\V  (0;t,x,v))}  \\
& \leq 
\mathbf{1}_{ gt/4 \leq   \sqrt{|v_3|^2 + g x_3}}
e^{  \frac{64 \beta}{g} 	\| \Delta_0^{-1} (\nabla_x \cdot b ) \|_{L^\infty_{t,x}}  ^2 }
e^{-\frac{\beta}{2}|v|^2}  
e^{-\frac{\beta}{2} g x_3} 
\leq  2 
e^{- \frac{g^2 \beta }{64} t^2}	e^{- \frac{\beta}{4}  ( |v|^2 + g  x_3)}
.
\end{split}	\Ee
We have used \eqref{est:tB} and \eqref{est:1/w_h} from the first to second line; and from the second to third line we have used \eqref{bound:e^Ddb} and the fact that if $t \leq \frac{4}{g} \sqrt{|v_3|^2 + g x_3}\leq \frac{4}{g} \sqrt{|v |^2 + g x_3} $ then 
\Be\notag
\begin{split}
\mathbf{1}_{t \leq \frac{4}{g} \sqrt{|v_3|^2 + g x_3}} e^{- \frac{\beta}{2} \big(|v|^2 + gx_3\big)} \leq \mathbf{1}_{t \leq \frac{4}{g} \sqrt{|v |^2 + g x_3}} e^{- \frac{\beta}{2} \big(|v|^2 + gx_3\big)}  \leq e^{- \frac{g^2\beta}{64}  t^2} e^{- \frac{\beta}{4} \big(|v|^2 + gx_3\big)}.
\end{split}
\Ee
Now we have, by completing the square, 
\Be
\begin{split}\notag
e^{- \frac{\beta g^2}{64} t^2} &= e^{- \frac{\beta g^2}{64} \big(t- \frac{32 \lambda_\infty}{\beta g^2}\big)^2} e^{\frac{16 \lambda_\infty^2}{\beta g^2}} e^{-\lambda_\infty t} \leq e^{\frac{16 \lambda_\infty^2}{\beta g^2}} e^{-\lambda_\infty t}.
\end{split}
\Ee
Combining this with \eqref{bound:tB/w}, we derive \eqref{est:I}.

\hide
Since $w^F (X^F(s;t,x,v),V^F (s;t,x,v))$ is invariant with respect to $s \in [ t - t^\b  _{F}(t,x,v)  ,  t+ t^\f  _{F}(t,x,v)]$, the denominator equals, for $C=4$, 
\Be
\begin{split}\label{est:1/w}
&\frac{\mathbf{1}_{ gt/C \leq   |v_{\f ,3}^F(t,x,v)|  }}{ w^F (X^F(t+ t_\f ^{F};t,x,v),V^F (t+t_\f^F;t,x,v))}\\
&= \mathbf{1}_{ gt/C \leq   |v_{\f ,3}^F(t,x,v)|  }
e^{- \beta \frac{| v_{\f ,3}^F(t,x,v)|^2}{2}}e^{- \beta \frac{| v_{\f,\parallel}^F (t,x,v)|^2}{2}}\\
& \leq e^{- \frac{\beta}{4} \frac{g^2 t^2}{C^2}}e^{- \beta \frac{| v_{\f,3}^F (t,x,v)|^2}{4}}e^{- \beta \frac{| v_{\f,\parallel}^F (t,x,v)|^2}{2}}\\
& \leq e^{- \frac{\beta}{4} \frac{g^2 t^2}{C^2}}
e^{- \frac{\beta}{4} \Big( 1- \frac{2  C^2 \| \nabla _x \phi_F \|_\infty^2}{  g^2}  \Big)  | v_{\f,3}^F (t,x,v)|^2}  
e^{- \beta \frac{| v_{ \parallel} |^2}{4}} \\
& \leq e^{- \frac{\beta}{4} \frac{g^2 t^2}{C^2}}
e^{- \frac{\beta}{4} \Big( 1- \frac{2  C^2 \| \nabla _x \phi_F \|_\infty^2}{  g^2}  \Big)  | v_3 |^2 } 
e^{- \beta \frac{| v_{ \parallel} |^2}{4}} 
\end{split}
\Ee
where we have used the following inequality to get the fourth line of \eqref{est:1/w}, for $C=4$,
\Be
\begin{split}\notag
&-|v_{\f, \parallel}^F (t,x,v)|^2\\
&\leq  -\Big| |v_\parallel | -  t_\f^F(t,x,v) \| \nabla_x \phi_F \|_\infty \Big|^2 \\
& 
\leq  - \Big| |v_\parallel |  - \frac{C |v^F_{\f, 3} (t,x,v)|}{ g} \| \nabla_x \phi_F \|_\infty\Big|^2\\
& = - |v_\parallel|^2 + 2\frac{C |v_{\f, 3}^F|}{g} \| \nabla_x  \phi_F \|_\infty | v_\parallel| - \frac{C^2 |v_{\f, 3}^F|^2}{g^2} \| \nabla_x \phi _F \|_\infty^2\\
&  \leq  - |v_\parallel|^2  + \frac{|v_\parallel|^2}{2} + 2\frac{C^2 |v_{\f, 3}^F|^2}{g^2} \| \nabla_x  \phi_F \|_\infty^2 - \frac{C^2 |v_{\f, 3}^F|^2}{g^2} \| \nabla_x \phi _F \|_\infty^2\\
& = -\frac{|v_\parallel|^2}{2}  +  \frac{C^2 |v_{\f, 3}^F|^2}{g^2} \| \nabla_x \phi _F \|_\infty^2
\end{split}
\Ee
and have used $|v_{\f,3}^F (t,x,v)| \geq |v_3|$ at the last line of \eqref{est:1/w}. 

Therefore 
\Be\begin{split}
& \int_{\R^3}  \frac{ \mathbf{1}_{ gt/C \leq   |v_{F,3}^\f (t,x,v)|  }}{w^F (X^F(0;t,x,v),V^F (0;t,x,v))}  \dd v \\
& \leq e^{- \frac{\beta}{4} \frac{g^2 t^2}{C^2}}  \frac{8}{\beta^{3/2}} \Big( 1- \frac{2  C^2 \| \nabla _x \phi_F \|_\infty^2}{  g^2}  \Big)^{-1} 
\end{split}
\Ee
\unhide

Next we prove \eqref{est:N}. Using \eqref{est:1/w} and \eqref{bound:e^Ddb}, we obtain 
\Be\label{h_v}
\begin{split}
&	\nabla_v h (\Zz (s;t,x,v) ) 
\leq \frac{ 1}{w _\beta (\X (s;t,x,v), \V (s;t,x,v))} \| w_\beta \nabla_v h \|_\infty
\\ 
& \leq  	e^{\frac{16^2 \beta }{2 g^2}
	\| \Delta_0^{-1} (\nabla_x \cdot b^\ell) \|_{L^\infty_{t,x}}^2	
}
e^{- \frac{\beta}{4}|v|^2 } e^{- \frac{\beta g}{4} x_3}  \| w_\beta \nabla_v h \|_\infty    \leq 4^{1/g}
e^{- \frac{\beta}{4}|v|^2 } e^{- \frac{\beta g}{4} x_3}  \| w_\beta \nabla_v h \|_\infty ,
\end{split}
\Ee 
where we have used that $
e^{\frac{16^2 \beta }{2 g^2}
\| \Delta_0^{-1} (\nabla_x \cdot b^\ell) \|_{L^\infty_{t,x}}^2	
}
\leq e^{\frac{2}{g} \ln 2} \leq  4^{\frac{1}{g}}$ from \eqref{bound:e^Ddb}.

Using \eqref{h_v}, we now bound $\mathcal{N}  $:
\Be\begin{split}\notag
|	\mathcal{N}  (t,x,v)| 
& \leq  4^{1/g} \int^t_{t-\tB (t,x,v)}
e^{- \lambda_\infty s}  
\sup_{s \in [0, t] } \| e^{\lambda_\infty s} \varrho  (s) \|_\infty
e^{- \frac{\beta}{4} |v|^2} e^{- \frac{\beta g}{4} x_3} \| w_\beta \nabla_v h \|_\infty \dd s \\
& \leq  \underline{\tB (t,x,v)
	e^{\lambda_\infty  {\tB (t,x,v)}}
	e^{- \frac{\beta}{4} |v|^2} e^{- \frac{\beta g}{4} x_3} } \times 4^{1/g}   e^{-\lambda_\infty t} 
\sup_{s \in [0, t] } \| e^{\lambda_\infty s} \varrho^\ell (s) \|_\infty  \| w_\beta \nabla_v h \|_\infty   .
\end{split}\Ee
Then using \eqref{est:tB} for $\tB (t,x,v)$, we bound the above underlined term as 
\Be
\begin{split}\notag
&\tB (t,x,v)
e^{\lambda_\infty  {\tB  (t,x,v)}}
\leq  \frac{4}{g} \sqrt{|v_3|^2 + g x_3} 
e^{ \frac{4 \lambda_\infty}{g} (|v_3| + \sqrt{g x_3})}
\\
& \leq \frac{8}{g \beta^{1/2}} e^{- \frac{\beta}{8} |v|^2 + \frac{4 \lambda_\infty}{g} |v|  } 
e^{- \frac{\beta g}{8} x_3 + \frac{4 \lambda_\infty}{g} \sqrt{g x_3}} 
e^{  \frac{\beta}{8} |v|^2} e^{  \frac{\beta g}{8} x_3}
\leq  \frac{8}{g \beta^{1/2}} e^{ \frac{16^2 \lambda_\infty^2}{ 4 g^2 \beta}}
e^{ \frac{\beta}{8} |v|^2} e^{ \frac{\beta g}{8} x_3} .
\end{split}
\Ee
Therefore we get \eqref{est:N}.\hide
\Be\label{est:N1}
\begin{split}
|	\mathcal{N}  (t,x,v)|  
\leq  \frac{8}{g \beta^{1/2}}   
4^{1/g}
e^{ \frac{16^2 \lambda_\infty^2}{ 4 g^2 \beta}}  	\sup_{s \in [0, t] } \| e^{\lambda_\infty s} \varrho (s) \|_\infty  \| w_\beta \nabla_v h \|_\infty
e^{-\lambda_\infty t}  e^{- \frac{\beta}{8} |v|^2} e^{- \frac{\beta g}{8} x_3}.
\end{split}\Ee
Finally we use \eqref{est:D^-1 Db} to conclude \eqref{est:Nell}.\unhide \hide
\Be\label{est:N1}
\begin{split}
& |	\mathcal{N} ^{\ell+1}(t,x,v)| \\
& \leq  \frac{8}{g \beta^{1/2}}   e^{
	\frac{16^2  C^2 (1+ \frac{1}{\delta \beta g})^2}{2 g^2 \delta^2 \beta } \| \w^\ell_{\frac{\delta \beta}{8}} f^\ell \|_\infty^2	
}    e^{ \frac{16^2 \lambda_\infty^2}{ 4 g^2 \beta}}  	\sup_{s \in [0, t] } \| e^{\lambda_\infty s} \varrho^\ell (s) \|_\infty  \| w_\beta \nabla_v h \|_\infty
e^{-\lambda_\infty t}  e^{- \frac{\beta}{8} |v|^2} e^{- \frac{\beta g}{8} x_3}
\end{split}\Ee

\Be\begin{split}\label{est:N}
&	\frac{e^{-\lambda_\infty t}}{\lambda_\infty} e^{\frac{16^2 \beta}{2 g^2} \| \Delta_0^{-1} (\nabla \cdot b^\ell) \|_\infty^2}    \| w_\beta \nabla_v h \|_\infty  
e^{- \frac{\beta}{8} |v|^2 }e^{- \frac{\beta}{8}  gx_3 }
e^{ - \frac{\beta}{8} 
	\left(
	|v|^2- \frac{64\lambda_\infty}{g \beta} |v|
	\right)
}
e^{ - \frac{\beta}{8} 
	\left(
	g x_3- \frac{64 \lambda_\infty}{g \beta} \sqrt{g x_3}
	\right)
}\\
&  \leq \frac{e^{-\lambda_\infty t}}{\lambda_\infty} e^{\frac{32^2 \beta}{4 g^2} \| \Delta_0^{-1} (\nabla \cdot b^\ell) \|_\infty^2}  e^{ \frac{16^2 \lambda_\infty^2}{ 2 g^2 \beta } }   \| w_\beta \nabla_v h \|_\infty  
\\
&	  \leq \frac{e^{-\lambda_\infty t}}{\lambda_\infty} e^{
	\frac{   C^2 32^2 \beta  (1+ \frac{1}{\delta \beta g} )^2}{ 4  \delta^2 \beta^2 g^2 }  \| \w^\ell_{\frac{\delta \beta}{8}} f^\ell   \|^2_{L^\infty_{t,x}}
}  e^{ \frac{16^2 \lambda_\infty^2}{ 2 g^2 \beta } }   \| w_\beta \nabla_v h \|_\infty  .
\end{split}\Ee

\unhide\end{proof}

\hide
\begin{proposition}\label{prop:decay_N}
Choose $p>3$. Assume there exist $0< \lambda_\infty < \lambda_p<\infty$ and $\mathfrak{C}_{\rho, p}, \mathfrak{C}_{\rho, \infty}  \in (0 , \infty)$ such that 
\begin{align}
\sup_{t \geq 0 }\| e^{\lambda_p t} \rho  (t)\|_{L^p(\O)}< \mathfrak{C}_{\rho, p} ,\\
\sup_{t \geq 0 }\| e^{\lambda_\infty t} \rho  (t)\|_{L^\infty(\O)}< \mathfrak{C}_{\rho, \infty} .
\end{align}
and $ \| w_h \nabla_v h \|_{L^\infty_{x,v}}< \mathfrak{C}_h < \infty$. Then 
\Be\begin{split}
&\Big\| \mathcal{N}[\rho_f, h ,Z^F] (t, \cdot ) \Big\|_{L^p_x} \\
&\leq
\end{split}\Ee
and 
\Be
\sup_{t\geq 0}
\Big\|	e^{\lambda_\infty t} \mathcal{N}[\rho_f, h ,Z^F] (t,\cdot)\Big\|_{L^\infty_x} \leq    \frac{\mathfrak{C}_h}{g  \beta^{2}} e^{\frac{32 \lambda _\infty^2  }{  g^2\beta}} \mathfrak{C}_{\rho, \infty}  e^{-\lambda_\infty t}. 
\Ee
\end{proposition}

\begin{proof}

\Be
\begin{split}\label{consistence} 
&\Big\|	 \mathcal{N} ^{\ell+1}(t,\cdot)\Big\|_{L^\infty_x}  \\
& \leq  \Big\| 
\tB^{\ell+1} (t,x,v) 
e^{\frac{16^2 \beta }{2 g^2}
	\| \Delta_0^{-1} (\nabla_x \cdot b^\ell) \|_{L^\infty_{t,x}}^2	
}
e^{- \frac{\beta}{4}|v|^2 } e^{- \frac{\beta g}{4} x_3}  \| w \nabla_v h \|_\infty 
\\
& \ \ \ \ \  \times 
\sup_{s \in [ t- \tB^{\ell+1} (t,x,v),t]}| \nabla_x \Delta_0^{-1} \varrho^\ell  (s) | 
\dd v
\Big\|_{L^\infty_x}
\\
&\leq \mathfrak{C}_h \Big\| \int_{\R^3}  \frac{4}{g}|v_{\b,3}^F(x,v)| 
e^{- \frac{\beta}{2} \Big(1 - 	  \frac{16 }{g}	\big(2+\frac{8\| \nabla_x \phi_F \|_\infty }{g} \big) - \frac{32 \| \nabla_x \phi_F \|_\infty^2}{g^2}
	\Big) |v_{\b,3}^F (t,x,v)|^2} e^{- \frac{\beta}{2} 
	\Big(1 - 
	\frac{16}{g} 
	\| \nabla_x \phi_f \|_\infty 
	\Big)
	|v_\parallel|^2} 
\\
& \ \ \ \ \  \times 
\mathfrak{C}_{\rho, \infty}  e^{-\lambda_\infty\Big(t- 4\frac{|v_{\b,3}^F(t,x,v)|}{g}\Big)}  
\dd v\Big\|_{L^\infty_x}
\\
& \leq \mathfrak{C}_h \mathfrak{C}_{\rho, \infty}   e^{-\lambda_\infty t} \int_{\R^3} e^{\lambda_\infty  \frac{ 4 |v_{\b,3}^F (t,x,v)|}{g}}  \frac{1}{\sqrt{\beta}}  \sqrt{\beta} \frac{ 4| v_{\b,3}^F(t,x,v) |}{g} e^{- \frac{\beta}{3} |v_{\b,3}^F (t,x,v)|^2 } e^{- \frac{\beta}{4} |v_\parallel|^2}  \dd v\\
&\leq  \mathfrak{C}_h\mathfrak{C}_{\rho, \infty}  
e^{-\lambda_\infty t}\frac{1}{g\beta^{3/2}} 
\int_{\R}
e^{- \frac{\beta}{4} \{|v^F _{\b,3}(t,x,v)|^2 - \frac{16 \lambda_\infty}{g \beta} |v^F_{\b,3}(t,x,v)| + \frac{ 16 \lambda_\infty^2 }{g^2\beta^2} \}}
e^{\frac{4 \lambda _\infty^2 }{g^2\beta}} \dd v_3\\
& \leq \frac{\mathfrak{C}_h \mathfrak{C}_{\rho, \infty}  }{g  \beta^{3/2}} e^{\frac{32 \lambda _\infty^2  }{  g^2\beta}} e^{-\lambda_\infty t}
\int_{\R} e^{ - \frac{\beta}{8} |v_{\b, 3}^F (t,x,v)|  ^2} \dd v_3\\
& \leq \frac{\mathfrak{C}_h \mathfrak{C}_{\rho, \infty}  }{g  \beta^{3/2}} e^{\frac{32 \lambda _\infty^2  }{  g^2\beta}} e^{-\lambda_\infty t}
\int_{\R} e^{ - \frac{\beta}{8} |v_3|  ^2} \dd v_3\\
& \leq \frac{\mathfrak{C}_h \mathfrak{C}_{\rho, \infty}  }{g  \beta^{2}} e^{\frac{32 \lambda _\infty^2  }{  g^2\beta}} e^{-\lambda_\infty t}
\end{split}
\Ee

\Be
\begin{split}\label{consistence}
&\Big\|	 \mathcal{N}[\rho_f, h ,Z^F] (t,\cdot)\Big\|_{L^\infty_x}  \\
& \leq  \Big\|
\int_{\R^3} 
\tB^{\ell+1} (t,x,v) \mathfrak{C}_h  e^{- \frac{\beta}{2} \Big(1 - 	  \frac{16 }{g}	\big(2+\frac{8\| \nabla_x \phi_F \|_\infty }{g} \big) - \frac{32 \| \nabla_x \phi_F \|_\infty^2}{g^2}
	\Big) |v_{\b,3}^F (t,x,v)|^2} e^{- \frac{\beta}{2} 
	\Big(1 - 
	\frac{16}{g} 
	\| \nabla_x \phi_f \|_\infty 
	\Big)
	|v_\parallel|^2}  \\
& \ \ \ \ \  \times 
\sup_{s \in [ t- t_\b^{F} (t,x,v),t]}| \nabla_x \Delta_0^{-1} \rho_f (s) | 
\dd v
\Big\|_{L^\infty_x}
\\
&\leq \mathfrak{C}_h \Big\| \int_{\R^3}  \frac{4}{g}|v_{\b,3}^F(x,v)| 
e^{- \frac{\beta}{2} \Big(1 - 	  \frac{16 }{g}	\big(2+\frac{8\| \nabla_x \phi_F \|_\infty }{g} \big) - \frac{32 \| \nabla_x \phi_F \|_\infty^2}{g^2}
	\Big) |v_{\b,3}^F (t,x,v)|^2} e^{- \frac{\beta}{2} 
	\Big(1 - 
	\frac{16}{g} 
	\| \nabla_x \phi_f \|_\infty 
	\Big)
	|v_\parallel|^2} 
\\
& \ \ \ \ \  \times 
\mathfrak{C}_{\rho, \infty}  e^{-\lambda_\infty\Big(t- 4\frac{|v_{\b,3}^F(t,x,v)|}{g}\Big)}  
\dd v\Big\|_{L^\infty_x}
\\
& \leq \mathfrak{C}_h \mathfrak{C}_{\rho, \infty}   e^{-\lambda_\infty t} \int_{\R^3} e^{\lambda_\infty  \frac{ 4 |v_{\b,3}^F (t,x,v)|}{g}}  \frac{1}{\sqrt{\beta}}  \sqrt{\beta} \frac{ 4| v_{\b,3}^F(t,x,v) |}{g} e^{- \frac{\beta}{3} |v_{\b,3}^F (t,x,v)|^2 } e^{- \frac{\beta}{4} |v_\parallel|^2}  \dd v\\
&\leq  \mathfrak{C}_h\mathfrak{C}_{\rho, \infty}  
e^{-\lambda_\infty t}\frac{1}{g\beta^{3/2}} 
\int_{\R}
e^{- \frac{\beta}{4} \{|v^F _{\b,3}(t,x,v)|^2 - \frac{16 \lambda_\infty}{g \beta} |v^F_{\b,3}(t,x,v)| + \frac{ 16 \lambda_\infty^2 }{g^2\beta^2} \}}
e^{\frac{4 \lambda _\infty^2 }{g^2\beta}} \dd v_3\\
& \leq \frac{\mathfrak{C}_h \mathfrak{C}_{\rho, \infty}  }{g  \beta^{3/2}} e^{\frac{32 \lambda _\infty^2  }{  g^2\beta}} e^{-\lambda_\infty t}
\int_{\R} e^{ - \frac{\beta}{8} |v_{\b, 3}^F (t,x,v)|  ^2} \dd v_3\\
& \leq \frac{\mathfrak{C}_h \mathfrak{C}_{\rho, \infty}  }{g  \beta^{3/2}} e^{\frac{32 \lambda _\infty^2  }{  g^2\beta}} e^{-\lambda_\infty t}
\int_{\R} e^{ - \frac{\beta}{8} |v_3|  ^2} \dd v_3\\
& \leq \frac{\mathfrak{C}_h \mathfrak{C}_{\rho, \infty}  }{g  \beta^{2}} e^{\frac{32 \lambda _\infty^2  }{  g^2\beta}} e^{-\lambda_\infty t}
\end{split}
\Ee 
where we have used Lemma \ref{lem:elliptic_est:C^1}
\Be
\|\nabla_x \Delta_0^{-1} \rho_f (s,\cdot ) \|_{L^\infty_x} \lesssim 
\| \rho_f (s,\cdot) \|_{L^\infty_x} ,
\Ee 
\Be
\sqrt{\beta}|v_{\b,3}^F| e^{- \frac{ \beta}{3} |v_{\b,3}^F|^2 } \leq e^{- \frac{\beta}{4} |v_{\b,3}^F|^2 } ,
\Ee 
\Be
- \frac{16 \lambda_\infty}{g \beta} |v_{\b,3}^F (t,x,v)| \geq  - \frac{1}{2}|v_{\b,3}^F (t,x,v)| ^2 - \frac{128 \lambda_\infty^2}{g^2 \beta^2},
\Ee
and $|v^\b _{F,3} (t,x,v)| \geq |v_3|$.

\hide

Hence 
\Be\begin{split}
& \int_{\R} e^{ - \frac{\beta}{2} ( |v_{F, 3}^\b (t,x,v)| - \frac{aC}{g \beta })^2} \dd v_3
\int_{\R^2}  e^{- \frac{\beta}{2} |v_\parallel|^2}\dd v _\parallel
\\
&\leq  \Big\{\int_{|v_3| \geq  \frac{aC}{g \beta }}e^{- \frac{\beta}{2} (|v_3| - \frac{aC}{g \beta})^2} \dd v_3 + \int_{|v_3| \leq \frac{aC}{g \beta }} \dd v_3\Big\} 	\int_{\R^2}  e^{- \frac{\beta}{2} |v_\parallel|^2}\dd v _\parallel\\
& \leq \Big\{ \int_{\tau\geq 0} e^{- \frac{\beta}{2} \tau^2} \dd \tau + \frac{a C}{g \beta }\Big\} 	\int_{\R^2}  e^{- \frac{\beta}{2} |v_\parallel|^2}\dd v _\parallel\\
& \leq \Big\{ \frac{1}{\sqrt{\beta}} + \frac{a C}{g \beta }\Big\} \frac{1}{\beta}.
\end{split}\Ee

Back to \eqref{consistence}:
\Be
\Big\|	 \mathcal{N}[\rho_f, h ,Z^F] (t,\cdot)\Big\|_{L^\infty_x}  \lesssim e^{-a_\infty t} \frac{1}{g \sqrt{\beta}} e^{\frac{a^2 C^2}{2 g^2\beta}}  \Big\{ \frac{1}{\sqrt{\beta}} + \frac{a C}{g \beta }\Big\} \frac{1}{\beta}
\Ee\unhide

Choosing $g \gg_\beta 1$ we might expect to have an exponential decay. 

Using the control of $e^{\frac{\beta}{2} |v|^2} e^{ \frac{\beta g}{2} x_3} |\nabla_v h_\pm (x,v)| \leq \mathfrak{C}_h \| w_G^{\pm} \nabla_{x_\parallel, v} G_\pm \|_{L^\infty_{\gamma_-}}$ in \eqref{est:hk_v}, 
\Be\begin{split}
&	\left| \int^t_{{\max\{0, t- \tb^{\ell+1,\pm } (t,x,v) \}}   } \nabla_x\Delta^{-1}_0 \rho ^\ell  (s, \X^{\ell+1} _\pm (s;t,x,v)) \cdot    \nabla_v h_\pm 
( \mathcal{Z}_\pm ^{\ell+1}  (s;t,x,v) ) 
\dd s   \right|\\
& \leq   \int^{t}_{ t- \tb^{\ell+1,\pm } (t,x,v)}\mathfrak{C}_{\rho, \infty} e^{-\lambda_\infty s}
\mathfrak{C}_h \| w_G^{\pm} \nabla_{x_\parallel, v} G_\pm \|_{L^\infty_{\gamma_-}}
e^{-\frac{\beta}{4} | v |^2} e^{- \frac{\beta g}{4} x_3}
\\
& \leq  \mathfrak{C}_{\rho, \infty}  e^{- \lambda_\infty t}   \tb^{\ell+1,\pm }  e^{ \lambda_\infty  \tb^{\ell+1,\pm }  }\mathfrak{C}_h \| w_G^{\pm} \nabla_{x_\parallel, v} G_\pm \|_{L^\infty_{\gamma_-}}
e^{-\frac{\beta}{4} | v |^2} e^{- \frac{\beta g}{4} x_3}\\
& \leq  \mathfrak{C}_{\rho, \infty} \mathfrak{C}_h \| w_G^{\pm} \nabla_{x_\parallel, v} G_\pm \|_{L^\infty_{\gamma_-}}   e^{- \lambda_\infty t} 
\frac{4   }{g}  \sqrt{|v_3|^2 + g x_3}  e^{ \frac{4 \lambda_\infty }{g}  \sqrt{|v_3|^2 + g x_3}  }e^{-\frac{\beta}{4} | v |^2} e^{- \frac{\beta g}{4} x_3}\\
& \leq \left(	\mathfrak{C}_h \| w_G^{\pm} \nabla_{x_\parallel, v} G_\pm \|_{L^\infty_{\gamma_-}} 
\frac{16 e^{-1/2}}{g \beta^{1/2} }  
e^{\frac{32\lambda_\infty^2}{ g^2 \beta }} 
\right)\mathfrak{C}_{\rho, \infty} e^{- \lambda_\infty t} ,
\end{split}
\Ee
where we have used  
\Be\notag
\begin{split}
&	\frac{4   }{g}  \sqrt{|v_3|^2 + g x_3}  e^{ \frac{4 \lambda_\infty }{g}  \sqrt{|v_3|^2 + g x_3}  }e^{-\frac{\beta}{4} | v |^2} e^{- \frac{\beta g}{4} x_3} \\
& \leq 	\frac{4}{g}  \sqrt{|v|^2 +   g x_3} 
e^{- \frac{\beta}{4} (|v|^2 +   g x_3)} 
e^{ \frac{4 \lambda_\infty}{g}\sqrt{|v|^2 +   g x_3}  }\\
&= 	\frac{4}{g \beta^{1/2}}  \sqrt{ \beta (|v|^2 + gx_3)}
e^{- \frac{1}{4}\sqrt{ \beta (|v|^2 + gx_3)}^2 }	e^{ \frac{4 \lambda_\infty}{g}\sqrt{|v|^2 +   g x_3}  }\\
& \leq \frac{4}{g \beta^{1/2}} \frac{2}{e^{1/2}} 	e^{- \frac{ \beta}{8}\sqrt{  |v|^2 + gx_3}^2 }e^{ \frac{4 \lambda_\infty}{g}\sqrt{|v|^2 +   g x_3}  }\\
& = \frac{4}{g \beta^{1/2}} \frac{2}{e^{1/2}} 
e^{- \frac{\beta}{8} \left(
	\sqrt{  |v|^2 + gx_3} - \frac{16 \lambda_\infty}{ g \beta }
	\right)^2} e^{\frac{32\lambda_\infty^2}{ g^2 \beta }}.
\end{split}
\Ee

\hide\Be
\begin{split}
e^{- \frac{\beta}{4} |v|^2} e^{ \frac{4 \lambda_\infty}{g} |v|} &= e^{- \frac{\beta}{4} \left( |v| - \frac{8 \lambda_\infty}{ g \beta}\right)^2} 
e^{ \frac{16 \lambda_\infty^2}{   g^2 \beta}}	, \\
e^{- \frac{\beta}{4}  g x_3} e^{ \frac{4 \lambda_\infty}{g} \sqrt{gx_3}} &= e^{- \frac{\beta}{4} \left(  \sqrt{gx_3} - \frac{8\lambda_\infty}{ g \beta}\right)^2} 
e^{ \frac{16  \lambda_\infty^2}{  g^2 \beta}}	.
\end{split}
\Ee\unhide

\end{proof}

\unhide

\begin{proof}[\textbf{Proof of Theorem \ref{theo:AS}}]
Taking $v$-integration to \eqref{form:f} and using \eqref{est:I}-\eqref{est:N}, we derive that 
\Be\begin{split}\label{est:varrho}
& e^{  \lambda_\infty t}	|\varrho (t,x)|  \leq  e^{  \lambda_\infty t}	 \left\{\int_{\R^3} | \mathcal{I}  (t,x,v)| \dd v +  \int_{\R^3} |\mathcal{N} (t,x,v) |\dd v\right\} \\
& \leq  \frac{ 1}{\beta^{3/2}}  \left\{ (4 \pi )^{3/2}  2  e^{-  \frac{\beta}{4}  g  x_3 } 
e^{ \frac{16 \lambda_\infty^2}{\beta g^2}}	  \|\w_{\beta, 0 }   f_{0  } \|_{L^\infty_{x,v}}  \right.\\
&\left. \ \ \ \ \ \ \  \ \ \ \ \  + 
(8 \pi )^{3/2} 
e^{- \frac{\beta }{8} g x_3}
\frac{8 \cdot 4^{1/g}}{g \beta^{1/2}}  
e^{ \frac{16^2 \lambda_\infty^2}{ 4 g^2 \beta}}   \| w_\beta \nabla_v h \|_\infty	\sup_{s \in [0, t] } \| e^{\lambda_\infty s} \varrho^\ell (s) \|_\infty \right\}\\
& \leq
\frac{
	16	   \pi^{3/2} 
}{\beta^{3/2}}	e^{ \frac{16 \lambda_\infty^2}{\beta g^2}}	  \|\w_{\beta, 0 }   f_{0  } \|_{L^\infty_{x,v}} 
+ \frac{1}{2} 	\sup_{s \in [0, t] } \| e^{\lambda_\infty s} \varrho^\ell (s) \|_\infty.
\end{split}	\Ee
Here, at the last line, we have used \eqref{condition:Dvh}. By absorbing the last term, we conclude \eqref{decay:varrho}. We can prove \eqref{decay_f} using \eqref{decay:varrho} and \eqref{est:I}-\eqref{est:N}.\hide

Suppose 
\Be
\max_{\ell  \leq  k}	\sup_{t\in [0, \infty) } \| e^{\lambda_\infty t} \varrho^\ell (t) \|_\infty \leq L.
\Ee
Then using \eqref{est:varrho} and our choice of $g$ in \eqref{condition:g}, we derive that 
\Be
\begin{split}
\sup_{t \in [0, \infty) }	e^{  \lambda_\infty t}	 \|  \varrho^{k+1} (t)\|_{\infty}  
& \leq   
\frac{2C}{\beta^{3/2}} 	e^{ \frac{16 \lambda_\infty^2}{\beta g^2}}	  \|w_{0 }   f_{0  } \|_{L^\infty_{x,v}}   +\left( 
\frac{ C 8^2 4^{1/g} }{g^2 \beta^{5/2} }  
e^{ \frac{16^2 \lambda_\infty^2}{ 4 g^2 \beta}}   \| w_\beta \nabla_v h \|_\infty \right)L \\
&\leq L  .
\end{split}
\Ee

\unhide	\end{proof}

\hide

Next, we record the trace theorem (\cite{
CIP,Ukai}). 

\begin{lemma} Suppose $E= E(x) \in L^\infty (\O)$. 
If $h
\in L^\infty(\O \times \R^3) $ and $ v\cdot \nabla_x h
+ (E - g m_\pm \mathbf{e}_3) \cdot \nabla_v h
\in L^\infty(\O \times \R^3)$ then the trace exists $h
\in L^\infty 
(\gamma_-  )$ 
such that 
\Be\label{strace:infty}
\| h 
\|_{L^\infty(\gamma_- )} \lesssim 
\| h 
\|_{L^\infty  (\O \times \R^3
	) 
} + 
\|	v\cdot \nabla_x h 
+ (E- g m_\pm \mathbf{e} _3 )\cdot \nabla_v h 
\|_{L^\infty(\O \times \R^3)}
.
\Ee\end{lemma}

\begin{lemma} Suppose $E= E(t,x) \in L^\infty ([0,T] \times \O)$. 
If $f
\in L^\infty([0,T] \times \O \times \R^3)  $ and $\p_t f+  v\cdot \nabla_x f+ (E - g m_\pm \mathbf{e}_3) \cdot \nabla_v f \in L^\infty([0,T] \times \O \times \R^3)$ then the trace exists $f  \in L^\infty  ( [0,T] \times  \gamma_-  )$  such that 
\Be\label{dtrace:infty}
\begin{split}
	\sup_{t \in [0,T]}\| f  (t) \|_{L^\infty(\gamma_- )}  
	&\lesssim   	\sup_{t \in [0,T]}  	\| f  (t)\|_{L^\infty ( \O \times \R^3  )
	}  \\
	&+ 
	\sup_{t \in [0,T]}  	\| [\p_t   + 	v\cdot \nabla_x   + (E- g m_\pm \mathbf{e} _3 )\cdot \nabla_v] f (t)  \|_{L^\infty(\O \times \R^3)}  .
\end{split}
\Ee
\end{lemma}
We note that the trace theorem in $L^\infty$ is special. In $L^p$ for $1 \leq p < \infty$, the trace is only insured locally in general (see Theorem 5.1.1 and Theorem  5.1.2. in \cite{Ukai}).

\unhide

\section{Steady solutions} 
For the construction of a solution to the steady problem \eqref{VP_h}-\eqref{eqtn:Dphi}, we use a sequence of solutions. For an arbitrary number $\ell \in \mathbb{N}$, we suppose that  
\begin{align}
\Phi^\ell \in C^2 (\O) \cap C^1 (\bar \O), \ 	\Phi^\ell   =0\label{bdry:phik}  \ \ \text{on} \ \p\O ,&\\
| \nabla_x \Phi^ \ell  |  \leq  {g} / 4  
\ \ \text{in} \ \bar \O: = \O \cup \p\O.&
\label{Uest:DPhi^k}
\end{align}
Then we can solve the characteristics $(X^{\ell+ 1}(t;x,v), V^{\ell+ 1}  (t;x,v)) $, as in \eqref{ODE_h}, to
\Be\label{ODE_k}
\frac{dX^{\ell+1} }{dt}=V^{\ell+1 } , \ \ \ \frac{dV^{\ell+1} }{dt} = - \nabla_x  \Phi^\ell(X^{\ell+1} ) + g   \mathbf{e}_3,
\Ee
with $X^{\ell+1}  (t;x,v) |_{t=0} = x$ and $V^{\ell+1} (t;x,v) |_{t=0} =v$ and $\mathbf{e}_3= (0,0,1)^T$. A continuous-in-$(t,x,v)$ solution exists uniquely due to the Picard theorem. 
As long as $(X^{\ell+1}, V^{\ell+1}) \in \O \times \R^3$ exists, then   
\Be
\begin{split}\label{form:XV}
V^{\ell+1} (t;x,v) &= v + \int^t_0 \Big(- \nabla_x  \Phi ^\ell(X^{\ell+1}(s;x,v)) + g  \mathbf{e}_3\Big)  \dd s ,\\
X^{\ell+1} (t;x,v) 
&= x+ vt + \int^t_0\int^s_0 \Big(- \nabla_x  \Phi^{\ell} (X^{\ell+1}(\tau;x,v)) + g  \mathbf{e}_3 \Big)  \dd \tau   \dd s.
\end{split}
\Ee

With $(X^{\ell+1}, V^{\ell+1})$, we define the backward exit time, position, and velocity as in Definition \ref{def:tb}:
\Be 
\begin{split}\label{def:tb_k}
&	\tb^{\ell+1}(x,v) : = \sup\{ s  \in [0,\infty) :	X _{ 3}^{\ell+ 1} (-\tau; x,v )  \ \text{for all }  \tau \in (0, s)\}  \geq 0, \\
&\xb^{\ell+ 1}(x,v) = X^{\ell+ 1} (-\tb^{\ell+ 1}(x,v); x,v ) ,\ \ 
\vb^{\ell+ 1} (x,v)= V^{\ell+ 1} (-\tb^{\ell+ 1}(x,v); x,v ).
\end{split}
\Ee
From \eqref{Uest:DPhi^k}, it is easy to check that $\tb^{\ell+ 1}(x,v) <\infty$ for each $(x,v) \in \bar{\O} \times \R^3$. 

Now we define that, for a given in-flow boundary datum $G$ in \eqref{bdry:F}, 
\Be\label{form:h^k}
\begin{split}
h^{\ell+1} (x,v)  = 
G (\xb^{\ell +1  }(x,v) ,\vb^{\ell+1} (x,v) ) . 
\end{split}
\Ee
Note that this $h^{\ell+1} (x,v)$ is a unique solution for given $\Phi^\ell$ to 
\begin{align}
v\cdot \nabla_x h^{\ell+1}   - \nabla_x (\Phi^\ell+ g   x_3) \cdot \nabla_v h^{\ell+1} =0& \ \ \text{in} \ \O, \label{eqtn:hk}\\
h^{\ell+1}   = G &
\ \  \text{on} \ \gamma_-
\label{bdry:hk}.
\end{align} 
Then we define the density 
\Be
\rho^{\ell+1} (x)= \int_{\R^3}    h ^{\ell+1} (x,v) \dd v    \ \  \text{in} \ \O. \label{eqtn:rhok} 
\Ee

\hide

Since a weak solution (Remark \ref{remark:W=L}), a weak solution $h^{\ell} $ to a linear problem \eqref{eqtn:hk}-\eqref{bdry:hk} with a given $ \nabla_x \Phi^\ell$ satisfies

With these $(\Phi^\ell, X^\ell, V^\ell)$ we

we consider $h^\ell$ solves
\hide\begin{align}
v\cdot \nabla_x h^{\ell+1}  - \nabla_x (\Phi^\ell + g  m_\pm x_3) \cdot \nabla_v h^{\ell+1}  =0,\\
h^{\ell+1}  |_{\gamma_-} = G ,\\
\Delta  \Phi^\ell = \int_{\R^3} (e_+ h_+^\ell + e_- h^\ell_-) \dd v,\\
\Phi^\ell |_{\p\O} =0.
\end{align}

From our simplification, we have \unhide

\hide First we set $(\rho^0,  \nabla_x \Phi^0) = (0,0)$. Then we construct $h^1$ solving \eqref{eqtn:hk}-\eqref{bdry:hk} using the Lagrangian formulation of a straight characteristics $( X^1(s;t,x,v), V^1(s;t,x,v))= (x-(t-s)v ,v)$. Then we define $\rho^1 $ by \eqref{eqtn:rhok} (as long as it is well-defined) and solve $ \nabla_x \Phi^1$ by \eqref{eqtn:phik} using the Green's function as in \eqref{phi_rho}. We iterate the process to construct $(h^{\ell+1}, \rho^\ell, \nabla_x \Phi^\ell)$ for $\ell=0,1,2,\cdots$ solving \eqref{eqtn:hk}-\eqref{bdry:phik}. More precise way of construction follows.\unhide

we have 
a representation along the characteristics for $h^{\ell+1}(x,v)$.

\hide

We could understand \eqref{ODE_k} as (see the regular Lagrangian flow in \cite{DiL_ODE}) 
\Be
\begin{split}\label{form:XV_k}
V^{\ell+1}_\pm (t;x,v) &= v + \int^t_0 \left(- \nabla_x  \Phi^{\ell}(X^{\ell+1}_\pm(s;x,v)) + g 
\mathbf{e}_3\right)  \dd s ,\\
X^{\ell+1}_\pm (t;x,v) 
&= x+ vt + \int^t_0\int^s_0 \left(- \nabla_x  \Phi^\ell(X^{\ell+1}(\tau;x,v)) + g 
\mathbf{e}_3\right)  \dd \tau   \dd s.
\end{split}
\Ee

\unhide

\unhide

\hide
It is a well-defined map for each $ \Phi $ satisfying \eqref{bdry:phik}, \eqref{Uest:DPhi^k}, and \eqref{ODE_k}:  $\Phi \mapsto \eta \Delta_0^{-1} \rho $. We try to find a fixed point of this map, which solves \eqref{VP_h}-\eqref{eqtn:Dphi} (note that \eqref{eqtn:Dphi} is equivalent to $\Delta \Phi = \eta \rho$). To apply the Schaefer's fixed point theorem, we show that the map is a continuous and compact map between a Banach space $\mathfrak{S} \subset W^{2, \infty}_{\text{loc} } (\O)$ with an associate norm: 
\Be\label{spaceS}
\begin{split}
\mathfrak{S}  := \left\{
\Phi  \text{ is Lipschitz continuous}:   \nabla_x^2 \Phi \in L^\infty (\O) , \  \| \nabla_x \Phi^ \ell  \|_{L^\infty (\bar \O)}  \leq  \frac{g}{2} ,   \ 	\Phi^\ell   =0  \  \text{on} \ \p\O
\right\},& \\
\| \Phi \|_{\mathfrak{S}}  : = 2g^{-1} \|  \langle x_3 \rangle^{-1} \Phi \|_{L^\infty (\O)} + 2g^{-1}
\| \nabla_x \Phi \|_{L^\infty (\O)}  + \| \nabla_x^2 \Phi \|_{L^\infty (\O)} .&
\end{split}
\Ee
We will also show that the following set is bounded with respect to $\| \cdot \|_{\mathfrak{S}}$:
\Be
\left\{\Phi^\ell  :   \Phi^\ell    = \frac{\eta}{1+ 1/\ell}  \Delta_0^{-1} \rho^\ell \ \text{for all} \ \ell  \in [0, \infty] \right\} 
\Ee

\Be
\Delta  \Phi^\ell  = \eta^\ell \rho^\ell:= \frac{\eta}{1+ 1/\ell} \rho^\ell  \text{in} \ \O. \label{eqtn:phik}	
\Ee
\unhide
Next, as \eqref{w^h}, we define a weight function which is invariant along the characteristics
\Be\label{def:w^k}
w^{\ell +1} _\beta (x,v)   = e^{\beta \big( |v|^2  + 2\Phi^{\ell} (x) + 2g x_3\big)}.
\Ee
Note that, as \eqref{w:bdry}, at the boundary 
\Be\label{wk:bdry}
w^{\ell +1}_\beta(x,v ) = e^{\beta |v|^2} \ \ \text{on} \ x \in \p\O.
\Ee
Using \eqref{form:h^k}, \eqref{bdry:hk}, and \eqref{wk:bdry}, as long as $\tb^{\ell+1} (x,v) < \infty$, then we have  
\Be
\begin{split}\label{form:h^k_G}
w^{\ell +1}_\beta(x,v)	h^{\ell +1}  (x,v) 
&= w^{\ell +1}_\beta  (x_{\b}^{\ell +1} (x,v) , v^{\ell +1 }_{\b } (x,v) )  G  (x_{\b}^{\ell+1} (x,v) , v^{\ell+1 }_{\b } (x,v) )\\
&
=  e^{\beta | v^{\ell+1}_{\b } (x,v)|^2}  G  (x_{\b}^{\ell+1} (x,v) , v^{\ell+1}_{\b } (x,v) ).
\end{split}	\Ee

\begin{lemma}\label{lemma:Unif_steady}
For an arbitrary $g>0$ and a given $\nabla_x  \Phi^\ell$, we assume that \eqref{bdry:phik} and \eqref{Uest:DPhi^k} hold. 
Then $h^{\ell+1 }$ solving \eqref{eqtn:hk}-\eqref{bdry:hk} and $\rho^{\ell+1}$ defined \eqref{eqtn:rhok} satisfy the following estimates:
\begin{align}
\|  h^{\ell+1 }   \|_{L^\infty (\bar \O \times \R^3)} &\leq \|  G \|_{L^\infty (\gamma_-)},
\label{Uest:h^k} \\
\| w^{\ell +1} _\beta h^{\ell +1}   \|_{L^\infty (\bar \O \times \R^3)} &\leq \| w_\beta  G \|_{L^\infty (\gamma_-)},
\label{Uest:wh^k}
\\
\|  e^{  \beta g x_3 } \rho^{\ell+1}  \|_{L^\infty (\bar \O)} & \leq    \pi ^{3/2}  \beta^{-3/2} \| w_\beta G \|_{L^\infty(\gamma_-)}.
\label{Uest:rho^k}
\end{align} 
Furthermore, $\Phi^\ell/ \langle x_3 \rangle  \in L^\infty $ and 
\Be\label{lower_w_st}
w_\beta^ {\ell+1 }(x,v)  
\geq e^{ \beta \big( |v|^2    + g   x_3\big)} .
\Ee

\end{lemma}

\begin{proof}
From \eqref{form:h^k} and \eqref{form:h^k_G}, we prove that \eqref{Uest:h^k} and \eqref{Uest:wh^k}, respectively.


Using \eqref{eqtn:rhok} and \eqref{Uest:wh^k}, we have
\begin{align}\notag
|	\rho^{\ell}(x)|   = \Big|\int_{\R^3}  w^{\ell}_\beta  h^{\ell} (x,v)    \frac{1}{w_\beta^{\ell} (x,v)} \dd v\Big|  
\leq \| w_\beta G\|_{L^\infty(\gamma_-)}
\int_{\R^3}  \frac{\dd v }{w^{\ell}_\beta  (x,v)}.
\end{align}
Since \eqref{Uest:DPhi^k} holds, using  \eqref{bdry:phik} we derive that 
\Be
\begin{split}\notag
|\Phi ^{\ell} (x_\parallel, x_3)|  =   \Big| \underbrace{\Phi^{\ell}   (x_\parallel, 0)}_{=0} + \int ^{x_3}_0 \p_{3} \Phi^{\ell}  (x_\parallel, y_3) \dd y_3 \Big|    \leq  \frac{g}{2}x_3. 
\end{split}
\Ee
Therefore we deduce that 
\Be\notag
w^ {\ell }_\beta(x,v)  \geq e^{ \beta \big( |v|^2  - 2 |\Phi^{\ell}(x)| +2  g   x_3\big)} 
\geq e^{ \beta \big( |v|^2    + g   x_3\big)} ,
\Ee
which implies 
\Be\notag
\int_{\R^3}  \frac{1 }{w_\beta^{\ell}  (x,v)}\dd v  \leq  \int_{\R^3}  
e^{- \beta \big( |v|^2  + g   x_3\big)}
\dd v  = \frac{ \pi ^{3/2}}{\beta^{3/2}} e^{-  \beta g x_3 }.
\Ee
Therefore we conclude \eqref{Uest:rho^k}. In addition, we have derived \eqref{lower_w_st}. 
\hide
Note that $\Phi^{k+1}$ is given by \eqref{phi_rho} so that $\Phi^{k+1}\in H^1_0 (\O)$ is the weak solution to \eqref{eqtn:phik}-\eqref{bdry:phik}. We apply Lemma \ref{lem:rho_to_phi}: Set $A=\frac{\pi^{3/2}}{\beta^{3/2}}  \| w_\beta G \|_{L^\infty (\gamma_-)} $ and $B = \beta g$. Using \eqref{est:nabla_phi} with such $A$ and $B$, we derive that  
\Be\label{est:nabla_Phi}
|\p_{x_j}   \Phi^{k+1} (x)|   \leq  \frac{\mathfrak{C} \pi^{3/2}}{\beta^{3/2}}  \| w_\beta G \|_{L^\infty (\gamma_-)}  \left\{
1+  \frac{1}{ \beta g}+ \delta_{3j} \frac{e^{-   \beta g  x_3} }{ \beta g}
\right\}.
\Ee
\hide
\Be
\begin{split}
|\nabla_x \Phi ^k (x) | &\leq \mathfrak{C} \left\{\| wG \|_{L^\infty (\gamma_-)} \frac{\pi^{3/2}}{\beta^{3/2}} 
\frac{|\ln  \beta g| }{ \beta g} e^{- \frac{\beta g}{2} x_3}
+  \| wG \|_{L^\infty (\gamma_-)} \frac{\pi^{3/2}}{\beta^{3/2}} \frac{e^{- \beta g x_3} - e^{-x_3}}{1-\beta g }\right\} \\
& \leq  \mathfrak{C} \| wG \|_{L^\infty (\gamma_-)}\frac{\pi^{3/2}}{\beta^{3/2}} 
\left\{
\frac{ |\ln  \beta g|  }{ \beta g}  + (\beta g)^{- \frac{\beta g}{ \beta g -1}}
\right\},
\end{split}\Ee
where at the last line we have used the remark below \eqref{est:nabla_phi}.\unhide From the assumption of $\beta$ in \eqref{condition:beta}, we conclude that \eqref{Uest:DPhi^k} holds for $\ell=k+1$. 

Using \eqref{est:tb^h} and Lemma \ref{lem:tb}, since $\| \nabla_x \Phi^{k} \|_{L^\infty (\bar \O)} \leq \frac{g}{2}$ holds, we derive that $\tb^{k+1}(x,v) \leq  \frac{2}{g} (\sqrt{|v_3|^2 + g x_3} - v_3)< \infty$ for each $(x,v)$. Hence we have \eqref{form:h^k_G} and therefore we prove \eqref{Uest:h^k} and \eqref{Uest:wh^k} for $\ell=k+1$. The proof is now completed by the mathematical induction. \unhide
\end{proof}

Next, we move to construct $h^{\ell+2}$. The starting point is defining $\Phi^{\ell+1} (x): = \eta \int_{\O} G(x,y) \rho^{\ell+1} (y) \dd y $ as \eqref{phi_rho}. From \eqref{Dphi}, we deduce that $\Phi^{\ell+1}$ is a weak solution to 
\Be\label{Dphi_ell}
\Delta \Phi^{\ell+1} = \eta \rho^{\ell+1 } \ \ \text{in $\O$, } \ \Phi^{\ell+1} =0 \ \ \text{on $\p\O$.}
\Ee
To repeat the process to construct $h^{\ell+2}$ as in \eqref{ODE_k}, we verify that $\Phi^{\ell+1} \in C^2 (\O) \cap C^1 (\bar \O)$. We achieve this by establishing the regularity estimate of $h^{\ell+1}$ and $\rho^{\ell+1}$ and then using an elliptic estimate to \eqref{Dphi_ell}.


\subsection{Regularity Estimate}\label{sec:RS}
In this section, we establish a regularity estimate of $(h^{\ell+1}, \rho^{\ell+1}  )$ which are given in \eqref{form:h^k} and \eqref{eqtn:rhok}. We utilize the kinetic distance function \eqref{alpha_k}:
\Be\label{alpha_ell}
\alpha^{\ell+1}
(x,v) =\sqrt{ |v_3|^2 +    |x_3|^2  + 2\p_{x_3} \Phi^{\ell}
(x_\parallel , 0) x_3 + 2g 
x_3 }.
\Ee

\begin{lemma}\label{VL}Suppose a condition \eqref{Uest:DPhi^k} holds. For all $(x,v) \in \O \times \R^3$ and $t  \in  [    - \tb^{\ell+1}  (x,v),0] $, 
\Be\label{est:alpha}
\begin{split}
&	\alpha^{\ell+1} (X^{\ell+1} (t;x,v),V^{\ell+1} (t;x,v)) \\
&\leq  \alpha^{\ell+1} (x,v) e^{   (1+ \| \p_{x_3}^2   \Phi^\ell   \|_{L^\infty({\O})} 	 ) |t| }e^{ \frac{1}{g  } \|  \nabla_{x_\parallel} \p_{x_3} \Phi ^\ell  \|_{L^\infty(\p\O)} 
	\int_0^{|t|} |V_\parallel^{\ell+1} (s;x,v)| \dd s  } ,\\
&\alpha^{\ell+1}  (X^{\ell+1}  (t;x,v),V^{\ell+1}  (t;x,v)) \\
&\geq  \alpha^{\ell+1}  (x,v) e^{  - (1+ \| \p_{x_3}^2   \Phi  ^\ell \|_{L^\infty({\O})} 	 ) |t| }e^{ \frac{-1}{g  } \|  \nabla_{x_\parallel} \p_{x_3} \Phi ^\ell  \|_{L^\infty(\p\O)} 
	\int^0_{|t|} |V^{\ell+1}  _\parallel (s;x,v)| \dd s  }.
\end{split}\Ee
In particular, the last inequality implies that 
{\small\Be	
|v_{\b, 3}^{\ell+1} (x,v)| 
\geq 
\alpha^{\ell+1}  (x,v) e^{  - (1+ \| \p_{x_3}^2   \Phi ^{\ell}  \|_{L^\infty(\bar{\O})} 	 ) t_{\b} ^{\ell+1}   (x,v) }e^{ \frac{-1}{g   } \|  \nabla_{x_\parallel} \p_{x_3} \Phi  ^{\ell}  \|_{L^\infty(\p\O)} 
	\int_{-t_{\b} ^{\ell+1}  (x,v) }^0 |V _\parallel^{\ell+1}  (s;x,v)| \dd s  }.\label{est1:alpha}
\Ee}
\end{lemma}
\begin{proof}
Note that 
\Be\notag
\begin{split}
&[v\cdot \nabla_x  - \nabla_x (\Phi^{\ell} (x)
+ g 
x_3) \cdot \nabla_v ] \sqrt{ {|v_3|^2} +  {|x_3|^2}  + 2\p_{x_3} \Phi^{\ell}  (x_\parallel , 0) x_3 + 2g  x_3 }\\
& = \frac{-  (\p_{x_3} \Phi^{\ell}  (x)+ g  )  v_3 +  v_3 x_3+   \p_3 \Phi^{\ell}  (x_\parallel, 0) v_3+  g v_3+  v_\parallel \cdot \nabla_{x_\parallel} \p_{3} \Phi  ^{\ell}  (x_\parallel, 0) x_3}{\sqrt{ {|v_3|^2} +  {|x_3|^2}   +   \p_{x_3} \Phi  ^{\ell}  (x_\parallel , 0) x_3 + 2g  x_3 }}\\
& \leq \frac{   \big(1 +  \| \p_3\p_3 \Phi^{\ell}    \|_{L^\infty (\O)}\big) |v_3| |x_3|+   |v_\parallel| \| \nabla_\parallel \p_3 \Phi  ^{\ell}  \|_{L^\infty(\p\O)}|x_3|}{\sqrt{ {|v_3|^2} +  {|x_3|^2}  +  2\p_{x_3} \Phi ^{\ell}  (x_\parallel , 0) x_3 + 2g    x_3  }},
\end{split}
\Ee
where we have used $- \p_{3} \Phi^{\ell}   (x_\parallel, x_3) + \p_3 \Phi ^{\ell}  (x_\parallel, 0) =\int^0_{x_3} \p_3 \p_3 \Phi  ^{\ell} (x_\parallel, y_3) \dd y_3$.

Using \eqref{Uest:DPhi}, we deduce that 
\Be\notag
\begin{split}
&\big|[v\cdot \nabla_x  - \nabla_x (\Phi^{\ell} (x)  + g   x_3) \cdot \nabla_v ] \alpha^{\ell+1}  (x,v)\big|\\
& \leq  \big(1+ \| \p_{x_3} \p_{x_3} \Phi^{\ell}  \|_{L^\infty( \O)} 
+ \frac{1}{g  }|v_\parallel| \|  \nabla_\parallel \p_{x_3} \Phi^{\ell}   \|_{L^\infty(\p\O)} 
\big) \alpha^{\ell+1}   (x,v)  .
\end{split}\Ee
By the Gronwall's inequality, we conclude both inequalities of \eqref{est:alpha}. The inequality \eqref{est1:alpha} is a direct consequence of \eqref{est:alpha}.\end{proof}

\begin{lemma}Suppose a condition \eqref{Uest:DPhi^k} holds. Then  
\Be 
\tb^{\ell+1}  (x,v) \leq  
\frac{2}{g} \min \Big\{\sqrt{|v_3|^2 + g x_3} -  v_3,  \sqrt{|v^{\ell+1}_{\mathbf{b}, 3} (x,v)|^2 - g x_3 } +  v^{\ell+1}_{\mathbf{b}, 3} (x,v)  \Big\}, \label{est:tb^ell}
\Ee 
{\small 	\Be 
|v_{\b, 3}^{\ell+1} (x,v)|  
\geq {\alpha ^{\ell+1}  (x,v) }
e^{ -\frac{4}{g} (1+ \| \p_{x_3}^2 \Phi ^{\ell}  \|_{L^\infty (\O)} ) |v_{\b, 3}^{\ell+1}  |}
e^{- \frac{4}{g^2} \| \nabla_{x_\parallel} \p_{x_3} \Phi ^{\ell}  \|_{L^\infty (\p \O)}
	(1+ \frac{2}{g} \| \nabla_x \Phi^{\ell}   \|_{L^\infty (\O)})
	|\vb^{\ell+1}  |^2	
}.\label{est1:xb_x/w}
\Ee}
\end{lemma}
\begin{proof}
The proof of \eqref{est:tb^ell} follows one for \eqref{est:tb^h}. Now we prove \eqref{est1:xb_x/w}. Using \eqref{est1:alpha}, \eqref{ODE_k}, and \eqref{est:tb^ell}, we have 
\Be\notag
\int_{-\tb^{\ell+1} (x,v) }^0 |V_\parallel ^{\ell+1} (s;x,v)| \dd s  \leq 
|\vb ^{\ell+1}  | \tb ^{\ell+1} +  \frac{|\tb^{\ell+1}  |^2}{2} \| \nabla_x \Phi ^{\ell} \|_{L^\infty (\O)}
\leq  \frac{4}{g} \Big(1+ \frac{2}{g} \| \nabla_x \Phi ^{\ell} \|_{L^\infty (\O)} \Big)|\vb^{\ell+1}   | ^2 ,
\Ee
and therefore we prove \eqref{est1:xb_x/w}. 
\end{proof}

\hide
Assume some decay condition (which can be verified by the bootstrap argument of the previous section):
\Be
\nabla_x \phi (x) \lesssim \text{some decay function} (x_3)
\Ee

Note that 
\Be
h_\pm(x,v)  = G(x_\b^{h, \pm}(x,v), v_\b^{h, \pm} (x,v))
\Ee
Taking $\nabla_v$ to this, we obtain that 
\Be
\p_{v_i} h_\pm(x,v)  = \p_{v_i} x_\b^{h, \pm} \cdot  \nabla_{x_\parallel} G(x_\b^{h, \pm}, v_\b^{h, \pm})
+\p_{v_i} v_\b^{h, \pm} \cdot  \nabla_{v } G(x_\b^{h, \pm}, v_\b^{h, \pm})
\Ee \unhide

\begin{lemma}\label{lem:XV_xv}Suppose \eqref{form:XV} holds for all $(t,x,v)$. Then 
\Be
\begin{split}\label{XV_x}
\p_{x_i} X^{\ell+1}  _{  j}(t;x,v)   &= \delta_{ij} +  \int^t_0 \p_{x_i} V ^{\ell+1} _{  j} (s;x,v) \dd s 	\\
&
=  \delta_{ij} - \int^t_0 \int^s_0  \p_{x_i} X^{\ell+1}  (\tau;x,v) \cdot \nabla_x   \p_{x_j} \Phi   (X^{\ell+1}  (\tau;x,v)) \dd \tau \dd s ,\\
\p_{x_i} V^{\ell+1}  _j  (t;x,v) &=   - \int^t_0  \p_{x_i} X^{\ell+1}  (s;x,v) \cdot \nabla_x   \p_{x_j} \Phi^\ell   (X^{\ell+1}  (s;x,v)) \dd s,
\end{split}
\Ee
and
\Be
\begin{split}\label{XV_v}
\p_{v_i} X ^{\ell+1} _{  j} (t;x,v)   &= \int^t_0 \p_{v_i} V^{\ell+1}  _{  j} (s;x,v) \dd s \\
&
= t \delta_{ij} - \int^t_0 \int^s_0  \p_{v_i} X^{\ell+1}  (\tau;x,v) \cdot \nabla_x   \p_{x_j} \Phi  ^{\ell}  (X ^{\ell+1} (\tau;x,v)) \dd \tau \dd s ,\\
\p_{v_i} V^{\ell+1}  _{  j}  (t;x,v) &= \delta_{ij} - \int^t_0  \p_{v_i} X^{\ell+1}  (s;x,v) \cdot \nabla_x   \p_{x_j} \Phi  ^{\ell}  (X^{\ell+1}  (s;x,v)) \dd s.
\end{split}
\Ee
Moreover,
\begin{align}
|\nabla_{x} X^{\ell+1}   (t;x,v)| &\leq   \min \big\{  e^{\frac{t^2}{2} \| \nabla_x^2 \Phi^{\ell}   \|_ {L^\infty (\O)} },
e^{ (1+ \| \nabla_x ^2 \Phi^{\ell}    \|_ {L^\infty (\O)})t}
\big\},\label{est:X_x}\\
|\nabla_{x} V ^{\ell+1}  (t;x,v)| &\leq   \min \big\{   t \| \nabla_x^2 \Phi^{\ell}   \|_ {L^\infty (\O)}e^{\frac{t^2}{2} \| \nabla_x^2 \Phi ^{\ell}   \|_ {L^\infty (\O)} },
e^{ (1+ \| \nabla_x ^2 \Phi \|_ {L^\infty (\O)})t}
\big\},\label{est:V_x}\\
|\nabla_{v} X^{\ell+1}   (t;x,v)| &\leq   \min \big\{ te^{ \frac{t^2}{2} \| \nabla_x^2 \Phi ^{\ell}  \|_ {L^\infty (\O)} },
e^{ (1+ \| \nabla_x ^2 \Phi  \|_ {L^\infty (\O)})t}
\big\},\label{est:X_v}\\
|\nabla_{v} V ^{\ell+1} (t;x,v)| &\leq  \min \big\{ e^{ \frac{t^2}{2} \| \nabla_x^2 \Phi ^{\ell}   \|_ {L^\infty (\O)} },
e^{ (1+ \| \nabla_x ^2 \Phi ^{\ell}   \| _ {L^\infty (\O)})t}
\big\}.\label{est:V_v}
\end{align}
\end{lemma}
\begin{proof}
From \eqref{form:XV}, we directly compute \eqref{XV_x} and \eqref{XV_v}.

For $\p_{x_i} X ^{\ell+1}_{ j}$ and $\p_{v_i} X ^{\ell+1}_{  j}$, we change the order of integrals in each last double integral to get
\Be\notag
\begin{split}
\p_{x_i} X_{  j} ^{\ell+1} &= \delta_{ij}   - \int^t_0(t-\tau) \p_{x_i} X^{\ell+1}  (\tau;x,v) \cdot \nabla_x   \p_{x_j} \Phi ^{\ell} (X^{\ell+1} (\tau;x,v)) \dd \tau  , \\
\p_{v_i} X_{  j}^{\ell+1}   &= t \delta_{ij} - \int^t_0(t-\tau) \p_{v_i} X ^{\ell+1} (\tau;x,v) \cdot \nabla_x   \p_{x_j} \Phi ^{\ell}(X^{\ell+1} (\tau;x,v)) \dd \tau  .
\end{split}
\Ee  
Now applying the Gronwall's inequality, we derive that 
\Be\label{est1:X_z}
\begin{split}
|\p_{x_i} X ^{\ell+1} (t;x,v)| &\leq  
e^{\int^t_0  (t-\tau) \|\nabla_x^2 \Phi ^{\ell}   \|_{L^\infty (\O)} \dd \tau}  \leq 
e^{\frac{t^2}{2}  \| \nabla_x ^2 \Phi  ^{\ell}\|_{L^\infty (\O)}} ,\\
|\p_{v_i} X ^{\ell+1} (t;x,v)| &\leq t 
e^{\int^t_0  (t-\tau) \|\nabla_x^2 \Phi   ^{\ell} \|_{L^\infty (\O)} \dd \tau} \leq  t
e^{\frac{t^2}{2}  \| \nabla_x ^2 \Phi^{\ell} \|_{L^\infty (\O)}}  .\end{split}
\Ee

Next using \eqref{est1:X_z} and the second identities of \eqref{XV_x}, 
we have 
\Be
|\p_{x_i} V _{  j} ^{\ell+1}(t;x,v)|    \leq  t \| \nabla_x^2 \Phi ^{\ell}\|_{L^\infty (\O)} e^{\frac{t^2}{2}  \| \nabla_x ^2 \Phi^{\ell}  \|_{L^\infty (\O)}} .\label{est1:V_x}
\Ee

From the second line of \eqref{XV_v} and the $v_i$-derivative to the first line of \eqref{form:XV}, we obtain that 
\Be
\begin{split}\notag
|\p_{v_i} V ^{\ell+1}_{  j}(t;x,v) | 
&=  \Big|\delta_{ij} - \int^t_0  
\int^s_0 \p_{v_i  } V^{\ell+1} _{  j}(\tau;x,v) 
\cdot \nabla_x   \p_{x_j} \Phi ^{\ell}(X^{\ell+1}(s;x,v)) \dd \tau \dd s\Big|\\
&  =   \Big| \delta_{ij} -  
\int^t_0\p_{v_i  } V ^{\ell+1}(\tau;x,v)\cdot   \int _\tau ^t 
\nabla_x   \p_{x_j} \Phi  ^{\ell}(X^{\ell+1}(s;x,v))  \dd s \dd \tau  \Big|\\
& \leq \delta_{ij}  + \int_0^ t (t- \tau ) \| \nabla_x ^2 \Phi  ^{\ell}  \| _{L^\infty (\O)} |\p_{v_i} V ^{\ell+1}(\tau;x,v) | \dd \tau ,
\end{split}
\Ee 
and hence, by the Gronwall's inequality,   
\Be
|\p_{v_i} V^{\ell+1}  (t; x,v)|  \leq   
e^{\frac{t^2}{2}  \| \nabla_x ^2 \Phi ^{\ell} \|_{L^\infty (\O)}}.\label{est1:V_v}
\Ee

On the other hand, we could derive the following inequalities by adding the first equalities of \eqref{XV_x} and \eqref{XV_v}: 
\Be
\begin{split}\notag
|\nabla_x X^{\ell+1}(t;x,v) | + |\nabla_x V^{\ell+1}(t;x,v)  | \leq 1 + \int^t_0 \{1+ |\nabla_x^2 \Phi^{\ell} (X (s;x,v))| \} \{|\nabla_x X^{\ell+1} | + |\nabla_x V ^{\ell+1}| \} \dd s,
\\
|\nabla_v X ^{\ell+1}(t;x,v) | + |\nabla_v V^{\ell+1}(t;x,v)  | \leq 1 + \int^t_0 \{1+ |\nabla_x^2 \Phi^\ell (X^{\ell+1} (s;x,v))|\} \{|\nabla_v V ^{\ell+1}| + |\nabla_v X^{\ell+1} | \}\dd s.
\end{split}
\Ee
Using the Gronwall's inequality, we derive that  
\Be\label{est2:X_v}
\begin{split}
|\nabla_x X ^{\ell+1}(t;x,v)| + |\nabla_x V ^{\ell+1}(t; x,v)| \leq e^{(1+ \|\nabla_x^2 \Phi ^{\ell}\|_{L^\infty (\O)}) t} ,\\
|\nabla_v X^{\ell+1} (t;x,v)| + |\nabla_v V^{\ell+1} (t; x,v)| \leq e^{(1+ \|\nabla_x^2 \Phi^{\ell} \|_{L^\infty (\O)}) t} .
\end{split}
\Ee

Finally we finish the proof of \eqref{est:X_x}-\eqref{est:V_v} by combining \eqref{est1:X_z}, \eqref{est1:V_x}, \eqref{est1:V_v}, and \eqref{est2:X_v}.
\end{proof}

\begin{lemma}\label{lem:exp_txvb}Recall $(\tb^{\ell+1} (x,v), \xb^{\ell+1}  (x,v), \vb^{\ell+1}   (x,v))$ in Definition \ref{def:tb}. The following identities hold:
\Be\label{tb_x}
\begin{split}
&	\p_{x_i} \tb^{\ell+1}   (x,v)  = \frac{ \p_{x_i} X_{ 3} ^{\ell+1} (- \tb^{\ell+1}  (x,v) ; x,v )}{ v^{\ell+1} _{\mathbf{b}, 3} (x,v) }\\
& \ \  = \frac{1}{  v^{\ell+1}  _{\b,3} (x,v)}
\Big\{
\delta_{i3}
- \int^{-\tb^{\ell+1}  (x,v)}_0 \int^s_0 \p_{x_i} X^{\ell+1}  (\tau;x,v) \cdot \nabla_x \p_{x_3}  \Phi ^{\ell}  (X^{\ell+1}  (\tau;x,v))   \dd \tau \dd s 
\Big\},
\end{split}
\Ee
{\small	\Be\begin{split}\label{tb_v}
	&	\p_{v_i}t_{\b} ^{\ell+1} (x,v) = \frac{\p_{v_i} X_{  3}^{\ell+1}  (-\tb^{\ell+1}  (x,v) ; x,v) }{v_{\b,3}^{\ell+1}  (x,v) }\\
	& 
	= \frac{1}{  v_{\b ,3}^{\ell+1}   (x,v)}
	\Big\{
	\tb^{\ell+1}  (x,v)  \delta_{i3}
	- \int^{- \tb ^{\ell+1}  (x,v)}_0 \int^s_0 \p_{v_i} X ^{\ell+1}  (\tau;x,v) \cdot \nabla_x \p_{x_3}  \Phi  ^{\ell}  (X ^{\ell+1} (\tau;x,v))   \dd \tau \dd s 
	\Big\}	,
\end{split}\Ee}
and 
\Be\label{xb_x}
\p_{x_i} \xb^{\ell+1}   (x,v) = - \frac{ \p_{x_i} X ^{\ell+1} _{  3} ( - \tb^{\ell+1}   (x,v);x,v) }{v_{\b, 3}^{\ell+1}  } \vb^{\ell+1}  (x,v)+ \p_{x_i} X ^{\ell+1} ( - \tb ^{\ell+1} (x,v) ;x,v),
\Ee
\Be\label{xb_v}
\p_{v_i} x_\b^{\ell+1}  (x,v) = - \frac{ \p_{v_i} X _{ 3}^{\ell+1}  ( - \tb^{\ell+1}   (x,v);x,v) }{v_{\b, 3}^{\ell+1}  }  v _\b^{\ell+1} (x,v) + \p_{v_i} X^{\ell+1}  (-t _\b^{\ell+1} (x,v);x,v),
\Ee 
and 
\Be\label{vb_x}
\begin{split}
\p_{x_i} \vb^{\ell+1}  (x,v)  & =  \frac{ \p_{x_i} X_{ 3} ^{\ell+1}  (- \tb ^{\ell+1} (x,v) ; x,v )}{ v^{\ell+1} _{\mathbf{b}, 3} (x,v) } 
\nabla_x \Phi^{\ell}  (   \xb^{\ell+1}  (x,v) ) \\
& \  \ 
- \int^{-\tb ^{\ell+1}  (x,v) }_0 ( \p_{x_i} X^{\ell+1}  (s;x,v) \cdot \nabla_x )  \nabla_x \Phi ^{\ell}   (X ^{\ell+1}  (s;x,v)) \dd s,
\end{split}\Ee
\Be\label{vb_v}
\begin{split}
\p_{v_i} \vb ^{\ell+1}  (x,v)  & =  \frac{ \p_{v_i} X_{ 3}^{\ell+1}   (- \tb ^{\ell+1}  (x,v) ; x,v )}{ v_{\mathbf{b}, 3}^{\ell+1}  (x,v) } 
\nabla_x \Phi ^{\ell} (   \xb ^{\ell+1}  (x,v) ) \\
& \  \  +e_i
- \int^{-\tb ^{\ell+1}  (x,v) }_0 ( \p_{v_i} X^{\ell+1}  (s;x,v) \cdot \nabla_x )  \nabla_x \Phi ^{\ell}   (X ^{\ell+1} (s;x,v)) \dd s.
\end{split}\Ee 	
\end{lemma}
\begin{proof}
By taking derivatives to $X^{\ell+1}_3  (-\tb^{\ell+1} (x,v); x,v )=0$ (see \eqref{def:tb}), we get 
\Be\notag
\begin{split}
- \p_{x_i} \tb^{\ell+1} (x,v) V_3 ^{\ell+1}  (- \tb^{\ell+1} (x,v); x,v) + \p_{x_i} X_3^{\ell+1} (- \tb^{\ell+1} (x,v) ; x,v ) &= 0 ,\\
- \p_{v_i} \tb^{\ell+1}(x,v) V_3 ^{\ell+1}(- \tb ^{\ell+1}(x,v); x,v) + \p_{v_i} X_3^{\ell+1} (- \tb^{\ell+1} (x,v) ; x,v ) &= 0,
\end{split}
\Ee
which imply the first identities of \eqref{tb_x} and \eqref{tb_v}. Then using the first identity of \eqref{XV_x}, we derive \eqref{tb_x}. Similarly, using the first identity of \eqref{XV_v}, we have \eqref{tb_v}.

\hide \Be\label{exp:tb_x}
\begin{split}
&	\p_{x_i} t^{k, \pm }_\b (x,v) \\
& = \frac{1}{  v^{k, \pm}_{\b,3} (x,v)}
\Big\{
\delta_{i3}
- \int^{-\tb^{k, \pm}(x,v)}_0 \int^s_0 \p_{x_i} X^k_\pm(\tau;x,v) \cdot \nabla_x \nabla_{x_j}  \phi^{k-1} (X^k_\pm(\tau;x,v))   \dd \tau \dd s 
\Big\}\\
& = \frac{t^\b (x,v)}{v^\b _3 (x,v) } \delta_{i3} - \frac{1}{v^\b_3 (x,v)} \int^{-\tb(x,v)}_0 \int^s_0 \p_{v_i} X(\tau;x,v) \cdot \nabla_x \p_{x_3}  \phi_h (X(\tau;x,v))   \dd \tau \dd s 
\end{split}
\Ee

\Be\label{exp:tb_v}
\begin{split}
&	\p_{v_i} t^\b (x,v) \\
& = \frac{1}{- v^\b _3 (x,v)}
\Big\{
-\tb(x,v) e_i \cdot (- e_3)  
- \int^{-\tb(x,v)}_0 \int^s_0 \p_{v_i} X(\tau;x,v) \cdot \nabla_x \nabla_x  \phi_h (X(\tau;x,v)) \cdot (- e_3) \dd \tau \dd s 
\Big\}\\
& = \frac{t^\b (x,v)}{v^\b _3 (x,v) } \delta_{i3} - \frac{1}{v^\b_3 (x,v)} \int^{-\tb(x,v)}_0 \int^s_0 \p_{v_i} X(\tau;x,v) \cdot \nabla_x \p_{x_3}  \phi_h (X(\tau;x,v))   \dd \tau \dd s 
\end{split}
\Ee
\unhide

By taking derivatives to $x_{\b} ^{\ell+1}(x,v)$ in \eqref{def:zb}, we derive that 
\Be\begin{split}\notag
\p_{x_i} x_{\b}  ^{\ell+1} &= - \p_{x_i} \tb^{\ell+1}  v_{\b} ^{\ell+1}+ \p_{x_i} X^{\ell+1}  ( - \tb ^{\ell+1}; x,v), \\
\p_{v_i} x_{\b} ^{\ell+1}  &= - \p_{v_i} \tb^{\ell+1}  v_{\b} ^{\ell+1} + \p_{v_i} X^{\ell+1}  ( - \tb ^{\ell+1}; x,v).
\end{split}\Ee 
Then using the first identities of \eqref{XV_x} and \eqref{XV_v}, we conclude \eqref{xb_x} and \eqref{xb_v}.

Finally we take derivatives to $v_{\b}^{\ell+1} (x,v)$ in \eqref{def:zb}
\Be\notag
\begin{split}
\p_{x_i} \vb ^{\ell+1} (x,v)  &= - \p_{x_i} \tb ^{\ell+1} (x,v)  
\dot{V}^{\ell+1}(-  \tb ^{\ell+1} (x,v);x,v)
+ \p_{x_i} V^{\ell+1} ( - \tb ^{\ell+1} (x,v); x,v)\\
&=  \p_{x_i} \tb ^{\ell+1}(x,v)  
\nabla_x \Phi ^{\ell}(   \xb^{\ell+1}  (x,v) )
+ \p_{x_i} V ^{\ell+1}( - \tb^{\ell+1}  (x,v); x,v),\\
\p_{v_i} \vb ^{\ell+1} (x,v)  &= \p_{v_i} \tb^{\ell+1}  (x,v)  
\nabla_x \Phi ^{\ell}(   \xb^{\ell+1}  (x,v) )
+ \p_{v_i} V^{\ell+1} ( - \tb  ^{\ell+1}(x,v); x,v).
\end{split}
\Ee
Finally we use the second identities of \eqref{XV_x} and \eqref{XV_v} and conclude \eqref{vb_x} and \eqref{vb_v}. \end{proof}

\begin{lemma}\label{lem:nabla_zb}
\Be\label{est:xb_x}
\begin{split}
&	|\p_{x_i} \xb ^{\ell+1} (x,v)|  \leq \frac{|\vb^{\ell+1} (x,v)|}{|v_{\b, 3}  ^{\ell+1}(x,v)|}
\delta_{i3} 
\\
& \ \ + \Big(1 +  \frac{|\vb^{\ell+1} (x,v)|}{|v_{\b, 3} ^{\ell+1} (x,v)|}
\frac{|\tb^{\ell+1} (x,v)|^2}{2} \| \nabla_x^2 \Phi  ^{\ell} \|_\infty\Big)  \min \big\{  e^{\frac{ |\tb^{\ell+1}|^2}{2} \| \nabla_x^2 \Phi ^{\ell} \|_\infty },
e^{ (1+ \| \nabla_x ^2 \Phi ^{\ell} \|_\infty)\tb}
\big\} ,
\end{split}
\Ee
\Be\label{est:xb_v}
\begin{split}
&	|\p_{v_i} \xb ^{\ell+1} (x,v)|  \leq \frac{|\vb^{\ell+1}  (x,v)| |\tb^{\ell+1} (x,v)|  }{|v_{\b, 3}  ^{\ell+1}(x,v)|}
\delta_{i3} 
\\
& \ \ + \Big(1 +  \frac{|\vb ^{\ell+1} (x,v)|}{|v_{\b, 3} ^{\ell+1} (x,v)|}
\frac{|\tb ^{\ell+1}(x,v)|^2}{2} \| \nabla_x^2 \Phi^{\ell}  \|_\infty\Big)   \min \big\{ \tb e^{ \frac{|\tb^{\ell+1}|^2}{2} \| \nabla_x^2 \Phi ^{\ell} \|_\infty },
e^{ (1+ \| \nabla_x ^2 \Phi^{\ell}  \|_\infty)\tb}
\big\} ,
\end{split}
\Ee
\Be\label{est:vb_x}
\begin{split}
&	|\p_{x_i} v _{\b, j} ^{\ell+1} (x,v)|  \leq 
\delta_{i3}	\frac{|\p_{x_j} \Phi ^{\ell}(\xb ^{\ell+1} (x,v))|}{|v_{\b, 3} ^{\ell+1} (x,v)|} \\
&  \ \ +
\Big(1+  \frac{|\tb ^{\ell+1}(x,v)| \| \nabla_x   \Phi ^{\ell}\|_\infty }{
	|v_{\b, 3}  ^{\ell+1}(x,v)| 
}\Big) 
|\tb ^{\ell+1}(x,v)| \| \nabla_x^2 \Phi^{\ell} \|_\infty \min \big\{  e^{\frac{|\tb^{\ell+1}|^2}{2} \| \nabla_x^2 \Phi^{\ell}  \|_\infty },
e^{ (1+ \| \nabla_x ^2 \Phi ^{\ell}\|_\infty)\tb}
\big\} ,
\end{split}
\Ee
{\small \Be\label{est:vb_v}
\begin{split}
	&	|\p_{v_i} v _{\b, j} ^{\ell+1}(x,v)|  \leq 
	\delta_{i3}	\frac{ |\tb ^{\ell+1}(x,v)| |\p_{x_j} \Phi^{\ell}  (\xb ^{\ell+1} (x,v))| }{|v_{\b, 3} ^{\ell+1} (x,v)|} + \delta_{ij}\\
	&  \ \ + \Big(1+  \frac{|\tb^{\ell+1} (x,v)| \| \nabla_x \Phi ^{\ell}\|_\infty}{ |v^{\ell+1}_{\b, 3}  (x,v)|}\Big) |\tb^{\ell+1} (x,v)| \| \nabla_x^2 \Phi  ^{\ell}\|_\infty  \min \big\{ \tb^{\ell+1} e^{ \frac{|\tb^{\ell+1}|^2}{2} \| \nabla_x^2 \Phi^{\ell} \|_\infty },
	e^{ (1+ \| \nabla_x ^2 \Phi ^{\ell} \|_\infty)\tb}
	\big\} .
\end{split}
\Ee }
\end{lemma}
\begin{proof}	
Lemma \ref{lem:XV_xv} and Lemma \ref{lem:exp_txvb} yield the estimates. \hide

Applying \eqref{est:X_v} to \eqref{est:tb_v}, we derive that 
\Be
\begin{split}
|\p_{v_i} t^\b (x,v)| & \leq \frac{t^\b (x,v)}{|v^\b _3 (x,v) |} \delta_{i3} + \frac{\tb(x,v)}{|v^\b_3 (x,v)|} \int^{-\tb(x,v)}_0   |\p_{v_i} X(\tau;x,v)|   \| \nabla^2_x \phi \|_\infty \dd \tau  \\
&\lesssim \frac{t^\b (x,v)}{|v^\b _3 (x,v) |}  \{1 +  e^{ \tb (x,v) (1+\| \nabla^2_x \phi \|_\infty) } 
\}.
\end{split}
\Ee
This implies  \unhide
\end{proof}

\hide

\begin{lemma}Suppose \eqref{Uest:DPhi} holds for $i=k-1$. 
Then we have \hide	\Be
t_\b^{k, \pm}(t,x,v) 
\leq  \frac{4}{g}
|v_{\b,3}^{k, \pm}(t,x,v)|
\Ee\unhide
\Be
\begin{split}\label{est:tb^h}
v_{\b,3}^{k, \pm }(x,v)   &\geq    \sqrt{g x_3} ,\\
v_{\b,3}^{k, \pm }(x,v) & \geq \frac{g }{ 4  }	t_\b^{k, \pm }(x,v)   .
\end{split}
\Ee
\end{lemma}
\begin{proof}
From the bootstrap assumption \eqref{Uest:DPhi}, the vertical acceleration is bounded above and below by 
\Be\notag
\frac{d V_3^{k, \pm} (s; x,v)}{ds} =  -g     -  \frac{1}{m_\pm} \p_{x_3} \phi^{k-1}  \leq - \frac{g }{2}  .
\Ee
Hence
\Be\notag
\begin{split}
0&=	X_3^{k, \pm } (- \tb^{k, \pm};x,v)  = x_3 -\int_{- \tb^{k, \pm}}^0 V_3^{k, \pm} (\tau; x,v) \dd \tau  \\
& = x_3 - \tb^{k, \pm} v_{\b,3}^{k, \pm}   - \int_{- \tb^{k, \pm}}^0 \int^\tau _{- t_\b^{k, \pm}}   	\frac{d V_3^F(\tau^\prime;x,v)}{d\tau^\prime} \dd  \tau^\prime\dd \tau \\
& \geq  x_3- \tb^{k, \pm} v_{\b,3}^{k, \pm}   +\frac{|\tb^{k, \pm} |^2}{2} \frac{g }{2}.
\end{split}
\Ee
Therefore we conclude \eqref{est:tb^h}.\end{proof}

\unhide

\begin{lemma}\label{lem:regularity} Recall $(h^{\ell+1}, \rho^{\ell+1}  )$, which are constructed in \eqref{form:h^k} and \eqref{eqtn:rhok}. Suppose the condition \eqref{Uest:DPhi^k} holds. 
For arbitrary numbers $\beta > \tilde \beta >0$, we assume that $\|   e^{{  \beta } |v|^2}  G (x,v) \|_{L^\infty (\gamma_-)}+\|   e^{{\tilde \beta } |v|^2}   \nabla_{x_\parallel, v} G (x,v) \|_{L^\infty (\gamma_-)}< \infty$ and 
\Be\label{choice_beta}
\frac{16}{g^2} \| \nabla_x^2 \Phi ^\ell \|_ {L^\infty (\O)}
\leq 
\tilde \beta  .
\Ee
Then for $(x,v) \in  \bar \O \times \R^3$
\Be
w^{\ell +1} _{{\tilde{\beta}} /{2}} (x,v)
|	\nabla_v h^{\ell+1} (x,v)|   \lesssim 
\left(1+	\frac{1}{g {\tilde\beta}^{1/2}} 
\right)
\|   e^{\tilde \beta |v|^2}  \nabla_{x_\parallel,v} G    \|_{L^\infty(\gamma_-)} , \label{est:hk_v}
\Ee
\vspace{-5pt}
{\small	\Be
w^{\ell +1} _{{\tilde{\beta}} /{2}} (x,v) |	\p_{x_i}h^{\ell+1 } (x,v)|
\lesssim \left[ 
\left(1+ \frac{1}{g \tilde{\beta}^{1/2}}\right)
+ \left(1+ \frac{1}{\tilde{\beta}^{1/2}}\right) \frac{\delta_{i3}}{\alpha^{\ell+1 } (x,v)}
\right]
\|   e^{\tilde \beta |v|^2}  \nabla_{x_\parallel,v} G    \|_{L^\infty(\gamma_-)} 
, \label{est:hk_x}
\Ee}
where $ w^{\ell +1} _{{\tilde{\beta}} /{2}} (x,v) \geq 	e^{ \frac{\tilde \beta}{2}|v|^2} e^{  \frac{\tilde \beta g}{2} x_3}$.

Moreover, for all $x \in \bar \O$,  
\Be\label{est:rho_x}
\begin{split}
&  	e^{  \frac{\tilde \beta g}{2} x_3} 	|\p_{x_i} \rho^{\ell+1 }   (x)|   
\lesssim 	\frac{1}{\tilde{\beta}^{3/2}}\Big(1+ \frac{1}{ g \tilde{\beta}^{1/2}}\Big)
\|  e^{\tilde \beta |v|^2}\nabla_{x_\parallel, v} G \|_{L^\infty ({\gamma_-})} 
\\
&  + 
\frac{\delta_{i3}}{ \tilde \beta }   \left(1+ \frac{1}{\tilde{\beta}^{1/2}}\right) 
\Big(
1+   \mathbf{1}_{|x_3| \leq 1}  |\ln    ( |x_3|^2 + g x_3 )|  +  \frac{1}{\tilde \beta^{1/2}}
\Big)\|  e^{\tilde \beta |v|^2}\nabla_{x_\parallel, v} G \|_{L^\infty ({\gamma_-})},
\end{split}
\Ee
For $0<\delta<1$, 
\Be\label{est:rho_Holder}
[\rho^{\ell+1} ]_{C^{0,\delta}} \lesssim_\delta
\frac{1}{\tilde{\beta}}\Big(1+ \frac{1}{\tilde{\beta}^{1/2}}+ \frac{1}{ g \tilde{\beta}}\Big)
\| e^{\tilde \beta |v|^2} \nabla_{x_\parallel, v}  G \|_{L^\infty (\gamma_-)}.
\Ee
Furthermore, $\Phi^{\ell} \in C^{1} (\bar\Omega) \cap C^2 (\O)$ and 
\Be\label{est:phi_C1}
\| \nabla_x \Phi^{\ell+1} \|_{L^\infty (\O)}  \leq  \mathfrak{C} \frac{\pi^{3/2}}{\beta^{3/2}} \Big(1 + \frac{2}{\beta g}\Big)
\| e^{\beta |v|^2} G \|_{L^\infty (\gamma_-)}
,  
\Ee
\vspace{-15pt}
\Be
\begin{split}
\label{est:phi_C2}
&  \| \nabla_x^2 \Phi^{\ell+1}  \|_{L^\infty (\O)} \\
&  \leq 
\frac{\mathfrak{C}_1}{\beta^{3/2}} 
\| e^{\beta |v|^2} G \|_{L^\infty (\gamma_-)} \bigg\{
\frac{1}{g \beta }  +
\log \bigg(
e+ 	\frac{1}{\tilde{\beta}}\Big(1+ \frac{1}{\tilde{\beta}^{1/2}}+ \frac{1}{ g \tilde{\beta}}\Big) 
\| e^{\tilde \beta |v|^2} \nabla_{x_\parallel, v}  G \|_{L^\infty (\gamma_-)}
\bigg)\bigg\}.
\end{split}	\Ee
\end{lemma}

\begin{remark}
To iterate this construction of sequence of solutions, we need to verify the iteration assumption \eqref{Uest:DPhi^k} and \eqref{choice_beta}.
\end{remark}

\begin{proof}Taking a derivative to \eqref{form:h^k_G}, we derive that 
\Be\label{form:nabla_h}
\nabla_{x,v} h^{\ell+1} (x,v)  = \nabla_{x,v} x_\b^{\ell+1}   \cdot  \nabla_{x_\parallel} G (z_\b^{\ell+1} )  + \nabla_{x,v} v_\b^{\ell+1}  \cdot \nabla_v G  (z_\b^{\ell+1} ),
\Ee
where $z_\b^{\ell+1} = (x_\b ^{\ell+1}, v_\b^{\ell+1} )$ and their derivatives in the right hand side were evaluated at $(x,v)$.

\medskip

\textit{Step 1. Proof of \eqref{est:hk_v}. } From \eqref{form:nabla_h}, 
\Be\begin{split}\label{est1:hk_v}
|\nabla_v h ^{\ell+1} (x,v)|  \leq \frac{ | \nabla_v \xb^{\ell+1} (x,v) | + |\nabla_v \vb^{\ell+1} (x,v)|}{
	e^{\tilde{\beta} |\vb ^{\ell+1} (x,v)|^2}	
}  \|  e^{\tilde{\beta} | v|^2}   \nabla_{x_\parallel,v} G   \|_{L^\infty (\gamma_-)}.
\end{split}\Ee
Using \eqref{est:xb_v}, \eqref{est:tb^ell}, and \eqref{choice_beta}, we have 
\Be\label{est1:xb_v}
\begin{split}
\frac{|\p_{v_i} \xb ^{\ell+1}  (x,v)|}{ e^{\tilde \beta |\vb ^{\ell+1} (x,v)|^2}} & \leq \frac{4 }{g} |\vb^{\ell+1} |e^{-\tilde \beta |\vb ^{\ell+1} |^2} \delta_{i3}  + \Big(
1+ \frac{8 \| \nabla_x^2 \Phi^{\ell}  \|_\infty }{g^2} |\vb ^{\ell+1} | |v_{\b,3}^{\ell+1}  | 
\Big)e^{- \tilde\beta |\vb^{\ell+1}  |^2}	\\
& \ \  \times 
\min \Big \{
\frac{4|v_{\b,3}^{\ell+1} | }{g} e^{ \frac{8}{g^2} \| \nabla_x^2 \Phi^{\ell} \|_\infty|v_{\b, 3} ^{\ell+1} |^2 }  , 
e^{\frac{4}{g}(1+ \| \nabla_x^2 \Phi ^{\ell} \|_\infty) |v_{\b, 3}^{\ell+1} | }
\Big \} \\
& \leq 
\frac{16}{g {\tilde\beta}^{1/2}}
\Big(1 + \frac{8}{g^2\tilde \beta } \| \nabla_x^2 \Phi ^{\ell}  \|_\infty  \Big) e^{- \frac{\tilde\beta}{2 }|\vb^{\ell+1}  |^2 }  \leq \frac{24}{g {\tilde\beta}^{1/2}}
w^{\ell+1}_{\tilde{\beta}/2} (x,v)
,
\end{split}
\Ee
Here, we have used the following lower bound, from \eqref{Uest:DPhi^k},  
\Be\label{lower:vb}
\frac{|\vb^{\ell+1}  (x,v)|^2}{2} = \frac{|v|^2}{2} +  \Phi ^{\ell }  (x) + g x_3
\geq  \frac{|v|^2}{2} + \big(g- \| \nabla_x  \Phi^{\ell }  \|_\infty\big) x_3 
\geq \frac{|v|^2}{2} + \frac{g}{2} x_3.
\Ee

Similarly, using \eqref{est:vb_v}, \eqref{est:tb^ell}, and \eqref{lower:vb}, we derive that 
\Be\label{est1:vb_v}
\begin{split}
\frac{	|\p_{v_i} \vb^{\ell+1} (x,v)| }{
	e^{\tilde \beta  |\vb^{\ell+1} (x,v)|^2}	
}& \leq  \Big( 1+ \frac{4}{g} \|  \nabla_x \Phi^{\ell}  \|_\infty \Big) e^{- \tilde\beta |\vb^{\ell+1} |^2}
\\
& \ + \Big(1+ \frac{4}{g} \| \nabla_x \Phi^{\ell} \|_\infty\Big) \frac{4}{g}  \| \nabla_ x^2 \Phi ^{\ell} \|_\infty  | v_{\b,3}^{\ell+1}  (x,v)|
e^{- \tilde\beta |\vb^{\ell+1} |^2}
\\
& \ \ \ \ \times 
\min \Big\{
\frac{4 |v_{\b, 3}^{\ell+1}  | }{g}  e^{ \frac{8}{g^2}  \| \nabla_x^2 \Phi ^{\ell}\|_\infty |v_{\b, 3}^{\ell+1}  | ^2}, e^{\frac{4}{g}  (1+ \| \nabla_x^2 \Phi^{\ell} \|_\infty)  |v_{\b, 3}^{\ell+1}  |}
\Big\}\\
& \leq \Big(1+ \frac{4}{g} \| \nabla_x \Phi^{\ell}  \|_\infty\Big)\Big(1+ \frac{32}{g^2 \tilde\beta }  \| \nabla_ x^2 \Phi ^{\ell} \|_\infty\Big) e^{- \frac{\tilde\beta}{2}|\vb |^2} \leq 
9	 w^{\ell+1}_{\tilde{\beta}/2} (x,v).
\end{split}
\Ee
Finally we complete the prove of \eqref{est:hk_v} using \eqref{est1:hk_v}, \eqref{est1:xb_v}, and \eqref{est1:vb_v} altogether.

\medskip

{\it{Step 2. }}From \eqref{form:nabla_h}, 
\Be\begin{split}\label{est1:hk_x}
|\partial_{x_i} h^{\ell+1}  (x,v)|  \leq \frac{ |\partial_{x_i} \xb^{\ell+1} (x,v) |  }{
	e^{\tilde \beta |\vb^{\ell+1} (x,v)|^2}  
}  \| e^{\tilde \beta |v|^2}  \nabla_{x_\parallel } G   \|_{L^\infty (\gamma_-)}
+\frac{   | \partial_{x_i} \vb^{\ell+1} (x,v)|}{
	e^{\tilde \beta |\vb^{\ell+1} (x,v)|^2}  
}  \| e^{\tilde \beta |v|^2}  \nabla_{ v} G   \|_{L^\infty (\gamma_-)}
.
\end{split}\Ee
Using \eqref{est1:xb_x/w}, \eqref{est:xb_x}, and \eqref{est:tb^ell}, we derive that 
\Be\label{est1:xb_x}
\begin{split}
\frac{|\p_{x_i} \xb ^{\ell+1} (x,v)|}{e^{\tilde \beta |\vb ^{\ell+1} (x,v)|^2}  } & \leq  \frac{|\vb ^{\ell+1} (x,v)|}{|v_{\mathbf{b},3}  ^{\ell+1} (x,v)|}e^{-\tilde \beta |\vb^{\ell+1}  |^2} \delta_{i3}    + \Big(
1+ \frac{8 \| \nabla_x^2 \Phi^{\ell}  \|_\infty }{g^2} |\vb^{\ell+1} | |v_{\b,3}^{\ell+1} | 
\Big)e^{- \tilde\beta |\vb^{\ell+1} |^2}	\\
& \ \ \  \times 
\min \Big \{
e^{ \frac{8}{g^2} \| \nabla_x^2 \Phi ^{\ell}\|_\infty|v_{\b, 3}^{\ell+1} |^2 }  , 
e^{\frac{4}{g}(1+ \| \nabla_x^2 \Phi ^{\ell} \|_\infty) |v_{\b, 3} ^{\ell+1} | }
\Big \} \\
& \leq 
\frac{|\vb^{\ell+1}(x,v)|}{\alpha ^{\ell+1} (x,v) }
e^{  \frac{4}{g} (1+ \| \nabla_x ^2 \Phi ^{\ell} \|_\infty ) |v_{\b, 3}^{\ell+1} |}
e^{\big(  - \tilde \beta + \frac{4}{g^2} \| \nabla_x^2 \Phi ^{\ell} \|_\infty
	(1+ \frac{2}{g} \| \nabla_x \Phi ^{\ell} \|_\infty)
	\big)	|\vb ^{\ell+1}|^2	
}
\\
& \ \ +\frac{16}{g {\tilde\beta}^{1/2}}
\Big(1 + \frac{8}{g^2\tilde \beta } \| \nabla_x^2 \Phi ^{\ell} \|_\infty  \Big) e^{- \frac{\tilde\beta}{2 }|\vb ^{\ell+1}|^2 }\\
& \leq 
\left(
\frac{\delta_{i3}}{\tilde{\beta}^{1/2}} \frac{1}{\alpha^{\ell+1} (x,v)}   +
\frac{16}{g {\tilde\beta}^{1/2}}
\Big(1 + \frac{8}{g^2 \tilde\beta } \| \nabla_x^2 \Phi ^{\ell} \|_\infty  \Big) \right)
e^{- \frac{\tilde\beta}{2}|v|^2} e^{- \frac{ \tilde\beta g}{2} x_3}
\\
& \leq 
\left(
\frac{\delta_{i3}}{\tilde{\beta}^{1/2}} \frac{1}{\alpha^{\ell+1} (x,v)}   +
\frac{32}{g {\tilde\beta}^{1/2}}
\right)
e^{- \frac{\tilde\beta}{2}|v|^2} e^{- \frac{ \tilde\beta g}{2} x_3}
,
\end{split}
\Ee
where we have used $ \frac{4}{g} (1+ \| \nabla_x ^2 \Phi  \|_\infty ) |v_{\b, 3} | + 
\big(  - \tilde \beta + \frac{4}{g^2} \| \nabla_x^2 \Phi  \|_\infty
(1+ \frac{2}{g} \| \nabla_x \Phi  \|_\infty)
\big)	|\vb |^2 \leq \frac{\tilde{\beta}}{2} |\vb|^2$.

Similarly, using \eqref{est1:xb_x/w}, \eqref{est:vb_x}, and \eqref{est:tb^h}, we derive that 
{\small	\Be\label{est1:vb_x}
\begin{split}
	&	\frac{	|\p_{x_i} \vb^{\ell+1} (x,v)| }{
		e^{\tilde \beta |\vb^{\ell+1} (x,v)|^2}	
	}  \leq  \frac{\|  \nabla_x \Phi^{\ell}  \|_\infty}{|v^{\ell+1}_{\mathbf{b}, 3} (x,v)|}   e^{- \tilde\beta |\vb ^{\ell+1}|^2} \delta_{i3}
	\\
	& \ + \Big(1+ \frac{4}{g} \| \nabla_x \Phi^{\ell} \|_\infty\Big) \frac{4}{g}  \| \nabla_ x^2 \Phi ^{\ell} \|_\infty  | v_{\b,3}  (x,v)|
	e^{- \tilde\beta |\vb^{\ell+1} |^2}
	\min \Big\{
	e^{ \frac{8}{g^2}  \| \nabla_x^2 \Phi^{\ell} \|_\infty |v_{\b, 3} ^{\ell+1} | ^2}, e^{\frac{4}{g}  (1+ \| \nabla_x^2 \Phi^{\ell} \|_\infty)  |v_{\b, 3}^{\ell+1}  |}
	\Big\}\\
	& \leq \left( 
	\frac{ \| \nabla_x \Phi^{\ell} \|_\infty }{\alpha^{\ell+1}(x,v)} \delta_{i3} +  \Big(1+ \frac{4}{g} \| \nabla_x \Phi^{\ell} \|_\infty\Big)
	\right) e^{- \frac{\tilde\beta}{2}|v|^2} e^{- \frac{\tilde\beta g}{2} x_3}  \leq \left( 
	\frac{ 1/2
	}{\alpha^{\ell+1}(x,v)} \delta_{i3} + 3
	\right) e^{- \frac{\tilde\beta}{2}|v|^2} e^{- \frac{\tilde\beta g}{2} x_3}.
\end{split}
\Ee}

\medskip

\textit{Step 3. Proof of \eqref{est:rho_x}. } 
From \eqref{Uest:DPhi^k}, we obtain a lower bound of $\alpha^{\ell+1}$ in \eqref{alpha_ell} as $
\alpha^{\ell+1} (x,v) \geq 	\sqrt{  |v_3|^2 +  |x_3|^2 +  g x_3}.$ We derive that 
{\small	\Be\label{est3:xb_x/w}
\begin{split}
	& 	\int_{\R^3}  \frac{1}{\alpha^{\ell+1} (x,v) }  e^{ - \tilde  \beta 
		\big(
		\frac{|v|^2}{2} + \frac{g}{2}x_3
		\big)
	}  \dd v 
	\leq    \int_{\R^3} \frac{1}{
		\sqrt{  |v_3|^2 +  |x_3|^2 +  g x_3}
	} e^{ - \tilde  \beta 
		\big(
		\frac{|v|^2}{2} + \frac{g}{2}x_3
		\big)
	} \dd  v \\
	& \leq \frac{2 \pi}{ \tilde \beta } e^{- \frac{\tilde \beta g}{2} x_3} \int_{\R} \frac{
		e^{- \frac{\tilde \beta}{2} |v_3|^2}
	}{
		\sqrt{ |v_3|^2 + |x_3|^2  +g x_3}
	}  \dd v_3 \leq \frac{8 \pi}{ \tilde \beta } e^{- \frac{\tilde \beta g}{2} x_3} 
	\Big(
	1+   \mathbf{1}_{|x_3| \leq 1}  |\ln    ( |x_3|^2 + g x_3 )|  +  \tilde \beta^{-1/2}
	\Big),
\end{split}
\Ee}
where we have used the following computation of $A=1$:
\Be
\begin{split}\notag
& \int_{\R} \frac{
	e^{- \frac{\tilde\beta}{2} |v_3|^2}
}{
	\sqrt{  {|v_3|^2} + {|x_3|^2}  + {g}  x_3}
}  \dd v_3  = \int_{|v_3| \leq A} + \int_{|v_3| \geq A} \\
& \leq   
4 \int^A_0
\frac{\dd r}{
	\sqrt{r^2 + |x_3|^2 + g x_3}
}
+ \frac{4e^{ - \frac{\tilde \beta A^2}{4}}}{\sqrt{ {A^2} 
		+ {|x_3|^2}  + g|x_3|
}} \int_{A}^\infty e^{- \frac{\tilde \beta}{4} r ^2}  \dd r\\
& \leq 4
\Big(
\ln | A + \sqrt{A^2 + |x_3|^2 + g x_3 }|
- \ln \sqrt{  |x_3|^2 + g x_3}
\Big) +\frac{\sqrt{  \frac{ 4}{\tilde \beta} }e^{- \frac{\tilde \beta}{4}A^2}}{\sqrt{  {A^2} 
		+ |x_3|^2  + g|x_3|
}} \\
& \leq 4\big\{1+   \mathbf{1}_{|x_3| \leq 1}  |\ln    ( |x_3|^2 + g x_3 )| \big\}+ 4 \tilde \beta^{-1/2}    .
\end{split}
\Ee
Then it is straightforward to derive \eqref{est:rho_x} using \eqref{est:hk_x} and \eqref{est3:xb_x/w}. \hide

From \eqref{form:nabla_h}, we have that 
\Be
\begin{split}\label{est1:rho_x}
&  | \nabla_x \rho^{\ell+1} (x) |   = \Big| \int_{\R^3}\nabla_x  h^{\ell+1} (x,v)    \dd v \Big| \\
&  \leq  \Big|\int_{\R^3}   \nabla_x x_\b^{\ell+1}  \cdot  \nabla_{x_\parallel}  G ({z^{\ell+1}_\mathbf{b} })   + \nabla_x v_\b^{\ell+1}  \cdot \nabla_v G ({z^{\ell+1}_\mathbf{b} })  \dd v \Big| 
\\
& \leq  \| e^{\tilde \beta |v|^2} \nabla_{x_\parallel,v}G \|_{L^\infty (\gamma_-)}  \bigg\{   \int_{\R^3} \frac{| \nabla_x x_\b^{\ell+1} (x,v)|}{ e^{\tilde \beta |\vb^{\ell+1}(x,v)|^2}} \dd v  +  \int_{\R^3}   \frac{|\nabla_x v _\b^{\ell+1}(x,v)|}{e^{\tilde \beta |\vb^{\ell+1}(x,v)|^2}}  \dd v  \bigg\},\\
\end{split}
\Ee 
where $z_\b ^{\ell+1} = (x_\b^{\ell+1} , v_\b^{\ell+1}  )$ were evaluated at $(x,v)$.

We first estimate $	\frac{|\nabla_x \xb^{\ell+1} (x,v)| }{ e^{\tilde \beta |\vb^{\ell+1}(x,v)|^2} }$. From \eqref{est:xb_x} and $\tb^{\ell+1}(x,v) \leq \frac{4}{g} v^{\ell+1}_{\mathbf{b},3}(x,v)$ in \eqref{est:tb^h}, we bound it by a singular term $\eqref{est:xb_x/w1}$ plus non-singular term $\eqref{est:xb_x/w2}$: 
\Be
\frac{|\vb ^{\ell+1}|}{|v_{\b, 3}^{\ell+1} |} e^{- \tilde  \beta |\vb ^{\ell+1}|^2}, \label{est:xb_x/w1} 
\Ee
{\small\Be
\left(1 +  \frac{8}{g^2} \| \nabla_x^2 \Phi^{\ell}  \|_\infty |\vb^{\ell+1} | |v_{\b, 3}^{\ell+1} | \right) 
\min \left\{
\frac{4}{g} |v_{\b, 3}^{\ell+1} |  e^{\frac{8}{g^2}  \| \nabla_x^2 \Phi ^{\ell} \|_\infty  |v_{\b, 3}^{\ell+1} | ^2 }, e^{\frac{4}{g}(1+ \| \nabla_x^2 \Phi ^{\ell} \|_\infty)  |v_{\b, 3}^{\ell+1} |}
\right\} e^{- \tilde \beta |\vb^{\ell+1} |^2}.
\label{est:xb_x/w2}
\Ee}

Using \eqref{est1:xb_x/w} and our choice of $\tilde \beta$ in \eqref{choice_beta}, we derive that 
\Be\label{est2:xb_x/w}
\begin{split}
\eqref{est:xb_x/w1} &\leq \frac{2}{\sqrt{\tilde \beta}}\frac{1}{\alpha^{\ell+1} (x,v)}e^{- \frac{\tilde \beta}{2} |\vb ^{\ell+1}(x,v)|^2 } \leq  \frac{2}{\sqrt{\tilde \beta}} \frac{1}{\alpha^{\ell+1} (x,v)} e^{ -  {\tilde \beta} 
	\big(
	\frac{|v|^2}{2} + \frac{g}{2}x_3
	\big)
},\\
\eqref{est:xb_x/w2}  &\leq \frac{1}{g {\tilde\beta}^{1/2}} \Big(1+ \frac{1}{g^{2} \tilde \beta } \| \nabla_x^2 \Phi ^{\ell} \|_\infty\Big)
e^{- \frac{\tilde \beta}{2} |\vb^{\ell+1} (x,v)|^2 }\leq 
\frac{1}{g {\tilde\beta}^{1/2}} \Big(1+ \frac{\| \nabla_x^2 \Phi ^{\ell} \|_\infty}{g^{2} \tilde \beta } \Big)
e^{ -\tilde  \beta 
	\big(
	\frac{|v|^2}{2} + \frac{g}{2}x_3
	\big)
}.
\end{split}
\Ee

\hide

\Be\label{lower:vb}
\begin{split}
e^{  \frac{\tilde \beta}{2} |\vb  (x,v)|^2 } &= \sqrt{\tilde w (\xb (x,v), \vb (x,v))}   = e^{ \tilde {\beta} 
	\left(
	\frac{|v|^2}{2} + \Phi  (x) + g x_3
	\right)
}\\
& \geq e^{   \tilde\beta 
	\left(
	\frac{|v|^2}{2} + (g-  \|  \Phi  \|_\infty ) x_3
	\right)
}  \geq  e^{   \tilde\beta 
	\left(
	\frac{|v|^2}{2} + \frac{g}{2}x_3
	\right)
}.
\end{split}
\Ee\unhide
Using the first bound of \eqref{est2:xb_x/w}, the form of $\alpha^{\ell+1}  $ in \eqref{alpha_ell}, and the condition \eqref{Uest:DPhi^k}, we derive that 
\Be\label{est3:xb_x/w}
\begin{split}
\int_{\R^3}  \eqref{est:xb_x/w1} \dd v 
& \leq  \frac{2}{\sqrt{\tilde \beta}}  \int_{\R^3} \frac{1}{
	\sqrt{  |v_3|^2 +  |x_3|^2 +  g x_3}
} e^{ - \tilde  \beta 
	\big(
	\frac{|v|^2}{2} + \frac{g}{2}x_3
	\big)
} \dd  v \\
& \leq \frac{4 \pi}{ \tilde \beta^{3/2}} e^{- \frac{\tilde \beta g}{2} x_3} \int_{\R} \frac{
	e^{- \frac{\tilde \beta}{2} |v_3|^2}
}{
	\sqrt{ |v_3|^2 + |x_3|^2  +g x_3}
}  \dd v_3\\
& \leq \frac{16 \pi}{ \tilde \beta^{3/2}} e^{- \frac{\tilde \beta g}{2} x_3} 
\Big(
1+   \mathbf{1}_{|x_3| \leq 1}  |\ln    ( |x_3|^2 + g x_3 )|  +  \tilde \beta^{-1/2}
\Big),
\end{split}
\Ee

Now we consider $\eqref{est:xb_x/w2}$. It is straightforward to bound $\eqref{est:xb_x/w2}$ using \eqref{est2:xb_x/w}: 
\Be\label{est4:xb_x/w}
\begin{split}
\int_{\R^3}\eqref{est:xb_x/w2}   \dd v &\leq \frac{1}{g {\tilde \beta}^{1/2}} \Big(1+ \frac{1}{g^{2} \tilde\beta } \| \nabla_x^2 \Phi^\ell \|_\infty\Big)
e^{ -  \frac{g \tilde \beta }{2}x_3 
}
\int_{\R^3} e^{- \frac{ \tilde \beta |v|^2}{2}} \dd v \\
&
= \frac{2 \sqrt{2} \pi^{3/2}}{ g {\tilde \beta}^{2}}  \Big(1+ \frac{1}{g^{2} \tilde \beta } \| \nabla_x^2 \Phi ^\ell \|_\infty \Big)
e^{ -  \frac{g\tilde \beta}{ 2}x_3 
}.
\end{split}
\Ee

Next we estimate $\frac{|\nabla_x \vb^{\ell+1} (x,v)| }{ e^{\tilde \beta |\vb^{\ell+1}(x,v)|^2}}$. From \eqref{est:vb_x} and \eqref{est:tb^ell}, we bound it by 
{\small	\Be\label{est:vb_x/w}
\begin{split}
	&	
	\frac{\|\nabla_x \Phi ^{\ell} \|_\infty}{|v_{\b, 3} ^{\ell+1} |}
	e^{-   \tilde \beta |\vb ^{\ell+1} |^2}
	\\
	&    +
	\frac{4}{g} \| \nabla_x^2 \Phi ^{\ell} \|_\infty	\Big(1+   \frac{4}{g}  \| \nabla_x   \Phi ^{\ell} \|_\infty  
	\Big)  |v_{\b,3}^{\ell+1}  | 
	\min \big\{  e^{\frac{8}{g^2} |v_{\b,3} ^{\ell+1} | ^2 \| \nabla_x^2 \Phi ^{\ell} \|_\infty },
	e^{\frac{4}{g} (1+ \| \nabla_x ^2 \Phi ^{\ell} \|_\infty)  |v_{\b,3}^{\ell+1} |}
	\big\}	e^{- \tilde   \beta |\vb ^{\ell+1}|^2} .
\end{split}
\Ee}
Using \eqref{est1:xb_x/w}, \eqref{lower:vb} and the choice of $\tilde \beta$ in \eqref{choice_beta}, we derive that 
\Be\label{est1:vb_x/w}
\eqref{est:vb_x/w} \leq 
\frac{\|\nabla_x \Phi ^{\ell} \|_\infty}{ \alpha^{\ell+1}  (x,v)}
e^{	- \frac{\tilde  \beta}{2}   (
|v|^2  +  gx_3
)}+  \frac{4}{g {\tilde\beta}^{1/2}} \| \nabla_x^2 \Phi ^{\ell} \|_\infty	\Big(1+   \frac{4}{g}  \| \nabla_x   \Phi ^{\ell} \|_\infty  
\Big) e^{	-\frac{ \tilde  \beta}{2}  (
{|v|^2}  + gx_3
)}.
\Ee
Following the same argument to obtain \eqref{est3:xb_x/w}-\eqref{est4:xb_x/w}, we derive that 
\Be\label{est1:vb_x/w}
\begin{split}
\int_{\R^3}	\frac{	|\nabla_x v _{\b}^{\ell+1} (x,v)|  }{
	e^{\tilde \beta |\vb^{\ell+1} (x,v)|^2}
} \dd v  
& \leq 	
\frac{ 4 \sqrt{2} \pi \|\nabla_x \Phi ^{\ell}  \|_\infty}{\tilde \beta}    
\Big(
1+   \mathbf{1}_{|x_3| \leq 1}  |\ln    ( |x_3|^2 + g x_3 )| +  \frac{1}{\sqrt{ \tilde \beta}}
\Big)e^{- \frac{ \tilde \beta g}{2} x_3}   \\
& \ \ +  \frac{8 \sqrt{2} \pi^{3/2}}{g {\tilde\beta}^{2}} \| \nabla_x^2 \Phi ^{\ell} \|_\infty	\Big(1+   \frac{4}{g}  \| \nabla_x   \Phi^{\ell}  \|_\infty  
\Big) e^{	   -  \frac{ \tilde \beta  g}{2}x_3
}  .
\end{split}
\Ee
\hide\Be\label{est:xb_v/w}
\begin{split}
|\p_{v_i} \xb^{k-1} (x,v)|  &\leq  \frac{4 |\vb^{k-1  } (x,v)|  }{g}
\delta_{i3} 
\\
& \ \ + \Big(1 +   \frac{8}{g^2} |\vb^{k-1  } (x,v)| |v_{\b, 3}^{k-1 } (x,v)| 
\| \nabla_x^2 \phi^{k-2}  \|_\infty\Big) \\
& \ \ \ \  \times   \min \Big\{ \frac{4 |v_{\b, 3}^{k-1 } (x,v)|  }{g}e^{ \frac{8}{g^2} \| \nabla_x^2 \phi^{k-2} \|_\infty |v_{\b, 3}^{k-1 }   |^2 },
e^{ \frac{4}{g}(1+ \| \nabla_x ^2 \phi^{k-2} \|_\infty)  |v_{\b, 3}^{k-1 }   |
}
\Big\} .
\end{split}
\Ee
Again our choice of $\beta$ in \eqref{choice_beta} and \eqref{lower:vb} ensure that 
\Be
\frac{	|\p_{v_i} \vb^{k-1} (x,v)|  }{w_G (\xb^{k-1}, \vb^{k-1})}  \leq 
\frac{16}{g \sqrt{\beta}} \Big( 1
+ \frac{32}{ g^2 \beta}  \| \nabla_x \phi^{k-2} \|_\infty
\Big)
e^{- \frac{\beta}{2} \big( \frac{|v|^2}{2} + \frac{g}{2} x_3\big)}
\Ee\unhide

Finally we conclude \eqref{est:rho_x} for $\p_{x_3} \rho^{\ell+1}$ using \eqref{est1:rho_x}, \eqref{est3:xb_x/w}, \eqref{est4:xb_x/w}, and \eqref{est1:vb_x/w}. We can easily derive the estimate for the tangential derivatives $\p_{x_1} \rho$ and $\p_{x_2} \rho$ in \eqref{est:rho_x} by noticing the singular terms $\frac{|\vb(x,v)|}{|v_{\mathbf{b}, 3} (x,v)|} \delta_{i3}$ in \eqref{est:xb_x} and $ \frac{| \p_{x_j} \Phi(\xb(x,v))|}{|v_{\mathbf{b}, 3} (x,v)|} \delta_{i3} $ in  \eqref{est:vb_x} are absent when $i \neq 3$. \unhide

\medskip

\textit{Step 4. Proof of \eqref{est:phi_C2}. } We will use \eqref{est:nabla^2phi}. For $0<|h|<1$, using \eqref{est:rho_x} we derive that  
\Be\label{rho_DQ}
\begin{split}
&\frac{|\rho^{\ell+1}(x+ h e_i) - \rho^{\ell+1}(x)|}{|h|^\delta}  \leq  \frac{1}{|h|^\delta} \int^{|h|}_0 | \nabla_x \rho^{\ell+1} (x+ \tau e_i) | \dd \tau \\
& \leq  \| e^{\tilde \beta |v|^2} \nabla_{x_\parallel, v} G \|_{L^\infty ({\gamma_-})} 
\bigg\{ |h|^{1-\delta} 
\frac{1}{\tilde{\beta}^{3/2}}\Big(1+ \frac{1}{ g \tilde{\beta}^{1/2}}\Big)
\\
& 
\ \  \ \ \ \  + 
\frac{\delta_{i3}}{ \tilde \beta }   \left(1+ \frac{1}{\tilde{\beta}^{1/2}}\right) 
\frac{1}{|h|^\delta} \int_0^{|h|} \big(
1+   \mathbf{1}_{ |x_3 + h|    \leq 1}  |\ln    ( |x_3+ \tau|^2 + g (x_3+ \tau) )|  +   {\tilde \beta} ^{-1/2}
\big)   \dd \tau \bigg\}  ,
\end{split}
\Ee
as long as $x_3+ h e_3 \geq 0$.

Note that, for $0< |h| < 1 $,   
\Be
\begin{split}\notag
|h| ^{-\delta } \Big|	\int^{ |h| }_0  \ln (x_3 + \tau) \dd \tau \Big|& \leq  { |h| }^{-\delta } \Big| { |h| } \ln (x_3 +{ |h| }) + x_3 \big( \ln (x_3 + { |h| }) - \ln x_3 \big) - { |h| }\Big|\\
&  \leq { |h| }^{1-\delta} | \ln (x_3 + { |h| }) | + 2 { |h| }^{1-\delta} \lesssim_{\delta} 1,
\end{split}
\Ee 
and 
\Be\begin{split}\notag
&|h|^{-\delta}\Big|	\int^0_{ - \min \{{ |h| }, x_3\}}  \ln (x_3 + \tau) \dd \tau \Big|\\
& \leq { |h| }^{-\delta} \big\{ x_3 | \ln x_3 - \ln (x_3- \min\{{ |h| }, x_3\})|\\
& \ \ 
+ { |h| }^{-\delta}   \min\{{ |h| }, x_3\} 
\Big(
| \ln x_3 - \ln (x_3- \min\{{ |h| }, x_3\})| +   |\ln x_3| + \min\{{ |h| }, x_3\}
\Big)\\
& \leq 2 { |h| }^{-\delta} \min\{|h|, x_3\} + \min\{{ |h| }, x_3\}^{1-\delta} | \ln (  \min\{{ |h| }, x_3\}) |
+ { |h| }^{1-\delta}  \lesssim_\delta 1.
\end{split}\Ee
Using these bounds, we bound \eqref{rho_DQ}, for all $i=1,2,3$, above by 
\Be\begin{split}\label{est:rho_Hol}
&\sup_{0 < |h| < 1}\frac{|\rho^{\ell+1}(x+ h e_i) - \rho^{\ell+1}(x)|}{|h|^\delta} \\
&
\lesssim_\delta
\left\{ 
\frac{1}{\tilde{\beta}^{3/2}}\Big(1+ \frac{1}{ g \tilde{\beta}^{1/2}}\Big)
+
\frac{\delta_{i3}}{ \tilde \beta }   \left(1+ \frac{1}{\tilde{\beta}^{1/2}}\right) 
\right\}\| e^{\tilde \beta |v|^2} \nabla_{x_\parallel, v}  G \|_{L^\infty (\gamma_-)}.
\end{split}\Ee

Using \eqref{est:rho_Hol}, \eqref{Uest:rho^k}, and \eqref{est:nabla^2phi}, we conclude \eqref{est:phi_C2}.
\hide
\Be\label{est:rho_x}
\begin{split}
&	e^{  \frac{\tilde \beta g}{2} x_3} 	|\nabla_x \rho  (x)|  \\
\lesssim	&   \left\{  
\left( {\tilde{\beta}^{-3/2}} + \tilde{\beta}^{-1} \| \nabla_x \Phi \|_\infty\right)
\Big(
1+   \mathbf{1}_{|x_3| \leq 1}  |\ln    ( |x_3|^2 + g x_3 )|  +   {\tilde \beta} ^{-1/2}
\Big)
\right.\\ 
& \ \ \    \left.
+  {g^{-1} \tilde{\beta}^{-2}}	\Big(
1+  g^{-1}{\| \nabla_x \Phi \|_\infty} +  {g^{-2}\tilde{\beta}^{-1}}	 
\Big)
\| \nabla_x ^2 \Phi \|_\infty 
\right\} \| w_{\tilde\beta} \nabla_{x_\parallel, v} G \|_{L^\infty ({\gamma_-})}  ,
\end{split}
\Ee

From \eqref{Uest:wh} and \eqref{lower:vb}, we derive that  
\Be\label{decay_h}
\begin{split}
|	h (x,v) | &\leq \| w_\beta  G  \|_{L^\infty ({\gamma_-})}\frac{1}{w_\beta (\xb (x,v), \vb  (x,v))}\\
&\leq  \| w_\beta  G  \|_{L^\infty(\gamma_-)} e^{ - \beta|\vb (x,v)|^2}\\
&
\leq  \| w_\beta  G  \|_{L^\infty(\gamma_-)}  e^{- \beta \big( |v|^2 + g x_3\big)}.
\end{split}
\Ee
Using \eqref{decay_h} together with \eqref{Uest:rho} and \eqref{est:nabla^2phi}, we derive that 
\Be
\begin{split}
\end{split}
\Ee \unhide 
\end{proof}

\subsection{Construction of Sequences and their Stability}\label{sec:CS}

Let us go back to the discussion right after \eqref{Dphi_ell}. Using $\Phi^{\ell+1} \in C^2 (\O) \cap C^1 (\bar \O)$ in \eqref{est:phi_C2}, now we can repeat the process to construct $h^{\ell+2}$ as in \eqref{ODE_k}. In order to achieve the uniform-in-$\ell$ estimates, we make sure the bound \eqref{est:phi_C2} guarantees \eqref{choice_beta}.
\begin{theorem}\label{prop:Reg}
Suppose \Be\label{bootstrapC1}
\mathfrak{C} \frac{\pi^{3/2}}{\beta^{3/2}} \Big(1 + \frac{1}{\beta g}\Big)
\| e^{\beta |v|^2} G \|_{L^\infty (\gamma_-)}
\leq \frac{g}{2},
\Ee 
\Be\label{bootstrapC2}
\frac{\mathfrak{C}_1}{\beta^{3/2}} 
\| e^{\beta |v|^2} G \|_{L^\infty (\gamma_-)} \bigg\{
\frac{1}{g \beta }  +
\log \bigg(
e+ 	\frac{1}{\tilde{\beta}}\Big(1+ \frac{1}{\tilde{\beta}^{1/2}}+ \frac{1}{ g \tilde{\beta}}\Big) 
\| e^{\tilde \beta |v|^2} \nabla_{x_\parallel, v}  G \|_{L^\infty (\gamma_-)}
\bigg)\bigg\} \leq \frac{\tilde{\beta}g^2}{16}.
\Ee
where $\mathfrak{C}, \mathfrak{C}_1>0$ are the computable constants, which appeared in \eqref{est:phi_C1} and \eqref{est:phi_C2}. Then we can construct $\Phi^\ell, h^{\ell+1}, \rho^{\ell+1}, X^{\ell+1}, V^{\ell+1}$ solve \eqref{eqtn:hk}, \eqref{bdry:hk}, \eqref{eqtn:rhok}, \eqref{Dphi_ell}, \eqref{ODE_k}. Moreover they satisfy \eqref{Uest:DPhi^k} and \eqref{choice_beta}-\eqref{est:phi_C2}.
\end{theorem} 
\begin{proof}
We set $\Phi^0 \equiv 0$ and $h^0 \equiv 0$. Then we solve the characteristics $(X^1, V^1)$ to \eqref{ODE_k} and initial condition with $\ell=0$. Clearly $(X^1, V^1) \in C^1$. Then now we define $h^1, \tb^1, \xb^1, \vb^1$ as in \eqref{form:h^k} and \eqref{def:tb_k} with $\ell=0$. Using Lemma \ref{lemma:Unif_steady} and \eqref{Uest:rho^k}, we derive that $\| e^{\beta g x_3} \rho^1 \|_{L^\infty (\bar \O)} \leq \frac{\pi^{3/2}}{\beta^{3/2}} \| w_\beta G \|_{L^\infty (\gamma_-)}$. Then using \eqref{est:phi_C1} and \eqref{est:phi_C2}, we verify the iteration assumptions \eqref{Uest:DPhi^k} and \eqref{choice_beta} for $\ell=1$. Therefore using Lemma \ref{lem:regularity}, we can iterate this process to construct $\Phi^\ell$, then $(X^{\ell+1}, V^{\ell+1})$ and $h^{\ell+1}$ for $\ell=1,2, \cdots$.\end{proof}


To pass a limit of the sequences we prove a stability lemma, which is very helpful to prove both the stability a la Cauchy and uniqueness of a limiting solution.

\begin{lemma}\label{lem:stability_seq}For given $\bar{h}_i (x,v)$ such that $\bar{\rho}_i := \int  \bar{h}_i  \dd v \in C^{0, \delta}(\O)$ for some $\delta>0$, suppose $ \Phi_i \in C^1(\bar \O) \times C^2 (\O)$ solves
\Be\notag
\Delta \Phi_i = \eta \bar{\rho}_i \  	  \text{in $\O \times \R^3$,} \ \  \ \   \Phi_i =0 \   \text{on $\p\O$.}
\Ee
Now we consider $h_i(x,v)$ solving, in the sense of \eqref{Lform:h},
\begin{align}
v\cdot \nabla_x h_i - \nabla_x (\Phi_i + g x_3 ) \cdot \nabla_v h_i =0
\   \text{in $\O \times \R^3$,} \  \   
h_i = G  \   \text{on $\gamma_-$}.
\end{align}
Suppose the following two condition hold for $g, \bar \beta, \e_0>0$
\Be
|  \Phi_1  (x)| \leq \frac{g}{2} x_3,\label{Uest:DPhi_2}
\Ee 
\vspace{-15pt}
\Be \label{condition_unique}
\frac{ 2^{3/2} \pi^{3/2} \mathfrak{C}
}{g  \bar{\beta}^{2} } 
\left\{ 1+ \frac{4}{\bar\beta g}	\right\} 
\| w_{\bar\beta} \nabla_v h_2 \|_{L^\infty (\O \times \R^3)}  \leq \frac{\e_0}{2}, 
\Ee
where $\mathfrak{C}$ had appeared in \eqref{est:phi_C1}.

Then for a small number $\e_0>1$, the following stability holds 
\Be\label{seq_stable}
\|  e^{  \frac{\bar{\beta }}{2} \big( |v|^2+g x_3\big)} (h_1(x,v) - h_2(x,v)) \|_{L^\infty (\O \times \R^3)}
\leq \frac{1}{2} \|  e^{  \frac{\bar{\beta }}{2}   \big( |v|^2+g x_3\big)} ( \bar{h}_1(x,v) - \bar{h}_2(x,v)) \|_{L^\infty (\O \times \R^3)}.
\Ee

\end{lemma}

\begin{proof}
Clearly the difference of two solutions solves 
\Be\label{VP_diff}
\begin{split}
v\cdot \nabla_x (h_1-h_2) - \nabla_x (\Phi_1 + g x_3) \cdot \nabla_v (h_1-h_2) = \nabla_x(\Phi_1 - \Phi_2) \cdot \nabla_v h_2 \ \ &\text{in } \O \times \R^3,\\
h_1- h_2=0 \ \ &\text{on } \p\O \times \{v_3>0\}.
\end{split}
\Ee
Let $(X_1 , V_1 )$ be the characteristics solving \eqref{ODE_h} with $ \nabla_x \Phi = \nabla_x \Phi_1$ and $t_{\mathbf{b},1}(x,v)$ (as \eqref{def:tb}) is the backward exit time of this characteristics $(X_1 , V_1 )$. Then, as $(h_1- h_2) (X_1(- t_{\mathbf{b},1} (x,v);x,v)) \equiv 0$,  
\Be \label{diff:h}
\begin{split}
&(h_1- h_2)(x,v) \\
&= \int^0_{- t_{\mathbf{b},1} (x,v)}
(\nabla_x\Phi_1(X _1(s;x,v)) -\nabla_x \Phi_2(X _1(s;x,v))) \cdot \nabla_ v h_2(X _1(s;x,v), V_1 (s;x,v))
\dd s .
\end{split}
\Ee
Now we bound above the right hand side of \eqref{diff:h} by 
{\small	\Be\label{bound:diff_h}
\begin{split}
	\underbrace{ t_{\mathbf{b},1} (x,v) \sup_{ s \in [-  t_{\mathbf{b},1} (x,v),0]}   \left( \frac{1}{ w_{\bar\beta,1}  (X_1(s;x,v), V_1(s;x,v))}\right )}_{\eqref{bound:diff_h}_*} \| w_{\bar\beta,1}   \nabla_v h_2 \|_{L^\infty (\O)} \| \nabla_x \Phi_1 - \nabla_x \Phi_2 \|_{L^\infty(\O)}. 
\end{split}
\Ee}
Note that \eqref{Uest:DPhi_2} implies 
\Be
w_{\bar\beta,1} (x,v) = e^{ \bar\beta  \big( |v|^2  + 2 \Phi_1(x) + 2 g   x_3\big)}  \geq e^{ \bar\beta  (|v|^2 + g x_3)}.\label{lower_w}\Ee
Using Lemma \ref{lem:tb} (and \eqref{est:tb^h}), \eqref{w:invar} and \eqref{lower_w}, we bound that 
\Be\label{est:tb/w}
\begin{split}
\eqref{bound:diff_h}_* & \leq \frac{  2 g^{-1}( \sqrt{|v_3|^2 + g x_3} - v_3)}{w_{\bar \beta,1} (x,v)} 
\leq \frac{4}{g}   \sqrt{|v_3|^2 + g x_3}   e^{-\bar \beta \big(|v|^2 + g x_3\big)}
\\&  \lesssim  \frac{1}{g \bar{\beta}^{1/2}}
\sqrt{ \bar{\beta} 
	\big(|v|^2 + g x_3\big)}
e^{- {\bar \beta} \big(|v|^2 + g x_3\big)} 
\lesssim  \frac{1}{g \bar{\beta}^{1/2}} 
e^{- \frac{\bar \beta}{2} \big(|v|^2 + g x_3\big)} 
.
\end{split}\Ee

On the other hand, using Lemma \ref{lem:rho_to_phi} (\eqref{est:nabla_phi} with $A= \| e^{\beta^\prime g x_3} ( \bar \rho_1 - \bar \rho_2) \|_\infty$ and $B= \beta^\prime g$ for $\beta^\prime<\bar \beta$), we derive that 
\Be
\begin{split}\label{diff:Phi} 
\| \nabla_x \Phi_1 - \nabla_x \Phi_2 \|_{L^\infty (\O)}
\leq   \mathfrak{C}
\left\{ 1+ \frac{2}{\beta^\prime g}	\right\}
\underbrace{  \|  e^{\beta^\prime g x_3} ( \bar\rho_1 - \bar\rho_2) \|_{L^\infty (\O)}}.
\end{split}
\Ee
Using \eqref{Uest:rho}, we bound the underbraced term above by 
\Be
\begin{split}\label{diff:bar_rho}
\|  e^{\beta^\prime g x_3} ( \bar\rho_1 - \bar\rho_2) \|_{L^\infty (\O)}
\leq  \frac{\pi^{3/2}}{(\beta^\prime)^{3/2}}  \|  e^{ \beta^\prime   \big( |v|^2+g x_3\big)} ( \bar{h}_1(x,v) - \bar{h}_2(x,v)) \|_{L^\infty (\O)}.
\end{split}
\Ee
Now combining above bounds together with \eqref{bound:diff_h} and \eqref{est:tb/w}, we conclude that  
\begin{align}
&|h_1 (x,v) - h_2 (x,v)|
\notag
\\
& \lesssim 
\underbrace{	 \frac{  \mathfrak{C} \pi^{3/2}}{g  \bar{\beta}^{1/2} (\beta^\prime)^{3/2 } } 
	\left\{ 1+ \frac{2}{\beta^\prime g}	\right\} 
	\| w_{\bar\beta} \nabla_v h_2 \|_\infty }   e^{- \frac{\bar \beta}{2} \big(|v|^2 + g x_3\big)} 
\|  e^{\beta^\prime  \big( |v|^2+g x_3\big)} (\bar{h}_1(x,v) - \bar{h}_2(x,v)) \|_{L^\infty (\O \times \R^3)}.\notag
\end{align} 
With a choice of $\beta^\prime = \bar \beta /2$ and \eqref{condition_unique}, we bound the underbraced term for a sufficiently small $\e_0>0$ to get \eqref{seq_stable}.\hide

Therefore we derive that 
\Be\notag
\|  e^{\beta^\prime  \big( |v|^2+g x_3\big)} (h_1(x,v) - h_2(x,v)) \|_{L^\infty (\O \times \R^3)}
\leq \frac{1}{2} \|  e^{\beta^\prime  \big( |v|^2+g x_3\big)} ( \bar{h}_1(x,v) - \bar{h}_2(x,v)) \|_{L^\infty (\O \times \R^3)},
\Ee

and finally conclude the uniqueness statement. In case that \eqref{condition_unique} holds for $i=1$, we utilize the following equation instead of \eqref{VP_diff} and follow the same argument:
\Be\label{VP_diff1}
\begin{split}
v\cdot \nabla_x (h_2-h_1) - \nabla_x (\Phi_2 + g x_3) \cdot \nabla_v (h_2-h_1) = \nabla_x(\Phi_2 - \Phi_1) \cdot \nabla_v h_1 \ \ &\text{in } \O \times \R^3.
\end{split}
\Ee

\unhide\end{proof}\hide


\begin{lemma}[Uniqueness Theorem]\label{theo:US}
Let $(h_1, \rho_1, \Phi_1)$ and $(h_2, \rho_2, \Phi_2)$ solve \eqref{VP_h}-\eqref{eqtn:Dphi} in the sense of Definition \ref{weak_sol}. Assume \eqref{Uest:DPhi} holds for $\Phi = \Phi_i$ of both $i=1,2$. There exists {\color{black}a sufficiently small constant $\e_0>0$} such that if the following bound holds for $i=1,2$, and $\bar \beta>0$,
\Be
\|	w_{\bar \beta} \nabla_v h_i \|_\infty < \e_0 g^2 (\bar \beta )^3
,
\Ee
then $h_1= h_2$ a.e. in $\O \times \R^3$ and $\Phi_1 = \Phi_2$ a.e. in $\O$. Here, $w_{\bar \beta}(x,v)$ is defined in \eqref{w^h}. 
\end{lemma}

\begin{proof}[\textbf{Proof of Theorem \ref{theo:RS}}]
The proof is a direct consequence of Theorem \ref{prop:Reg}. Note that using the condition \eqref{condition:G} we can verify \eqref{choice_beta} through \eqref{est:phi_C2}.\end{proof}


\unhide
Finally, we prove the existence of a unique solution by passing a limit of the sequences in Theorem \ref{prop:Reg} and using the stability in Lemma \ref{lem:stability_seq}.

\hide
\begin{theorem}\label{theo:CS}Suppose all three conditions \eqref{condition:beta}, $\| G \|_{L^\infty (\gamma_-)}<\infty$ and $\| w_\beta G \|_{L^\infty (\gamma_-)}<\infty$ hold. 
Then there exists at least one solution $(h, \rho, \Phi)$ to \eqref{VP_h}-\eqref{eqtn:Dphi} in the sense of Definition \eqref{weak_sol}. Moreover, we have 
\begin{align}
\|    h    \|_{L^\infty ( \bar \O \times \R^3)   } &\leq \|    G \|_{L^\infty (\gamma_-)},
\label{Uest:h}
\\
\| w_\beta  h    \|_{L^\infty (  \O \times \R^3)} &\leq \| w _\beta  G \|_{L^\infty (\gamma_-)},
\label{Uest:wh}
\\
\|  e^{  \beta g x_3 } \rho  \|_{L^\infty (\bar \O)} &
\leq   \frac{\pi ^{3/2}}{   \beta^{ 3/2} }	\| w_\beta   h    \|_{L^\infty (  \O \times \R^3)}
\leq \frac{\pi ^{3/2}}{   \beta^{ 3/2} } \| w_\beta  G \|_{L^\infty(\gamma_-)},
\label{Uest:rho}
\\
\| \nabla_x \Phi \|_{L^\infty (\bar{\O})} &\leq  \frac{g}{2}  , \label{Uest:DPhi}
\end{align}
and 
\Be
w_\beta (x,v) = e^{ \beta \big( |v|^2  + 2 \Phi(x) + 2 g   x_3\big)}  \geq e^{\beta (|v|^2 + g x_3)}.\label{lower_w}
\Ee\end{theorem}\unhide

\begin{proof}[\textbf{Proof of Theorem \ref{theo:CS}}] Let us first check that, if \eqref{est:hk_v} and \eqref{cond:stability} hold then for $\e_1 \ll \e_0$ we have that 
\Be\notag
\frac{ 2^{3/2} \pi^{3/2} \mathfrak{C}
}{g  \bar{\beta}^{2} } 
\left\{ 1+ \frac{4}{\bar\beta g}	\right\} 	w^{\ell +1} _{{\tilde{\beta}} /{2}} (x,v)
|	\nabla_v h^{\ell+1} (x,v)|  
\leq \frac{\e_0}{2} .
\Ee
Note that this bound guarantees \eqref{condition_unique} in Lemma \ref{lem:stability_seq}. Therefore now we can apply Lemma \ref{lem:stability_seq} to the sequences of Theorem \ref{prop:Reg} with $\bar\beta = \tilde{\beta}/2$: $
\| e^{\frac{\tilde \beta}{4}
\big( |v|^2 + g x_3\big)	
} [h^{\ell+1} (x,v) - h^{\ell} (x,v) ]  \|_{L^\infty}
\leq \frac{1}{2^\ell }	\| e^{\frac{\tilde \beta}{4}
\big( |v|^2 + g x_3\big)	
} h^{1} (x,v)   \|_{L^\infty}.$ Then $h^\ell$ is Cauchy: for all $\ell ,m \in \mathbb{N}$,
\Be\label{Cauchy}
\Big\| e^{\frac{\tilde \beta}{4}
\big( |v|^2 + g x_3\big)	
} [h^{\ell} (x,v) - h^{m} (x,v) ] \Big\|
_{L^\infty(\O \times \R^3)}
\leq \frac{2}{2^{\min\{\ell, m\}} }	\Big\| e^{\frac{\tilde \beta}{4}
\big( |v|^2 + g x_3\big)	
} h^{1} (x,v)   \Big\|_{L^\infty(\O \times \R^3)}.
\Ee
With this strong convergence together with uniform-upper-bounds of Theorem \ref{prop:Reg}, it is standard to prove the convergence of the sequences and prove that their limiting function $(h, \rho, \Phi)$ is a strong solution to \eqref{VP_h}-\eqref{eqtn:Dphi}. Moreover, every upper bound of Theorem \ref{prop:Reg} is valid for the limiting function. Finally Lemma \ref{lem:stability_seq} implies the uniqueness of solution.\hide

{\it{Step 1. }}Using the uniform bounds of \eqref{Uest:h^k} and \eqref{Uest:rho^k},  we extract a subsequence $(h^{i^\prime }, \rho^{i^\prime} )$ satisfying the following convergence results:
\begin{align}
h^{i^\prime }  \overset{\ast}{\rightharpoonup} 
h&
\ \ \text{weak}-* \text{ in } \ L^\infty (\O \times \R^3) \cap L^\infty (\p \O \times \R^3; \dd \gamma)
\label{weakconv_hst} \\
\rho^{i^\prime} \rightharpoonup \rho&
\ \ \text{weak}-* \text{ in } \ L^\infty (\O \times \R^3)
\label{weakconv_rhost}.
\end{align} 

Recall that $h^{i^\prime}(x,v)$ in Lemma \ref{lemma:Unif_steady} is a Lagrangian solution to \eqref{eqtn:hk}-\eqref{bdry:hk} for a given field $\nabla_x \Phi^{i^\prime}$ and characteristics generated by the field. Since this Lagrangian solution is a weak solution, we have that, for any $\psi \in  C^\infty_{c} (\bar \O \times \R^3)$, 
\Be\label{weak_stk}
\iint_{\O \times \R^3} h^{i^\prime }  v\cdot \nabla_x \psi \dd v \dd x - \iint_{\O \times \R^3} h^{i^\prime }  \nabla_x (\Phi^{i^\prime }  + g x_3) \cdot \nabla_v \psi
\dd v \dd x
= \int_{\gamma_+} h^{i^\prime }  \psi \dd \gamma
- \int_{\gamma_-} G \psi \dd \gamma,
\Ee
and, for any $\varphi \in H^1_0 (\O) \cap C^\infty_c (\bar \O)$, 
\Be\label{weak_Poisson_k}
- \int_{\O} \nabla_x \Phi^{i^\prime} \cdot \nabla_x \varphi  \dd x =  \eta \int_{\O} \rho^{i^\prime}  \varphi \dd x.
\Ee
We require that test functions have compact spatial support: 
\begin{align}
\text{spt}_x (\psi)&:= \overline{  \{x \in \O:  \psi(x,v) \neq 0 \ \text{for some }  v \in \R^3\}} \subset_{\text{compact}} \O,\notag\\
\text{spt} (\varphi)&:= \overline{\{x \in \O:  \varphi(x) \neq0 \}} \subset_{\text{compact}} \O.\notag
\end{align}

The convergence of each terms in \eqref{weak_stk} is straightforward using \eqref{weakconv_hst} and \eqref{weakconv_rhost}, except a convergence of the nonlinear term
\Be\label{conv:hDPDpsi}
\int_{\O}  \nabla_x  \Phi^{i^\prime}  \cdot \left(\int_{\R^3} \nabla_v \psi h^{i^\prime + 1}  \dd v\right) \dd x
\rightarrow 
\int _\O  \nabla_x  \Phi\cdot \left( \int_{\R^3}   \nabla_v \psi h  \dd v\right) \dd x.
\Ee
For \eqref{conv:hDPDpsi},	It suffices to prove the convergence of the following terms:
\begin{align}
& \iint_{\O \times \R^3}	[\nabla_x \Phi^{i^\prime} (x) - \nabla_x \Phi  (x) ] \cdot \nabla_v \psi(x,v) h^{i^\prime+1}
(x,v) \dd v \dd x \label{diff1:phih}\\
-&  \iint_{\O \times \R^3}[ h (x,v) - h^{i^\prime+1}(x,v)]	\nabla_x \Phi  (x) \cdot \nabla_v \psi(x,v) \dd v \dd x .\label{diff2:phih}
\end{align} 	
This \eqref{diff1:phih} needs some strong convergence of $ \nabla_x  \Phi^{i^\prime} $ in the compact set $\text{spt}_x (\psi)$. Indeed, from \eqref{CZ_infty}, for $p \in (1, \infty)$ we have $
\| \nabla_x^2 \Phi^{i^\prime} \|_{L^p(\O)} \lesssim   \frac{A}{ B^{1/p}  } \ \ \text{for all } i^\prime.$ Choose $p>3$. We also have \eqref{Uest:DPhi^k}. Then we can apply the Rellich-Kondrachov compactness theorem and get
\Be\label{weakconv_Phist}
\| \nabla_x \Phi^{i^\prime}  -   \nabla_x \Phi\|_{L^q(\text{spt}_x (\psi) )} \rightarrow 0 \ \ \text{as} \   \ i^\prime \rightarrow \infty \ \ \text{for any } q \in [1, \infty).
\Ee 
Using \eqref{weakconv_Phist}, we control the first part: for $1/q+ 1/q^*=1$,
\Be\notag
|\eqref{diff1:phih}| \leq  \| \nabla_x \Phi^{i^\prime}   - \nabla_x \Phi   \|_{L^q(\text{spt}_x (\psi) )}  \|   h^{i^\prime+1} \|_{L^\infty } \|  \nabla_v \psi  \|_{L^{q^*}}  \rightarrow 0.
\Ee

For the second part \eqref{diff2:phih}, using $ \| \nabla_x \Phi \cdot \nabla_v\psi \|_{L^1 (\O \times \R^3)} \leq  \| \nabla_x \Phi \|_\infty \| \nabla_v \psi \|_{L^1}< \infty$ and \eqref{weakconv_hst}, we conclude that $|\eqref{diff2:phih}|\rightarrow 0$. Therefore, we conclude \eqref{conv:hDPDpsi}.

Similarly, we prove the convergence of two each terms in \eqref{weak_Poisson_k} using \eqref{weakconv_Phist} and \eqref{weakconv_rhost}. It is much elementary and we skip the detail. 

\medskip

\textit{Step 2. }Now we prove \eqref{Uest:wh}. From \eqref{Uest:wh^k}, there exists (not necessarily unique) $g \in L^\infty $ and a subsequence $\{i^\prime\}$ such that $w^{i^\prime+1} (x,v)  h^{i^\prime+1}  \overset{\ast}{\rightharpoonup}  g$ weak$-*$ in $L^\infty$.  

Note that \eqref{Uest:DPhi^k} and \eqref{bdry:phik} imply that $\Phi^{i^\prime} (x) \rightarrow \Phi(x)$ almost everywhere in $\O$. Therefore we have $w^{i+1} (x,v) = e^{\beta (|v|^2 + 2 g x_3)} e^{2 \beta \Phi^{i^\prime}(x)} \rightarrow  e^{\beta (|v|^2 + 2 g x_3)} e^{2 \beta \Phi (x)} $ a.e. in $\O$. We also note that \eqref{weakconv_hst}  implies $h^{i^\prime+1} (x,v) \rightarrow h(x,v)$ almost everywhere in $\O \times \R^3.$ Hence we conclude that $w^{i+1} (x,v)  h^{i^\prime+1} (x,v) \rightarrow w(x,v) h(x,v)$ a.e. in $\O \times \R^3$. This implies that 
\Be\label{weakconv_wh}
w^{i+1}   h^{i^\prime+1}  \overset{\ast}{\rightharpoonup}  wh	\ \ \text{weak}-* \text{ in } \ L^\infty (\O \times \R^3).
\Ee
Combining this fact with \eqref{Uest:wh^k} and the weak-star lower semicontinuity in $L^\infty$, finally we prove \eqref{Uest:wh}.

\medskip

\textit{Step 3. }Finally we consider the convergence of 
\Be
\rho^{i^\prime} (x) = \int _{\R^3}  h^{i^\prime} (x,v)\dd v.\notag
\Ee

Due to \eqref{weakconv_hst} and \eqref{weakconv_rhost}, it suffices to prove that, for $\tilde{\varphi} \in C^\infty_c (\O)$, 
\Be\begin{split}\label{conv:rho_k}
\int_{\O}\tilde{\varphi}(x)\int_{\R^3} h^{i^\prime}(x,v)  \dd v \dd x   \rightarrow   \int_{\O}\tilde{\varphi}(x)  \int_{\R^3} h   (x,v)\dd v \dd x .
\end{split}\Ee
For this we will use the uniform weighted bound \eqref{Uest:wh^k} crucially. For $N \gg1$,
\Be\begin{split}\notag
&  \int_{\O}\tilde{\varphi}(x)\int_{\R^3}   \{ h^ {i^\prime}(x,v) - h  (x,v) \} \dd v \dd x \\
& = \int_{\O}\int_{ \R^3 }\tilde{\varphi}(x) \mathbf{1}_{|v|  \leq N }  \{ h^ {i^\prime}(x,v) - h  (x,v) \} \dd v \dd x 
+  \int_{\O}\tilde{\varphi}(x)\int_{ |v|  \geq N }  \{ h^{i^\prime}(x,v) - h  (x,v) \} \dd v \dd x .
\end{split}\Ee
Using \eqref{weakconv_hst}, the first term converges to zero as $i^\prime \rightarrow \infty$ for any fixed $N$. Using \eqref{Uest:wh^k}, \eqref{Uest:wh}, and \eqref{lower_w_st}, we control the second term as 
\Be\begin{split}\notag
&  
\int_{\O} | \tilde{\varphi } (x)| \int_{|v| \geq N} 
e^{- \beta |v|^2} e^{- \beta g x_3} \{ | w^{i^\prime} h^{i^\prime}(x,v)| + | w h (x,v)|  \}\dd v \dd x
\lesssim    2 e^{- \beta N^2 }   \rightarrow 0
.
\end{split}
\Ee
Therefore we prove that $\rho(x) = \int_{\R^3} h(x,v) \dd v$. \unhide
\end{proof}

\hide

\subsection{Proof of the Main Theorem: Stationary problem}
Now we assemble established propositions to conclude the main theorem of the steady problem.

\begin{theorem}\label{main_theo_steady}

\end{theorem}
\begin{proof}[{Proof of (S1) in Theorem \ref{main_theo_steady}}]

\end{proof}

\begin{proof}[{Proof of (S2) in Theorem \ref{main_theo_steady}}]

\end{proof}

\begin{proof}[{Proof of (S3) in Theorem \ref{main_theo_steady}}]

\end{proof}

\unhide

\hide

Now we use Lemma \ref{lem:nabla_zb}. From \eqref{est:nabla_zb}, we derive that  
\Be
\begin{split}
&\eqref{est1:rho_x}_1  \\
& \lesssim \int_{\R^3} \frac{1}{|v_{\mathbf{b},3} (x,v)|} \frac{1}{ \sqrt{w(\xb(x,v), \vb(x,v))}} \frac{1}{ \sqrt{w(x,v)}}  \dd v \\
& \lesssim \int_{\R^3} \frac{e^{\frac{\| \nabla_x^2 \phi \|_\infty}{g} (|v_\parallel|+  \frac{4}{g} |v_{\mathbf{b}, 3}| ) \frac{|v_{\mathbf{b}, 3}|}{g}}}{\alpha(x,v)} 	 e^{- \frac{\beta}{4} \Big(1 - 	  \frac{16 }{g}	\big(2+\frac{8\| \nabla_x \phi_F \|_\infty }{g} \big) - \frac{32 \| \nabla_x \phi_F \|_\infty^2}{g^2}
\Big) |v_{\b,3}^F (t,x,v)|^2} e^{- \frac{\beta}{4} 
\Big(1 - 
\frac{16}{g} 
\| \nabla_x \phi_f \|_\infty 
\Big)
|v_\parallel|^2}  
e^{- \frac{\beta}{4}|v|^2} e^{- \frac{\beta g }{4} x_3} \dd v \\
&\lesssim e^{- \frac{\beta g }{4} x_3} \big| \ln |x_1| \big|
\end{split}
\Ee

Recall 
\Be
\begin{split} 
|V_3^F (\tau;t,x,v)| &\leq |v_{\b,3}^F (t,x,v)  |,\\
\big|	|V_\parallel^F (\tau;t,x,v)| -|v_\parallel| \big|&\leq    \{t^F_\b (t,x,v) + t^F_\f (t,x,v) \} \| \nabla_x \phi_F \|_\infty 
\leq   \frac{8\| \nabla_x \phi_F \|_\infty }{g}|v_{\b,3}^F (t,x,v)  | .
\end{split}
\Ee

Recall that $w_\pm (x,v)= w(\frac{|v|^2}{2} + \phi_h (x)+ g m_{\pm} x_3)$, which is invariant along the trajectory.  Assume $1/w(x_\b, v_\b)$ can control the growth term $e^{\frac{\| \nabla_x^2 \phi \|_\infty}{g}|v_\parallel| t_\b}$ namely
\Be\begin{split}
&\Big|\frac{1}{w(x_\b, v_\b)} e^{\frac{\| \nabla_x^2 \phi \|_\infty}{g}|V_\parallel| t_\b}\Big|
\lesssim \Big|\frac{1}{w(x_\b(x,v), v_\b(x,v))} e^{\frac{\| \nabla_x^2 \phi \|_\infty}{g} (|v_\parallel|+  \frac{4}{g} |v_{\mathbf{b}, 3}| ) \frac{|v_{\mathbf{b}, 3}|}{g}}\Big|\\
&
\lesssim  \frac{e^{\frac{\| \nabla_x^2 \phi \|_\infty}{g}|v_\parallel| \frac{|v_3|}{g}}}{w(|v|^2/2)}  
.
\end{split}
\Ee
as long as the bootstrap assumption \eqref{Bootstrap} holds.

As long as 
\Be
\frac{\| \nabla_x ^2 \phi \|_\infty}{g^2} \ll 1,
\Ee 
we can control \eqref{est1:rho_x} since
\Be\begin{split}
&\int_{\R^3}  \frac{ | \nabla_x x_\b (x,v) | }{w (x_\b, v_\b)} \dd v \lesssim \int_{|v| \leq 100} \frac{1}{\alpha(x,v)} \dd v
+ \int_{|v| \geq 100} \frac{1}{\alpha(x,v)} \frac{1}{w(|v|^2/2)} \dd v
\\
&
\lesssim 
\mathbf{1}_{|x_3| \leq 1/10} \ln \frac{1}{|x_3|} \in L^p (\O) \ \ for  \ all \ p<\infty. 
\end{split} \Ee

Therefore we deduce that 
\Be
\| \nabla_x \rho(x) \|_{L^p} \lesssim \| w \nabla_{x,v} G \|_\infty .
\Ee

Therefore we can conclude that 
\Be\label{est1:h_v}
\nabla_v h_\pm (x,v) \leq \frac{ |v_\b(x,v)|t_\b (x,v)}{|v_{\mathbf{b},3}  (x,v) |}  \{1 +  e^{ \tb (x,v) (1+\| \nabla^2_x \phi \|_\infty) } 
\}    \frac{1}{w (x_\b, v_\b)}w\nabla_{x,v}G(\xb, \vb) 
\Ee
Roughly 
\Be
|\nabla_v h_\pm (x,v)| \leq 
\frac{1}{g} |v_\b^{h, \pm} (x,v)|
e^{ \frac{1+\| \nabla^2_x \phi \|_\infty}{g} |v_\b^{h, \pm} (x,v)|} \frac{1}{w (x_\b, v_\b)}  \| w \nabla_{x,v} G \|_{L^\infty(\gamma_-)}
\Ee 
From \eqref{est:V_v}, recall that 
\Be
\begin{split} 
|V_3^F (\tau;t,x,v)| &\leq |v_{\b,3}^F (t,x,v)  |,\\
\big|	|V_\parallel^F (\tau;t,x,v)| -|v_\parallel| \big|&\leq    \{t^F_\b (t,x,v) + t^F_\f (t,x,v) \} \| \nabla_x \phi_F \|_\infty 
\leq   \frac{8\| \nabla_x \phi_F \|_\infty }{g}|v_{\b,3}^F (t,x,v)  | .
\end{split}
\Ee
As long as $\| \nabla_x \phi \|_\infty < \infty$, we have that 
\Be
\begin{split}\label{est:1/w_h}
&	\frac{1}{w_h (X^F (s;t,x,v), V^F(s;t,x,v))}  \\
&
\leq e^{- \frac{\beta}{2} \Big(1 - 	  \frac{16 }{g}	\big(2+\frac{8\| \nabla_x \phi_F \|_\infty }{g} \big) - \frac{32 \| \nabla_x \phi_F \|_\infty^2}{g^2}
\Big) |v _3|^2} e^{- \frac{\beta}{2} 
\Big(1 - 
\frac{16}{g} 
\| \nabla_x \phi_f \|_\infty 
\Big)
|v_\parallel|^2}  
\end{split}\Ee
we should be able to have
\Be
|\nabla_v h (x,v)| \lesssim e^{-  o(1) |v|^2}.
\Ee

\textit{(Or we can use the weight function to have an estimate
\Be
\begin{split}
&\nabla_v (	w h_{\pm} (x,v)) = \nabla_v w h + w \nabla_v h 
\\
&= [\nabla_v w(x^\b, v^\b) ]G(x^\b, v^\b)  + 
w(x^\b, v^\b) [\nabla_v  G(x^\b, v^\b)  ]\\
& =  [\nabla_v x^\b \cdot \nabla_{x^\b} + \nabla_v v^\b \cdot \nabla_{v^\b} ]w(x^\b, v^\b) G(x^\b, v^\b)\\
&  \ \ 
+ w(x^\b, v^\b) [\nabla_v x^\b \cdot \nabla_{x^\b} + \nabla_v v^\b \cdot \nabla_{v^\b} ] G(x^\b, v^\b)
\end{split}
\Ee)}

We need to bound $\| \nabla_x^2 \phi \|_\infty$.

\subsubsection{Estimate of $\| \nabla_x^2 \phi \|_\infty$}

Potential estimate (has to be checked for our domain) says
\Be
\| \nabla_x^2 \phi \|_\infty \lesssim  \|\rho \|_\infty \log  \|\rho \|_{C^{0,\gamma}}
\Ee 
Sobolev embedding: We need $\rho \in W^{1, p} \subset C^{0, \gamma}$ for $p>3$ as $\frac{1}{p} - \frac{1}{3}<0$.

and 
\Be
\| \nabla_x ^2 \phi \|_\infty \lesssim  \| w G \|_\infty  \ln \| w \nabla_{x,v} G \|_\infty  (need \ to \ verify \ that \ \ll g^2)
\Ee

[[NO USE YET: Note that $\frac{|v|^2}{ 2} + \phi (x,v) + g m_\pm x_3$ is invariant along the trajectory:
\Be
\frac{|v|^2}{ 2} + \phi (x ) + g m_\pm x_3 = \frac{|v^\b|^2}{ 2} + \phi (x^\b ) + g m_\pm x_3|_{x_3=0} =  \frac{|v^\b(x,v)|^2}{ 2} 
\Ee
This implies that 
\begin{align}
|v| = |v^\b (x,v)| \frac{\p |v^\b (x,v)| }{\p |v|}
\end{align}
]]
\unhide

\hide

\subsubsection{Uniqueness (incomplete)}
Suppose $h_{\pm}$ and $\tilde{h}_\pm$ solve the same problem. Then $h_\pm - \tilde{h}_\pm$ solves 
\Be
\begin{split}
v\cdot \nabla_x (h- \tilde{h}) -  \nabla_x \phi_{\tilde{h}} \cdot \nabla_v (h-\tilde{h}) =  \nabla_x \phi_{h - \tilde{h}} \cdot \nabla_v h ,\\
(h- \tilde{h})|_{\gamma_-}= 0.
\end{split}
\Ee
Using $L^1$ estimate, we get 
\Be
\| h- \tilde{h} \|_{L^1 (\gamma_+)}   \leq \|  \nabla_x \phi_{h - \tilde{h}} \cdot \nabla_v h \|_{L^1} 
\Ee

Trajectory: 
\Be
\frac{dX}{d t}  = V ,  \ \ \frac{d V}{d t} = -  \nabla_x \phi_{\tilde{h}}(t, X)
\Ee
Then 
\Be
(h-\tilde{h})(x,v) = \int^0_{-\tb(x,v)} \nabla_x \phi_{h - \tilde{h}} ( X(s;x,v)) \cdot \nabla_v h ( X(s;x,v), V(s;x,v))
\dd s .
\Ee

\subsection{Special solution}
If the boundary condition is space-homogeneous 
$$
G_\pm (x,v) = \mu (|v|),
$$
From the other papers' argument, what we can get? {\color{red}(CHECK)}

\unhide


\section{Dynamic Solutions}  
In this section, we construct a global-in-time strong solution to the dynamic problem \eqref{eqtn:f}-\eqref{Poisson_f}, and study their properties such as regularity and uniqueness.

\subsection{Construction of Sequences}\label{sec:DC}
We construct a solution to the dynamic problem \eqref{eqtn:f}-\eqref{Poisson_f} via the following sequences: starting with $f^0  \equiv 0$, we set $(\varrho^0,  \Psi^0)= (0,0)$; and then $f^1$ solves 
\Be\label{f_1}
\p_t f^1 + v\cdot \nabla_x f^1 - \nabla_x ( \Phi + g x_3) \cdot \nabla_v f^1 =0, \  \ f^1 |_{\gamma_-} =0, \  \ f^1|_{t=0} = f_0. 
\Ee
Since $ \Phi + g x_3 \in C^1 (\bar \O) \cap C^2 (\O)$, the characteristics to \eqref{ODE_F} equals the steady characteristics $(X,V)$ of \eqref{ODE_h} and hence $ f^1 $ is defined as in \eqref{Lform:f} along the characteristics. 

Suppose that $\Psi^\ell  \in C^1 (\bar \O) \cap C^2 (\O)$ satisfies
\begin{align}
\Delta \Psi^\ell = \varrho^\ell  := \eta \int_{\R^3}   f ^\ell  \dd v,   
\ \ \
\Psi^\ell |_{\p\O}   =0. \label{Poisson_fell}\vspace{-10pt}
\end{align}
Note that $
\phi_{F^\ell} = \Psi^\ell + \Phi. $

The corresponding characteristics is 
\Be
\mathcal{Z}^{\ell+1} (s;t,x,v) = (\X^{\ell+1} (s;t,x,v) , \V^{\ell+1} (s;t,x,v) ) , \label{Zell}
\Ee
solving
\Be\begin{split}\label{ODEell}
\frac{d \X^{\ell+1} }{ds} = \V^{\ell+1} ,&\ \ \  \frac{d \V^{\ell+1} }{ds}  =
- \nabla_x \Psi^\ell - \nabla_x \Phi  - g   \mathbf{e}_3 ,\\
\X ^{\ell+1}  |_{s=t} =x, &\ \ \ \V ^{\ell+1}  |_{s=t}  = v.
\end{split}\Ee
We define $t_{\mathbf{B} }^{\ell+1}(t,x,v),t_{\mathbf{F} }^{\ell+1}(t,x,v),  x_{\mathbf{B} }^{\ell+1}(t,x,v),$ and $ v_{\mathbf{B} }^{\ell+1}(t,x,v)$ as in Definition \ref{def:tb} but for the characteristics $\Zz^{\ell+1}=(\X^{\ell+1},\V^{\ell+1})$ in \eqref{Zell}.

Then we successively construct solutions in the sense of Definition \ref{def:mild} along the characteristics as in \eqref{Lform:f} to the problem
\begin{align}
\p_t f ^{\ell+1} + v\cdot \nabla_x f ^{\ell+1 } - \nabla_x (\Psi^\ell + \Phi+ g   x_3)\cdot \nabla_v f ^{\ell+1} &= \nabla_x \Psi^\ell \cdot \nabla_v h ,
\label{eqtn:fell}
\\
f ^{\ell+1} |_{\gamma_-} &= 0,\label{bdry:fell}\\
f ^{\ell+1} |_{t=0} &= f_{ 0} :=  F_{  0} - h .\label{initial:fell}
\end{align}

From \eqref{def:flux}, \eqref{cont_eqtn}, and \eqref{identity:Psi_t}, we have 
\begin{align}
b^\ell(t,x) := \int_{\R^3} v   f^\ell (t,x,v)  \dd v \ \ \text{in} \ \R_+ \times \O,\label{def:flux_ell}\\
\p_t\varrho^\ell + \nabla_x \cdot b^\ell =0 \ \ \text{in} \ \R_+ \times \O.\label{cont_eqtn_ell}
\end{align}
\hide
\Be\label{def:flux_ell}
b^\ell(t,x) := \int_{\R^3} v   f^\ell (t,x,v)  \dd v ,
\Ee
\Be\label{cont_eqtn_ell}
\p_t\varrho^\ell + \nabla_x \cdot b^\ell =0 \ \ \text{in} \ \R_+ \times \O,
\Ee 
\Be
\begin{split}\label{identity:Psi_t_ell}
\p_t \Psi^\ell (t,x)  = \eta \Delta_0^{-1} \p_t \varrho^\ell(t,x)  = -  \eta \Delta_0^{-1}  (\nabla_x \cdot b^\ell) (t,x).
\end{split}
\Ee

\unhide



\begin{remark}
The continuity equation \eqref{cont_eqtn_ell} should hold in a weak sense against smooth test function with compact support. As what we have done for the steady solution construction, we will prove that the sequence $(f^{\ell+1}, \varrho^\ell, \Psi^\ell)$ belongs to some regularity space. Then, in Lemma \ref{lem:D3tphi_F} and Remark \ref{remark:CE}, we will derive that the continuity equation \eqref{cont_eqtn_ell} holds in a strong sense so that the following identity is valid:
\Be\label{identity:Psi_t_ell}
\p_t \Psi^\ell (t,x)  = \eta \Delta_0^{-1} \p_t \varrho^\ell(t,x)  = -  \eta \Delta_0^{-1}  (\nabla_x \cdot b^\ell) (t,x)\ \ \text{in} \ \R_+ \times \O.\Ee
\end{remark}


Applying Lemma \ref	{lem:tb}, we have the following result:
\begin{lemma}Assume a bootstrap assumptions $ \Psi^\ell (t,\cdot )\in C^1 (\bar \O) \cap C^2 (\O)$ and 
\Be\label{Bootstrap_ell}
\sup_{0 \leq \tau \leq t}	\| \nabla_x \phi_{F^\ell}(\tau )   \| _{L^\infty (\O)}
=	\sup_{0 \leq \tau \leq t}	\| \nabla_x \Psi^\ell(\tau ) + \nabla_x \Phi   \|_{L^\infty (\O)}\leq   \frac{g }{2}.
\Ee 
Then we have that for all $0 \leq s \leq t$
\Be\label{est:tB_ell}
\begin{split}
&t_{\mathbf{B }}^{\ell + 1}	(s,x,v) \leq
\frac{2}{g} \min \Big\{\sqrt{|v_3|^2 + g x_3} -  v_3,  \sqrt{|v_{\mathbf{B}, 3}^{\ell+1} (s,x,v)|^2 - g x_3 }+ v_{\mathbf{B}, 3} ^{\ell+1}(s,x,v)  \Big\} ,
\\
& \tB^{\ell + 1}(s,x,v) 	+ \tF^{\ell + 1} (s,x,v) \leq \frac{4}{g} \sqrt{|v_3|^2 + g x_3}.
\end{split}
\Ee

\hide Then we have 	\Be\label{est:tB_ell}
\max\big\{	t_\mathbf{B}^{\ell+1}(t,x,v) , 	t_\f^{\ell+1}(t,x,v) \big\} \leq  \frac{4}{g}\min\big\{ |v_{\b,3}^{\ell+1}(t,x,v)|,  |v_{\f,3}^{\ell+1}(t,x,v)|\big\}.
\Ee\unhide
\end{lemma}
\hide\begin{proof}

We have 
\Be
\frac{d V_3^F(s;t,x,v)}{ds}  \leq - \frac{g }{2}  
\Ee
Hence
\Be
\begin{split}
0=	X_3^F (s;t,x,v)|_{s= t-t_\b^F} &= x_3 + \int^s_t V_3^F (\tau; t,x,v) \dd \tau\Big|_{s= t-t_\b^F} \\
& = x_3 + v_{\b,3}^F (s-t)\big|_{s= t-t_\b^F} + \int^s_t \int^\tau _{t- t_\b^F}   	\frac{d V_3^F(\tau^\prime;t,x,v)}{ds} \dd  \tau^\prime\dd \tau\Big|_{s= t-t_\b^F}\\
& \leq  x_3- t_\b^F(t,x,v) v_{\b, 3}^F (t,x,v) +\frac{|t_\b^F(t,x,v)|^2}{2} \frac{g }{2}.
\end{split}
\Ee
Therefore we derive that 
\Be
t_\b^F(t,x,v)  \leq \frac{v_{\b,3}^F+ \sqrt{(v_{\b,3}^F)^2 - 4 \frac{g }{4} x_3}}{2\frac{g }{4}} \leq 4 \frac{|v_{\b,3}^F (t,x,v)|}{g }.
\Ee

\end{proof}

{\color{red} Do we need this:
\begin{lemma}Define the maximum height of the trajectory
\Be
\begin{split}\label{MH}
\mathring{x}^{  \ell+1}_{ \pm,3} (t,x,v) 
& = \X^{\ell+1}_{\pm,3}(  \mathring{t}^{  \ell+1}_{ \pm,3} (t,x,v)   ;t,x,v)\\
&
= \max 	\big\{\X^{\ell+1}_{\pm,3}(s;t,x,v) \geq 0:  s \in [t-t_{\mathbf{B}, \pm}^{\ell+1} (t,x,v), t+t_{\mathbf{F}, \pm}^{\ell+1} (t,x,v)  ] \big\}.
\end{split}	\Ee
A position attaining the maximum height is denoted by 
\Be
\label{position_MH}
\mathring{x}^{  \ell+1}_{ \pm } (t,x,v) 
= \X^{\ell+1}_{\pm}( \mathring{t}^{  \ell+1}_{ \pm,3} (t,x,v)   ;t,x,v).
\Ee
Then 
\Be\label{est:MH}
\mathring{x}^{  \ell+1}_{ \pm,3} (t,x,v)   \leq \frac{2}{g m_\pm}
\left(
\frac{|v|^2}{2} + 2 gm_{\pm} x_3+  4  \sqrt{|v_3|^2 + g x_3 }
\right).
\Ee	
\end{lemma}
\begin{proof}
When the trajectory reaches the maximum height, the vertical velocity equals zero:
\Be\label{V=0atMH}
\V_{\pm, 3}^{\ell+1} (   \mathring{t}^{  \ell+1}_{ \pm,3} (t,x,v)   ;t,x,v)=0
\Ee
Taking a time-integration of \eqref{dEC} from $t$ to $\mathring{t}^{  \ell+1}_{ \pm,3} (t,x,v)$ and using \eqref{V=0atMH}, we have 
\Be\label{int:energy_MH}
\begin{split}
&	\frac{|v|^2}{2} + \phi_{f^\ell} (t,x) + \phi_h (x)+ gm_\pm x_3\\
&=  \phi_{f^\ell} ( \mathring{t}^{  \ell+1}_{ \pm,3}  , \mathring{x}^{  \ell+1}_{ \pm }) + \phi_h ( \mathring{x}^{  \ell+1}_{ \pm })+ gm_\pm   \mathring{x}^{  \ell+1}_{ \pm,3}
\\
& + \int_t^{ \mathring{t}^{  \ell+1}_{ \pm,3}  } 
\Delta_0^{-1} (\nabla_x \cdot J^\ell) (s, \X_\pm^{\ell+1} (s;t,x,v)) 
\dd s ,
\end{split}
\Ee
where $ \mathring{t}^{  \ell+1}_{ \pm,3} =  \mathring{t}^{  \ell+1}_{ \pm,3} (t,x,v)$ and $ \mathring{x}^{  \ell+1}_{ \pm,3} =  \mathring{x}^{  \ell+1}_{ \pm,3} (t,x,v)$ were abbreviated.

We expand \eqref{int:energy_MH} using the zero Dirichlet boundary condition of the potentials \eqref{bdry:phi}, \eqref{Poisson_fell} and the upper bound of $t_{\mathbf{B}, \pm}^{\ell+1}$ in \eqref{est:tB}:
\Be\begin{split}\notag
& \big(  g m_\pm  - \{ \| \p_{x_3} \phi_{f^\ell} \|_{L^\infty_{t,x}} + \| \p_{x_3}\phi_{h} \|_{L^\infty_{ x}} \} \big)  \mathring x_{\pm, 3}^{\ell+1} (t,x,v)\\
&\leq 	\frac{|v|^2}{2} + \big(gm_\pm +  \| \p_{x_3} \phi_{f^\ell}  \|_{L^\infty_{t,x}} + \| \p_{x_3} \phi_h \|_{L^\infty_{ x}} \big) x_3    +  \| \Delta_0^{-1} (\nabla_x \cdot J^\ell) \|_{L^\infty_{t,x}} \frac{4}{g}  \sqrt{|v_3|^2 + g x_3 }.
\end{split}\Ee
Finally we conclude \eqref{est:MH} from \eqref{Bootstrap_ell} and \eqref{Bootstrap_J}.
\end{proof}}

{\color{red}Do we need this:

\Be
\begin{split}\label{dsE}
&	\frac{d}{ds}  \big(
|\V^{\ell+1} (s;t,x,v)|^2 
+ 2 \Phi (\X^{\ell+1} (s;t,x,v))
+ 2 g \X^{\ell+1}_{ 3} (s;t,x,v)
\big)\\
& = 2 \V^{\ell+1} (s) \cdot \big( - \nabla_x \Phi (\X^{\ell+1}(s)) - \nabla_x \Psi^\ell (s,\X^{\ell+1}(s) ) - g \mathbf{e}_3 \big)\\
& \  \ + 2 \V^{\ell+1} (s) \cdot \nabla_x \Phi (\X^{\ell+1}(s))+ 2 g \V^{\ell+1}_3 (s)\\
& = -2 \V^{\ell+1} (s  ;t,x,v)  \cdot \nabla_x \Psi^\ell (s, \X^{\ell+1} (s;t,x,v)).
\end{split}
\Ee} \unhide
Define a dynamic weight for the sequence (cf. $\w_\beta$ in \eqref{w^F})
\Be\label{w^ell}
\w^{\ell+1}_\beta (s,x,v )= w_\beta(|v|^2 + 2 \Phi (x) + 2 \Psi^\ell (s,x) + 2 g  x_3) = e^{\beta  \big(|v|^2 + 2 \Phi (x) + 2 \Psi^\ell (s,x) + 2 g  x_3\big)  }.
\Ee
As \eqref{dDTE}, we have 
\Be\label{dEC_ell}
\begin{split}
&\frac{d}{ds} \big(
|\V^{\ell+1} (s;t,z)|^2 
+  2\phi_{F^\ell}(\X^{\ell+1} (s;t,z)) 
+ 2 g  \X^{\ell+1}_{  3} (s;t,z)
\big) = 
2 \p_t \Psi^\ell (s,\X^{\ell+1} (s;t,z))  
.
\end{split}
\Ee

\begin{lemma}\label{lem:w/w_ell}Suppose the assumption \eqref{Bootstrap_ell} holds. 
Then, for $s , s^\prime\in [ \max\{0, t-\tB^{\ell+1} (t,x,v) \}, t 
]$ and $\beta>0$,
\Be
\begin{split}\label{est:1/w_h_ell}
\frac{\w ^{\ell+1}_\beta  (s^\prime,\Zz^{\ell+1}  (s^\prime;t,x,v) )}{\w ^{\ell+1} _\beta (s ,\Zz^{\ell+1}  (s ;t,x,v))}
&\leq 
e^{ \frac{8\beta}{g}  	\| 
\p_t \Psi^\ell	
\|_{L^\infty_{t,x}}   \sqrt{|v_3|^2 + g x_3}
},
\\
\frac{1}{\w ^{\ell+1}_\beta  (s,\Zz^{\ell+1}  (s;t,x,v))}  
&	\leq
e^{  \frac{64 \beta}{g} 	\|	\p_t \Psi^\ell	 
\|_{L^\infty_{t,x}}  ^2 }
e^{-\frac{\beta}{2}|v|^2}  
e^{-\frac{\beta}{2} g x_3}
,
\end{split}\Ee 
and
\Be\label{est:1/w_ell}
\frac{1}{w  _\beta  ( \Zz^{\ell+1}  (s;t,x,v))}   \leq 
e^{\frac{16^2 \beta }{2 g^2}
\| \p_t \Psi^\ell	 
\|_{L^\infty_{t,x}}^2	
}
e^{- \frac{\beta}{4}|v|^2 } e^{- \frac{\beta g}{4} x_3}.
\Ee
Here, we have used the notation $L^\infty_{t,x}$ defined in \eqref{notation}.
\end{lemma}

\begin{proof} Using \eqref{dEC_ell}, we derive that 
\Be	\begin{split}\label{w(Z)_s}
\frac{d}{ds} \w _\beta^{\ell+1} (s, \Zz ^{\ell+1} (s;t,z) )
=  2 \beta 
\p_t \Psi^\ell
(s, \Zz^{\ell+1} (s;t,z)) 
\w _\beta^{\ell+1}  (s, \Zz^{\ell+1} (s;t,z) )
.
\end{split}\Ee

\hide	If $w(\tau)= e^{- \beta \tau }$ then 
\Be
\frac{d}{ds}w_h (X^F (s), V^F(s))  =  w_h (X^F (s), V^F(s))  \beta (V^F(s;t,x,v) \cdot \nabla_x \phi^f (s,X^F(s;t,x,v)))
\Ee  \unhide
Hence, if $\max \{0, t-\tB  ^{\ell+1} (t,x,v) \} \leq s, s^\prime \leq 
t
$ then 
\Be\label{est:w_ell}
\begin{split}
\w _\beta ^{\ell+1}   (s,\Zz^{\ell+1}   (s;t,z) )
=\w _\beta^{\ell+1}    (s^\prime,\Zz ^{\ell+1}  (s^\prime;t,z),  ) 
e^{   2 \beta  \underline{\int^s_{s^\prime}
	\p_t \Psi^\ell(\tau, \X^{\ell+1}  (\tau;t,z)) \dd \tau   }
}.
\end{split}
\Ee


Now we estimate the underlined term in the exponent of \eqref{est:w_ell}: Using \eqref{cont_eqtn}, \eqref{est:tB}, we bound it above by 
\Be\label{est:expo_w_ell}
\begin{split}
|s-s^\prime|  \|  \p_t \Psi^\ell\|_{L^\infty ([s^\prime, s] \times \bar\O)}
\leq |t_{\mathbf{B }} (t,x,v)+ t_{\mathbf{F} }  (t,x,v)| \| \p_t \Psi^\ell \|_{L^\infty_{t,x}}
\leq \|  \p_t \Psi^\ell \|_{L^\infty_{t,x}}  \frac{4}{g} \sqrt{|v_3|^2 + g x_3}   
.
\end{split}
\Ee 
\hide\Be
\begin{split}
&\int_{s^\prime}^s	|\V^{\ell+1}_{\pm}(\tau;t,x,v) \cdot \nabla_x \phi_{f^\ell } (\tau;t,x,v)| \dd \tau \\
& \leq 
\| \nabla_x \phi_{f^\ell} \|_{L^\infty_{t,x}} |t_{\mathbf{B}, \pm }^{\ell+1}  + t_{\mathbf{F}, \pm }^{\ell+1} |  \{ |v| +\| \nabla_x \phi_{f^\ell} \|_{L^\infty_{t,x}}   |t_{\mathbf{B}, \pm }^{\ell+1}  + t_{\mathbf{F}, \pm }^{\ell+1} |   \}\\
& \leq   \Big(\frac{8}{g} \| \nabla_x \phi_{f^\ell} \|_{L^\infty_{t,x}} \Big)^2 \{|v_3|^2 + g x_3\}
+ \frac{8}{g} \| \nabla_x \phi_{f^\ell} \|_{L^\infty_{t,x}} |v| \sqrt{|v_3|^2 + g x_3}
\end{split}
\Ee\unhide
Finally, we conclude \eqref{est:1/w_h_ell} by evaluating \eqref{est:w_ell} at $s^\prime = t$ and using \eqref{est:expo_w_ell}:
\Be\notag
\begin{split}
&\w _\beta ^{\ell+1}  (s,\X  ^{\ell+1} (s;t,x,v), \V ^{\ell+1}  (s;t,x,v)) 
\geq \w  _\beta ^{\ell+1}  (t,x,v) e^{-8 \frac{\beta}{g} \|  \p_t \Psi^\ell \|_{L^\infty_{t,x}}   \sqrt{|v_3|^2 + g x_3} }  \\
& \geq e^{\beta |v|^2 - 8\frac{\beta}{g} 
\| \p_t \Psi^\ell \|_{L^\infty_{t,x}} |v|
} e^{ \beta g  x_3
- 8\frac{\beta}{g} 
\|  \p_t \Psi^\ell \|_{L^\infty_{t,x}} \sqrt{g x_3}
} 
\geq
e^{- \frac{64 \beta}{g} 	\| \p_t \Psi^\ell\|_{L^\infty_{t,x}}  ^2 }
e^{\frac{\beta}{2}|v|^2}  
e^{\frac{\beta}{2} g x_3}
.
\end{split}
\Ee

Now we prove \eqref{est:1/w_ell}. Using \eqref{ODE_F} and \eqref{dEC_ell}, we compute that 
\Be	\begin{split}\label{w(Z)_s:h}
&	\frac{d}{ds} w _\beta   (\Zz^{\ell+1}  (s;t,x,v)
) 
= 	\beta   w _\beta   (\Zz ^{\ell+1} (s;t,x,v) ) 
\frac{d}{ds}  \Big( |\V ^{\ell+1}  (s)|^2  +   2\Phi (\X ^{\ell+1}  (s)) + 2g  \X ^{\ell+1} _{ 3} (s)  \Big) \\
&= -2\beta w _\beta   (\Zz ^{\ell+1}  (s;t,x,v)
) 	 \big( \V ^{\ell+1}  (s ;t,x,v )  \cdot \nabla_x \Psi  (s, \X ^{\ell+1}  (s;t,x,v)) \big)
\\
&  = 2\beta w _\beta (\Zz ^{\ell+1}  (s;t,x,v)
) 	 \big( 
\p_t \Psi ^\ell(s, \X  ^{\ell+1} (s;t,x,v)) - \frac{d}{ds} \Psi ^\ell(s, \X ^{\ell+1} (s;t,x,v)) 
\big)	,
\end{split}\Ee 
where we have also used 
{\small	\Be\label{dsPsi}
\frac{d}{ds} \Psi ^\ell(x, \X ^{\ell+1}(s;t,x,v)) = \p_t \Psi ^\ell(s, \X^{\ell+1} (s;t,x,v))
+ \V^{\ell+1} (s;t,x,v)  \cdot \nabla_x \Psi ^\ell (s, \X ^{\ell+1} (s;t,x,v)).
\Ee}
This implies 
{\small 	\Be\label{est:w_h} 
w_\beta  (\Zz ^{\ell+1} (s;t,x,v) ) 
=w _\beta (\Zz ^{\ell+1} (s^\prime;t,x,v)) 
e^{  2 \beta \int^s_{s^\prime} 
\p_t \Psi ^\ell(\tau , \X^{\ell+1} (\tau ;t,x,v)) 	\dd \tau - 2 \beta  \underline{ \int^s_{s^\prime} \frac{d}{ds} \Psi ^\ell (\tau ,\X^{\ell+1} (\tau ;t,x,v))  
	\dd \tau }
}. 
\Ee}
Using the Dirichlet boundary condition \eqref{Poisson_f}, we estimate the underlined term in the exponent:
\Be
\begin{split}
&	\big|\Psi ^\ell (s,\X ^{\ell+1} (s;t,x,v)) - \Psi^\ell (s^\prime,\X^{\ell+1}  (s^\prime;t,x,v)) \big|   \leq 
2  \| \p_{x_3} \Psi ^\ell \|_{L^\infty_{t,x}} \max_{s}  \X_{  3} ^{\ell+1} (s; t,x,v)\\
& \leq  2 \| \p_{x_3} \Psi^\ell  \|_{L^\infty_{t,x}}  
\Big(
x_3 + |v_3|  \frac{4}{g} \sqrt{|v_3|^2 + g x_3 }
\Big)   \leq \frac{10}{g} \| \p_{x_3} \Psi  ^\ell\|_{L^\infty_{x,v}} \big(|v_3|^2 + g x_3  \big)
.\label{est:expo_w}
\end{split}
\Ee 
Therefore, we conclude \eqref{est:1/w_ell} from \eqref{est:w_h}, \eqref{est:expo_w}, and \eqref{est:expo_w_ell}:
\Be\begin{split}
&w_\beta(\Zz ^{\ell+1}  (s;t,x,v)) \\
& \geq e^{\frac{\beta}{2} |v|^2} e^{\frac{\beta g}{2} x_3} \times e^{\frac{\beta}{2} |v|^2}
e^{- 20 \frac{\beta}{g} \| \p_{x_3} \Psi^\ell  \|_{L^\infty_{x,v}} |v|^2}
e^{- \frac{8\beta}{g} \| \p_t \Psi^\ell\|_{L^\infty_{x,v}} |v|}
\\	& \  \ 
\times 	
e^{\frac{\beta g}{2} x_3} 
e^{- 20 \frac{\beta}{g} \| \p_{x_3} \Psi^\ell  \|_{L^\infty_{x,v}}  gx_3}
e^{- \frac{8\beta}{g} \|  \p_t \Psi^\ell \|_{L^\infty_{x,v}} \sqrt{gx_3}}\\
& \geq e^{\frac{\beta}{2} |v|^2} e^{\frac{\beta g}{2} x_3}   \times e^{- \frac{16^2 \beta}{2 g^2}
\| \p_t \Psi^\ell \|_{L^\infty_{x,v}}^2	 
}  e^{\frac{\beta}{4}  \big( |v| - \frac{16}{g} \| \p_t \Psi^\ell \|_{L^\infty_{x,v}}\big)^2}
e^{\frac{\beta}{4}  \big( \sqrt{gx_3}- \frac{16}{g} \| \p_t \Psi^\ell \|_{L^\infty_{x,v}}\big)^2}\\
& \geq  e^{- \frac{16^2 \beta}{2 g^2}
\|  \p_t \Psi^\ell \|_{L^\infty_{x,v}}^2	 
} e^{\frac{\beta}{2} |v|^2} e^{\frac{\beta g}{2} x_3}.\notag
\end{split}\Ee\hide

where we have used \eqref{est:tB}.

\Be\label{inv:w}
w (\Zz^{\ell+1} (s;t,x,v) ) = w(x,v) e^{2 \beta \int^t_s \V^{\ell+1} (\tau ;t,x,v)
\cdot \nabla_x \Psi^\ell(\tau, \X^{\ell+1} (\tau;t,x,v)) \dd \tau 
}.
\Ee

which implies

For the proof of \eqref{est:1/w}, we compute that

\hide

Using \eqref{est:tB}, we derive that 
\Be
\begin{split}
& \beta \Big|	\int^t_{s }
\V^{\ell+1}_\pm(\tau;t,x,v) \cdot \nabla_x \phi_{f^\ell} (\tau,\X^{\ell+1}_\pm(\tau;t,x,v))
\dd \tau \Big| \\
& \leq \beta t_{\b, \pm }^{\ell+1}(t,x,v)  
\Big\{
|v| + \frac{3g}{2} t_{\b, \pm }^{\ell+1}(t,x,v)
\Big\}\| \nabla_x \phi_{f^\ell} \|_\infty
\\
& \leq  \frac{4 \beta}{g} \| \nabla_x \phi_{f^\ell} \|_\infty \sqrt{|v_3|^2 + g x_3} \{7 |v| + 6 \sqrt{g x_3}\}.\label{growth_w_t}\\
& \leq \frac{8 \beta }{g} \Big\{1+\frac{8\| \nabla_x \phi_F \|_\infty }{g} \Big\}	|v^F_{\b,3}(x,v)| ^2 \| \nabla_x \phi^f \|_\infty	+ \frac{8 \beta }{g} |v_\parallel|  |v^F_{\b,3}(x,v)| \| \nabla_x \phi^f \|_\infty\\	& \leq \frac{8 \beta }{g} 	\| \nabla_x \phi^f \|_\infty \Big\{	\big(2+\frac{8\| \nabla_x \phi_F \|_\infty }{g} \big)	|v^F_{\b,3}(x,v)| ^2  	+    |v_\parallel|^2\Big\},
\end{split}
\Ee
\hide	where we have used, for any $\tau \in [t-t_\b^F(t,x,v), t+ t_\f^F (t,x,v)]$,
\Be
\begin{split}\label{est:V-v}
|V_3^F (\tau;t,x,v)| &\leq |v_{\b,3}^F (t,x,v)  |,\\
\big|	|V_\parallel^F (\tau;t,x,v)| -|v_\parallel| \big|&\leq    \{t^F_\b (t,x,v) + t^F_\f (t,x,v) \} \| \nabla_x \phi_F \|_\infty 
\leq   \frac{8\| \nabla_x \phi_F \|_\infty }{g}|v_{\b,3}^F (t,x,v)  | .
\end{split}
\Ee\unhide
Hence, from \eqref{est:w_h}, we derive that 
\Be
\begin{split}
&w^h_\pm (\X^{\ell+1}_\pm (s;t,x,v), \V^{\ell+1} _\pm (s;t,x,v))\\
& \geq w^h_\pm (x,v)
e^{ -\frac{8 \beta}{g} \| \nabla_x \phi_{f^\ell} \|_\infty \sqrt{|v_3|^2 + g x_3} \{7 |v| + 6 \sqrt{g x_3}\}}\\
& \geq e^{\beta \Big( 1- \frac{8^3}{g} \| \nabla_x \phi _{f^\ell} \|_\infty
\Big) |v|^2}
e^{\beta g\Big(1 - \frac{8^3}{g} \| \nabla_x \phi _{f^\ell} \|_\infty
\Big) x_3}
\end{split}	\Ee

\hide	\Be
\begin{split} 
&\frac{1}{w_h (X^F (s;t,x,v), V^F(s;t,x,v))} \\
&
\leq \frac{ 
e^{ \frac{8 \beta }{g} 
	\| \nabla_x \phi^f \|_\infty \Big\{
	\big(2+\frac{8\| \nabla_x \phi_F \|_\infty }{g} \big)	|v^F_{\b,3}(x,v)| ^2  
	+    |v_\parallel|^2\Big\} }	
}{w_h (X^F (t-t^\b_F (t,x,v);t,x,v), V^F(t-t^\b_F (t,x,v);t,x,v))} 
\\
& = e^{- \frac{\beta}{2} |v_{\b,3}^F (t,x,v)|^2 }
e^{- \frac{\beta}{2} |v_{\b,\parallel}^F (t,x,v)|^2   }
e^{ \frac{8 \beta }{g} 
\| \nabla_x \phi^f \|_\infty \Big\{
\big(2+\frac{8\| \nabla_x \phi_F \|_\infty }{g} \big)	|v^F_{\b,3}(x,v)| ^2  
+    |v_\parallel|^2\Big\} } 
\\
& = e^{- \frac{\beta}{2} \Big(1 - 	  \frac{16 }{g}	\big(2+\frac{8\| \nabla_x \phi_F \|_\infty }{g} \big) - \frac{32 \| \nabla_x \phi_F \|_\infty^2}{g^2}
\Big) |v_{\b,3}^F (t,x,v)|^2} e^{- \frac{\beta}{2} 
\Big(1 - 
\frac{16}{g} 
\| \nabla_x \phi_f \|_\infty 
\Big)
|v_\parallel|^2}  .
\end{split}\Ee\unhide 
\unhide
\unhide
\vspace{-10pt}
\end{proof}

\hide
\begin{proof} Using \eqref{dEC}, we derive that 
\Be	\begin{split}\label{w(Z)_s}
&	\frac{d}{ds} \w^{\ell+1}_\beta (s, \X^{\ell+1} (s;t,x,v), \V^{\ell+1} (s;t,x,v))\\
&=  \beta 	e^{ \beta  ( |\V^{\ell+1}  (s)|^2  +   2\Phi  (\X^{\ell+1}(s)) 
+2 \Psi^{\ell} (s, \X^{\ell+1} (s) )	
+ 2g  \X^{\ell+1}_{ 3} (s)  ) }\\
& \ \  \ \times  \frac{d}{ds}  \Big( |\V^{\ell+1}  (s)|^2  +   2\Phi (\X^{\ell+1} (s))
+  2  \Psi^\ell (s,\X^{\ell+1} (s ))  + 2g \X^{\ell+1}_{3} (s)  \Big) \\
&= 2\beta	 
\p_t \Psi^\ell (s, \X^{\ell+1} (s;t,x,v)) 
\w^{\ell+1}_\beta  (s, \X^{\ell+1} (s;t,x,v), \V^{\ell+1} (s;t,x,v))\\
&= -2 \beta \Delta_0^{-1} (\nabla_x \cdot b^\ell )(s, \X^{\ell+1} (s;t,x,v)) 
\w^{\ell+1} _\beta  (s, \X^{\ell+1} (s;t,x,v), \V^{\ell+1} (s;t,x,v))
.
\end{split}\Ee

\hide	If $w(\tau)= e^{- \beta \tau }$ then 
\Be
\frac{d}{ds}w_h (X^F (s), V^F(s))  =  w_h (X^F (s), V^F(s))  \beta (V^F(s;t,x,v) \cdot \nabla_x \phi^f (s,X^F(s;t,x,v)))
\Ee  \unhide
Hence, if $\max \{0, t-\tB^{\ell+1} (t,x,v) \} \leq s, s^\prime \leq 
t+\tF^{\ell+1} (t,x,v)$ then 
\Be\label{est:w_ell}
\begin{split}
&\w^{\ell+1}_\beta  (s,\X^{\ell+1}  (s;t,x,v), \V^{\ell+1}   (s;t,x,v))\\
& =\w^{\ell+1}_\beta   (s^\prime,\X^{\ell+1}  (s^\prime;t,x,v), \V^{\ell+1} (s^\prime;t,x,v)) 
e^{  -2 \beta  \underline{\int^s_{s^\prime}
	\Delta_0^{-1} (\nabla_x \cdot b^\ell )(\tau, \X^{\ell+1} (\tau;t,x,v)) \dd \tau   }
}.
\end{split}
\Ee


Now we estimate the underlined term in the exponent of \eqref{est:w_ell}: Using \eqref{cont_eqtn_ell}, \eqref{est:tB_ell} and \eqref{est:MH}, we bound it above by 
\Be\label{est:expo_w_ell}
\begin{split}
|s-s^\prime| \| \Delta_0^{-1} (\nabla_x \cdot b^\ell) \|_{L^\infty_{t,x}}
& \leq |t_{\mathbf{B }}^{\ell+1} (t,x,v)+ t_{\mathbf{F} }^{\ell+1} (t,x,v)| \| \Delta_0^{-1} (\nabla_x \cdot b^\ell) \|_{L^\infty_{t,x}}
\\
& \leq \| \Delta_0^{-1} (\nabla_x \cdot b^\ell) \|_{L^\infty_{t,x}}  \frac{4}{g} \sqrt{|v_3|^2 + g x_3}   
.
\end{split}
\Ee 
\hide\Be
\begin{split}
&\int_{s^\prime}^s	|\V^{\ell+1}_{\pm}(\tau;t,x,v) \cdot \nabla_x \phi_{f^\ell } (\tau;t,x,v)| \dd \tau \\
& \leq 
\| \nabla_x \phi_{f^\ell} \|_{L^\infty_{t,x}} |t_{\mathbf{B}, \pm }^{\ell+1}  + t_{\mathbf{F}, \pm }^{\ell+1} |  \{ |v| +\| \nabla_x \phi_{f^\ell} \|_{L^\infty_{t,x}}   |t_{\mathbf{B}, \pm }^{\ell+1}  + t_{\mathbf{F}, \pm }^{\ell+1} |   \}\\
& \leq   \Big(\frac{8}{g} \| \nabla_x \phi_{f^\ell} \|_{L^\infty_{t,x}} \Big)^2 \{|v_3|^2 + g x_3\}
+ \frac{8}{g} \| \nabla_x \phi_{f^\ell} \|_{L^\infty_{t,x}} |v| \sqrt{|v_3|^2 + g x_3}
\end{split}
\Ee\unhide
Finally, we conclude \eqref{est:1/w_h} by evaluating \eqref{est:w_ell} at $s^\prime = t$ and using \eqref{est:expo_w_ell}:
\Be\notag
\begin{split}
&\w^{\ell+1} _\beta  (s,\X ^{\ell+1} (s;t,x,v), \V ^{\ell+1} (s;t,x,v))\\
&\geq \w^{\ell+1}  _\beta  (t,x,v) e^{-8 \frac{\beta}{g} \| \Delta_0^{-1} (\nabla_x \cdot b^\ell) \|_{L^\infty_{t,x}}   \sqrt{|v_3|^2 + g x_3} }  \\
& \geq e^{\beta |v|^2 - 8\frac{\beta}{g} 
\| \Delta_0^{-1} (\nabla_x \cdot b^\ell) \|_{L^\infty_{t,x}} |v|
} e^{ \beta g  x_3
- 8\frac{\beta}{g} 
\| \Delta_0^{-1} (\nabla_x \cdot b^\ell) \|_{L^\infty_{t,x}} \sqrt{g x_3}
}
\\
& \geq
e^{- \frac{64 \beta}{g} 	\| \Delta_0^{-1} (\nabla_x \cdot b^\ell) \|_{L^\infty_{t,x}}  ^2 }
e^{\frac{\beta}{2}|v|^2}  
e^{\frac{\beta}{2} g x_3}
.
\end{split}
\Ee

Now we prove \eqref{est:1/w}. Using \eqref{ODEell} and \eqref{dsE}, we compute that 
\Be	\begin{split}\label{w(Z)_s:h}
&	\frac{d}{ds} w _\beta   (\Zz^{\ell+1} (s;t,x,v)
)\\
&= 	\beta   w _\beta   (\Zz^{\ell+1} (s;t,x,v) ) 
\frac{d}{ds}  \Big( |\V^{\ell+1}  (s)|^2  +   2\Phi (\X^{\ell+1} (s)) + 2g  \X^{\ell+1}_{ 3} (s)  \Big) \\
&= -2\beta w _\beta   (\Zz^{\ell+1} (s;t,x,v)
) 	 \big( \V^{\ell+1} (s ;t,x,v )  \cdot \nabla_x \Psi^\ell (s, \X^{\ell+1} (s;t,x,v)) \big)
\\
&  = 2\beta w _\beta (\Zz^{\ell+1} (s;t,x,v)
) 	 \big( 
\p_t \Psi^\ell (s, \X^{\ell+1} (s;t,x,v)) - \frac{d}{ds} \Psi^\ell (s, \X^{\ell+1} (s;t,x,v)) 
\big)	,
\end{split}\Ee 
where we have also used 
\Be\label{dsPsi}
\frac{d}{ds} \Psi^\ell (x, \X^{\ell+1} (s;t,x,v)) = \p_t \Psi^\ell (s, \X^{\ell+1} (s;t,x,v))
+ \V^{\ell+1} \cdot \nabla_x \Psi^\ell  (s, \X^{\ell+1} (s;t,x,v)).
\Ee
This implies 
\Be\label{est:w_h}
\begin{split}
&w_\beta  (\Zz^{\ell+1}  (s;t,x,v) )\\
& =w _\beta (\Zz^{\ell+1}  (s^\prime;t,x,v)) 
e^{  2 \beta \int^s_{s^\prime} 
\p_t \Psi^\ell (\tau , \X^{\ell+1} (\tau ;t,x,v)) 	\dd \tau - 2 \beta  \underline{ \int^s_{s^\prime} \frac{d}{ds} \Psi^\ell (\tau ,\X^{\ell+1} (\tau ;t,x,v))  
	\dd \tau }
}.
\end{split}
\Ee
Using the zero Dirichlet boundary condition \eqref{Poisson_fell}, we estimate the underlined term in the exponent:
\Be
\begin{split}
&  
\big|\Psi^\ell (s,\X^{\ell+1} (s;t,x,v)) - \Psi^\ell (s^\prime,\X^{\ell+1} (s^\prime;t,x,v)) \big|\\
& \leq 
2  \| \p_{x_3} \Psi^\ell \|_{L^\infty_{t,x}} \max_{s}  \X_{  3}^{\ell+1}(s; t,x,v)\\
& \leq  2 \| \p_{x_3} \Psi^\ell \|_{L^\infty_{t,x}}  
\Big(
x_3 + |v_3|  \frac{4}{g} \sqrt{|v_3|^2 + g x_3 }
\Big) \\
& \leq \frac{10}{g} \| \p_{x_3} \Psi^\ell \|_{L^\infty_{x,v}} \big(|v_3|^2 + g x_3  \big)
.\label{est:expo_w}
\end{split}
\Ee 
Therefore, we conclude \eqref{est:1/w} from \eqref{est:w_h}, \eqref{est:expo_w}, and \eqref{est:expo_w_ell}:
\Be\begin{split}
w_\beta(\Zz^{\ell+1} (s;t,x,v)) & \geq e^{\frac{\beta}{2} |v|^2} e^{\frac{\beta g}{2} x_3} \\
& \   \ \times e^{\frac{\beta}{2} |v|^2}
e^{- 20 \frac{\beta}{g} \| \p_{x_3} \Psi^\ell \|_\infty |v|^2}
e^{- \frac{8\beta}{g} \| \Delta_0^{-1} (\nabla \cdot b^\ell) \|_\infty |v|}
\\
& \  \  \times 	
e^{\frac{\beta g}{2} x_3} 
e^{- 20 \frac{\beta}{g} \| \p_{x_3} \Psi^\ell \|_\infty  gx_3}
e^{- \frac{8\beta}{g} \| \Delta_0^{-1} (\nabla \cdot b^\ell) \|_\infty \sqrt{gx_3}}\\
& \geq e^{\frac{\beta}{2} |v|^2} e^{\frac{\beta g}{2} x_3}  \\
& \  \  \times e^{- \frac{16^2 \beta}{2 g^2}
\| \Delta_0^{-1} (\nabla \cdot b^\ell) \|_\infty^2	 
}  e^{\frac{\beta}{4}  \big( |v| - \frac{16}{g} \| \Delta_0^{-1} (\nabla \cdot b^\ell) \|_\infty\big)^2}
e^{\frac{\beta}{4}  \big( \sqrt{gx_3}- \frac{16}{g} \| \Delta_0^{-1} (\nabla \cdot b^\ell) \|_\infty\big)^2}\\
& \geq  e^{- \frac{16^2 \beta}{2 g^2}
\| \Delta_0^{-1} (\nabla \cdot b^\ell) \|_\infty^2	 
} e^{\frac{\beta}{2} |v|^2} e^{\frac{\beta g}{2} x_3}.\notag
\end{split}\Ee

\hide

where we have used \eqref{est:tB}.

\Be\label{inv:w}
w (\Zz^{\ell+1} (s;t,x,v) ) = w(x,v) e^{2 \beta \int^t_s \V^{\ell+1} (\tau ;t,x,v)
\cdot \nabla_x \Psi^\ell(\tau, \X^{\ell+1} (\tau;t,x,v)) \dd \tau 
}.
\Ee

which implies

For the proof of \eqref{est:1/w}, we compute that

\hide

Using \eqref{est:tB}, we derive that 
\Be
\begin{split}
& \beta \Big|	\int^t_{s }
\V^{\ell+1}_\pm(\tau;t,x,v) \cdot \nabla_x \phi_{f^\ell} (\tau,\X^{\ell+1}_\pm(\tau;t,x,v))
\dd \tau \Big| \\
& \leq \beta t_{\b, \pm }^{\ell+1}(t,x,v)  
\Big\{
|v| + \frac{3g}{2} t_{\b, \pm }^{\ell+1}(t,x,v)
\Big\}\| \nabla_x \phi_{f^\ell} \|_\infty
\\
& \leq  \frac{4 \beta}{g} \| \nabla_x \phi_{f^\ell} \|_\infty \sqrt{|v_3|^2 + g x_3} \{7 |v| + 6 \sqrt{g x_3}\}.\label{growth_w_t}\\
& \leq \frac{8 \beta }{g} \Big\{1+\frac{8\| \nabla_x \phi_F \|_\infty }{g} \Big\}	|v^F_{\b,3}(x,v)| ^2 \| \nabla_x \phi^f \|_\infty	+ \frac{8 \beta }{g} |v_\parallel|  |v^F_{\b,3}(x,v)| \| \nabla_x \phi^f \|_\infty\\	& \leq \frac{8 \beta }{g} 	\| \nabla_x \phi^f \|_\infty \Big\{	\big(2+\frac{8\| \nabla_x \phi_F \|_\infty }{g} \big)	|v^F_{\b,3}(x,v)| ^2  	+    |v_\parallel|^2\Big\},
\end{split}
\Ee
\hide	where we have used, for any $\tau \in [t-t_\b^F(t,x,v), t+ t_\f^F (t,x,v)]$,
\Be
\begin{split}\label{est:V-v}
|V_3^F (\tau;t,x,v)| &\leq |v_{\b,3}^F (t,x,v)  |,\\
\big|	|V_\parallel^F (\tau;t,x,v)| -|v_\parallel| \big|&\leq    \{t^F_\b (t,x,v) + t^F_\f (t,x,v) \} \| \nabla_x \phi_F \|_\infty 
\leq   \frac{8\| \nabla_x \phi_F \|_\infty }{g}|v_{\b,3}^F (t,x,v)  | .
\end{split}
\Ee\unhide
Hence, from \eqref{est:w_h}, we derive that 
\Be
\begin{split}
&w^h_\pm (\X^{\ell+1}_\pm (s;t,x,v), \V^{\ell+1} _\pm (s;t,x,v))\\
& \geq w^h_\pm (x,v)
e^{ -\frac{8 \beta}{g} \| \nabla_x \phi_{f^\ell} \|_\infty \sqrt{|v_3|^2 + g x_3} \{7 |v| + 6 \sqrt{g x_3}\}}\\
& \geq e^{\beta \Big( 1- \frac{8^3}{g} \| \nabla_x \phi _{f^\ell} \|_\infty
\Big) |v|^2}
e^{\beta g\Big(1 - \frac{8^3}{g} \| \nabla_x \phi _{f^\ell} \|_\infty
\Big) x_3}
\end{split}	\Ee
\hide	\Be
\begin{split} 
&\frac{1}{w_h (X^F (s;t,x,v), V^F(s;t,x,v))} \\
&
\leq \frac{ 
e^{ \frac{8 \beta }{g} 
	\| \nabla_x \phi^f \|_\infty \Big\{
	\big(2+\frac{8\| \nabla_x \phi_F \|_\infty }{g} \big)	|v^F_{\b,3}(x,v)| ^2  
	+    |v_\parallel|^2\Big\} }	
}{w_h (X^F (t-t^\b_F (t,x,v);t,x,v), V^F(t-t^\b_F (t,x,v);t,x,v))} 
\\
& = e^{- \frac{\beta}{2} |v_{\b,3}^F (t,x,v)|^2 }
e^{- \frac{\beta}{2} |v_{\b,\parallel}^F (t,x,v)|^2   }
e^{ \frac{8 \beta }{g} 
\| \nabla_x \phi^f \|_\infty \Big\{
\big(2+\frac{8\| \nabla_x \phi_F \|_\infty }{g} \big)	|v^F_{\b,3}(x,v)| ^2  
+    |v_\parallel|^2\Big\} } 
\\
& = e^{- \frac{\beta}{2} \Big(1 - 	  \frac{16 }{g}	\big(2+\frac{8\| \nabla_x \phi_F \|_\infty }{g} \big) - \frac{32 \| \nabla_x \phi_F \|_\infty^2}{g^2}
\Big) |v_{\b,3}^F (t,x,v)|^2} e^{- \frac{\beta}{2} 
\Big(1 - 
\frac{16}{g} 
\| \nabla_x \phi_f \|_\infty 
\Big)
|v_\parallel|^2}  .
\end{split}\Ee\unhide 
\unhide
\unhide
\end{proof}
\unhide

\begin{lemma}\label{lem:Linfty_dyn}For an arbitrary $\ell \in \N$, we suppose the bootstrap assumption \eqref{Bootstrap_ell} holds for $\Psi^\ell \in C^1 (\bar \O) \cap C^2 (\O)$. Then $(f^{\ell+1}, \varrho^\ell)$ solving \eqref{eqtn:fell}-\eqref{initial:fell} and \eqref{Poisson_fell} in the sense of Definition \ref{def:mild} satisfies that, for all $(s,x,v) \in [0,t] \times  \bar\O \times \R^3$,
\Be\label{Uest:wf}
\begin{split} 
&e^{  \frac{
	\beta}{2}|v|^2}
e^{  \frac{
	\beta}{2} g x_3}
| f^{\ell+1}   (s,x,v)|
+ \beta^{-3/2}  e^{  \frac{
	\beta}{2} g x_3}
| \varrho^{\ell+1}   (s,x)|\\
&\lesssim  
e^{\frac{64
	\beta}{g}
\sup_{\tau \in [0,s]}	\| \p_t \Psi^\ell 
(\tau) \|_{L^\infty (\O)}^2
}  \big\{   \| \w_{\beta, 0 } F_{0} \|_{L^\infty (\O \times \R^3)}
+  \| 
e^{\beta|v|^2}
G  \|_{L^\infty (\gamma_-)}\big\}  
.
\end{split}	\Ee
\end{lemma}
\begin{proof}\hide
From \eqref{Bootstrap_ell}, \eqref{w^ell}, and \eqref{w^h}, we have that, for any $0< \hat \beta < \beta$, 
\Be\notag
\begin{split}
\w^{\ell+1}_{\beta/2} (t,x,v) & \leq  w_{ \beta }(x,v) e^{
- \beta \big( g x_3 -  \Psi^{\ell} (t,x) \big)
}  
\\
&
\leq  w_{ \beta }(x,v) e^{
- \beta \big( g   - \|  \nabla_x  \Psi^{\ell} \|_\infty \big) x_3
}  \\
& \leq  w_{ \beta }(x,v) .
\end{split}
\Ee\unhide
We express $
f^{\ell+1}(t,x,v)$ as  
\Be
\begin{split}\label{decom:wf}
|
f^{\ell+1}(t,x,v)|  
= |  F^{\ell+1} (t,x,v)-   h(x,v)|
\leq  |  F^{\ell+1} (t,x,v)  |  + e^{- \beta |v|^2} e^{- \beta g x_3} \| w_\beta h  \|_\infty,
\end{split}\Ee
where we have used \eqref{Bootstrap_ell}. Since we already have bounded $\| w_\beta h\|_\infty$ in \eqref{Uest:wh}, it suffices to estimate $|   F^{\ell+1}   (t,x,v)| $.

Along the characteristics \eqref{ODEell}, for $s \in [ \max\{0, t- t^{\ell+1 }_{\mathbf{B} } (t,x,v)\}, t- t^{\ell+1 }_{\mathbf{F} } (t,x,v)]$, 
\Be\begin{split}\notag
\frac{d}{ds} 
F^{\ell+1}  (s, \mathcal{Z} ^{\ell+1} (s;t,x,v) )
=0.
\end{split}\Ee
Using the boundary condition \eqref{bdry:F} and the initial datum, we derive that  
\begin{align}\label{form:Fell}
F^{\ell+1 }  (t,x,v) 
= 
\mathbf{1}_{   t_{\mathbf{B}}^{\ell+1  } (t,x,v)\geq t }
F _{0 }  (\mathcal{Z} ^{\ell +1} (0;t,x,v)) 
+ \mathbf{1}_{t > t_{\mathbf{B}}^{\ell+1  } (t,x,v)}
G  
(
\mathcal{Z} ^{\ell }(t-  t_{\mathbf{B} }^{\ell +1}(t,x,v);t,x,v) ) . 
\end{align} 

Using Lemma \ref{lem:w/w_ell} and \eqref{est:1/w_h_ell}, we derive that 
$ | F ^{\ell+1} (t,x,v)|  \leq I_1 + I_2  ,$ 
where
\begin{align}
I_1  
& \leq
\frac{	\|  \w_{  { \beta}   , 0 }
F_{0 } \|_{L^\infty_{x,v}}}{\w^{\ell+1}_{
	\beta} (0,\Zz^{\ell+1} (0;t,z))} 
\leq   
e^{\frac{64
		\beta}{g}
	\| 
	\p_t \Psi^\ell 
	\|_{L^\infty_{t,x}}^2
} 
e^{- \frac{
		\beta}{2}|v|^2}
e^{- \frac{
		\beta}{2} g x_3}
\|  \w_{  { \beta}   , 0 }
F_{0 } \|_{L^\infty_{x,v}},
\label{est:Fell1}
\\
I_2 &  \leq  	 
\frac{ \| w_{\beta  }   
G  \|_{L^\infty (\gamma_-)}
}{    \w^{\ell+1} _{\beta}   (
t-  t_{\mathbf{B} }^{\ell+1 } 	,
\mathcal{Z} ^{\ell+1} (t-  t_{\mathbf{B} }^{\ell+1 };t,z))  }
\leq   
e^{  \frac{64 \beta}{g} 	\|	\p_t \Psi^\ell	 
\|_{L^\infty_{t,x}}  ^2 }
e^{-\frac{\beta}{2}|v|^2}  
e^{-\frac{\beta}{2} g x_3}
\| w_{\beta  }   
G  \|_{L^\infty (\gamma_-)}. \label{est:Fell2} 
\end{align} 
Here we have used the following facts from \eqref{W_t=0} and \eqref{w:bdry}:
\begin{align}
\w  ^{\ell+1}_\beta (0, x,v)= \w_{\beta, 0 }(x,v)
\ \ \text{in}&   \ (x,v) \in \bar \O \times \R^3,
\label{w_initial}\\
\w ^{\ell+1}_\beta (t 	,x,v) = e^{\beta |v|^2} = w_\beta (x,v) \ \ \text{on}& \  (x,v) \in \p\O \times \R^3. \label{w_bdry} 
\end{align} 

\hide
Now we bound $\eqref{est:Fell1}_*$ by 
completing the square:
\Be\begin{split}\notag
\eqref{est:Fell1}_* & \leq   e^{- \frac{(1-\delta)\beta }{4}|v|^2} 	e^{ -\frac{(1-\delta) \beta}{4} g   x_3}
e^{\frac{64 (1- \delta) \beta}{g}  \| \Delta_0^{-1} (\nabla \cdot b^\ell) \|^2_{L^\infty_{t,x}} }\\
& \ \  \times 
e^{- \frac{(1-\delta)\beta }{4}|v|^2}  e^{ 8   \frac{\delta \beta}{g}
\| \Delta_0^{-1} ( \nabla  \cdot b^\ell) \|_{L^\infty_{t,x}}  |v|   
}
e^{ -\frac{(1-\delta) \beta}{4} g   x_3}
e^{ 8   \frac{\delta\beta}{g}
\| \Delta_0^{-1} ( \nabla  \cdot b^\ell) \|_{L^\infty_{t,x}} \sqrt{  g x3}   
}\\
& \leq   e^{- \frac{\delta\beta |v|^2}{8}} 	e^{ -\frac{\delta  \beta g  x_3}{8}} 
e^{ \frac{16^2  \delta^2 \beta}{2g^2  (1-\delta)}	\| \Delta_0^{-1} (\nabla \cdot b^\ell) \|_{L^\infty_{t,x}}^2}
\\
&  \ \ \    \times 
e^{- \frac{ (1-\delta) \beta}{4} \left(|v| - \frac{16  \delta }{g (1-\delta)}
\| \Delta_0^{-1} (\nabla  \cdot b^\ell) \|_{L^\infty_{t,x}}
\right)^2} 
e^{-   \frac{ (1-\delta) \beta}{4}  \left(
\sqrt{g x_3} -  \frac{16  \delta }{g (1-\delta)}	\| \Delta_0^{-1} (\nabla \cdot b^\ell) \|_{L^\infty_{t,x}}
\right)^2}\\
& \leq    e^{- \frac{(1-\delta)\beta }{4}|v|^2} 	e^{ -\frac{(1-\delta) \beta}{4} g   x_3}
e^{ \frac{16^2  \delta^2 \beta}{2g^2  (1-\delta)}	\| \Delta_0^{-1} (\nabla \cdot b^\ell) \|_{L^\infty_{t,x}}^2}.
\end{split}\Ee
\unhide
Using this bound of $\eqref{est:Fell1}_*$, together with \eqref{est:Fell1} and \eqref{est:Fell2}, we derive a bound of $F^{\ell+1} (t,x,v)$. Then combining with \eqref{decom:wf}, we can conclude this lemma. \end{proof}

\hide
\begin{proof}
We derive the first identity of \eqref{D^-1 Db} from \eqref{phi_rho}; and then the second identity from the integration by parts and the zero Dirichlet boundary condition of \eqref{Dphi}. 

Note that 
\Be\begin{split}\notag
| b^\ell (t,x) |  & 
\leq \int_{\R^3} |v| |f^\ell (t,x,v)| \dd v 
\\
& \lesssim  \left(e^{-\frac{  \beta  }{2} g x_3} 
\int_{\R^3} |v| e^{- \frac{\beta}{2}|v|^2} \dd v \right)
\sup_{0 \leq t \leq T } \|   e^{  \frac{\beta}{2} (|v|^2 + gx_3)}  f^\ell(t) \|_{\infty}\\
& \lesssim  \frac{e^{-\frac{  \beta  }{2} g x_3}  }{\beta^2}
\sup_{0 \leq t \leq T } \|   e^{  \frac{\beta}{2} (|v|^2 + gx_3)}  f^\ell(t) \|_{\infty}.
\end{split}\Ee
Now we utilize Lemma \ref{lemma:G}, and follow the proof of Lemma \ref{lem:rho_to_phi} to conclude the lemma.
\end{proof}
\unhide
\hide

First, we bound \eqref{wF1}. For any $\delta>0$, we have 
\Be
\begin{split}
|\eqref{wF1}| & \leq  \frac{e^{ 2\beta  |t_{\mathbf{B} } ^{\ell+1, \pm} (t,x,v)|  \| \p_t \phi_{F^\ell} \|_{L^\infty_{t,x}} } 
}{ |w_{0, \pm}(\X_\pm^{\ell+1} (0;t,x,v), \V_\pm^{\ell+1} (0;t,x,v) )|^\delta}
\| w_{0, \pm }^{1+ \delta} F_{0, \pm} \|_{L^\infty_x}\\
&  \leq 
\end{split}
\Ee
Here, in the second line, we have used $t_{\mathbf{B}, \pm}^{\ell+1} (t,x,v) \leq \frac{4}{g} \sqrt{|v_3|^2 + gx_3}$ from \eqref{est:tB}, and from \eqref{w_ell}.

\Be\begin{split}
&\frac{d}{ds} w^{\ell+1} f^{\ell+1} (s, \mathcal{Z}^{\ell+1} (s;t,x,v) )\\
& \  \ -  
2 \beta \p_t \phi_{F^\ell} (s, \X^{\ell+1} (s;t,x,v)_\pm) w^{\ell+1} 
f^{\ell+1} (s, \mathcal{Z}^{\ell+1} (s;t,x,v) ) \\
&= w^{\ell+1}  (s, \X^{\ell+1} (s;t,x,v)_\pm)  \nabla_x \phi_{f^\ell} \cdot \nabla_v h  (  \mathcal{Z}_\pm^{\ell+1} (s;t,x,v))
\end{split}\Ee

\Be\begin{split}
&\frac{d}{ds}  \Big[
e^{\int^t_s  2 \beta  \p_t \phi_{F^\ell} (\tau, \X^{\ell+1} (\tau;t,x,v)_\pm) \dd \tau }
w^{\ell+1} f^{\ell+1} (s, \mathcal{Z}^{\ell+1} (s;t,x,v) )
\Big] \\
&= 	e^{\int^t_s  2 \beta  \p_t \phi_{F^\ell} (\tau, \X^{\ell+1} (\tau;t,x,v)_\pm) \dd \tau }  w^{\ell+1}  (s, \X^{\ell+1} (s;t,x,v)_\pm)  \nabla_x \phi_{f^\ell} \cdot \nabla_v h  (  \mathcal{Z}_\pm^{\ell+1} (s;t,x,v))
\end{split}\Ee

\Be\begin{split}
&w^{\ell+1} f^{\ell+1}_\pm (t,x,v) \\
&= \mathbf{1}_{t <  t_{\mathbf{B} }} e^{  \int^t_{0}  2 \beta \p_t \phi_{F^\ell} (s, \X^{\ell+1} (s;t,x,v)_\pm)  \dd s }w^{\ell+1}_\pm(0,\X(0),\V(0))   f^{\ell+1} _\pm(0,\X_\pm(0),\V_\pm(0))  \\
& \ \  + \mathbf{1}_{t <  t_{\mathbf{B} }}
\int^t_0
e^{ 2 \beta\int_s^t  \p_t \phi_{F^\ell} (\tau) \dd \tau }
w^{\ell+1}  (s, \X^{\ell+1} (s;t,x,v)_\pm)  \nabla_x \phi_{f^\ell} \cdot \nabla_v h  (  \mathcal{Z}_\pm^{\ell+1} (s;t,x,v))
\dd s \\
& \ \  + \mathbf{1}_{t  \geq  t_{\mathbf{B} }}
\int^t_{t-t_{\mathbf{B}}}
e^{ 2 \beta\int_s^t  \p_t \phi_{F^\ell} (\tau) \dd \tau }
w^{\ell+1}  (s, \X^{\ell+1} (s;t,x,v)_\pm)  \nabla_x \phi_{f^\ell} \cdot \nabla_v h  (  \mathcal{Z}_\pm^{\ell+1} (s;t,x,v))
\dd s 
\end{split}\Ee

\unhide

\subsection{Regularity Estimate}\label{sec:RD} In this section we study the higher regularity of $F^{\ell+1}(t,x,v)= h(x,v)+f^{\ell+1}(t,x,v)$ that we have constructed in the previous step. 
Note that 
\Be\begin{split}\label{form:F}
F^{\ell+1}(t,x,v) & = \mathbf{1}_{t \leq \tB(t,x,v) } F_0 (\X^{\ell+1}(0;t,x,v), \V^{\ell+1}(0;t,x,v))\\
&  + \mathbf{1}_{t >\tB^{\ell+1}(t,x,v) }  G(  \xB^{\ell+1}(t,x,v), \vB^{\ell+1}(t,x,v)).
\end{split}\Ee

Assume a compatibility condition \eqref{CC}. 
Due to this compatibility condition \eqref{CC}, weak derivatives of $F^{\ell+1}(t,x,v)$ in \eqref{form:F} are
\Be\begin{split}\label{DF_x}
\p_{x_i}F^{\ell+1}(t,x,v) & = \mathbf{1}_{t \leq \tB(t,x,v) }  \{\p_{x_i} \X^{\ell+1} (0 ) \cdot \nabla_x F_0 (\Zz^{\ell+1}(0 )) + \p_{x_i} \V^{\ell+1} (0 ) \cdot \nabla_v F_0 (\Zz^{\ell+1}(0 ))
\}\\
&  + \mathbf{1}_{t >\tB^{\ell+1}(t,x,v) }  \{ 
\p_{x_i} \xB^{\ell+1} \cdot \nabla_{x_\parallel} G 
+ \p_{x_i} \vB^{\ell+1} \cdot \nabla_vG  
\},
\end{split}\Ee
\Be\begin{split}\label{DF_v}
\p_{v_i}F^{\ell+1}(t,x,v) & = \mathbf{1}_{t \leq \tB(t,x,v) }  \{\p_{v_i} \X ^{\ell+1}(0 ) \cdot \nabla_x F_0 (\Zz^{\ell+1}(0 )) + \p_{v_i} \V^{\ell+1} (0 ) \cdot \nabla_v F_0 (\Zz^{\ell+1}(0 ))
\}\\
&  + \mathbf{1}_{t >\tB^{\ell+1}(t,x,v) }  \{
\p_{v_i} \xB^{\ell+1} \cdot \nabla_{x_\parallel} G 
+ \p_{v_i} \vB^{\ell+1} \cdot \nabla_vG  
\},
\end{split}\Ee
where $\Zz^{\ell+1}(0) =\Zz^{\ell+1}(0;t,x,v)$ and $(\X^{\ell+1}(0), \V^{\ell+1}(0)) = (\X^{\ell+1}(0;t,x,v), \V^{\ell+1}(0;t,x,v))$; and every $G$ is evaluated at $(t-\tB^{\ell+1}(t,x,v), \xB^{\ell+1}(t,x,v),\vB^{\ell+1}(t,x,v))$.

Following the same proof of Lemma \ref{lem:XV_xv}, we can derive the following estimate (\cite{CKL_VPB,ChKL}):
\begin{align}
|\nabla_{x} \X^{\ell+1}  (s;t,x,v)| &\leq   \min \big\{  e^{\frac{|t-s|^2}{2} \| \nabla_x^2 \phi_{F^\ell}  \|_\infty },
e^{ (1+ \| \nabla_x ^2 \phi_{F^\ell}  \|_\infty)|t-s|}
\big\},\label{est:X_x:dyn}\\
|\nabla_{x} \V^{\ell+1}   (s;t,x,v)| &\leq   \min \big\{   |t-s| \| \nabla_x^2 \phi_{F^\ell}\|_\infty e^{\frac{|t-s|^2}{2} \| \nabla_x^2\phi_{F^\ell} \|_\infty },
e^{ (1+ \| \nabla_x ^2\phi_{F^\ell} \|_\infty)|t-s|}
\big\},\label{est:V_x:dyn}\\
|\nabla_{v} \X ^{\ell+1}  (s;t,x,v)| &\leq   \min \big\{ |t-s|e^{ \frac{|t-s|^2}{2} \| \nabla_x^2\phi_{F^\ell} \|_\infty },
e^{ (1+ \| \nabla_x ^2 \phi_{F^\ell}  \|_\infty)|t-s|}
\big\},\label{est:X_v:dyn}\\
|\nabla_{v} \V^{\ell+1}  (s;t,x,v)| &\leq  \min \big\{ e^{ \frac{|t-s|^2}{2} \| \nabla_x^2\phi_{F^\ell}  \|_\infty },
e^{ (1+ \| \nabla_x ^2\phi_{F^\ell} \|_\infty)|t-s|}
\big\}.\label{est:V_v:dyn}
\end{align}
Here, we have used the notation $L^\infty_{t,x}$ defined in \eqref{notation}.

We also follow the same proof of Lemma \ref{lem:nabla_zb} to get (\cite{CKL_VPB,ChKL}) 
\Be\label{est:xb_x:dyn}
\begin{split}
&	|\p_{x_i} \xB^{\ell+1}  (t,x,v)|  \leq \frac{|\vB^{\ell+1} (t,x,v)|}{|v^{\ell+1} _{\mathbf{B}  , 3}  (t,x,v)|}
\delta_{i3} 
\\
&  + \Big(1 +  \frac{|\vB^{\ell+1} (t,x,v)|}{|v_{\mathbf{B}, 3} ^{\ell+1}  (t,x,v)|}
\frac{|\tB^{\ell+1}  |^2}{2} \| \nabla_x^2 \phi_{F^\ell}  \|_\infty\Big)  \min \big\{  e^{\frac{|\tB^{\ell+1} |^2}{2} \| \nabla_x^2 \phi_{F^\ell}  \|_\infty },
e^{ (1+ \| \nabla_x ^2 \phi_{F^\ell}  \|_\infty)\tB}
\big\} ,
\end{split}
\Ee
\Be\label{est:xb_v:dyn}
\begin{split}
&	|\p_{v_i} \xB^{\ell+1}   (t,x,v)|  \leq \frac{|\vB ^{\ell+1}  (t,x,v)| |\tB^{\ell+1}  (t,x,v)|  }{|v_{\mathbf{B}, 3} ^{\ell+1}  (t,x,v)|}
\delta_{i3} 
\\
&  + \Big(1 +  \frac{|\vB  ^{\ell+1} (t,x,v)|}{|v_{\mathbf{B}, 3} ^{\ell+1}  (t,x,v)|}
\frac{|\tB ^{\ell+1}  |^2}{2} \| \nabla_x^2  \phi_{F^\ell}   \|_\infty\Big)   \min \big\{ \tB^{\ell+1} e^{ \frac{|\tB^{\ell+1} |^2}{2} \| \nabla_x^2 \phi_{F^\ell}  \|_\infty },
e^{ (1+ \| \nabla_x ^2 \phi_{F^\ell}  \|_\infty)\tB}
\big\} ,
\end{split}
\Ee
\Be\label{est:vb_x:dyn}
\begin{split}
&	|\p_{x_i} v _{\mathbf{B}, j} ^{\ell+1} (t,x,v)|  \leq 
\delta_{i3}	\frac{|\p_{x_j} \phi_{F^\ell}(
t-\tB^{\ell+1} (t,x,v),
\xB ^{\ell+1}  (t,x,v))|}{|v_{\mathbf{B}, 3}^{\ell+1}   (t,x,v)|} \\
&  \ \ +
\Big(1+  \frac{|\tB^{\ell+1}  | \| \nabla_x   \phi_{F^\ell}\|_\infty }{
|v_{\mathbf{B}, 3}^{\ell+1}   (t,x,v)| 
}\Big) 
|\tB^{\ell+1}   | \| \nabla_x^2 \phi_{F^\ell}  \|_\infty \min \big\{  e^{\frac{|\tB^{\ell+1} |^2}{2} \| \nabla_x^2 \phi_{F^\ell}  \|_\infty },
e^{ (1+ \| \nabla_x ^2 \phi_{F^\ell} \|_\infty)\tB}
\big\} ,
\end{split}
\Ee
\Be\label{est:vb_v:dyn}
\begin{split}
&	|\p_{v_i} v _{\mathbf{B}, j}^{\ell+1}  (t,x,v)|  \leq 
\delta_{i3}	\frac{ |\tB ^{\ell+1}  (t,x,v)| |\p_{x_j} \phi_{F^\ell}  (t-\tB^{\ell+1} (t,x,v), \xB ^{\ell+1}  (t,x,v))| }{|v_{\mathbf{B}, 3} ^{\ell+1}  (t,x,v)|} + \delta_{ij}\\
&    + \Big(1+  \frac{|\tB^{\ell+1}   | \| \nabla_x \phi_{F^\ell} \|_\infty}{ |v_{\mathbf{B}, 3} ^{\ell+1}  (t,x,v)|}\Big) |\tB^{\ell+1}   | \| \nabla_x^2 \phi_{F^\ell}  \|_\infty  \min \big\{ \tB^{\ell+1}  e^{ \frac{|\tB^{\ell+1} |^2}{2} \| \nabla_x^2 \phi_{F^\ell}\|_\infty },
e^{ (1+ \| \nabla_x ^2 \phi_{F^\ell}  \|_\infty)\tB}
\big\} ,
\end{split}
\Ee 
where we have abbreviated $\tB^{\ell+1} =\tB^{\ell+1} (t,x,v)$, $ \| \nabla_x^2\phi_{F^\ell}  \|_\infty =  \sup_{\tau \in [t-\tB^{\ell+1} (t,x,v),t]} \| \nabla_x^2\phi_{F^\ell}(\tau)  \|_{L^{\infty}(\O \times \R^3)}$ and $ \| \nabla_x \phi_{F^\ell}  \|_\infty =  \sup_{\tau \in [t-\tB^{\ell+1} (t,x,v),t]} \| \nabla_x \phi_{F^\ell}(\tau)  \|_{L^{\infty}(\O \times \R^3)}$.

Again we utilize a kinetic distance for the dynamic problem \eqref{VP_F}. We define a dynamic kinetic distance (\cite{Guo94,CK_VM})
\Be\label{alpha_k:dyn}
\alpha^{\ell+1}_{F^\ell} 
(t,x,v) =\sqrt{ |v_3|^2 +    |x_3|^2  + 2\p_{x_3} \phi_{F^\ell}
(t,x_\parallel , 0) x_3 + 2g 
x_3 },
\Ee 
In particular, $\alpha^{\ell+1}_{F^\ell} (t,x,v) =  |v_3|$ when $x \in \p\O$ (i.e. $x_3=0$) due \eqref{Dbc:F}. 

\hide\begin{definition}
\label{def:alpha_k:dyn} For any given $T>0$, assume that \eqref{Bootstrap} holds. Also suppose the Dirichlet boundary condition \eqref{Dbc:F} holds for $\phi_F$. 
\hide	\Be\label{Uest:DPhi}
\| \nabla_x \Phi \|_{L^\infty( \bar \O)} \leq \frac{g}{2}.
\Ee
\unhide
We define a kinetic distance
\Be\label{alpha_k:dyn}
\alpha_F 
(t,x,v) =\sqrt{ |v_3|^2 +    |x_3|^2  + 2\p_{x_3} \phi_F
(t,x_\parallel , 0) x_3 + 2g 
x_3 }.
\Ee
In particular, $\alpha_F (t,x,v) =  |v_3|$ when $x \in \p\O$ (i.e. $x_3=0$) due \eqref{Dbc:F}. 
\end{definition}\unhide

\begin{lemma}\label{VL:dyn}
Assume \eqref{Bootstrap_ell} holds for $\Psi^\ell \in C^1 (\bar \O) \cap C^2 (\O)$. Suppose the Dirichlet boundary condition \eqref{Dbc:F} holds for $\phi_{F^\ell}$. 
\hide	\Be\label{Uest:DPhi}
\| \nabla_x \Phi \|_{L^\infty( \bar \O)} \leq \frac{g}{2}.
\Ee
\unhide
Recall the characteristics $\Zz^{\ell+1}(s;t,x,v) = (\X^{\ell+1}(s;t,x,v), \V^{\ell+1}(s;t,x,v))$ solving \eqref{ODE_F}. For all $(x,v) \in \O \times \R^3$ and $s \in  [  t  - \tB^{\ell+1}  (t,x,v),t] $, 
\Be\label{est:alpha:dyn}
\begin{split}
&	\alpha^{\ell+1}_{F^\ell}  (s,\X^{\ell+1} (s;t,x,v),\V ^{\ell+1}(s;t,x,v))  \\
&\leq  \alpha ^{\ell+1}_{F^\ell}  (t,x,v) e^{  \sup_{\tau \in [s,t]}  \big(1+  \| \p_{x_3}^2   \phi_{F^\ell}(\tau)   \|_{L^\infty( {\O})}  
+ \frac{1}{g} \| \p_t \p_{x_3} \phi_{F^\ell} (\tau) \|_{L^\infty(\p\O)} 
\big)|t-s|  }\\
& \ \  \times e^{ \frac{1}{g  } \sup_{\tau \in [s,t]}  \|  \nabla_{x_\parallel} \p_{x_3} \phi_{F^\ell} (\tau)   \|_{L^\infty(\p\O)} 
\int^{t}_s |\V^{\ell+1} _\parallel (s^\prime;t,x,v)| \dd s^\prime  } ,\\
&	\alpha ^{\ell+1} _{F^\ell} (s, \X^{\ell+1} (s;t,x,v),\V^{\ell+1} (s;t,x,v))\\
&\geq  \alpha ^{\ell+1}_{F^\ell} (t,x,v) e^{  -\sup_{\tau \in [s,t]}  \big(1+  \| \p_{x_3}^2   \phi_{F^\ell}(\tau)   \|_{L^\infty( {\O})}  
+ \frac{1}{g} \| \p_t \p_{x_3} \phi_{F^\ell} (\tau) \|_{L^\infty(\p\O)} 
\big)|t-s|  }\\
& \ \  \times e^{- \frac{1}{g  } \sup_{\tau \in [s,t]}  \|  \nabla_{x_\parallel} \p_{x_3} \phi_{F^\ell} (\tau)   \|_{L^\infty(\p\O)} 
\int^{t}_s |\V^{\ell+1}_\parallel (s^\prime;t,x,v)| \dd s^\prime  }  .
\end{split}\Ee
In particular, the last inequality implies that 
\begin{align}
&|v_{\mathbf{B}, 3} ^{\ell+1} (t,x,v)|  \geq 
\alpha_{F^\ell}  ^{\ell+1} (t,x,v) 
e^{  -\sup_{\tau \in [t-\tB,t]}  \big(1+  \| \p_{x_3}^2   \phi_{F^\ell}(\tau)   \|_{L^\infty( {\O})}  
+ \frac{1}{g} \| \p_t \p_{x_3} \phi_{F^\ell} (\tau) \|_{L^\infty(\p\O)} 
\big)\tB } \notag \\
& \ \   \ \ \ \  \ \ \ \ \ \ \ \ \ \  \   \   \times e^{- \frac{1}{g  } \sup_{\tau \in [t-\tB^{\ell+1} ,t]}  \|  \nabla_{x_\parallel} \p_{x_3} \phi_{F^\ell} (\tau)   \|_{L^\infty(\p\O)} 
\int^{t}_{t-\tB} |\V ^{\ell+1}_\parallel (s^\prime;t,x,v)| \dd s^\prime  } 
\label{est1:alpha:dyn}\\
& \geq\alpha_{F^\ell}  ^{\ell+1}(t,x,v) 
e^{  - \frac{4}{g} |v_{\mathbf{B}, 3}^{\ell+1} (t,x,v)| \sup_{\tau \in [t-\tB^{\ell+1},t]}  \big(1+  \| \p_{x_3}^2   \phi_{F^\ell}(\tau)   \|_{L^\infty( {\O})}  
+ \frac{1}{g} \| \p_t \p_{x_3} \phi_{F^\ell} (\tau) \|_{L^\infty(\p\O)} 
\big)  } \notag \\
& \ \  \times e^{- \frac{4}{g^2  }|\vB^{\ell+1}(t,x,v)|^2 \sup_{\tau \in [t-\tB^{\ell+1},t]}  \|  \nabla_{x_\parallel} \p_{x_3} \phi_{F^\ell} (\tau)   \|_{L^\infty(\p\O)} 
\big(1+ \frac{2}{g} \| \nabla_x \phi_{F^\ell} (\tau) \|_{L^\infty(\O)}  \big) 
}  .\label{est1:xb_x/w:dyn}
\end{align}
Here, we abbreviate $\tB^{\ell+1} = \tB^{\ell+1}(t,x,v)$.

\end{lemma}

\begin{proof}
Note that 
\Be\notag
\begin{split}
&[ \p_t + v\cdot \nabla_x  - \nabla_x (\phi_{F^\ell}(t,x)
+ g 
x_3) \cdot \nabla_v ] \sqrt{ {|v_3|^2} +  {|x_3|^2}  + 2\p_{x_3} \phi_{F^\ell}(t,x_\parallel , 0) x_3 + 2g  x_3 }\\
& = \frac{ 
\p_t \p_{x_3} \phi_{F^\ell} (t, x_\parallel, 0) x_3
-  \big( \p_{x_3} \phi_{F^\ell}(t,x)  -  \p_3  \phi_{F^\ell}(t,x_\parallel, 0)   \big)v_3
+  v_3 x_3 +   v_\parallel \cdot \nabla_{x_\parallel} \p_{3}  \phi_{F^\ell}  (t,x_\parallel, 0) x_3
}{ 
\alpha^{\ell+1}_{F^\ell}  (t,x,v)
}\\ 
& \leq \frac{  \big(1+  \| \p_3\p_3 \phi_{F^\ell}(t)  \|_{L^\infty (\O)} \big)|v_3| |x_3|
+  \big( \|\p_t \p_{x_3} \phi_{F^\ell} (t )\|_{L^\infty (\p\O)} +  |v_\parallel| \| \nabla_\parallel \p_3 \phi_{F^\ell} (t)  \|_{L^\infty (\p\O)}\big)
|x_3|
}{\alpha^{\ell+1} _{F^\ell}  (t,x,v)} ,
\end{split}
\Ee
where we use $- \p_{3} \phi_{F^\ell}  (t,x_\parallel, x_3) + \p_3  \phi_{F^\ell} (t,x_\parallel, 0) =\int^0_{x_3} \p_3 \p_3 \phi_{F^\ell}(t,x_\parallel, y_3) \dd y_3$. Then 
\Be\notag
\begin{split}
&\big|[\p_t + v\cdot \nabla_x  - \nabla_x (\phi_{F^\ell} + g   x_3) \cdot \nabla_v ] \alpha_{F^\ell} ^{\ell+1} (t,x,v)\big|\\
& \leq  
\Big(1+ \| \p_{x_3} \p_{x_3} \phi_{F^\ell} (t)\|_{L^\infty( \O)} + \frac{1}{g} \| \p_t \p_{x_3} \phi_{F^\ell} (t) \|_{L^\infty (\p\O)}	
+ \frac{1}{g  }|v_\parallel| \|  \nabla_\parallel \p_{x_3} \phi_{F^\ell}(t)  \|_{L^\infty(\p\O)} 
\Big) \alpha_{F^\ell} ^{\ell+1}  (t,x,v) 
.
\end{split}\Ee
By the Gronwall's inequality, we conclude both inequalities of \eqref{est:alpha:dyn}. Then \eqref{est1:alpha:dyn} is a direct result of the second estimate in \eqref{est:alpha:dyn}. For \eqref{est1:xb_x/w:dyn}, we use \eqref{est:tB}.
\hide
\Be\notag
\int_{-\tb }^0 |V_\parallel  (s;x,v)| \dd s  \leq 
|\vb  | \tb +  \frac{|\tb |^2}{2} \| \nabla_x \phi_F \|_\infty  
\leq  \frac{4}{g} \Big(1+ \frac{2}{g} \| \nabla_x \phi_F \|_\infty \Big)|\vb  | ^2 ,
\Ee

Then the proof follows the proof of Lemma \ref{VL}.\unhide
\end{proof}

\begin{lemma}\label{RE:dyn}Assume that \eqref{Bootstrap_ell} holds for $\Psi^\ell \in C^1 (\bar \O) \cap C^2 (\O)$, and the compatibility condition \eqref{CC} holds. Let $(h,   \Phi)$ and $(f^{\ell+1},  \Psi^\ell)$ be the solutions constructed in Theorem \ref{theo:CS} and \eqref{Poisson_fell}-\eqref{initial:fell} respectively. Recall that $\phi_{F^\ell} = \Phi + \Psi^\ell$. Suppose that
\Be\label{condition:DDphi:dyn}
\begin{split}
\frac{  \tilde \beta ^{1/2} }{g^{1/2} }  	\| \p_t \phi_{F^\ell} \|_{L^\infty ([0,t] \times \O)}  + \frac{1}{g^2  \tilde{\beta}^{1/2}} 	\| \p_t  \p_{x_3} \phi_{F^\ell} \|_{L^\infty ([0,t] \times \p \O)} 
+  \frac{1}{g   \tilde{\beta}^{1/2}} 	\|  \nabla_x^2 \phi_{F^\ell} \|_{L^\infty ([0,t] \times   \O)} 
\lesssim 1.
\end{split}
\Ee
\hide

\Be
\frac{64 \tilde \beta}{g}  	\| \p_t \phi_{F^\ell} \|_{L^\infty_{t,x}}  ^2  + \frac{32 }{g^2 \tilde \beta}
(1+   \| \nabla_x^2 \phi_{F^\ell}\|_{L^\infty_{t,x}})^2	 \lesssim 1
\Ee
\Be
\frac{1}{g^2 \tilde{\beta}} \| \nabla_x^2 \phi_{F^\ell} \|_{L^\infty_{t,x}} \lesssim 1 
\Ee
\Be
\frac{2^5}{g^2 \tilde \beta }   
\sup_{\tau \in [t-\tB^{\ell+1}(t,x,v),t]}  \big(1+  \| \p_{x_3}^2   \phi_{F^\ell}(\tau)   \|_{L^\infty( {\O})}  
+ \frac{1}{g} \| \p_t \p_{x_3} \phi_{F^\ell} (\tau) \|_{L^\infty(\p\O)} 
\big)^2 \lesssim 1
\Ee

\Be
\sup_{t\in [0,\infty]}  \|  \nabla_{x_\parallel} \p_{x_3} \phi_{F^\ell} (t)   \|_{L^\infty(\p\O)} 
\Big(1+ \frac{2}{g} \| \nabla_x \phi_{F^\ell} (t) \|_{L^\infty(\O)}  \Big)  \leq \frac{\tilde \beta g^2}{32}.
\Ee\unhide 

Then, for $F^{\ell+1} = h + f^{\ell+1}$ and $\phi_{F^\ell} = \Phi + \Psi^\ell$, we have that, for all $s \in [0,t]$,
\hide\Be
\begin{split}\label{est:F_v:dyn}
&\| e^{\frac{\tilde \beta}{4} (|v|^2 + g x_3)} \nabla_v F^{\ell+1}(t)\|_{L^\infty(\O \times \R^3)} \\
& \lesssim e^{  \frac{64 \tilde \beta}{g}  	\| \Delta_0^{-1} (\nabla_x \cdot b^\ell ) \|_{L^\infty_{t,x}}  ^2 }
e^{\frac{32 }{g^2 \tilde \beta}
(1+   \| \nabla_x^2 \phi_{F^\ell}\|_{L^\infty_{t,x}})^2	
}  \| \w_{\tilde \beta, 0}   \nabla_{x,v} F_0  \|_{L^\infty (\O \times \R^3)} \\
&+\Big(1+ \frac{1}{g\tilde \beta ^{1/2}} + \frac{4}{g}  \| \nabla_x \phi_{F^\ell}   \|_{L^\infty_{t,x}}\Big)  \Big(1+ \frac{1}{g^2 \tilde\beta }   \| \nabla_ x^2 \phi_{F^\ell}   \|_{L^\infty_{t,x}}\Big)  \|  e^{\tilde{\beta} | v|^2}   \nabla_{x_\parallel,v} G   \|_{L^\infty (\gamma_-)},
\end{split}\Ee

\unhide {\small\Be
\begin{split}\label{est:F_v:dyn}
e^{\frac{\tilde \beta}{4} (|v|^2 + g x_3)} |\nabla_v F^{\ell+1}(s,x,v) | 
\lesssim  
\| \w_{\tilde \beta, 0}   \nabla_{x,v} F_0  \|_{L^\infty (\O \times \R^3)}  
+\Big(1+ \frac{1}{g\tilde \beta ^{1/2}} 
\Big)  
\|  e^{\tilde{\beta} | v|^2}   \nabla_{x_\parallel,v} G   \|_{L^\infty (\gamma_-)},
\end{split}\Ee  }
\Be
\label{est:F_x:dyn} 
\begin{split}
&  e^{\frac{\tilde \beta}{4} (|v|^2 + g x_3)} |\p_{x_i} F^{\ell+1}(s,x,v)|  
\lesssim 
\| \w_{\tilde \beta, 0}   \nabla_{x,v} F_0  \|_{L^\infty (\O \times \R^3)} \\
& \ \ \  \ \ \   \ \ \   \ \ \   \ \ \  \ \ \  \ \ \  \ \ \ \  
+ \left[ 
\left(1+ \frac{1}{g \tilde{\beta}^{1/2}}\right)
+ \left(1+ \frac{1}{\tilde{\beta}^{1/2}}\right) \frac{\delta_{i3}}{\alpha^{\ell+1 }_{F^\ell} (s,x,v)}
\right]
\|   e^{\tilde \beta |v|^2}  \nabla_{x_\parallel,v} G    \|_{L^\infty(\gamma_-)} .
\end{split}
\Ee 
\hide where we have abbreviated notations $\| \nabla_x \phi_{F^\ell}   \|_{L_{t,x}^\infty} :=\sup_{s \in [0
,t]}\| \nabla_x \phi_{F^\ell} (s) \|_{L^\infty(\O)}$ and $\| \nabla_x^2 \phi_{F^\ell}   \|_{L_{t,x}^\infty} := \sup_{s \in [0
,t]} \| \nabla_x^2 \phi_{F^\ell} (s) \|_{L^\infty(\O)}$.\unhide

Moreover, for all $(s,x) \in  [0,t] \times \bar\O$, 
\Be\label{est:rho_x:dyn}
\begin{split}
&  	e^{\frac{\tilde{\beta}g}{4} x_3 } |\p_{x_i} \varrho^{\ell+1} (s,x)| 
\lesssim 	\frac{1}{\tilde{\beta}^{3/2}}
\left[	 \| \w_{\tilde \beta, 0}   \nabla_{x,v} F_0  \|_{L^\infty (\O \times \R^3)}  +
\Big(1+ \frac{1}{ g \tilde{\beta}^{1/2}}\Big)
\|  e^{\tilde \beta |v|^2}\nabla_{x_\parallel, v} G \|_{L^\infty ({\gamma_-})} 
\right]
\\
& \ \ \ \ \ \ \ \ \  \ \ \ \ \     + 
\frac{\delta_{i3}}{ \tilde \beta }   \left(1+ \frac{1}{\tilde{\beta}^{1/2}}\right) 
\Big(
1+  \frac{1}{\tilde \beta^{1/2}} +   \mathbf{1}_{|x_3| \leq 1}  |\ln    ( |x_3|^2 + g x_3 )|  
\Big)\|  e^{\tilde \beta |v|^2}\nabla_{x_\parallel, v} G \|_{L^\infty ({\gamma_-})}.
\end{split}
\Ee
For $0 < \delta <1$ and for all $s \in [0,t]$
\Be\label{est:rho_holder_dyn}
[\varrho^{\ell+1} (s)]_{C^{0,\delta}}
\lesssim_\delta  \frac{1}{\tilde{\beta}^{3/2}} \| \w_{\tilde \beta, 0}   \nabla_{x,v} F_0  \|_{L^\infty (\O \times \R^3)} + \frac{1}{\tilde{\beta}} \left(
1+ \frac{1}{\tilde{\beta}^{1/2}} + \frac{1}{g \tilde{\beta}}
\right) \| e^{\tilde{\beta} |v|^2} \nabla_{x_\parallel, v} G \|_{L^\infty (\gamma_-)}.
\Ee

\hide	\Be\label{Uest:wf}
\begin{split} 
&e^{  \frac{
	\beta}{2}|v|^2}
e^{  \frac{
	\beta}{2} g x_3}
| f^{\ell+1}   (s,x,v)|
+ \beta^{-3/2}  e^{  \frac{
	\beta}{2} g x_3}
| \varrho^{\ell+1}   (s,x)|\\
&\lesssim  
e^{\frac{64
	\beta}{g}
\sup_{\tau \in [0,s]}	\| \p_t \Psi^\ell 
(\tau) \|_{L^\infty (\O)}^2
}  \big\{   \| \w_{\beta, 0 } F_{0} \|_{L^\infty (\O \times \R^3)}
+  \| 
e^{\beta|v|^2}
G  \|_{L^\infty (\gamma_-)}\big\}  
.
\end{split}	\Ee\unhide

Furthermore, for all $s \in [0,t]$, we have that $\phi_{F^{\ell+1}}(s) \in C^1 (\bar{\O})\cap C^2 (\O)$ and
\Be\label{est:phi_C1_dyn}
\| \nabla_x \phi_{F^{\ell+1}}(s) \|_{L^\infty (\O)}  \leq    \frac{1}{\beta^{3/2}} \Big(1 + \frac{1}{\beta g}\Big)
\big\{   \| \w_{\beta, 0 } F_{0} \|_{L^\infty (\O \times \R^3)}
+  \| 
e^{\beta|v|^2}
G  \|_{L^\infty (\gamma_-)}\big\}  
,  
\Ee
{\small\Be
\begin{split}
\label{est:phi_C2_dyn}
& \| \nabla_x^2  \phi_{F^{\ell+1}}(s) \|_{L^\infty (\O)} 
\leq 
\frac{\mathfrak{C}_1}{\beta^{3/2}} 
\big\{   \| \w_{\beta, 0 } F_{0} \|_{L^\infty (\O \times \R^3)}
+  \| 
e^{\beta|v|^2}
G  \|_{L^\infty (\gamma_-)}\big\}  \\
& \times \bigg\{
\frac{1}{g \beta }  +
\log \bigg(
e+ 
\frac{1}{\tilde{\beta}^{3/2}} \| \w_{\tilde \beta, 0}   \nabla_{x,v} F_0  \|_{L^\infty (\O \times \R^3)}
+	\frac{1}{\tilde{\beta}}\Big(1+ \frac{1}{\tilde{\beta}^{1/2}}+ \frac{1}{ g \tilde{\beta}}\Big) 
\| e^{\tilde \beta |v|^2} \nabla_{x_\parallel, v}  G \|_{L^\infty (\gamma_-)}
\bigg)\bigg\}.
\end{split}	\Ee}

\end{lemma}

\begin{proof}
We bound \eqref{DF_x} and \eqref{DF_v} following the argument of the proof of Lemma \ref{lem:regularity}. Recall $\w^{\ell+1}_{\tilde{\beta}}$ defined in \eqref{w^ell}. From \eqref{DF_v}, 
\begin{align}
|\nabla_v F^{\ell+1}  (t,z)|  
&\leq 
\mathbf{1}_{t \leq \tB^{\ell+1}  (t,z)}
\frac{| \nabla_v\X^{\ell+1}  (0;t,z)| +| \nabla_v \V^{\ell+1}  (0;t,z)|  }{\w^{\ell+1}_{\tilde \beta} (0, \Zz^{\ell+1}  (0;t,z) )} \| \w^{\ell+1}_{\tilde \beta, 0}   \nabla_{x,v} F_0  \|_{L^\infty (\O \times \R^3)} \label{est1:F_v1}
\\
&+	\mathbf{1}_{t >\tB^{\ell+1} (t,z)}	\frac{ | \nabla_v \xB ^{\ell+1} (t,z) | + |\nabla_v \vB^{\ell+1}  (t,z)|}{
e^{\tilde{\beta} |\vB^{\ell+1}  (t,z)|^2}	
}  \|  e^{\tilde{\beta} | v|^2}   \nabla_{x_\parallel,v} G   \|_{L^\infty (\gamma_-)}.
\label{est1:F_v2}
\end{align} 
Using \eqref{est:X_v:dyn}-\eqref{est:V_v:dyn} and \eqref{est:1/w_h_ell}, we bound 
\Be\label{est2:F_v1}
\begin{split}
\eqref{est1:F_v1} 
& \lesssim	e^{  \frac{64 \tilde \beta}{g} 	\|
\p_t \Psi^\ell 
\|_{L^\infty_{t,x}}  ^2 }
e^{-\frac{\tilde\beta}{2}|v|^2}  
e^{-\frac{\tilde\beta}{2} g x_3}  \mathbf{1}_{t \leq \tB^{\ell+1}(t,x,v)} 
e^{ (1+ \| \nabla_x ^2 \phi_{F^\ell}  \|_{L^\infty_{t,x}})t}  \| \w_{\tilde \beta, 0}   \nabla_{x,v} F_0  \|_{L^\infty (\O \times \R^3)} \\
& \lesssim e^{  \frac{64 \tilde \beta}{g} 	\| 
\p_t \Psi^\ell 
\|_{L^\infty_{t,x}}  ^2 }
e^{\frac{32 }{g^2 \tilde \beta}
(1+ \| \nabla_x^2 \phi_{F^\ell} \|_{L^\infty_{t,x}})^2	
}
e^{-\frac{\tilde \beta}{4}|v|^2}  
e^{-\frac{\tilde \beta}{4} g x_3}  \| \w_{\tilde \beta, 0}   \nabla_{x,v} F_0  \|_{L^\infty (\O \times \R^3)} ,
\end{split}
\Ee
where we use the abbreviation $L^\infty_{t,x}$ of \eqref{notation}. 
In \eqref{est2:F_v1}, we have used \eqref{est:tB_ell} at the second line: $ 
e^{ (1+ \| \nabla_x ^2 \phi_{F^\ell}  \|_ {L^\infty_{t,x}})\tB^{\ell+1}(t,x,v)} \leq e^{  \frac{4}{g}(1+ \| \nabla_x ^2 \phi_{F^\ell}  \|_{L^\infty_{t,x}}) (|v_3| + \sqrt{g x_3})}$; and used the completing-square trick to get the last line of \eqref{est2:F_v1}. 

Now we consider \eqref{est1:F_v2}. We follow the same argument in \eqref{est1:xb_v} and \eqref{est1:vb_v}: Using \eqref{est:xb_v:dyn}, \eqref{est:vb_v:dyn}, and \eqref{est:tB_ell}, we derive that 
\begin{align}
\frac{|\p_{v_i} \xB ^{\ell+1} (t, x,v)|}{ e^{\tilde \beta |\vB^{\ell+1} (t,x,v)|^2}} 
& \leq \frac{16}{g {\tilde\beta}^{1/2}}
\Big(1 + \frac{8}{g^2 \tilde\beta } \| \nabla_x^2 \phi_{F^\ell}  \|_{L^\infty_{t,x}}  \Big) 
e^{- \frac{\tilde\beta}{2}|v|^2} e^{- \frac{ \tilde\beta g}{2} x_3},\label{est1:xb_v:dyn}\\
\frac{	|\p_{v_i} \vB^{\ell+1}  (t,x,v)| }{
e^{\tilde \beta  |\vB^{\ell+1} (t,x,v)|^2}	
}
& \leq \Big(1+ \frac{4}{g} \| \nabla_x \phi_{F^\ell} \|_{L^\infty_{t,x}}\Big)\Big(1+ \frac{32}{g^2 \tilde\beta }  \| \nabla_ x^2 \phi_{F^\ell}  \|_ {L^\infty_{t,x}}\Big) e^{- \frac{\tilde\beta}{2}|v|^2} e^{- \frac{\tilde\beta g}{2} x_3}.\label{est1:vb_v:dyn}
\end{align}

Finally applying \eqref{est1:xb_v:dyn}, \eqref{est1:vb_v:dyn}, \eqref{est2:F_v1} to \eqref{est1:F_v1}-\eqref{est1:F_v2} we conclude \eqref{est:F_v:dyn} under the condition of \eqref{condition:DDphi:dyn}.

Next, to get \eqref{est:F_x:dyn}, we bound \eqref{DF_x}. We can bound the first line of \eqref{DF_x}, following the argument of \eqref{est2:F_v1} and using \eqref{est:X_x:dyn}-\eqref{est:V_x:dyn}:
\Be\begin{split}	 \label{est1:F_x1}
&	\mathbf{1}_{t \leq \tB^{\ell+1}(t,x,v)}
\frac{| \nabla_x\X^{\ell+1} (0;t,x,v)| +| \nabla_x \V ^{\ell+1} (0;t,x,v)|  }{\w_{\tilde \beta} (0, \Zz ^{\ell+1}  (0;t,x,v) )} \| \w_{\tilde \beta, 0}   \nabla_{x,v} F_0  \|_{L^\infty (\O \times \R^3)}\\
& \lesssim e^{  \frac{64 \tilde \beta}{g} \|
\p_t \Psi^\ell
\|_{L^\infty_{t,x}}  ^2 }
e^{\frac{32 }{g^2 \tilde \beta}
(1+ \| \nabla_x^2 \phi_{F^\ell}  \|_{L^\infty_{t,x}})^2	
} 
e^{-\frac{\tilde \beta}{4}|v|^2}  
e^{-\frac{\tilde \beta}{4} g x_3}  \| \w_{\tilde \beta, 0}   \nabla_{x,v} F_0  \|_{L^\infty (\O \times \R^3)} .
\end{split}\Ee

Now we consider the second line of \eqref{DF_x}. We follow the same argument of \eqref{est1:xb_x} and \eqref{est1:vb_x}. Using \eqref{condition:DDphi:dyn}, \eqref{est1:xb_x/w:dyn}, and the completing-square trick, we have that 
\Be\label{est1:xb_x:dyn}
\begin{split}
&	\frac{|\p_{x_i} \xB^{\ell+1}  (t,x,v)|}{e^{\tilde \beta |\vB^{\ell+1} (t,x,v)|^2}  }    \leq  \frac{|\vB^{\ell+1} |}{|v_{\mathbf{B},3} ^{\ell+1} |}e^{-\tilde \beta |\vB^{\ell+1} |^2} \delta_{i3}   
+ \Big(	1+ \frac{8\| \nabla_x^2 \phi_{F^\ell} \|_{L^\infty_{t,x}}  }{g^2} |\vB ^{\ell+1}| |v_{\mathbf{B},3}^{\ell+1} | 	\Big)e^{- \tilde\beta |\vB^{\ell+1} |^2}
\\&  \ \ \ \   \ \ \ \  \ \ \ \    \ \ \ \ \    \ \ \ \  \ \ \ \ \ \ \ \ \  \ \ \ \  \ \ \    \ \ \ \  \ \ \ \ \  \ \ \ \ \times 
\min \Big \{
e^{ \frac{8}{g^2} \| \nabla_x^2 \phi_{F^\ell} \|_{L^\infty_{t,x}}  |v_{\mathbf{B}, 3}^{\ell+1} |^2 }  , 
e^{\frac{4}{g}(1+  \| \nabla_x^2 \phi_{F^\ell} \|_{L^\infty_{t,x}}  ) |v_{\mathbf{B}, 3}^{\ell+1} | }
\Big \}  
\\
\hide	& \leq 
\frac{|\vb(x,v)|}{\alpha  (x,v) }
e^{  \frac{4}{g} (1+ \| \nabla_x ^2 \Phi  \|_\infty ) |v_{\b, 3} |}
e^{\big(  - \tilde \beta + \frac{4}{g^2} \| \nabla_x^2 \Phi  \|_\infty
(1+ \frac{2}{g} \| \nabla_x \Phi  \|_\infty)
\big)	|\vb |^2	
}
\\
& \ \ +\frac{16}{g {\tilde\beta}^{1/2}}
\Big(1 + \frac{8}{g^2\tilde \beta } \| \nabla_x^2 \Phi  \|_\infty  \Big) e^{- \frac{\tilde\beta}{2 }|\vb |^2 }\\
& \leq 
\left(
\frac{\delta_{i3}}{\tilde{\beta}^{1/2}} \frac{1}{\alpha(x,v)}   +
\frac{16}{g {\tilde\beta}^{1/2}}
\Big(1 + \frac{8}{g^2 \tilde\beta } \| \nabla_x^2 \Phi  \|_\infty  \Big) \right)
e^{- \frac{\tilde\beta}{2}|v|^2} e^{- \frac{ \tilde\beta g}{2} x_3}
\\
\unhide
& \leq 
\left(
e^{\frac{2^5}{g^2 \tilde \beta }   
\sup_{\tau \in [t-\tB^{\ell+1},t]}  \big(1+  \| \p_{x_3}^2   \phi_{F^\ell}(\tau)   \|_{L^\infty( {\O})}  
+ \frac{1}{g} \| \p_t \p_{x_3} \phi_{F^\ell}(\tau) \|_{L^\infty(\p\O)} 
\big)^2 
}
\frac{\delta_{i3}}{\tilde \beta^{1/2}}\frac{1}{\alpha^{\ell+1} _{F^\ell}(t,x,v)}   +
\frac{32}{g {\tilde\beta}^{1/2}}
\right)\\
&  \ \ \times 
e^{- \frac{\tilde\beta}{2}|v|^2} e^{- \frac{ \tilde\beta g}{2} x_3}
,
\end{split}
\Ee
\hide

\begin{align}
|v_{\mathbf{B}, 3} ^{\ell+1} (t,x,v)| &\geq 
\alpha ^{\ell+1} (t,x,v) 
e^{  -\sup_{\tau \in [t-\tB,t]}  \big(1+  \| \p_{x_3}^2   \phi_{F^\ell}(\tau)   \|_{L^\infty( {\O})}  
+ \frac{1}{g} \| \p_t \p_{x_3} \phi_{F^\ell} (\tau) \|_{L^\infty(\p\O)} 
\big)\tB } \notag \\
& \ \  \times e^{- \frac{1}{g  } \sup_{\tau \in [t-\tB^{\ell+1} ,t]}  \|  \nabla_{x_\parallel} \p_{x_3} \phi_{F^\ell} (\tau)   \|_{L^\infty(\p\O)} 
\int^{t}_{t-\tB} |\V ^{\ell+1}_\parallel (s^\prime;t,x,v)| \dd s^\prime  } 
\label{est1:alpha:dyn}\\
& \geq\alpha ^{\ell+1}(t,x,v) 
e^{  - \frac{4}{g} |v_{\mathbf{B}, 3}^{\ell+1} (t,x,v)| \sup_{\tau \in [t-\tB^{\ell+1},t]}  \big(1+  \| \p_{x_3}^2   \phi_{F^\ell}(\tau)   \|_{L^\infty( {\O})}  
+ \frac{1}{g} \| \p_t \p_{x_3} \phi_{F^\ell} (\tau) \|_{L^\infty(\p\O)} 
\big)  } \notag \\
& \ \  \times e^{- \frac{4}{g^2  }|\vB^{\ell+1}(t,x,v)|^2 \sup_{\tau \in [t-\tB^{\ell+1},t]}  \|  \nabla_{x_\parallel} \p_{x_3} \phi_{F^\ell} (\tau)   \|_{L^\infty(\p\O)} 
\big(1+ \frac{2}{g} \| \nabla_x \phi_{F^\ell} (\tau) \|_{L^\infty(\O)}  \big) 
}  .\label{est1:xb_x/w:dyn}
\end{align}

\unhide where we have abbreviated 
$\vB^{\ell+1} = \vB^{\ell+1}(t,x,v)$ and used \eqref{notation}. 
\hide:
\Be\begin{split}
\frac{1}{|v_{\mathbf{B}, 3} (t,x,v)| } \leq &\frac{1}{\alpha_F (t,x,v)} 
e^{   \frac{4}{g} |v_{\mathbf{B}, 3} (t,x,v)| \sup_{\tau \in [t-\tB,t]}  \big(1+  \| \p_{x_3}^2   \phi_F(\tau)   \|_{L^\infty( {\O})}  
+ \frac{1}{g} \| \p_t \p_{x_3} \phi_F (\tau) \|_{L^\infty(\p\O)} 
\big)  } \notag \\
& \ \  \times e^{  \frac{4}{g^2  }|\vB(t,x,v)|^2 \sup_{\tau \in [t-\tB,t]}  \|  \nabla_{x_\parallel} \p_{x_3} \phi_F (\tau)   \|_{L^\infty(\p\O)} 
\big(1+ \frac{2}{g} \| \nabla_x \phi_F (\tau) \|_{L^\infty(\O)}  \big) 
} .
\end{split}\Ee
\unhide
\hide

\Be
e^{\frac{2^5}{g^2 \tilde \beta }  \sup_{\tau \in [t-\tB,t]}  \|  \nabla_{x_\parallel} \p_{x_3} \phi_F (\tau)   \|_{L^\infty(\p\O)} ^2
\big(1+ \frac{2}{g} \| \nabla_x \phi_F (\tau) \|_{L^\infty(\O)}  \big) ^2}
\frac{\delta_{i3}}{\tilde \beta^{1/2}}\frac{1}{\alpha_F (t,x,v)}
\Ee

where we have used $ \frac{4}{g} (1+ \| \nabla_x ^2 \Phi  \|_\infty ) |v_{\b, 3} | + 
\big(  - \tilde \beta + \frac{4}{g^2} \| \nabla_x^2 \Phi  \|_\infty
(1+ \frac{2}{g} \| \nabla_x \Phi  \|_\infty)
\big)	|\vb |^2 \leq \frac{\tilde{\beta}}{2} |\vb|^2$.

\unhide

Similarly, we derive that 
\Be\label{est1:vb_x:dyn}
\begin{split}
&  e^{  \frac{\tilde\beta}{2}|v|^2} e^{  \frac{\tilde\beta g}{2} x_3} 	\frac{	|\p_{x_i} \vB^{\ell+1} (t,x,v)| }{
e^{\tilde \beta |\vB^{\ell+1} (t,x,v)|^2}	
}
\\
\hide
& \leq  \frac{\|  \nabla_x \Phi  \|_\infty}{|v_{\mathbf{b}, 3} (x,v)|}   e^{- \tilde\beta |\vb |^2} \delta_{i3}
\\
& \ + \Big(1+ \frac{4}{g} \| \nabla_x \Phi \|_\infty\Big) \frac{4}{g}  \| \nabla_ x^2 \Phi  \|_\infty  | v_{\b,3}  (x,v)|
e^{- \tilde\beta |\vb |^2}
\\
& \ \ \ \ \times 
\min \Big\{
e^{ \frac{8}{g^2}  \| \nabla_x^2 \Phi \|_\infty |v_{\b, 3}  | ^2}, e^{\frac{4}{g}  (1+ \| \nabla_x^2 \Phi \|_\infty)  |v_{\b, 3}  |}
\Big\}\\
& \leq \left( 
\frac{ \| \nabla_x \Phi \|_\infty }{\alpha_F(t,x,v)} \delta_{i3} +  \Big(1+ \frac{4}{g} \| \nabla_x \Phi \|_\infty\Big)
\right) e^{- \frac{\tilde\beta}{2}|v|^2} e^{- \frac{\tilde\beta g}{2} x_3}\\
\unhide
& \leq \left( 
e^{\frac{2^5}{g^2 \tilde \beta }   
\sup_{\tau \in [t-\tB^{\ell+1} ,t]}  \big(1+  \| \p_{x_3}^2   \phi_{F^\ell}(\tau)   \|_{L^\infty( {\O})}  
+ \frac{1}{g} \| \p_t \p_{x_3} \phi_{F^\ell} (\tau) \|_{L^\infty(\p\O)} 
\big)^2 
}
\frac{ \| \nabla_x \phi^{\ell+1} \|_{L^\infty_{t,x}} }{\alpha^{\ell+1}_{F^\ell}(t,x,v)} \delta_{i3} + 3
\right)  .
\end{split}
\Ee
Finally we conclude \eqref{est:F_x:dyn} using \eqref{est1:F_x1}, \eqref{est1:xb_x:dyn}, and \eqref{est1:vb_x:dyn}, under the condition of \eqref{condition:DDphi:dyn}.

We prove \eqref{est:rho_x:dyn}-\eqref{est:phi_C2_dyn} following the proof of \eqref{est:rho_x}-\eqref{est:phi_C2}.\end{proof}

\begin{lemma}\label{lem:D3tphi_F}
Suppose $g \frac{\tilde{\beta}}{4} \geq 1$. Suppose \eqref{theo:hk_x} and \eqref{est:F_x:dyn} hold. Then \eqref{identity:Psi_t_ell} holds and $\nabla_x \cdot b^{\ell+1}(t, \cdot ) \in L^\infty (\O)$ and $\p_{x_3} \p_t \phi_{F^{\ell+1}}(t,\cdot ) \in C^{0, \delta}(\O)$ for some $\delta>0$. 

Moreover, for all $(s,x) \in [0,t] \times \bar \O$,
\begin{align}
&
e^{ \frac{\tilde \beta}{4}  g x_3 }	|\nabla_x \cdot b^{\ell+1}(s,x)| + \tilde{\beta}^{3/2} e^{x_3}| \p_{x_3} \p_t \phi_{F^{\ell+1}}(s,x)| 
+ \tilde{\beta}^{3/2} |\p_t \phi_{F^{\ell+1}}(s,x)|
\notag\\
&\lesssim
\bigg(
\frac{1}{\tilde{\beta}^2}	\| \w_{\tilde \beta, 0}   \nabla_{x,v} F_0  \|_{L^\infty (\O \times \R^3)} 
+
\frac{1}{\tilde{\beta}^{3/2}} \Big( 1+ 	  \frac{1}{\tilde{\beta}^{1/2}} + 	\frac{1}{g {\tilde\beta} } \Big)
\|  e^{\tilde{\beta} | v|^2}   \nabla_{x_\parallel,v} G   \|_{L^\infty (\gamma_-)} \bigg)	
.\label{est:D.b} 
\end{align}

\hide
where
\Be\label{express:C}
\begin{split}
C&= 	\left(1+  \sup_s  \| \nabla_x \phi_{F^{\ell-1}}  (s)\|_{L^\infty_x}	e^{\frac{2^5}{g^2 \tilde \beta }   
\big(1+ \sup_s \| \p_{x_3}^2   \phi_{F^{\ell-1}}  (s)  \|_{L^\infty _{x}}  
+ \frac{1}{g} \sup_s \| \p_t \p_{x_3} \phi_{F^{\ell-1}} (s)  \|_{L^\infty  (\p\O)} 
\big)^2 
}
\right)
\\&\  \ \  \times 
\|  e^{\tilde{\beta} | v|^2}   \nabla_{x_\parallel,v} G   \|_{L^\infty (\gamma_-)} 
\frac{1 }{\tilde \beta^{3/2}},
\end{split}\Ee\Be\begin{split}\label{express:D}
D&= 	\left(1+ \sup_s  \| \nabla_x \phi_{F^{\ell-1}}(s) \|_{L^\infty_{x}}		e^{\frac{2^5}{g^2 \tilde \beta }   
\big(1+  \| \p_{x_3}^2   \phi_{F^{\ell-1}}   \|_{
	L^\infty_{t,x}
}  
+ \frac{1}{g} \sup_s \| \p_t \p_{x_3} \phi_{F^{\ell-1}} (s) \|_{L^\infty_{x}} 
\big)^2 
}
\right)
\\& \ \    \times 
\|  e^{\tilde{\beta} | v|^2}   \nabla_{x_\parallel,v} G   \|_{L^\infty (\gamma_-)} 
\frac{1 }{\tilde \beta^{3/2}}
\Big(1 + \tilde{\beta}^{-1/2}\Big)\\
&	+ \bigg(
e^{  \frac{64 \tilde \beta}{g} \sup_s \| \Delta_0^{-1} (\nabla_x \cdot b^{\ell-1}  ) (s) \|_{L^\infty_{x}}  ^2 }
e^{\frac{32 }{g^2 \tilde \beta}
(1+ \sup_s \| \nabla_x^2 \phi_{F^{\ell-1}}(s) \|_{L^\infty_{x}})^2	
} 
\| \w_{\tilde \beta, 0}   \nabla_{x,v} F_0  \|_{L^\infty (\O \times \R^3)} 
\\	&   \ \ \ \ \ \   +
\frac{1}{g {\tilde\beta}^{1/2}} 	 \|  e^{\tilde{\beta} | v|^2}   \nabla_{x_\parallel,v} G   \|_{L^\infty (\gamma_-)} \bigg)	
\frac{1}{\tilde{\beta}^2}.
\end{split}
\Ee
Here, $\sup_s\| z(s) \|_{L^\infty } = \sup_{s \in [\max\{0, t-\tB^{\ell+1} (t,x,v)\},t]} \|  z(s, \cdot ) \|_{L^\infty  }$.\unhide
\hide
\Be\label{est:D.b}
\begin{split}
&e^{ \frac{\tilde \beta g  }{4}  x_3 } 	|\nabla_x \cdot b(t,x)| \\
\hide
&\lesssim  \| e^{\tilde \beta |v|^2} \nabla_{x_\parallel,v}  G \|_{L^\infty(\gamma_-)}  \frac{e^{ - \frac{\tilde \beta g}{2} x_3} }{\tilde \beta^{3/2}}
\Big(1+ \mathbf{1}_{|x_3| \leq 1} |\ln (|x_3|^2 + g x_3)| + \tilde{\beta}^{-1/2}\Big)\\
&	+ \bigg(
e^{  \frac{64 \tilde \beta}{g} \| \Delta_0^{-1} (\nabla_x \cdot b(t) ) \|_{L^\infty(\O)}  ^2 }
e^{\frac{32 }{g^2 \tilde \beta}
(1+ \| \nabla_x^2 \phi_F (t)\|_{L^\infty(\O)})^2	
} 
\| \w_{\tilde \beta, 0}   \nabla_{x,v} F_0  \|_{L^\infty (\O \times \R^3)} 
\\
&   \ \ \ \ \ \ \   +
\frac{1}{g {\tilde\beta}^{1/2}} 	 \|  e^{\tilde{\beta} | v|^2}   \nabla_{x_\parallel,v} G   \|_{L^\infty (\gamma_-)} \bigg)	
\frac{1}{\tilde{\beta}^2}	e^{-\frac{\tilde \beta}{4}  g x_3 } 
\\
&   + 
e^{\frac{2^5}{g^2 \tilde \beta }   
\sup_{\tau \in [t-\tB,t]}  \big(1+  \| \p_{x_3}^2   \phi_F(\tau)   \|_{L^\infty( {\O})}  
+ \frac{1}{g} \| \p_t \p_{x_3} \phi_F (\tau) \|_{L^\infty(\p\O)} 
\big)^2 
} \\
& \ \ \times   \| \nabla_x \phi_F \|_\infty  \|  e^{\tilde{\beta} | v|^2}   \nabla_{x_\parallel,v} G   \|_{L^\infty (\gamma_-)} 
\frac{e^{ - \frac{\tilde \beta g}{4} x_3} }{\tilde \beta^{3/2}}
\Big(1+ \mathbf{1}_{|x_3| \leq 1} |\ln (|x_3|^2 + g x_3)| + \tilde{\beta}^{-1/2}\Big)\\
\unhide
& \lesssim  
\left(1+   \| \nabla_x \phi_F \|_\infty 	e^{\frac{2^5}{g^2 \tilde \beta }   
\sup_{\tau \in [t-\tB,t]}  \big(1+  \| \p_{x_3}^2   \phi_F(\tau)   \|_{L^\infty( {\O})}  
+ \frac{1}{g} \| \p_t \p_{x_3} \phi_F (\tau) \|_{L^\infty(\p\O)} 
\big)^2 
}
\right)
\\& \ \    \times 
\|  e^{\tilde{\beta} | v|^2}   \nabla_{x_\parallel,v} G   \|_{L^\infty (\gamma_-)} 
\frac{1 }{\tilde \beta^{3/2}}
\Big(1+ \mathbf{1}_{|x_3| \leq 1} |\ln (|x_3|^2 + g x_3)| + \tilde{\beta}^{-1/2}\Big)\\
&	+ \bigg(
e^{  \frac{64 \tilde \beta}{g} \| \Delta_0^{-1} (\nabla_x \cdot b(t) ) \|_{L^\infty(\O)}  ^2 }
e^{\frac{32 }{g^2 \tilde \beta}
(1+ \| \nabla_x^2 \phi_F (t)\|_{L^\infty(\O)})^2	
} 
\| \w_{\tilde \beta, 0}   \nabla_{x,v} F_0  \|_{L^\infty (\O \times \R^3)} 
\\
&   \ \ \ \ \ \   +
\frac{1}{g {\tilde\beta}^{1/2}} 	 \|  e^{\tilde{\beta} | v|^2}   \nabla_{x_\parallel,v} G   \|_{L^\infty (\gamma_-)} \bigg)	
\frac{1}{\tilde{\beta}^2}	,
\end{split}\Ee\unhide

\end{lemma}
\hide\begin{remark}
\begin{align}
\p_t \Psi^{\ell+1} (t,x)  = \eta \Delta_0^{-1} \p_t \varrho^\ell(t,x)  = -  \eta \Delta_0^{-1}  (\nabla_x \cdot b^\ell) (t,x)\ \ \text{in} \ \R_+ \times \O.\label{identity:Psi_t_ell}
\end{align}
\end{remark}\unhide
\begin{remark}\label{remark:CE}
Since $\nabla_x \cdot b^{\ell+1}$ is bounded, a weak solution $\varrho^{\ell+1}$ to \eqref{cont_eqtn_ell} should satisfy $\varrho^{\ell+1} (t,x) = \eta \int_{\R^3} f_0 (x,v) \dd v- \int^t_0 \nabla_x \cdot b^{\ell+1} (s,x ) \dd s$. Therefore $\varrho^{\ell+1} (\cdot, x)$ is absolutely continuous in time. From \eqref{Poisson_fell}, $\Delta \Psi^{\ell+1}(\cdot, x)$ is also absolutely continuous in time and hence \eqref{identity:Psi_t_ell} holds almost everywhere. 

\end{remark}

\begin{proof}[Proof of Lemma \ref{lem:D3tphi_F}]Note that 
\Be\begin{split}\label{est:D.b1}
|\nabla_x \cdot b^{\ell+1} (t,x)|  
&  \leq \int_{\R^3} |v \cdot \nabla_x F^{\ell+1}(t,x,v)| \dd v + \int_{\R^3} |v \cdot \nabla_x h(x,v)| \dd v.\end{split}\Ee 
Using \eqref{theo:hk_x} and \eqref{est3:xb_x/w}, we bound 
\Be\begin{split}\label{est:D.b2}
\int_{\R^3} |v \cdot \nabla_x h(x,v)| \dd v 
& \leq 	 \| e^{\tilde \beta |v|^2} \nabla_{x_\parallel,v}  G \|_{L^\infty(\gamma_-)}   \int_{\R^3} 
e^{ -\frac{\tilde \beta}{2}|v|^2} e^{ - \frac{\tilde \beta g}{2} x_3} \Big(|v |   + \frac{|v_3|}{\alpha(x,v)}\Big) 
\dd v\\
& \lesssim  \| e^{\tilde \beta |v|^2} \nabla_{x_\parallel,v}  G \|_{L^\infty(\gamma_-)}  
e^{ - \frac{\tilde \beta g}{2} x_3} 
\int_\R (|v|+1 ) e^{ -\frac{\tilde \beta}{2}|v|^2} \dd v_3 \\
& \lesssim  \| e^{\tilde \beta |v|^2} \nabla_{x_\parallel,v}  G \|_{L^\infty(\gamma_-)}   \frac{1}{\tilde \beta^{3/2}} 
\Big(1 + \frac{1}{\tilde{\beta}^{1/2}}\Big)
e^{ - \frac{\tilde \beta g}{2} x_3},
\end{split}\Ee
where we have used that $|v_3| \leq \alpha (x,v)$.

Similarly, using \eqref{est3:xb_x/w} and \eqref{est:F_x:dyn}, we bound 
\hide	 \Be\notag
\begin{split}
&
\int_{\R^3} |v \cdot \nabla_x F^{\ell+1} (t,x,v)| \dd v  
\\
&\lesssim \bigg(
e^{  \frac{64 \tilde \beta}{g} \| \Delta_0^{-1} (\nabla_x \cdot b(t) ) \|_{L^\infty(\O)}  ^2 }
e^{\frac{32 }{g^2 \tilde \beta}
(1+ \| \nabla_x^2 \phi_F (t)\|_{L^\infty(\O)})^2	
} 
\| \w_{\tilde \beta, 0}   \nabla_{x,v} F_0  \|_{L^\infty (\O \times \R^3)} 
\\
&   \ \ \ \ \ \ \   +
\frac{1}{g {\tilde\beta}^{1/2}} 	 \|  e^{\tilde{\beta} | v|^2}   \nabla_{x_\parallel,v} G   \|_{L^\infty (\gamma_-)} \bigg)	 \int_{\R^3}	|v|e^{-\frac{\tilde \beta}{4} (|v|^2 + g x_3)} \dd v
\\
&   + 
e^{\frac{2^5}{g^2 \tilde \beta }   
\sup_{\tau \in [t-\tB,t]}  \big(1+  \| \p_{x_3}^2   \phi_F(\tau)   \|_{L^\infty( {\O})}  
+ \frac{1}{g} \| \p_t \p_{x_3} \phi_F (\tau) \|_{L^\infty(\p\O)} 
\big)^2 
} \\
& \ \  \ \times   \| \nabla_x \phi_F \|_\infty  \|  e^{\tilde{\beta} | v|^2}   \nabla_{x_\parallel,v} G   \|_{L^\infty (\gamma_-)} 
\int_{\R^3} \frac{ |v| e^{-\frac{\tilde \beta}{4} (|v|^2 + g x_3)}  }{\alpha_F(t,x,v)} \dd v \\ 
\end{split}
\Ee\unhide \Be\begin{split}\label{est:D.b3}
&
\int_{\R^3} |v\cdot \nabla_x F^{\ell+1} (t,x,v)| \dd v  
\\
&\lesssim \bigg(
\frac{1}{\tilde{\beta}^2}	\| \w_{\tilde \beta, 0}   \nabla_{x,v} F_0  \|_{L^\infty (\O \times \R^3)} 
+
\frac{1}{\tilde{\beta}^{3/2}} \Big( 1+ 	  \frac{1}{\tilde{\beta}^{1/2}} + 	\frac{1}{g {\tilde\beta} } \Big)
\|  e^{\tilde{\beta} | v|^2}   \nabla_{x_\parallel,v} G   \|_{L^\infty (\gamma_-)} \bigg)	
e^{-\frac{\tilde \beta}{4}  g x_3 } .
\end{split}
\Ee 
In summary, we can conclude a $\nabla_x \cdot b^{\ell+1}$-bound of \eqref{est:D.b} from \eqref{est:D.b1}-\eqref{est:D.b3}.

From \eqref{identity:Psi_t_ell} and Lemma \ref{lemma:G}, we have 
\Be \label{D3tphi}
\begin{split}
\p_{x_3} \p_t \phi_{F^{\ell+1}} (t,x)&=   \p_{x_3} \p_t \Psi^{\ell+1} (t,x)
=
\eta \int_{\O} \nabla\cdot b^{\ell+1}(y)  \p_{x_3} G(x,y) \dd y .
\end{split} \Ee
Now applying Lemma \ref{lem:rho_to_phi} and using \eqref{est:nabla_phi} and the $\nabla_x \cdot b^{\ell+1}$-bound of \eqref{est:D.b}, we prove the bound of $|\p_{x_3} \p_t \phi_{F^{\ell+1}}|$ in \eqref{est:D.b}. Next, using the Dirichlet boundary condition $\p_t \phi_{F^{\ell+1}}|_{\p\O}=0$ and the bound of $|\p_{x_3} \p_t \phi_{F^{\ell+1}}|$ in \eqref{est:D.b}, we conclude the bound of $| \p_t \phi_{F^{\ell+1}}|$. 
\hide

Following the decomposition \eqref{est:p3phi}, we bound \eqref{D3tphi} as 
\Be\label{est:D.b4}
|	\p_{x_3} \p_t \phi_{F^\ell} (t,x)| \leq 
\Big|\int_{\O}  \nabla\cdot b^\ell(y)  \p_{x_3} G(x,y)  \dd y\Big| \leq I_1 + I_2 + I_3,
\Ee
where they are bounded as, using \eqref{est:I_1}-\eqref{est:I_3},
\Be\notag
\begin{split}
I_1 &\leq \int_{\T^2} \int_{x_3}^\infty | \nabla\cdot b^\ell(y_\parallel, y_3)|  \dd y_3 \dd y_\parallel,\\
I_2 &\leq    \int_{\T^2} \int_{0}^\infty  \frac{1}{|x-y|^2} |\nabla\cdot b^\ell(y_\parallel, y_3)|  \dd y_3 \dd y_\parallel,\\
I_3  &\lesssim   \int_{0}^\infty
\mathbf{1}_{|x_3- y_3| \leq 1} \| \nabla\cdot b^\ell(\cdot, y_3) \|_{L^\infty (\T^2)}
\dd y_3  
+ \int_{0}^\infty
\mathbf{1}_{|x_3- y_3| \geq  1} e^{-|x_3 - y_3|} \|\nabla\cdot b^\ell(\cdot, y_3)  \|_{L^\infty (\T^2)}
\dd y_3  .
\end{split}
\Ee

We will use \eqref{est:D.b}. With $C$ and $D$ in \eqref{express:C} and \eqref{express:D}, we have that 
\Be\label{est:I13_t}
\begin{split}
I_1 &\lesssim C \mathbf{1}_{x_3 \leq 1} \int_{x_3}^1  |\ln (|y_3|^2 + g y_3)|  \dd y_3
+ D \int_{x_3}^\infty e^{- \frac{\tilde{\beta} g }{4} y_3} \dd y_3  \lesssim C \mathbf{1}_{x_3 \leq 1}  +  D (\tilde \beta g)^{-1} e^{-\frac{\tilde \beta g}{4} x_3} ,\\
I_3 & \lesssim C \mathbf{1}_{|x_3| \leq 2} \int_0^1 |\ln(|y_3|^2 + g y_3)| \dd y_3
+ C \sqrt{ \int_0^1 |\ln(|y_3|^2 + g y_3)|^2 \dd y} \sqrt{\int_0^\infty e^{-2 |x_3-y_3|} \dd  y_3}\\
& \ \ \ + D \int^{x_3+1}_{\max\{0, 1- x_3\}} e^{- \frac{\tilde \beta g}{4} y_3 } \dd y_3
+ D \int_0^\infty e^{-|x_3 - y_3|} e^{- \frac{\tilde \beta g}{4}y_3} \dd y_3\\
& \lesssim C  + D (\tilde \beta g)^{-1}.
\end{split}
\Ee

For $I_2$, using \eqref{est:I_2}, 
\Be\label{est:I2_t}
\begin{split}
I_2 &\lesssim C \int_{\T^2} \int_0^{1}\frac{1}{|x-y|^2}  |\ln (|y_3|^2 + g y_3)|  \dd y_3 \dd y_\parallel + D (\tilde \beta g )^{-1}\\
&	\lesssim C \left(\int_{\T^2} \int_0^{1}\frac{1}{|x-y|^{5/2}}  \dd y \right)^{4/5}
\left(\int_{\T^2} \int_0^{1}  \big|\ln (|y_3|^2 + g y_3)\big|^5  \dd y \right)^{1/5}+ D (\tilde \beta g )^{-1}\\
& \lesssim C+ D (\tilde \beta g )^{-1}.
\end{split}	 \Ee

Finally we conclude \eqref{est:D3tphi_F} from \eqref{est:D.b4}-\eqref{est:I2_t}.\hide
Using \eqref{est:D.b}
\Be
\begin{split}
I_1 \lesssim 
\end{split}
\Ee

\eqref{est:I_2}, and \eqref{est:I_3},

\Be
\begin{split}\label{est:I_1}
| I_1  | & = \delta_{j3} \Big|	\frac{\p}{\p x_3} \int_0^\infty \int_{\T^2} \Big( \frac{|x_3 - y_3| }{2}- \frac{|x_3 + y_3|}{2}\Big)
\rho (y_\parallel , y_3)  \dd y_\parallel \dd y_3\Big|\\
& = \delta_{j3} \Big| \frac{1}{2} \int_{\T^2} \int_0^\infty 
\Big(
\mathbf{1}_{x_3> y_3} - \mathbf{1}_{x_3 < y_3}- \mathbf{1}_{x_3> - y_3} + \mathbf{1}_{x_3< - y_3}
\Big)
\rho(y_\parallel, y_3) \dd y_3 \dd y_\parallel \Big|\\
& =  \delta_{j3} \Big|  - \int_{\T^2} \int^\infty_{x_3} \rho  (y_\parallel, y_3) \dd y_3 \dd y_\parallel\Big|  \leq\delta_{j3}  \frac{A}{B} e^{- B x_3}.
\end{split}
\Ee

\hide
From \eqref{elliptic_est:C1}, 
\Be\begin{split}
&| \Delta_0^{-1} \rho (x_\parallel, x_3)|\\
&= \Big| \Delta_0^{-1} \rho (x_\parallel, x_3=0)
+ \int_0^{x_3}   \p_{x_3} \Delta_0^{-1} \rho (x_\parallel, y_3) \dd y_3 \Big|\\
& \leq 0+ 10 \frac{g}{20}x_3  = \frac{g}{2} x_3.
\end{split}	\Ee
Then we 
\Be
w_h (x,v) \geq e^{\frac{\beta |v|^2}{2}} e^{\beta (g - \frac{g}{2} ) x_3}
= e^{\frac{\beta |v|^2}{2}} e^{\beta   \frac{g}{2}  x_3}\label{lower:wh}
\Ee
Hence we deduce that 
\Be\begin{split}
| \rho_h(x)| &\leq  \int_{\R^3} \frac{1}{w_h (x,v)} |w_h (x,v)[ h_+ (x,v)- h_- (x,v)] |\dd v \\
& \leq  \frac{e^{- \frac{\beta g}{2} x_3  }}{\beta} \| w_h [h_+ - h_-] \|_{L^\infty}  .
\end{split}\Ee
\unhide

For the second term, we derive that, for $0<\delta \ll1 $,
\Be\begin{split}\notag
|I_2|
&\leq 
2 \int_0^\infty \int_{\T^2}  \frac{1}{|x-y|^2} A e^{- B y_3}
\dd y_\parallel \dd y_3
\\
& \leq  
2 \int_0^\infty A e^{- B y_3}  \int_0^{\frac{1}{\sqrt{2}}}  \frac{2 \pi r \dd r}{r^2 + |x_3 - y_3|^2} \dd y_3  
\\
&= 2 \pi  A  \int_0^\infty e^{- B y_3}	   \ln \left(1+ \frac{1}{2 |x_3 - y_3|^2}\right)
\dd y_3
\\
& = 2 \pi  A \left\{
\int_{\max\{0, x_3-1\}}^{ x_3+1} + \int_{x_3+1}^\infty+ \int_{0}^{\max\{0, x_3-1\}}
\right\},
\\
\hide
&
\| w [G_+ -   G_-] \|_{L^\infty_{\gamma_-}}
\int_{0}^\infty   \frac{(2\pi)^{3/2}}{\beta^{3/2}} e^{- \frac{\beta g}{2} y_3 } 	\int_{\T^2}  \frac{1}{|x_\parallel-y_\parallel|^2 + |x_3- y_3|^2}\dd y_\parallel \dd y_3 \\
& \leq
\| w [G_+ -   G_-] \|_{L^\infty_{\gamma_-}}
\int_{0}^\infty   \frac{(2\pi)^{3/2}}{\beta^{3/2}} e^{- \frac{\beta g}{2} y_3 }  	\int_{0}^2  \frac{1}{r + |x_3- y_3|^2}\dd r\dd y_3\\
&=  \frac{(2\pi)^{3/2}}{\beta^{3/2}}  \| w [G_+ -   G_-] \|_{L^\infty_{\gamma_-}}   \int_{0}^\infty  e^{- \frac{\beta g}{2} y_3 }   
\ln \Big( 1+ \frac{2}{|x_3- y_3|^2} \Big) \dd y_3 \\ 
&  =  \frac{(2\pi)^{3/2}}{\beta^{3/2}}  \| w [G_+ -   G_-] \|_{L^\infty_{\gamma_-}}   \left\{  \int_{0}^{\frac{x_3}{2}}  +  \int_{ 2 x_3}^{\infty}  +  \int^{ 2 x_3}_{\frac{x_3}{2}}
\right\}
\unhide
\end{split}\Ee
where we have decomposed $y_3$-integral into three parts. We bound each of them:
\Be\begin{split}\notag
\int_{\max\{0, x_3-1\}}^{ x_3+1}  \leq 
1,\ \
\int_{x_3+1}^\infty+ \int_{0}^{\max\{0, x_3-1\}}   \lesssim 1/B . 
\end{split}\Ee
Hence, we derive that 
\Be\label{est:I_2}
|I_2| \lesssim  
A+ \frac{A}{B} .
\Ee

Lastly, using \eqref{b0:3} and \eqref{b0:4}, we derive that 
\Be
\begin{split}\label{est:I_3}
|I_3| & = \Big| \int_0^\infty  \int_{\T^2} \nabla_x b_0 (x, y) \rho(y) \dd y_\parallel \dd y_3\Big|\\
& \lesssim  \Big| \int_{0}^\infty \mathbf{1}_{|x_3 - y_3| \leq 1}  A e^{-B y_3} \dd y_3\Big|  + \Big| \int_{0}^\infty \mathbf{1}_{|x_3 - y_3| \geq 1}e^{- |x_3- y_3|}  A e^{-B y_3} \dd y_3\Big| 
\\
&
\lesssim  AB^{-1}.
\end{split}
\Ee

\unhide\unhide\end{proof}

\begin{theorem}\label{theo:RD} Assume $\beta> \tilde{\beta}> \max\{1,\frac{4}{g}\}$.  Suppose $\e>0$ is sufficiently small such that \eqref{bootstrapC1_dyn:thrm}, \eqref{bootstrapC2_dyn:thrm}, and \eqref{bootstrapC3_dyn:thrm} hold. \hide
\Be \label{bootstrapC1_dyn}
\frac{1}{\beta^{3/2}} 
\big\{   \| \w_{\beta, 0 } F_{0} \|_{L^\infty (\O \times \R^3)}
+  \| 
e^{\beta|v|^2}
G  \|_{L^\infty (\gamma_-)}\big\}  \leq \e g 
,  
\Ee
\Be
\label{bootstrapC2_dyn} 
\frac{1}{\tilde{\beta}^3}	\| \w_{\tilde \beta, 0}   \nabla_{x,v} F_0  \|_{L^\infty (\O \times \R^3)} 
+
\frac{1}{\tilde{\beta}^{5/2}} 
\|  e^{\tilde{\beta} | v|^2}   \nabla_{x_\parallel,v} G   \|_{L^\infty (\gamma_-)}  	\leq   \e {g^{1/2}   }  
\Ee
\Be	\label{bootstrapC3_dyn} 
\begin{split}
&	\frac{1
}{\beta^{2}} 
\big\{   \| \w_{\beta, 0 } F_{0} \|_{L^\infty (\O \times \R^3)}
+  \| 
e^{\beta|v|^2}
G  \|_{L^\infty (\gamma_-)}\big\}  \\
& \times 
\log \bigg(
e+ 
\frac{1}{\tilde{\beta}^{3/2}} \| \w_{\tilde \beta, 0}   \nabla_{x,v} F_0  \|_{L^\infty (\O \times \R^3)}
+	\frac{1}{\tilde{\beta}}
\| e^{\tilde \beta |v|^2} \nabla_{x_\parallel, v}  G \|_{L^\infty (\gamma_-)}
\bigg)
\leq \e  g  
\end{split}
\Ee\unhide
Then we can construct $\Psi^\ell, f^{\ell+1}, \varrho^\ell, \mathcal{X}^{\ell+1}, \mathcal{V}^{\ell+1}$ solve \eqref{Poisson_fell}-\eqref{initial:fell} for all $\ell=0,1,2, \cdots$. Moreover, they satisfy \eqref{Bootstrap_ell} and \eqref{condition:DDphi:dyn}. Therefore all the results in Lemma \ref{lem:Linfty_dyn}, Lemma \ref{RE:dyn}, and Lemma \ref{lem:D3tphi_F} hold. 
\end{theorem}
\begin{proof}
The proof is a consequence of Lemma \ref{lem:Linfty_dyn}, Lemma \ref{RE:dyn}, and Lemma \ref{lem:D3tphi_F}. We ought to check the conditions \eqref{Bootstrap_ell} and \eqref{condition:DDphi:dyn} to iterate our construction of sequences in \eqref{f_1}-\eqref{initial:fell}. If \eqref{bootstrapC1_dyn:thrm} holds then using \eqref{est:phi_C1_dyn} we can verify the condition \eqref{Bootstrap_ell}. Moreover if \eqref{bootstrapC2_dyn:thrm}
holds then using \eqref{est:phi_C2_dyn}and \eqref{est:D.b} we can verify the condition \eqref{condition:DDphi:dyn}. \hide

{\small\Be
\begin{split}
\label{est:phi_C2_dyn}
&  \| \nabla_x^2  \phi_{F^{\ell+1}}(s) \|_{L^\infty (\O)} \\
&  \leq 
\frac{\mathfrak{C}_1}{\beta^{3/2}} 
\big\{   \| \w_{\beta, 0 } F_{0} \|_{L^\infty (\O \times \R^3)}
+  \| 
e^{\beta|v|^2}
G  \|_{L^\infty (\gamma_-)}\big\}  \\
& \times \bigg\{
\frac{1}{g \beta }  +
\log \bigg(
e+ 
\frac{1}{\tilde{\beta}^{3/2}} \| \w_{\tilde \beta, 0}   \nabla_{x,v} F_0  \|_{L^\infty (\O \times \R^3)}
+	\frac{1}{\tilde{\beta}}\Big(1+ \frac{1}{\tilde{\beta}^{1/2}}+ \frac{1}{ g \tilde{\beta}}\Big) 
\| e^{\tilde \beta |v|^2} \nabla_{x_\parallel, v}  G \|_{L^\infty (\gamma_-)}
\bigg)\bigg\}.
\end{split}	\Ee}
\begin{align}
&
e^{ \frac{\tilde \beta}{4}  g x_3 }	|\nabla_x \cdot b^{\ell+1}(s,x)| + \tilde{\beta}^{3/2} e^{x_3}| \p_{x_3} \p_t \phi_{F^{\ell+1}}(s,x)| 
+ \tilde{\beta}^{3/2} |\p_t \phi_{F^{\ell+1}}(s,x)|
\notag\\
&\lesssim
\bigg(
\frac{1}{\tilde{\beta}^2}	\| \w_{\tilde \beta, 0}   \nabla_{x,v} F_0  \|_{L^\infty (\O \times \R^3)} 
+
\frac{1}{\tilde{\beta}^{3/2}} \Big( 1+ 	  \frac{1}{\tilde{\beta}^{1/2}} + 	\frac{1}{g {\tilde\beta} } \Big)
\|  e^{\tilde{\beta} | v|^2}   \nabla_{x_\parallel,v} G   \|_{L^\infty (\gamma_-)} \bigg)	
.\label{est:D.b} 
\end{align}

\Be
\begin{split}
\frac{  \tilde \beta ^{1/2} }{g^{1/2} }  	\| \p_t \phi_{F^\ell} \|_{L^\infty ([0,t] \times \O)}  + \frac{1}{g^2  \tilde{\beta}^{1/2}} 	\| \p_t  \p_{x_3} \phi_{F^\ell} \|_{L^\infty ([0,t] \times \p \O)} 
+  \frac{1}{g   \tilde{\beta}^{1/2}} 	\|  \nabla_x^2 \phi_{F^\ell} \|_{L^\infty ([0,t] \times   \O)} 
\lesssim 1.
\end{split}
\Ee

\Be
\sup_{0 \leq \tau \leq t}	\| \nabla_x \Psi^\ell(\tau ) \|_{L^\infty(\O  )}  + \| \nabla_x \Phi   \|_{L^\infty (\O)}\leq   \frac{g }{2}.
\Ee 
\unhide
\end{proof}
\hide
\begin{proof}[\textbf{Proof of Proposition \ref{theo:RD}}]


We will use the same sequence in the proof of Proposition \ref{propo:DC}. We will derive a uniform-in-$\ell$ bound of the sequences $(F^{\ell+1} , \phi_{F^\ell})= (h + f^{\ell+1},  \Psi^\ell + \Phi) $ where $(f^{\ell+1},  \Psi^\ell)$ solves \eqref{eqtn:fell}-\eqref{Poisson_fell} and prove Theorem \ref{theo:RD} by passing a limit. We prove the theorem using Proposition \ref{RE:dyn}, Lemma \ref{lem:D3tphi_F}, Theorem \ref{lemma:G} (\eqref{GreenF} and \eqref{b0:4}, more precisely) and Lemma \ref{lem:rho_to_phi} (\eqref{est:nabla^2phi}, precisely). 


\medskip

\textit{Step 1.}  	Recall \eqref{est:e^b}:
\Be	\max_{\ell}e^{\frac{64
\beta}{g}
\| \Delta_0^{-1} (\nabla\cdot b^\ell) \|_{L^\infty ([0,\infty) \times \O)}^2
}   \leq 2 
\Ee

From \eqref{Uest:wfell}, we derive that $
| \varrho^\ell (t,x)| \lesssim \frac{1}{\beta^{3/2}} e^{- \frac{\beta g}{2} x_3 }
M
.$ Now applying Lemma \ref{lem:rho_to_phi} and \eqref{est:nabla_phi} to this bound, we derive that 
\Be\label{est:DPsi}
\begin{split}
\|	\nabla_x \Psi ^\ell(t)\|_{L^\infty (\O)} 
\lesssim \frac{ M}{g\beta^{5/2}}  
.
\end{split}
\Ee
Applying Lemma \ref{lem:rho_to_phi} to \eqref{Uest:rho}, we have 
\Be\label{est:DPhi}
\| \nabla_x \Phi \|_{L^\infty (\O)}   \lesssim \frac{ M}{g\beta^{5/2}}  .
\Ee
Due to our choice \eqref{choice_ML}, we verify \eqref{Bootstrap_ell}.

\medskip

\textit{Step 2.} 
Applying Lemma \ref{lem:rho_to_phi} and \eqref{est:nabla^2phi}, 
we derive that 
\Be
\| \nabla_x^2 \phi_{F^\ell}(t, \cdot ) \|_{L^\infty(\bar \O)} 	
\leq 	 
\|\rho +  \varrho^\ell (t ) \|_{L^\infty(\O)} \log \big(e+ \| \rho+ \varrho^\ell (t )  \|_{C^{0,\delta}(\O)}\big)
.\label{bound:D2phi}
\Ee
\hide

Using \eqref{Uest:wf:dyn}, we derive that 
\Be\begin{split}	
e^{\frac{\beta}{2} g x_3} |\varrho (t,x)| & \leq 
\| e^{\frac{\beta}{2 }|v|^2} e^{\frac{\beta}{2} g x_3} f(t) \|_{L^\infty (\bar \O \times \R^3)}
\int_{\R^3} e^{- \frac{\beta}{2} |v|^2} \dd v\\
& \leq  \frac{2^{\frac{3}{2}} \pi^{\frac{3}{2}}}{\beta^{\frac{3}{2}}}\max\big\{  \| w_\beta  h \|_{L^\infty (\O)}, 
\| \w_{\beta, 0 } F_{0} \|_{L^\infty (\O \times \R^3)}
+  \| e^{\beta |v|^2}
G  \|_{L^\infty  (\gamma_-)} \big\}.
\end{split}\Ee

\unhide
We can easily bound the $L^\infty$-norm term, using \eqref{Uest:rho} and \eqref{Uest:wf:dyn}, above by 
\Be  
\begin{split}\label{est:rho:infty}
\| \rho (\cdot ) + \varrho^\ell (t, \cdot ) \|_{L^\infty (\O)} &\leq 
\| \rho (\cdot ) \|_{L^\infty (\O)} 
+	\left\| \int_{\R^3} f^\ell(t,\cdot ,v) \dd v\right\|_{L^\infty(\O) }\\
&  \leq \frac{\pi^{3/2}}{\beta^{3/2}} \| e^{\beta |v|^2} G(t) \|_{L^\infty (\gamma_-)}  +  \frac{2^5 \pi^{\frac{3}{2}}}{\beta^{3/2}}
M
.\end{split}\Ee

For	a $C^{0,\delta}(\O)$-control of $\rho + \varrho^\ell$, we decompose the domain into $\T^2 \times (0,1)$ and $\T^2 \times [1,\infty)$. Using the Morrey's inequality, when $x_3   \in (0,1)$, we bound a $C^{0, \delta} (\T^2 \times (0,1))$-norm by a $W^{1,p} (\T^2 \times (0,1))$-norm:
\Be\begin{split}\notag
\| \rho(\cdot)+\varrho^\ell (t, \cdot ) \|_{C^{0, \delta} (\T^2 \times (0,1))}  
& \lesssim \| \rho(\cdot)+\varrho^\ell  (t, \cdot ) \|_{W^{1,p}(\T^2 \times (0,1)) } 
= \left\| \int_{\R^3} F^\ell (t,\cdot ,v) \dd v\right\|_{W^{1,p}(\T^2 \times (0,1)) }
\\
&\lesssim 	\| \rho (\cdot ) + \varrho ^\ell (t, \cdot ) \|_{L^\infty (\O)}
+ \left\| \int_{\R^3}  \nabla_x F^\ell (t,\cdot ,v) \dd v\right\|_{L^p(\T^2 \times (0,1)) }.
\end{split}\Ee
Here, $\| \rho (\cdot ) + \varrho^\ell (t, \cdot ) \|_{L^\infty (\O)}$ is already controlled in \eqref{est:rho:infty}.

For $\nabla_x F^\ell$, we use \eqref{est:F_x:dyn}: we derive that  
\Be
\label{est:rho:holder}
\begin{split}
&
\left\|
\int_{\R^3} \nabla_x F^\ell(t, \cdot, v) \dd v 
\right\|_{L^p(\T^2 \times (0,1))} 
\\
&\lesssim  \bigg(
e^{  \frac{64 \tilde \beta}{g} \| \Delta_0^{-1} (\nabla_x \cdot b^{\ell-1}(t) ) \|_{L^\infty(\O)}  ^2 }
e^{\frac{32 }{g^2 \tilde \beta}
(1+ \| \nabla_x^2 \phi_{F^{\ell-1}} (t)\|_{L^\infty(\O)})^2	
} 
\| \w_{\tilde \beta, 0}   \nabla_{x,v} F_0  \|_{L^\infty (\O \times \R^3)} 
\\
& \ \ \   \ \ \ 
+  \frac{1}{g {\tilde\beta}^{1/2}} \|  e^{\tilde{\beta} | v|^2}   \nabla_{x_\parallel,v} G   \|_{L^\infty (\gamma_-)}    \bigg)
\left\|	\int_{\R^3} e^{- \frac{\tilde \beta}{4} (|v|^2 + g x_3)} \dd v \right\|_{L^p (\T^2 \times (0,1))}
\\
&+ \delta_{i3}  
e^{\frac{2^5}{g^2 \tilde \beta }   
\sup_{\tau \in [t-\tB,t]}  \big(1+  \| \p_{x_3}^2   \phi_{F^{\ell-1}}(\tau)   \|_{L^\infty( {\O})}  
+ \frac{1}{g} \| \p_t \p_{x_3} \phi_{F^{\ell-1}} (\tau) \|_{L^\infty(\p\O)} 
\big)^2 
}
\|  e^{\tilde{\beta} | v|^2}   \nabla_{x_\parallel,v} G   \|_{L^\infty (\gamma_-)} \\
& \ \ \ \times  \| \nabla_x \phi_{F^{\ell-1}} \|_\infty 
\left\|  \int_{\R^3} \frac{e^{- \frac{\tilde \beta}{4}|v|^2 }}{ \alpha ^{\ell } (t,\cdot ,v)} \dd v   \right\|_{L^p (\T^2 \times (0,1))}
.
\end{split}
\Ee

Since \eqref{Bootstrap_ell} holds, $\alpha^\ell (t,x,v) \geq \sqrt{|v_3|^2 + |x_3|^2 + 
g x_3}$. Then  \Be
\begin{split}
&  \int_{\R^3} \frac{e^{- \frac{\tilde \beta}{4}|v|^2 }}{ \alpha^\ell (t,x,v)} \dd v  
\leq \int_{\R^3} \frac{e^{- \frac{\tilde \beta}{4}|v|^2 }}{  \sqrt{|v_3|^2 + |x_3|^2 + 
g x_3}} \dd v\\
&\lesssim  \left(\int_{\R^2} e^{- \frac{\tilde \beta}{4}|v_\parallel|^2 }
\dd v_\parallel
\right)
\bigg\{
\int_{1}^\infty \frac{e^{- \frac{\tilde{\beta}}{4} |v_3|^2} }{ \sqrt{|v_3|^2 + |x_3|^2 + 
g x_3}} \dd v_3 
+ 	\int^{1}_0 \frac{1}{ \sqrt{|v_3|^2 + |x_3|^2 + 
g x_3}} \dd v_3 
\bigg\}\\
& \lesssim  \frac{1}{\tilde \beta }  \bigg\{ \frac{1}{\tilde \beta ^{1/2}} + \big| \ln ( |x_3|^2 + 	g x_3 )  \big|	\bigg\}.\notag
\end{split}
\Ee
Hence we derive that, for any $1 \leq p<\infty$, 
\Be\label{int:1/a:dyn}
\left\|  \int_{\R^3} \frac{e^{- \frac{\tilde \beta}{4}|v|^2 }}{ \alpha^\ell (t,\cdot ,v)} \dd v   \right\|_{L^p (\T^2 \times (0,1))} \lesssim \frac{1}{\tilde{\beta}} \left(1+ \frac{1}{\tilde{\beta} ^{1/2}}  \right).
\Ee

To control a $C^{0,\delta}(\O)$-norm of $\rho + \varrho^\ell$ in $\T^2 \times [1,\infty)$ is much simpler as \eqref{est:F_x:dyn} is now bounded pointwisely. As $\alpha^\ell(t,x,v) \geq 1$ when $x_3 \geq 1$, we derive that 

\Be
\label{est:rho:holder2} 
\begin{split} 
&	 [	 \rho (\cdot) + \varrho^\ell (t, \cdot) ]_{C^{0, \delta} (\T^2 \times [1,\infty))}\\
&	 \leq  \|	 \nabla_x  \rho (\cdot) + \nabla_x  \varrho^\ell (t, \cdot) \|_{L^\infty(\T^2 \times [1,\infty))}\\
&
\leq \left\| \int_{\R^3} e^{- \frac{\tilde \beta}{4} (|v|^2 + g x_3)}   \dd v \right\|_{L^\infty (\T^2 \times  [1,\infty))}
\| e^{ \frac{\tilde \beta}{4} (|v|^2 + g x_3)}  \nabla_x F^\ell(t)\|_{L^\infty (\T^2 \times  [1,\infty))}
\\
&\lesssim \frac{1}{\tilde \beta^{3/2}} e^{  \frac{64 \tilde \beta}{g} \| \Delta_0^{-1} (\nabla_x \cdot b^{\ell-1}(t) ) \|_{L^\infty(\O)}  ^2 }
e^{\frac{32 }{g^2 \tilde \beta}
(1+ \| \nabla_x^2 \phi_{F^{\ell-1}} (t)\|_{L^\infty(\O)})^2	
} 
\| \w_{\tilde \beta, 0}   \nabla_{x,v} F_0  \|_{L^\infty (\O \times \R^3)} 
\\
&+ \frac{1}{\tilde \beta^{3/2}}  \left(   \| \nabla_x \phi_{F^{\ell-1}} \|_\infty   
e^{\frac{2^5}{g^2 \tilde \beta }   
\sup_{\tau \in [t-\tB^\ell,t]}  \big(1+  \| \p_{x_3}^2   \phi_{F^{\ell-1}}(\tau)   \|_{L^\infty( {\O})}  
+ \frac{1}{g} \| \p_t \p_{x_3}   \phi_{F^{\ell-1}}  (\tau) \|_{L^\infty(\p\O)} 
\big)^2 
} +  \frac{1}{g {\tilde\beta}^{1/2}}
\right) \\
& \ \  \times \|  e^{\tilde{\beta} | v|^2}   \nabla_{x_\parallel,v} G   \|_{L^\infty (\gamma_-)} .
\end{split}
\Ee 

In conclusion, applying \eqref{est:rho:infty}-\eqref{est:rho:holder2} to \eqref{bound:D2phi}, we conclude
\Be\begin{split}\label{est:D2phi_F}
&\| \nabla_x^2 \phi_{F^{\ell}}(t) \|_{L^\infty (\O)}\\
& \lesssim  \left(\frac{\pi^{3/2}}{\beta^{3/2}} \| e^{\beta |v|^2} G(t) \|_{L^\infty (\gamma_-)}  +  \frac{2^5 \pi^{\frac{3}{2}}}{\beta^{3/2}}
M\right)\\
&   \times \bigg\{ \frac{ \tilde \beta}{g} \| \Delta_0^{-1} (\nabla_x \cdot b^{\ell-1}(t) ) \|_{L^\infty(\O)}  ^2 
+ \frac{1 }{g^2 \tilde \beta}
(1+ \| \nabla_x^2 \phi_{F^{\ell-1}} (t)\|_{L^\infty(\O)})^2	
\\
&  \ \ \ \   + \ln \Big(		\| \w_{\tilde \beta, 0}   \nabla_{x,v} F_0  \|_{L^\infty (\O \times \R^3)} /\tilde{\beta}^{ 3/2} 
\Big)
+ \ln 
\Big(	
\|  e^{\tilde{\beta} | v|^2}   \nabla_{x_\parallel,v} G   \|_{L^\infty (\gamma_-)} /(g {\tilde\beta}^{2})
\Big)
\\
&  \ \ \ \  +
\ln \Big(
\| \nabla_x \phi_{F^{\ell-1}} \|_\infty / \tilde{\beta}^{3/2}
\Big) 
\\
&  \ \ \ \ + \frac{1}{g^2 \tilde \beta }   
\sup_{s \in [0,t]}  \big(1+  \| \p_{x_3}^2   \phi_{F^{\ell-1}}(s)   \|_{L^\infty( {\O})}  
+ \frac{1}{g} \| \p_t \p_{x_3} \phi_{F^{\ell-1}} (s) \|_{L^\infty(\p\O)} 
\big)^2  	\bigg\}.
\end{split}\Ee

\hide
\Be\begin{split}\notag
&	\| \rho(\cdot)+\varrho (t, \cdot ) \|_{C^{0, \delta} (\O)}  \\
&\lesssim  \frac{\pi^{3/2}}{\beta^{3/2}} \| e^{\beta |v|^2} G(t) \|_{L^\infty (\gamma_-)}  +  \frac{2^5 \pi^{\frac{3}{2}}}{\beta^{3/2}}
M
\\
&+	 \bigg(
e^{  \frac{64 \tilde \beta}{g} \| \Delta_0^{-1} (\nabla_x \cdot b(t) ) \|_{L^\infty(\O)}  ^2 }
e^{\frac{32 }{g^2 \tilde \beta}
(1+ \| \nabla_x^2 \phi_F (t)\|_{L^\infty(\O)})^2	
} 
\| \w_{\tilde \beta, 0}   \nabla_{x,v} F_0  \|_{L^\infty (\O \times \R^3)} 
\\
& \ \ \   \ \ \ 
+  \frac{1}{g {\tilde\beta}^{1/2}} \|  e^{\tilde{\beta} | v|^2}   \nabla_{x_\parallel,v} G   \|_{L^\infty (\gamma_-)}    \bigg)
(\tilde{\beta})^{-3/2}
\\
&+  
e^{\frac{2^5}{g^2 \tilde \beta }   
\sup_{\tau \in [t-\tB,t]}  \big(1+  \| \p_{x_3}^2   \phi_F(\tau)   \|_{L^\infty( {\O})}  
+ \frac{1}{g} \| \p_t \p_{x_3} \phi_F (\tau) \|_{L^\infty(\p\O)} 
\big)^2 
}
\|  e^{\tilde{\beta} | v|^2}   \nabla_{x_\parallel,v} G   \|_{L^\infty (\gamma_-)} \\
& \ \ \ \times  \| \nabla_x \phi_F \|_\infty \frac{1}{\tilde{\beta}} \left(1+ \frac{1}{\tilde{\beta} ^{1/2}}  \right)\\
&+	\frac{1}{\tilde \beta^{3/2}} e^{  \frac{64 \tilde \beta}{g} \| \Delta_0^{-1} (\nabla_x \cdot b(t) ) \|_{L^\infty(\O)}  ^2 }
e^{\frac{32 }{g^2 \tilde \beta}
(1+ \| \nabla_x^2 \phi_F (t)\|_{L^\infty(\O)})^2	
} 
\| \w_{\tilde \beta, 0}   \nabla_{x,v} F_0  \|_{L^\infty (\O \times \R^3)} 
\\
&+ \frac{1}{\tilde \beta^{3/2}}  \left(   \| \nabla_x \phi_F \|_\infty   
e^{\frac{2^5}{g^2 \tilde \beta }   
\sup_{\tau \in [t-\tB,t]}  \big(1+  \| \p_{x_3}^2   \phi_F(\tau)   \|_{L^\infty( {\O})}  
+ \frac{1}{g} \| \p_t \p_{x_3} \phi_F (\tau) \|_{L^\infty(\p\O)} 
\big)^2 
}+  \frac{1}{g {\tilde\beta}^{1/2}}
\right) \\
& \ \  \times \|  e^{\tilde{\beta} | v|^2}   \nabla_{x_\parallel,v} G   \|_{L^\infty (\gamma_-)} .
\end{split}\Ee
\unhide

\medskip

\textit{Step 3.}  Now we prove a uniform-in-$\ell$ bounds for $\| \p_t \p_{x_3} \phi_{F^\ell} \|_{L^\infty ([0,T] \times \p\O)}$ and $\| \nabla_x^2 \phi_{F^\ell} \|_{L^\infty ([0,T] \times \O)}$. Set the induction hypothesis 
\Be\label{IH:pt3phi+D2phi}
\begin{split}
\max_{   j\leq \ell-1}	\| \p_t \p_{x_3} \phi_{F^j} \|_{L^\infty ([0,T] \times \p\O)}&\leq 2 L \tilde\beta^{-3/2},  \\
\max_{   j\leq \ell-1}	  \| \nabla_x^2 \phi_{F^j} \|_{L^\infty ([0,T] \times \O)} & \leq 2 M \beta^{-3/2}.
\end{split}\Ee
From this hypothesis and using Lemma \ref{lem:D3tphi_F} (\eqref{est:D3tphi_F}, precisely), \eqref{est:e^b}, and \eqref{est:D2phi_F} we try to derive 
\Be\label{est:pt3phi+D2phi}
\begin{split}
\| \p_t \p_{x_3} \phi_{F^\ell} \|_{L^\infty ([0,T] \times \p\O)}&\leq 2 L \tilde\beta^{-3/2},  \\
\| \nabla_x^2 \phi_{F^\ell} \|_{L^\infty ([0,T] \times \O)} & \leq 2 M \beta^{-3/2}.
\end{split}\Ee

Applying \eqref{IH:pt3phi+D2phi}, \eqref{est:e^b}, \eqref{est:DPhi}, \eqref{est:DPsi} to \eqref{est:D2phi_F}, we get that 
\Be\begin{split}\notag
\| \nabla_x^2 \phi_{F^{\ell}}(t) \|_{L^\infty (\O)} 
& \lesssim    \frac{1}{\beta^{3/2}}
M  \bigg\{ 
\frac{ \tilde \beta}{\beta} 
+\frac{1}{g^2 \tilde \beta }   
\big(1+  2M \beta^{-3/2}
+  2 L  g^{-1}\tilde \beta^{-3/2}
\big)^2  
\\
&  \ \ \ \  \ \ \ \ \ \ \  \ \ \   + \ln \Big(	L  \tilde{\beta}^{ -3/2} 
\Big)
+ \ln 
\Big(	
L g^{-1} {\tilde\beta}^{-2}
\Big) +
\ln \Big( M g^{-1}  \beta^{-5/2} \tilde{\beta}^{-3/2}
\Big)  	\bigg\}.
\end{split}\Ee
Applying to \eqref{est:D3tphi_F}
\Be
\begin{split}
\notag
\| \p_{x_3} \p_t \phi_{F^\ell}(t)\|_{L^\infty (\p\O)}  & \lesssim
\left(1+ \frac{M}{g \beta^{5/2}}e^{\frac{2^5}{g^2 \tilde \beta }   
\big(1+ 2 M \beta^{-3/2}
+ \frac{2 L \tilde \beta^{-3/2}}{g}
\big)^2 
}
\right) 
\frac{L }{\tilde \beta^{3/2}}
\\
&+
\frac{1}{\tilde{\beta}^3 g }L \left\{
\frac{1}{g {\tilde\beta}^{1/2}}  +   \frac{M}{g \beta^{5/2}}e^{\frac{2^5}{g^2 \tilde \beta }   
\big(1+ 2 M \beta^{-3/2}
+ \frac{2 L \tilde \beta^{-3/2}}{g}
\big)^2 
} 
+   
e^{\frac{32 }{g^2 \tilde \beta}
(1+
2 M \beta^{-3/2}
)^2	
}    \right\}
.
\end{split}
\Ee
\hide

where
\Be
\begin{split}
C = 	\left(1+ \frac{M}{g \beta^{5/2}}e^{\frac{2^5}{g^2 \tilde \beta }   
\big(1+ 2 M \beta^{-3/2}
+ \frac{2 L \tilde \beta^{-3/2}}{g}
\big)^2 
}
\right) 
\frac{L }{\tilde \beta^{3/2}},
\end{split}\Ee\Be\begin{split}
D&= 	\left(1+  \frac{M}{g \beta^{5/2}}	e^{\frac{2^5}{g^2 \tilde \beta }   
\big(1+  2 M \beta^{-3/2}
+ \frac{1}{g} 2 L \tilde{\beta}^{-3/2}
\big)^2 
}
\right) 
\frac{L }{\tilde \beta^{3/2}}
\Big(1 + \tilde{\beta}^{-1/2}\Big)\\
&	+ \bigg( 
e^{\frac{32 }{g^2 \tilde \beta}
(1+
2 M \beta^{-3/2}
)^2	
} L  +
\frac{L}{g {\tilde\beta}^{1/2}}  \bigg)	
\frac{1}{\tilde{\beta}^2}.
\end{split}
\Ee\unhide

\medskip
\textit{Step 4.} We can check the assumption \eqref{condition:DDphi:dyn} of Proposition \ref{RE:dyn} using the uniform bounds \eqref{est:pt3phi+D2phi} and our choice \eqref{choice_ML}. Applying Proposition \ref{RE:dyn} (\eqref{est:F_v:dyn}-\eqref{est:F_x:dyn}) and using \eqref{est:pt3phi+D2phi}, we derive that 

\Be
\begin{split}\notag\label{est_final:F_v:dyn}
&\| e^{\frac{\tilde \beta}{4} (|v|^2 + g x_3)} \nabla_v F^{\ell+1}(t)\|_{L^\infty(\O \times \R^3)} \\
& \lesssim  
e^{\frac{32 }{g^2 \tilde \beta}
(1+  2 M \beta^{-3/2})^2	
}  \| \w_{\tilde \beta, 0}   \nabla_{x,v} F_0  \|_{L^\infty (\O \times \R^3)} 
\\&
\ \ +\Big(1+ \frac{1}{g\tilde \beta ^{1/2}} +  
\frac{M}{g^2 \beta^{5/2}}  
\Big)  \Big(1+ \frac{M }{g^2 \tilde\beta  \beta^{ 3/2} }
\Big)  \|  e^{\tilde{\beta} | v|^2}   \nabla_{x_\parallel,v} G   \|_{L^\infty (\gamma_-)},
\end{split}\Ee 
\Be\notag
\label{est_final:F_x:dyn} 
\begin{split}
& e^{\frac{\tilde \beta}{4} (|v|^2 + g x_3)} |\nabla_x F^{\ell+1}(t,x,v)| \\
&\lesssim  
e^{\frac{32 }{g^2 \tilde \beta}
(1+ 2M \beta^{-3/2})^2	
} 
\| \w_{\tilde \beta, 0}   \nabla_{x,v} F_0  \|_{L^\infty (\O \times \R^3)} 
\\
&+\left( 
e^{\frac{2^5}{g^2 \tilde \beta }   
\big(1
+  
2M \beta^{-3/2}
+2 L g^{-1} \tilde \beta ^{-3/2}
\big)^2 
}
\frac{g^{-1} \beta^{-5/2}M
}{\alpha^{\ell+1}(t,x,v)}  +  \frac{1}{g {\tilde\beta}^{1/2}}
\right)  \|  e^{\tilde{\beta} | v|^2}   \nabla_{x_\parallel,v} G   \|_{L^\infty (\gamma_-)} ,
\end{split}
\Ee 
Finally we use our choice of $M$ and $L$ in \eqref{choice_ML} to conclude that, for all $t\geq 0$ 
\Be
\begin{split}\label{est_ell:F_v:dyn}
\| e^{\frac{\tilde \beta}{4} (|v|^2 + g x_3)} \nabla_v F^{\ell+1}(t)\|_{L^\infty(\O \times \R^3)}  
\lesssim  
\| \w_{\tilde \beta, 0}   \nabla_{x,v} F_0  \|_{L^\infty (\O \times \R^3)} 
+  \|  e^{\tilde{\beta} | v|^2}   \nabla_{x_\parallel,v} G   \|_{L^\infty (\gamma_-)},
\end{split}\Ee 
\Be
\label{est_ell:F_x:dyn} 
\begin{split}
e^{\frac{\tilde \beta}{4} (|v|^2 + g x_3)} |\nabla_x F^{\ell+1}(t,x,v)|  
& \lesssim  
\| \w_{\tilde \beta, 0}   \nabla_{x,v} F_0  \|_{L^\infty (\O \times \R^3)} \\
&
+ \Big( 1+\frac{1}{\alpha^{\ell+1} (t,x,v)} \Big) \|  e^{\tilde{\beta} | v|^2}   \nabla_{x_\parallel,v} G   \|_{L^\infty (\gamma_-)} .
\end{split}
\Ee 
Then we pass a limit (see the proof of Theorem \ref{theo:CS} for the detail) and conclude \eqref{est_final:F_v:dyn}-\eqref{est_final:F_v:dyn}. Since we have used the same sequence of Proposition \ref{propo:DC}, the limit equals the solution in Theorem \ref{theo:CD}.\end{proof}\unhide

\hide

\Be
\begin{split}\label{est:1/w_h}
\frac{\w  _\beta  (s^\prime,\Zz  (s^\prime;t,x,v) )}{\w   _\beta (s ,\Zz  (s ;t,x,v))}
&\leq 
e^{ \frac{8\beta}{g}  	\| \Delta_0^{-1} (\nabla_x \cdot b ) \|_{L^\infty_{t,x}}   \sqrt{|v_3|^2 + g x_3}
},
\\
\frac{1}{\w  _\beta  (s,\Zz   (s;t,x,v))}  
&	\leq
e^{  \frac{64 \beta}{g} 	\| \Delta_0^{-1} (\nabla_x \cdot b ) \|_{L^\infty_{t,x}}  ^2 }
e^{-\frac{\beta}{2}|v|^2}  
e^{-\frac{\beta}{2} g x_3}
,
\end{split}\Ee

Using \eqref{est:xb_v} and \eqref{est:tb^h}, we have  

where we have used a lower bound \eqref{lower:vb}.

Similarly, using \eqref{est:vb_v}, \eqref{est:tb^h}, and \eqref{lower:vb}, we derive that 
\Be
\begin{split}
\frac{	|\p_{v_i} \vB (t,x,v)| }{
e^{\tilde \beta  |\vB(t,x,v)|^2}	
}
& \leq \Big(1+ \frac{4}{g} \| \nabla_x \phi_F \|_\infty\Big)\Big(1+ \frac{32}{g^2 \tilde\beta }  \| \nabla_ x^2 \phi_F  \|_\infty\Big) e^{- \frac{\tilde\beta}{2}|v|^2} e^{- \frac{\tilde\beta g}{2} x_3}.\label{est1:vb_v:dyn}
\end{split}
\Ee

Finally we complete the prove of \eqref{est:hk_v} using \eqref{est1:hk_v}, \eqref{est1:xb_v}, and \eqref{est1:vb_v} altogether.

\medskip

{\it{Step 4. }}From \eqref{form:nabla_h}, 
\Be\begin{split}\label{est1:hk_x}
|\partial_{x_i} h  (x,v)|  \leq \frac{ |\partial_{x_i} \xb (x,v) |  }{
e^{\tilde \beta |\vb(x,v)|^2}  
}  \| e^{\tilde \beta |v|^2}  \nabla_{x_\parallel } G   \|_{L^\infty (\gamma_-)}
+\frac{   | \partial_{x_i} \vb (x,v)|}{
e^{\tilde \beta |\vb(x,v)|^2}  
}  \| e^{\tilde \beta |v|^2}  \nabla_{ v} G   \|_{L^\infty (\gamma_-)}
.
\end{split}\Ee

Using \eqref{est1:xb_x/w}, \eqref{est:xb_x}, and \eqref{est:tb^h}, we derive that 
\Be\label{est1:xb_x}
\begin{split}
\frac{|\p_{x_i} \xb  (x,v)|}{e^{\tilde \beta |\vb(x,v)|^2}  } & \leq  \frac{|\vb (x,v)|}{|v_{\mathbf{b},3} (x,v)|}e^{-\tilde \beta |\vb |^2} \delta_{i3}    + \Big(
1+ \frac{8 \| \nabla_x^2 \Phi \|_\infty }{g^2} |\vb | |v_{\b,3} | 
\Big)e^{- \tilde\beta |\vb |^2}	\\
& \ \ \ \   \ \ \ \  \ \ \ \    \ \ \ \ \    \ \ \ \  \ \ \ \ \  \ \ \ \ \times 
\min \Big \{
e^{ \frac{8}{g^2} \| \nabla_x^2 \Phi \|_\infty|v_{\b, 3} |^2 }  , 
e^{\frac{4}{g}(1+ \| \nabla_x^2 \Phi  \|_\infty) |v_{\b, 3} | }
\Big \} \\
& \leq 
\frac{|\vb(x,v)|}{\alpha  (x,v) }
e^{  \frac{4}{g} (1+ \| \nabla_x ^2 \Phi  \|_\infty ) |v_{\b, 3} |}
e^{\big(  - \tilde \beta + \frac{4}{g^2} \| \nabla_x^2 \Phi  \|_\infty
(1+ \frac{2}{g} \| \nabla_x \Phi  \|_\infty)
\big)	|\vb |^2	
}
\\
& \ \ +\frac{16}{g {\tilde\beta}^{1/2}}
\Big(1 + \frac{8}{g^2\tilde \beta } \| \nabla_x^2 \Phi  \|_\infty  \Big) e^{- \frac{\tilde\beta}{2 }|\vb |^2 }\\
& \leq 
\left(
\frac{\delta_{i3}}{\tilde{\beta}^{1/2}} \frac{1}{\alpha(x,v)}   +
\frac{16}{g {\tilde\beta}^{1/2}}
\Big(1 + \frac{8}{g^2 \tilde\beta } \| \nabla_x^2 \Phi  \|_\infty  \Big) \right)
e^{- \frac{\tilde\beta}{2}|v|^2} e^{- \frac{ \tilde\beta g}{2} x_3}
\\
& \leq 
\left(
\frac{\delta_{i3}}{\tilde{\beta}^{1/2}} \frac{1}{\alpha(x,v)}   +
\frac{32}{g {\tilde\beta}^{1/2}}
\right)
e^{- \frac{\tilde\beta}{2}|v|^2} e^{- \frac{ \tilde\beta g}{2} x_3}
,
\end{split}
\Ee
where we have used $ \frac{4}{g} (1+ \| \nabla_x ^2 \Phi  \|_\infty ) |v_{\b, 3} | + 
\big(  - \tilde \beta + \frac{4}{g^2} \| \nabla_x^2 \Phi  \|_\infty
(1+ \frac{2}{g} \| \nabla_x \Phi  \|_\infty)
\big)	|\vb |^2 \leq \frac{\tilde{\beta}}{2} |\vb|^2$.

Similarly, using \eqref{est1:xb_x/w}, \eqref{est:vb_x}, and \eqref{est:tb^h}, we derive that 
\Be\label{est1:vb_x}
\begin{split}
\frac{	|\p_{x_i} \vb (x,v)| }{
e^{\tilde \beta |\vb(x,v)|^2}	
}& \leq  \frac{\|  \nabla_x \Phi  \|_\infty}{|v_{\mathbf{b}, 3} (x,v)|}   e^{- \tilde\beta |\vb |^2} \delta_{i3}
\\
& \ + \Big(1+ \frac{4}{g} \| \nabla_x \Phi \|_\infty\Big) \frac{4}{g}  \| \nabla_ x^2 \Phi  \|_\infty  | v_{\b,3}  (x,v)|
e^{- \tilde\beta |\vb |^2}
\\
& \ \ \ \ \times 
\min \Big\{
e^{ \frac{8}{g^2}  \| \nabla_x^2 \Phi \|_\infty |v_{\b, 3}  | ^2}, e^{\frac{4}{g}  (1+ \| \nabla_x^2 \Phi \|_\infty)  |v_{\b, 3}  |}
\Big\}\\
& \leq \left( 
\frac{ \| \nabla_x \Phi \|_\infty }{\alpha(x,v)} \delta_{i3} +  \Big(1+ \frac{4}{g} \| \nabla_x \Phi \|_\infty\Big)
\right) e^{- \frac{\tilde\beta}{2}|v|^2} e^{- \frac{\tilde\beta g}{2} x_3}\\
& \leq \left( 
\frac{ \| \nabla_x \Phi \|_\infty }{\alpha(x,v)} \delta_{i3} + 3
\right) e^{- \frac{\tilde\beta}{2}|v|^2} e^{- \frac{\tilde\beta g}{2} x_3}.
\end{split}
\Ee\unhide

\subsection{Stability of the Sequence} \label{Sec:SU}
The following lemma is useful to prove that i) the sequence in Theorem \ref{theo:RD} is Cauchy; and ii) the solution (as a limit of the sequence) is unique. \hide

\begin{lemma}\label{lem:stability_seq}For given $\bar{h}_i (x,v)$ such that $\bar{\rho}_i := \int  \bar{h}_i  \dd v \in C^{0, \delta}(\O)$ for some $\delta>0$, suppose $ \Phi_i \in C^1(\bar \O) \times C^2 (\O)$ solves
\Be\notag
\Delta \Phi_i = \eta \bar{\rho}_i \  	  \text{in $\O \times \R^3$,} \ \  \ \   \Phi_i =0 \   \text{on $\p\O$.}
\Ee
Now we consider $h_i(x,v)$ solving, in the sense of \eqref{Lform:h},
\begin{align}
v\cdot \nabla_x h_i - \nabla_x (\Phi_i + g x_3 ) \cdot \nabla_v h_i =0
\   \text{in $\O \times \R^3$,} \  \   
h_i = G  \   \text{on $\gamma_-$}.
\end{align}
Suppose the following two condition hold for $g, \bar \beta, \e_0>0$
\Be
|  \Phi_1  (x)| \leq \frac{g}{2} x_3,\label{Uest:DPhi_2}
\Ee 
\vspace{-15pt}
\Be \label{condition_unique}
\frac{ 2^{3/2} \pi^{3/2} \mathfrak{C}
}{g  \bar{\beta}^{2} } 
\left\{ 1+ \frac{4}{\bar\beta g}	\right\} 
\| w_{\bar\beta} \nabla_v h_2 \|_{L^\infty (\O \times \R^3)}  \leq \frac{\e_0}{2}, 
\Ee
where $\mathfrak{C}$ had appeared in \eqref{est:phi_C1}.

Then for a small number $\e_0>1$, the following stability holds 
\Be\label{seq_stable}
\|  e^{  \frac{\bar{\beta }}{2} \big( |v|^2+g x_3\big)} (h_1(x,v) - h_2(x,v)) \|_{L^\infty (\O \times \R^3)}
\leq \frac{1}{2} \|  e^{  \frac{\bar{\beta }}{2}   \big( |v|^2+g x_3\big)} ( \bar{h}_1(x,v) - \bar{h}_2(x,v)) \|_{L^\infty (\O \times \R^3)}.
\Ee

\end{lemma}\unhide

\begin{lemma}
\label{theo:UD}
Suppose $\bar{F}_i (t,x,v)$ is defined in $\R_+ \times \bar{\O} \times \R^3$ and satisfies $\int_{\R^3} \bar{F}_i (t,x,v)  \dd v \in C^{0, \delta} (\O)$ for some $\delta>0$ and any $t \in \R_+$. Suppose $\phi_{\bar{F}_i} (t, \cdot )\in C^{1}(\bar\O) \cap C^2(\O)$ for any $t \in \R_+$ and solves 
\Be\notag
\Delta \phi_{\bar{F}_i}  = \eta \int_{\R^3} \bar{F}_i  \dd v ,\ \ \phi_{\bar{F}_i}|_{\p\O} = 0. 
\Ee
Now we consider $F_i(t,x,v)$ solving that, in the sense of Definition \ref{def:mild},  
\Be \label{eqtn:Fbar}
\p_t F_i + v\cdot \nabla_x F_i - \nabla_x (\phi_{\bar{F}_i} + g x_3) \cdot \nabla_v F_i = 0 , \  \ 
F_i|_{t=0} = F_0  , \  \ 
F _i |_{\gamma_-} =  G.
\Ee
Suppose the following condition hold for $g, \bar{\beta} >0$ 
\begin{align}
| \phi_{\bar{F}_1} (t,x) | \leq \frac{g}{2} x_3, \label{Uest:Dphi_dyn}
\\
\| e^{\frac{\tilde{\beta}}{4} (|v|^2 + g x_3)} \nabla_v F_2 \|_{L^\infty}<\infty. \label{stab_con2}
\end{align}\hide
\Be 

\Be
\p_t \int \int \bar{F}_i + \nabla_x\cdot 
\Ee\unhide
Then there exists $C= \| e^{\frac{\tilde{\beta}}{4} (|v|^2 + g x_3)} \nabla_v F_2 \|_{L^\infty}  \left\{ 1+ \frac{16}{\tilde \beta  g}	\right\} \frac{\mathfrak{C}  (8\pi)^{3/2}     }{(\tilde \beta )^{3/2}} \exp{\left(
\frac{8}{g} (\frac{\tilde \beta}{8} + (C^\prime)^2\frac{\tilde{\beta}}{g})\| \p_t \phi_{\bar F_1} \|_{L^\infty_{t,x}}^2
\right)}>0$ such that for all $t \in \R_+$, we have 
\Be\label{est:F1-2}
\| e^{  \frac{\tilde{\beta}}{8}   (|v|^2 + g x_3 )} (F_1 (t )- F_2 (t )) \|_{L^\infty (\bar\O \times \R^3)} \leq C \int^t_0
\| e^{  \frac{\tilde{\beta}}{8}  (|v|^2 + g x_3 )} (\bar{F}_1 (s )-\bar{F}_2 (s )) \|_{L^\infty (\bar \O \times \R^3)}
\dd s .
\Ee\hide

\eqref{VP_F}, \eqref{Poisson_F}, \eqref{bdry:F}, and \eqref{Dbc:F}.

As long as both solutions exist, we have that, for $1<p<\infty$, 
\Be\begin{split}\label{est:F1-2}
&	\| \w_1 (F_1-F_2) (t) \|_{L^p (\O \times \R^3)}^p  \leq
\| \w_1 (F_1-F_2) (0) \|_{L^p (\O \times \R^3)}^p \\
& \ \   \ \ \ \ \ \times   e^{ \big(2 \beta  \sup_{s \in [0,t]} \| \Delta_0^{-1} (\nabla_x \cdot b_1(s)) \|_{L^\infty(\O)}
+   \sup_{s \in [0,t]}\big\| \| \w_1 \nabla_v F_2 (s) \|_{L^p(\R^3)} \big\|_{L^3(\O)}
\big)t}.
\end{split}\Ee
Here $b_1$ is a flux associated to $\bar b_1 (t,x) = \int_{\R^3} v(\bar F_1(t,x,v) - h(x,v)) \dd v $ as in \eqref{def:flux}.\unhide
\end{lemma}

\begin{proof}
Set $\bar{\beta}= \frac{\tilde{\beta}}{8}$ and $\w_{\bar{\beta}, 1} (t,x,v) = e^{\bar\beta \left(
|v|^2 + 2 \phi_{\bar F_1} (t,x) + 2 g x_3
\right)} $. Note that the difference of solutions solves
\Be\label{eqtn:F1-F2}
\begin{split}
& \big[
\p_t + v\cdot \nabla_x - \nabla_x (\phi_{\bar F_1} + g x_3) \cdot \nabla_v 
\big] (\w_{\bar{\beta}, 1} (F_1-F_2))\\
& =  2\bar\beta 
\p_t \phi_{\bar{F}_1} 
\w_{\bar{\beta}, 1} (F_1-F_2)
+ \nabla_x (\phi_{\bar F_1} - \phi_{\bar F_2}) \cdot \w_{\bar{\beta}, 1} \nabla_v F_2,\\
& \ \ \ \  \ \ \ \  \ \ \ \ 	\w_{\bar{\beta}, 1} (F_1-F_2)|_{t=0} = 0 , \ \  \w_{\bar{\beta}, 1} (F_1-F_2)|_{\gamma_-} =0.
\end{split}
\Ee
\hide	\Be\label{eqtn:F1-F2}
\begin{split}
&	\big[
\p_t + v\cdot \nabla_x - \nabla_x (\phi_{\bar F_1} + g x_3) \cdot \nabla_v 
\big]  (F_1-F_2) =   
+ \nabla_x (\phi_{\bar F_1} - \phi_{\bar F_2}) \cdot  \nabla_v F_2,\\
&	 (F_1-F_2)|_{t=0} = 0 , \ \  (F_1-F_2)|_{\gamma_-} =0.
\end{split}
\Ee\unhide

From \eqref{Uest:Dphi_dyn}, we have that
$\w_{\bar{\beta}, 1} (s,y,u)  \geq  e^{\bar\beta \left(
|u|^2 +   g y_3
\right)}
$ and  $e^{\frac{\tilde{\beta}}{4} (|u|^2 + g y_3)} \geq e^{\frac{\tilde{\beta}}{4} (|u|^2 + 2 \phi_{\bar{F}_1} (s,y)  +2 g y_3)} = \w_{ \tilde{\beta}/4, 1 } (s, y, u)$, and therefore
\Be
\begin{split}\notag
\w_{\bar{\beta}, 1} |\nabla_v F_2 |(s, y, u)  \leq \frac{\w_{\bar{\beta}, 1} }{
\w_{ \tilde{\beta}/4, 1 }  
}  |e^{\frac{\tilde{\beta}}{4} (|u|^2 + g y_3)}\nabla_v F_2 (s, y,u  )| 
= \frac{ e^{\frac{\tilde{\beta}}{4} (|u|^2 + g y_3)}|\nabla_v F_2(s, y,u  )| }{
\w_{  {\tilde{\beta} } / {8}, 1 }  (s, y,u  )
}  .
\end{split}\Ee

Along the characteristics $\Zz_1 = (\mathcal{X}_1, \mathcal{V}_1)$ associated with a field $- \nabla_x (\phi_{\bar{F}_1} + g x_3)$, we have a form 
\Be\label{Lag:diff}
\begin{split}
\w_{\bar{\beta}, 1} (F_1-F_2)(t,z)  
= \int^t_{\max\{0, t-t_{\mathbf{B}, 1} (t,z)\}}
&\frac{e^{ \int^t_s 2 \bar\beta \p_t \phi_{\bar{F}_1 } (\tau, 
\Zz_1(\tau;t,z) 
) \dd \tau } }{\w_{\tilde{\beta}/8,1} 
(s, \Zz_1 (s;t,z))
}\nabla_x (\phi_{\bar F_1} - \phi_{\bar F_2})  (s, \Zz_1(s;t,z) )
\\		 &  \cdot  
[  
e^{\frac{\tilde{\beta}}{4} \left( |\mathcal{V}_1 (s;t,z)|^2 + g \mathcal{X}_1 (s;t,z)
\right)}
\nabla_v F_2 (s, \Zz_1(s;t,z) )]
\dd s.
\end{split}
\Ee
Here, from \eqref{stab_con2} we know that the second line of \eqref{Lag:diff} is bounded. 
\hide
Using the boundary condition \eqref{bdry:F} and the initial datum, we derive that  
\begin{align}\label{form:Fell}
F^{\ell+1 }  (t,x,v) 
= 
\mathbf{1}_{   t_{\mathbf{B}}^{\ell+1  } (t,x,v)\geq t }
F _{0 }  (\mathcal{Z} ^{\ell +1} (0;t,x,v)) 
+ \mathbf{1}_{t > t_{\mathbf{B}}^{\ell+1  } (t,x,v)}
G  
(
\mathcal{Z} ^{\ell }(t-  t_{\mathbf{B} }^{\ell +1}(t,x,v);t,x,v) ) . 
\end{align} 

Using Lemma \ref{lem:w/w_ell} and \eqref{est:1/w_h_ell}, we derive that 
$ | F ^{\ell+1} (t,x,v)|  \leq I_1 + I_2  ,$ 
where
\begin{align}
I_1  
& \leq
\frac{	\|  \w_{  { \beta}   , 0 }
F_{0 } \|_{L^\infty_{x,v}}}{\w^{\ell+1}_{
\beta} (0,\Zz^{\ell+1} (0;t,z))} 
\leq   
e^{\frac{64
	\beta}{g}
\| 
\p_t \Psi^\ell 
\|_{L^\infty_{t,x}}^2
} 
e^{- \frac{
	\beta}{2}|v|^2}
e^{- \frac{
	\beta}{2} g x_3}
\|  \w_{  { \beta}   , 0 }
F_{0 } \|_{L^\infty_{x,v}},
\label{est:Fell1}
\\
I_2 &  \leq  	 
\frac{ \| w_{\beta  }   
G  \|_{L^\infty (\gamma_-)}
}{    \w^{\ell+1} _{\beta}   (
t-  t_{\mathbf{B} }^{\ell+1 } 	,
\mathcal{Z} ^{\ell+1} (t-  t_{\mathbf{B} }^{\ell+1 };t,z))  }
\leq   
e^{  \frac{64 \beta}{g} 	\|	\p_t \Psi^\ell	 
\|_{L^\infty_{t,x}}  ^2 }
e^{-\frac{\beta}{2}|v|^2}  
e^{-\frac{\beta}{2} g x_3}
\| w_{\beta  }   
G  \|_{L^\infty (\gamma_-)}. \label{est:Fell2} 
\end{align} 

\unhide

We bound the first line of integrand in \eqref{Lag:diff} term by term. Using Lemma \ref{lem:tb} and \eqref{est:tB}, we derive that, for $s \in [ t-t_{\mathbf{B}, 1} (t,z)  , t]$ 
\Be\label{Lag:diff2}
\begin{split}
\int^t_s 2 \bar\beta \p_t \phi_{\bar{F}_1 } (\tau, 
\mathcal{X}_1(\tau;t,x,v), \mathcal{V}_1(\tau;t,x,v)
) \dd \tau 
\leq C^\prime \frac{\bar\beta^{1/2}}{ g^{1/2}}
\sqrt{|v_3|^2 + g x_3 } \frac{  \tilde \beta ^{1/2} }{g^{1/2} }  	\| \p_t \phi_{\bar F_1} \|_{L^\infty_{t,x} }. 
\end{split}
\Ee
Then following a proof of Lemma \ref{lem:stability_seq} (\eqref{diff:Phi} and \eqref{diff:bar_rho}, in particular), we derive that  
\Be\label{Lag:diff1}
\|\nabla_x (\phi_{\bar F_1} - \phi_{\bar F_2}) (s)\|_{L^\infty (\bar \O)} \leq 
\left\{ 1+ \frac{2}{\bar \beta  g}	\right\} \frac{\mathfrak{C} \pi^{3/2}     }{(\bar \beta )^{3/2}}
\|  e^{  {\bar\beta} \big( |v|^2+g x_3\big)} ( \bar{F}_1 - \bar{F}_2)(s) \|_{L^\infty (\O \times \R^3)}.
\Ee
Finally, using \eqref{est:1/w_h} with $\bar\beta /8$ instead of $\beta$,  we derive that 
\Be\label{Lag:diff3}
\frac{1}{\w  _{\frac{ \bar \beta}{8}, 1}  (s,\Zz   (s;t,x,v))}  
\leq
e^{  \frac{8 \bar\beta}{ \tilde \beta  } 
\Big(	 \frac{  \tilde \beta ^{1/2} }{g^{1/2} }  	\| \p_t \phi_{\bar F_1} \|_{L^\infty_{t,x}  }\Big)^2
}
e^{-\frac{\bar\beta}{32}|v|^2}  
e^{-\frac{\bar\beta}{32} g x_3}.
\Ee

Then applying \eqref{Lag:diff1}-\eqref{Lag:diff3} to \eqref{Lag:diff} we conclude \eqref{est:F1-2} as
\Be\notag
\begin{split}
&\| e^{ \bar{\beta}  (|v|^2 + g x_3 )} (F_1  - F_2 )(t ) \|_{L^\infty (\O \times \R^3)}
\bigg/
\int^t_0  \|  e^{  \bar\beta    \big( |v|^2+g x_3\big)} ( \bar{F}_1  - \bar{F}_2 ) (s)\|_{L^\infty (\O \times \R^3)}
\dd s 
\\
& \leq  \| e^{\frac{\tilde{\beta}}{4} (|v|^2 + g x_3)} \nabla_v F_2 \|_{L^\infty} \left\{ 1+ \frac{2}{\bar \beta  g}	\right\} \frac{\mathfrak{C}  \pi^{3/2}     }{(\bar \beta )^{3/2}} \\
& \ \ \  \times e^{
\frac{8}{g} (\bar \beta + (C^\prime)^2\frac{\tilde{\beta}}{g})\| \p_t \phi_{\bar F_1} \|_{L^\infty_{t,x}}^2
}
e^{
- \frac{\bar\beta}{32}
\left(
\sqrt{|v|^2  + g x_3 }
- 16 \frac{C^\prime \tilde{\beta}^{1/2}}{ g \bar{\beta}^{1/2}} \| \p_t \phi_{\bar F_1} \|_{L^\infty_{t,x}}\right)^2
} .
\end{split}
\Ee\vspace{-10pt}
\hide

$\frac{(Ct)^\ell}{\ell !} \leq (Ct)^\ell \left(\frac{e}{\ell}\right)^\ell  \frac{1}{e^{\frac{ 1}{12 \ell + 1}} \sqrt{2 \pi \ell}} $

For $p>1$, we have 
\Be\label{energy:wF1-1}
\begin{split}
&	\| \w_1 (F_1- F_2  )(t)  \|_{L^p (\O \times \R^3)}^p -  {\| \w_1 (F_1 - F_2  ) (0)\|_{L^p (\O \times \R^3)}^p}\\
& \leq 2 \beta \sup_{s \in [0,t]} \| \Delta_0^{-1} (\nabla_x \cdot b_1(s)) \|_{L^\infty(\O)}
\int^t_0 	\| \w_1 (F_1  - F_2 )(s) \|_{L^p (\O \times \R^3)}^p   \dd s \\
&  \ \ + \int_0^t \| \nabla_x (\phi_{F_1} - \phi_{F_2}) (s) \|_{L^{p^*} _x(\O)}
\big\| \|\w_1 \nabla_v F_2 (t)\|_{L_v^p(\R^3)} \big\|_{L_x^3(\O)}
\| \w_1 (F_1  - F_2 )(s) \|_{L^p (\O \times \R^3)}^{p-1} \dd s,
\end{split}
\Ee
where $1/p^* = 1/p- 1/3$. Here, we have used the H\"older's inequality to get, with $1/p+ 1/q=1$, 
\Be
\begin{split}
& \int_0^t \iint_{\O \times \R^3}
|	 \nabla_x (\phi_{F_1} - \phi_{F_2})(s,x) \cdot \w_1 \nabla_v F_2(s,x,v)||\w_1 (F_1- F_2) (s,x,v)|^{p-1} \dd x \dd v \dd s  \\
&\leq   \int_0^t 
\|	 \nabla_x (\phi_{F_1} - \phi_{F_2})(s ) \cdot \w_1 \nabla_v F_2(s )\|_{L^p(\O \times \R^3)}
\| |\w_1 (F_1- F_2) (s )|^{p-1} \|_{L^q (\O \times \R^3)}   \dd s \\
& \leq  \int_0^t \|	 \nabla_x (\phi_{F_1} - \phi_{F_2})(s) \|_{L^{p^*} (\O)}
\big\|
\| \w_1 \nabla_v F_2 (s) \|_{L^p(\R^3)}\big\|_{L^3_x(\O)}
\|  \w_1 (F_1- F_2)(s,x, \cdot )  \|_{L^{p}( \O \times \R^3)}^{p-1}
\dd s . \notag
\end{split}
\Ee

Now we bound $ \| \nabla_x (\phi_{F_1} - \phi_{F_2}) (s) \|_{L^{p^*} _x(\O)}$ using Lemma \ref{lemma:G}. Using $\w_1 (s,y,v) \geq e^{\beta (|v|^2+ g x_3 )}$ and H\"older's inequality, we derive that, for $1/p+ 1/q=1$,  
\Be\label{est:Dphi1-2:1}
\begin{split}
& \| \nabla_x (\phi_{F_1} - \phi_{F_2}) (s) \|_{L^{p^*} _x(\O)}\\
& \leq  \left\| 
\int_{\O}  \nabla_x G(x,y) 
\int_{\R^3} \frac{1}{\w_1 (s,y,v)} \w_1 (s,y,v)(F_1(s,y,v)- F_2(s,y,v)) \dd v 
\dd y
\right\|_{L^{p^*} _x(\O)}\\
& \leq   \left\|    \int_{\O}  |\nabla_x G(x,y) |
\left\|  \frac{1}{\w_1 (s,y,\cdot )}  \right\|_{L^q_v (\R^3)}
\| \w_1 (s,y,\cdot )(F_1(s,y,\cdot )- F_2(s,y,\cdot ))\|_{L^p_v (\R^3)} \dd y	\right\|_{L^{p^*} _x(\O)} \\
& \leq    \left\|  \int_{\O}  |\nabla_x G(x,y) |
\frac{e^{- \beta g y_3}}{ q^{\frac{3}{2q}} \beta^{\frac{3}{2q}} }
\| \w_1 (s,y,\cdot )(F_1(s,y,\cdot )- F_2(s,y,\cdot ))\|_{L^p_v (\R^3)} \dd y	\right\|_{L^{p^*} _x(\O)}.
\end{split}
\Ee
Using the form of $\nabla_x G(x,y)$ in \eqref{est:p3phi}-\eqref{est:I_3}, we further bound \eqref{est:Dphi1-2:1}:
\Be
\begin{split}\notag
&  \left\| 	\int_{\T^2} \int_{x_3}^\infty e^{- \beta g   x_3}	\frac{e^{- \beta g  (y_3-x_3)}}{ q^{\frac{3}{2q}} \beta^{\frac{3}{2q}} }
\| \w_1 (s,y,\cdot )(F_1(s,y,\cdot )- F_2(s,y,\cdot ))\|_{L^p_v (\R^3)} \dd y\right\|_{L^{p^*} _x(\O)}\\
&+  \left\|  2 \int^\infty_0 \int_{\T^2} \frac{1}{|x-y|^2}
\frac{e^{- \beta g y_3}}{ q^{\frac{3}{2q}} \beta^{\frac{3}{2q}} }
\| \w_1 (s,y,\cdot )(F_1(s,y,\cdot )- F_2(s,y,\cdot ))\|_{L^p_v (\R^3)} \dd y\right\|_{L^{p^*} _x(\O)}
\\
&+  \left\|  \int_{\O} |\nabla b_0 (x,v)|	 	\frac{e^{- \beta g y_3}}{ q^{\frac{3}{2q}} \beta^{\frac{3}{2q}} }	\| \w_1 (s,y,\cdot )(F_1(s,y,\cdot )- F_2(s,y,\cdot ))\|_{L^p_v (\R^3)} \dd y\right\|_{L^{p^*} _x(\O)}.
\end{split}
\Ee
For the first and third term, using the H\"older's inequality, we can bound term straightforwardly by $  \| \w_1  (F_1(s )- F_2(s  ))\|_{L^p  (\O \times \R^3)}$ with a constant multiplier depending on $\beta, g, p$. For the second term, using the Young's inequality of convolution ($1+\frac{1}{p^*} = \frac{1}{p} + \frac{1}{3/2}$), we bound it above by $\| \w_1  (F_1(s )- F_2(s  ))\|_{L^p  (\O \times \R^3)}$ with a constant multiplier depending on $\beta, g, p$. Hence, from \eqref{est:Dphi1-2:1} , we get 
\Be\label{est:dF1-2}
\| \nabla_x (\phi_{F_1} - \phi_{F_2}) (s) \|_{L^{p^*} _x(\O)} \lesssim_{\beta, g, p} \| \w_1  (F_1(s )- F_2(s  ))\|_{L^p  (\O \times \R^3)}.
\Ee

Finally, applying the Gronwall's inequality to \eqref{est:dF1-2} and \eqref{est:Dphi1-2:1}, we prove \eqref{est:F1-2}.

\Be
\begin{split}
&	\| \w_1 (F_1- F_2  )(t)  \|_{L^p (\O \times \R^3)}^p -  {\| \w_1 (F_1 - F_2  ) (0)\|_{L^p (\O \times \R^3)}^p}\\
& \leq 2 \beta \sup_{s \in [0,t]} \| \Delta_0^{-1} (\nabla_x \cdot b_1(s)) \|_{L^\infty(\O)}
\int^t_0 	\| \w_1 (F_1  - F_2 )(s) \|_{L^p (\O \times \R^3)}^p   \dd s \\
&  \ \ + \int_0^t
C_{\beta, g, p} \| \w_1  (\bar F_1(s )- \bar F_2(s  ))\|_{L^p  (\O \times \R^3)}
\big\| \|\w_1 \nabla_v F_2 (t)\|_{L_v^p(\R^3)} \big\|_{L_x^3(\O)}
\| \w_1 (F_1  - F_2 )(s) \|_{L^p (\O \times \R^3)}^{p-1} \dd s,
\end{split}
\Ee
\unhide \end{proof}

\hide

\begin{proposition}\label{propo:DC}
For $\beta, g>0$, suppose both \eqref{choice:g} holds. Choose
\Be\label{choice:M}
M = 4\max\big\{  \| w_\beta  h \|_{L^\infty (\O)}, 
\| \w_{\beta, 0 } F_{0} \|_{L^\infty (\O \times \R^3)}
+  \| e^{\beta |v|^2}
G  \|_{L^\infty  (\O \times \R^3)} \big\}.
\Ee
\hide\Be\label{choice:g}
\max\big\{  \| w_\beta  h \|_{L^\infty (\O)}, 
\| \w_{\beta, 0 } F_{0} \|_{L^\infty (\O \times \R^3)}
+  \| e^{\beta |v|^2}
G  \|_{L^\infty  (\O \times \R^3)} \big\} \leq  \frac{\sqrt{\ln 2}}{2^{17/2}  \pi  } g^{3/2} \beta^{5/2}.
\Ee
Let
\Be\label{choice:M}
M = 4\max\big\{  \| w_\beta  h \|_{L^\infty (\O)}, 
\| \w_{\beta, 0 } F_{0} \|_{L^\infty (\O \times \R^3)}
+  \| e^{\beta |v|^2}
G  \|_{L^\infty  (\O \times \R^3)} \big\}.
\Ee\unhide
Then we have a uniform-in-$\ell$ bound for \eqref{Bootstrap_ell} for all $\ell$. Moreover, we have a uniform-in-$\ell$ bounds for $b^\ell$ of \eqref{identity:Psi_t_ell} and $f^{\ell+1}$ constructed in \eqref{f_1}-\eqref{Poisson_fell}:
\begin{align}
\max_{\ell}e^{\frac{64
\beta}{g}
\| \Delta_0^{-1} (\nabla\cdot b^\ell) \|_{L^\infty ([0,\infty) \times \O)}^2
}  &\leq 2,\label{est:e^b}\\
\max_{\ell} \sup_{t\geq 0}	\|	e^{  \frac{
\beta}{2}|v|^2}
e^{  \frac{
\beta}{2} g x_3} f^\ell \|_{L^\infty ( \O \times \R^3)} &\leq M.\label{Uest:wfell}
\end{align}

\end{proposition}

\begin{proof}
For an arbitrary number $k \in \N$, suppose an induction hypothesis holds: \eqref{Bootstrap_ell} for all $\ell \leq k$ and $
\max_{\ell \leq k}
\sup_{t \geq 0}
\| 	e^{  \frac{
\beta}{2}|v|^2}
e^{  \frac{
\beta}{2} g x_3} f^\ell(t) \|_{L^\infty_{x,v}} \leq M$.

Using both two preceding lemmas, and \eqref{Uest:wf} and \eqref{est:D^-1 Db_ell}, we derive that 
\Be\label{est:wf}
\begin{split}
&   
e^{  \frac{
\beta}{2}|v|^2}
e^{  \frac{
\beta}{2} g x_3}
| f^{\ell+1}   (t,x,v)|\\
&\leq  \  e^{\frac{64 C^2
(1+ \frac{1}{\beta g})^2
}{g\beta^3}
\|  	e^{  \frac{
	\beta}{2}|v|^2}
e^{  \frac{
	\beta}{2} g x_3} f^\ell  \|_{L^\infty_{t,x,v}}^2
}  \big\{   \| \w_{\beta, 0 } F_{0} \|_{L^\infty_{x,v}}
+  \| e^{\beta |v|^2}
G  \|_{L^\infty_{x,v}}\big\} 
+2 \| w_\beta  h \|_{L^\infty (\O)}.
\end{split}	\Ee

\hide

\Be
\begin{split}
& \w^{\ell+1}_{\frac{\delta \beta }{8}}  (t,x,v) | f^{\ell+1}   (t,x,v)|\\
&\leq  \ e^{ \frac{32^2 (1+ \delta)^2 (1+ \frac{1}{ \delta \beta g})^2  }{ 4g^2  \delta^2   \beta^2} 
\| \w^{\ell}_{\frac{\delta \beta}{8}} f^\ell(t) \|_{L^\infty_{t,x,v}}^2
} \big\{   \| \w_{0 , \frac{ \beta (1+ \delta)}{2}} F_{0} \|_{L^\infty_{x,v}}
+  \| \w_{ \frac{ \beta (1+ \delta)}{2}} 
G  \|_{L^\infty_{x,v}}\big\} 
+2 \| w_\beta  h \|_{L^\infty (\O)}.
\end{split}	\Ee
\unhide
Due to our choice of $g$ in \eqref{choice:g}, we can get that 
\Be
e^{ {\frac{64 C^2
(1+ \frac{1}{\beta g})^2
}{g\beta^3}
\sup_{t \geq 0}
\| 	e^{  \frac{
	\beta}{2}|v|^2}
e^{  \frac{
	\beta}{2} g x_3} f^\ell(t) \|_{L^\infty_{x,v}}
^2
}  
}   \leq e^{ {\frac{128 C^2
}{g^3\beta^5}
M^2
}  
} 
\leq 2. \notag
\Ee
Hence, we already prove \eqref{est:e^b} due to \eqref{est:D^-1 Db_ell}. Together with this and \eqref{choice:M}, \eqref{est:wf}, we conclude \eqref{Uest:wfell}:
\Be\notag
\begin{split}
\sup_{t \geq 0}	\|e^{  \frac{
\beta}{2}|v|^2}
e^{  \frac{
\beta}{2} g x_3}  f^k(t) \|_{L^\infty_{x,v}} &\leq \frac{M}{4}   	e^{ {\frac{64
	(1+ \frac{1}{\beta g})^2
}{g\beta^3}
\|  	e^{  \frac{
		\beta}{2}|v|^2}
e^{  \frac{
		\beta}{2} g x_3} f^\ell  \|_{L^\infty_{t,x,v}}^2
}  
}  + \frac{M}{2} \leq  \frac{M}{2} +  \frac{M}{2} \leq M.
\end{split}\Ee 
Now following the argument to get \eqref{Uest:DPhi^k}, we can derive \eqref{Bootstrap_ell} for $\ell= k+1$. \end{proof}
\unhide
\begin{proof}[\textbf{Proof of Theorem \ref{theo:CD}}]
Using Theorem \ref{theo:RD} and Lemma \ref{theo:UD}, it is standard to deduce that, for $\ell\geq m$,  
\vspace{-10pt}
\Be\notag
\begin{split}
&\| e^{  \frac{\tilde{\beta}}{8}   (|v|^2 + g x_3 )} (F^\ell (t )- F^m (t )) \|_{L^\infty (\bar\O \times \R^3)} \\
&  \leq \frac{(Ct)^{m}}{m !} \| e^{  \frac{\tilde{\beta}}{8}   (|v|^2 + g x_3 )}  f^{\ell-m} (t )  \|_{L^\infty (\bar\O \times \R^3)}\\
& \leq  \frac{(Ct)^{m}}{m !} 	e^{\frac{C
\beta}{\tilde{\beta }}
}  \big\{   \| \w_{\beta, 0 } F_{0} \|_{L^\infty (\O \times \R^3)}
+  \| 
e^{\beta|v|^2}
G  \|_{L^\infty (\gamma_-)}\big\}  ,
\end{split}\Ee
where we have used \eqref{Uest:wf} at the last step above. With this strong convergence together with uniform-upper-bounds of Theorem \ref{theo:RD}, it is standard to prove the convergence of the sequences and prove that their limiting function $(F, \phi_F)$ is a strong solution to \eqref{VP_F}-\eqref{Dbc:F}. Moreover, every upper bound of Theorem \ref{theo:RD} is valid for the limiting function. A proof of uniqueness is straightforward from Lemma \ref{theo:UD}. We omit the proof.
\end{proof}

\hide

From Lemma \ref{lem:Ddb}, we derive that 
\begin{align}
\label{D^-1 Db_ell}
\Delta_0^{-1}  (\nabla_x \cdot b^\ell) (t,x)   =    \int_{\O} G(x,y) \nabla  \cdot b^\ell (y) \dd y =  -\int_{\O} b^\ell (y)   \cdot  \nabla_y G(x,y)     \dd y,\\
\sup_{0 \leq t \leq T }	\|	 \Delta_0^{-1}  (\nabla_x \cdot b^\ell (t,x)) \|_{L^\infty (\O)} \leq 8 \pi  
\frac{1+ \frac{1}{  \beta g}}{  \beta^2 } \sup_{0 \leq t \leq T } \|   e^{  \frac{\beta}{2} (|v|^2 + gx_3)}  f^\ell(t) \|_{L^\infty(\O \times \R^3)} .\label{est:D^-1 Db_ell}
\end{align} 

\unhide

\noindent\textbf{Acknowledgment. }
The author thanks Professor Hyung-Ju Hwang for her interest to this work. The author thanks Dr. Jiaxin Jin and Dr. Jongchon Kim for discussions helpful to write Section \ref{Sec:Green}. He also thanks Professor Antoine Cerfon for his insightful comments in the author's talk on the occasion of a kinetic theory workshop at Madison in October 2019. He thanks Dr. Trinh Nguyen for his presentation about recent development of Landau damping on the occasion of kinetic theory working seminars in Madison. The author also thanks Professor Seung Yeal Ha, Donghyun Lee, and In-Jee Jeong, for their kind hospitality during his stay at the Postech and Seoul National University in 2021-2022 where/when the author has conducted this project partly. This project is partly supported by NSF-DMS 1900923/2047681 (CAREER), the Simons fellowship in Mathematics, and the Brain Pool fellowship funded by the Korean Ministry of Science and ICT.

\hide

\subsection{Asymptotic Stability of the Dynamic perturbation}
The Lagrangian formulation of $f^{\ell+1}$ solving \eqref{eqtn:fell}-\eqref{initial:fell} for given $\nabla_x \Phi,  \nabla_x \Psi^\ell$ and $\nabla_v h$ is 
\Be\label{form:f^ell}
\begin{split}
f^{\ell+1}   (t,x,v)   
&= \mathcal{I}^{\ell+1} (t,x ,v) + \mathcal{N}^{\ell+1}  (t,x,v ), 
\end{split}
\Ee where 
\begin{align} 
\mathcal{I}^{\ell+1} (t,x,v ) 	& :=  
\mathbf{1}_{t <  \tB^{\ell+1 } (t,x,v)}	f _{ 0} ( \Zz ^{\ell+1}  (0;t,x,v) ) 
,\label{form:f^ellI} \\
\mathcal{N}^{\ell+1} (t,x,v ) 	&  :=
\int^t_
{\max\{0, t- \tB^{\ell+1 } (t,x,v) \}} 
\nabla_x  \Psi ^\ell
(s, \X^{\ell+1}  (s;t,x,v)) \cdot \nabla_v h 
( \Zz ^{\ell+1}  (s;t,x,v) )
\dd s 
.\label{form:f^ellN}
\end{align}

\hide

Taking a $v-$integration to the form of $f_+^{\ell+1} - f_-^{\ell+1}$ in \eqref{form:f^ell}, we get that 
\begin{align}\label{form:rho}
\rho ^{\ell+1} (t,x)=\mathcal{I}^{\ell+1}(t,x) + \mathcal{N}^{\ell+1} (t,x)
,
\end{align}
where
\Be
\begin{split}
\mathcal{I} ^{\ell+1} (t,x) 
=  \int_{\R^3} 
\mathbf{1}_{t \leq t_\b^{\ell+1, \pm} (t,x,v)}
[f_{0,+ } - f_{0,- }]
(  \mathcal{Z}^{\ell+1}(0;t,x,v) ) \dd v, \label{rhof1}
\end{split}\Ee 
\Be
\begin{split}
&\mathcal{N} ^{\ell+1} (t,x)\\
& = \int_{\R^3} \int^t_{{\max\{0, t- \tb^{\ell+1 } (t,x,v) \}}   } \nabla_x\Delta^{-1}_0 \rho ^\ell  (s, \X^{\ell+1}  (s;t,x,v)) \cdot  [ \nabla_v h_+ - \nabla_v h_-]  
( \mathcal{Z} ^{\ell+1}  (s;t,x,v) ) 
\dd s 
\dd v.\label{rhof2}
\end{split}\Ee
\unhide

The main purpose of this section is to prove the temporal decay of $\mathcal{I}^{\ell+1}$ (Proposition \ref{prop:decay_I}) and $\mathcal{N}^{\ell+1}$ (Proposition \ref{prop:decay_N}) uniformly-in-$\ell$.

\begin{proposition}\label{prop:decay}
Suppose $\|w^F   [f_{0,+} - f_{0,-}]   \|_{L^\infty (\bar \O \times \R^3)} < \infty$ and $\| \nabla_x \phi_F \|_{L^\infty (\O \times \R^3)} < \infty$. Furthermore we assume that 
\Be  \frac{  4  \| \nabla _x \phi_F \|_\infty }{  g } \leq \frac{1}{2}.
\Ee

Then 
\begin{align}
\mathcal{I}^{\ell+1}  (t,x,v)  &\leq 	2 e^{-  \frac{\beta}{4}  ( |v|^2 + g  x_3)} 
e^{- \lambda_\infty t}
e^{ \frac{16 \lambda_\infty^2}{\beta g^2}}	  \|w_{0 }   f_{0  } \|_{L^\infty_{x,v}}  ,\label{est:Iell}\\
\mathcal{N}^{\ell+1} (t,x,v)  & 
\leq   	 e^{- \frac{\beta}{8} |v|^2} e^{- \frac{\beta g}{8} x_3}
e^{-\lambda_\infty t}  
\frac{8 \cdot 4^{1/g}}{g \beta^{1/2}}  
e^{ \frac{16^2 \lambda_\infty^2}{ 4 g^2 \beta}}  	\sup_{s \in [0, t] } \| e^{\lambda_\infty s} \varrho^\ell (s) \|_\infty  \| w_\beta \nabla_v h \|_\infty
.\label{est:Nell}
\end{align}

\hide
\Be
\mathcal{I}^{\ell+1}  (t,x) \leq e^{- \frac{\beta}{4} \frac{g^2 t^2}{C^2}}  \frac{1}{\beta^{3/2}}
\Ee \unhide
\end{proposition}
\begin{proof} 
First we prove \eqref{est:Iell}. From \eqref{est:tB}, \eqref{est:1/w_h}, \eqref{est:e^b}, and \eqref{est:D^-1 Db}, we derive that \hide
\Be
\begin{split}
\eqref{form:f^ellI}& \leq \|\w_{\beta, 0 }   f_0 \|_{L^\infty}  \int_{\R^3}  \frac{ \mathbf{1}_{t \leq t_\mathbf{B}^{\ell+1}  (t,x,v)}}{w_{ {\beta} / {2}} (\Zz^{\ell+1} (0;t,x,v) )  }  \dd v\\
& \leq    \|\w_{\beta, 0 }   f_0 \|_{L^\infty}  \int_{\R^3}  
\mathbf{1}_{  
t \leq \frac{2}{g} \big( \sqrt{|v_3|^2 + g x_3} - v_3\big)
}
\dd v\\
& \leq \frac{1}{\beta^{3/2}} e^{- \frac{\beta g}{4}  x_3 }
\mathbf{1}_{ gt/4 \leq   \sqrt{|v_3|^2 + g x_3}}
e^{- \frac{\beta}{4}  ( |v|^2 + g  x_3)}\\
& \leq  \frac{1}{\beta^{3/2}} e^{- \frac{\beta g}{4}  x_3 } e^{-   \frac{ \beta g^2 }{64} t^2} 
\end{split}	\Ee
\unhide
\Be
\begin{split}\label{bound:tB/w}
& \mathbf{1}_{t< \tB^{\ell+1 } (t,x,v)}  | f_{ 0} (\mathcal{Z} ^{\ell+1} (0;t,x,v))|
\\
& \leq \|w_{0 }   f_{0  } \|_{L^\infty_{x,v}}    \frac{ \mathbf{1}_{t \leq \tB^{\ell+1 } (t,x,v)}}{w _\beta (0,\X^{\ell+1} (0;t,x,v),\V^{\ell+1}  (0;t,x,v))}  \\
& \leq  \|w_{0 }   f_{0 } \|_{L^\infty_{x,v}}    \mathbf{1}_{ gt/4 \leq   \sqrt{|v_3|^2 + g x_3}}
e^{  \frac{64 \beta}{g} 	\| \Delta_0^{-1} (\nabla_x \cdot b^\ell) \|_{L^\infty_{t,x}}  ^2 }
e^{-\frac{\beta}{2}|v|^2}  
e^{-\frac{\beta}{2} g x_3} 
\\
& \leq  2 \|w_{0 }   f_{0  } \|_{L^\infty_{x,v}}  	
e^{- \frac{g^2 \beta }{64} t^2}	e^{- \frac{\beta}{4}  ( |v|^2 + g  x_3)}
,
\end{split}	\Ee
where we have used that, if $t \leq \frac{4}{g} \sqrt{|v_3|^2 + g x_3}$ then 
\Be\notag
\begin{split}
&\mathbf{1}_{t \leq \frac{4}{g} \sqrt{|v_3|^2 + g x_3}} e^{- \frac{\beta}{2} \big(|v|^2 + gx_3\big)}\\
& \leq \mathbf{1}_{t \leq \frac{4}{g} \sqrt{|v |^2 + g x_3}} e^{- \frac{\beta}{2} \big(|v|^2 + gx_3\big)}\\
& \leq e^{- \frac{g^2\beta}{64}  t^2} e^{- \frac{\beta}{4} \big(|v|^2 + gx_3\big)}.
\end{split}
\Ee
By completing the square, we have 
\Be
\begin{split}\notag
e^{- \frac{\beta g^2}{64} t^2} &= e^{- \frac{\beta g^2}{64} \big(t- \frac{32 \lambda_\infty}{\beta g^2}\big)^2} e^{\frac{16 \lambda_\infty^2}{\beta g^2}} e^{-\lambda_\infty t} \leq e^{\frac{16 \lambda_\infty^2}{\beta g^2}} e^{-\lambda_\infty t}.
\end{split}
\Ee
Combining this with \eqref{bound:tB/w}, we derive \eqref{est:Iell}.

\hide
Since $w^F (X^F(s;t,x,v),V^F (s;t,x,v))$ is invariant with respect to $s \in [ t - t^\b  _{F}(t,x,v)  ,  t+ t^\f  _{F}(t,x,v)]$, the denominator equals, for $C=4$, 
\Be
\begin{split}\label{est:1/w}
&\frac{\mathbf{1}_{ gt/C \leq   |v_{\f ,3}^F(t,x,v)|  }}{ w^F (X^F(t+ t_\f ^{F};t,x,v),V^F (t+t_\f^F;t,x,v))}\\
&= \mathbf{1}_{ gt/C \leq   |v_{\f ,3}^F(t,x,v)|  }
e^{- \beta \frac{| v_{\f ,3}^F(t,x,v)|^2}{2}}e^{- \beta \frac{| v_{\f,\parallel}^F (t,x,v)|^2}{2}}\\
& \leq e^{- \frac{\beta}{4} \frac{g^2 t^2}{C^2}}e^{- \beta \frac{| v_{\f,3}^F (t,x,v)|^2}{4}}e^{- \beta \frac{| v_{\f,\parallel}^F (t,x,v)|^2}{2}}\\
& \leq e^{- \frac{\beta}{4} \frac{g^2 t^2}{C^2}}
e^{- \frac{\beta}{4} \Big( 1- \frac{2  C^2 \| \nabla _x \phi_F \|_\infty^2}{  g^2}  \Big)  | v_{\f,3}^F (t,x,v)|^2}  
e^{- \beta \frac{| v_{ \parallel} |^2}{4}} \\
& \leq e^{- \frac{\beta}{4} \frac{g^2 t^2}{C^2}}
e^{- \frac{\beta}{4} \Big( 1- \frac{2  C^2 \| \nabla _x \phi_F \|_\infty^2}{  g^2}  \Big)  | v_3 |^2 } 
e^{- \beta \frac{| v_{ \parallel} |^2}{4}} 
\end{split}
\Ee
where we have used the following inequality to get the fourth line of \eqref{est:1/w}, for $C=4$,
\Be
\begin{split}\notag
&-|v_{\f, \parallel}^F (t,x,v)|^2\\
&\leq  -\Big| |v_\parallel | -  t_\f^F(t,x,v) \| \nabla_x \phi_F \|_\infty \Big|^2 \\
& 
\leq  - \Big| |v_\parallel |  - \frac{C |v^F_{\f, 3} (t,x,v)|}{ g} \| \nabla_x \phi_F \|_\infty\Big|^2\\
& = - |v_\parallel|^2 + 2\frac{C |v_{\f, 3}^F|}{g} \| \nabla_x  \phi_F \|_\infty | v_\parallel| - \frac{C^2 |v_{\f, 3}^F|^2}{g^2} \| \nabla_x \phi _F \|_\infty^2\\
&  \leq  - |v_\parallel|^2  + \frac{|v_\parallel|^2}{2} + 2\frac{C^2 |v_{\f, 3}^F|^2}{g^2} \| \nabla_x  \phi_F \|_\infty^2 - \frac{C^2 |v_{\f, 3}^F|^2}{g^2} \| \nabla_x \phi _F \|_\infty^2\\
& = -\frac{|v_\parallel|^2}{2}  +  \frac{C^2 |v_{\f, 3}^F|^2}{g^2} \| \nabla_x \phi _F \|_\infty^2
\end{split}
\Ee
and have used $|v_{\f,3}^F (t,x,v)| \geq |v_3|$ at the last line of \eqref{est:1/w}. 

Therefore 
\Be\begin{split}
& \int_{\R^3}  \frac{ \mathbf{1}_{ gt/C \leq   |v_{F,3}^\f (t,x,v)|  }}{w^F (X^F(0;t,x,v),V^F (0;t,x,v))}  \dd v \\
& \leq e^{- \frac{\beta}{4} \frac{g^2 t^2}{C^2}}  \frac{8}{\beta^{3/2}} \Big( 1- \frac{2  C^2 \| \nabla _x \phi_F \|_\infty^2}{  g^2}  \Big)^{-1} 
\end{split}
\Ee
\unhide

Next we prove \eqref{est:Nell}. Using \eqref{est:1/w}, we obtain 
\Be\label{h_v}
\begin{split}
&\nabla_v h (\Zz^{\ell+1}(s;t,x,v) ) \\
&\leq \frac{ 1}{w  (X^{\ell+1}(s;t,x,v), V^{\ell+1}(s;t,x,v))} \| w \nabla_v h \|_\infty
\\ 
& \leq  	e^{\frac{16^2 \beta }{2 g^2}
\| \Delta_0^{-1} (\nabla_x \cdot b^\ell) \|_{L^\infty_{t,x}}^2	
}
e^{- \frac{\beta}{4}|v|^2 } e^{- \frac{\beta g}{4} x_3}  \| w \nabla_v h \|_\infty  \\
& \leq 4^{1/g}
e^{- \frac{\beta}{4}|v|^2 } e^{- \frac{\beta g}{4} x_3}  \| w \nabla_v h \|_\infty ,
\end{split}
\Ee 
where we have used that $
e^{\frac{16^2 \beta }{2 g^2}
\| \Delta_0^{-1} (\nabla_x \cdot b^\ell) \|_{L^\infty_{t,x}}^2	
}
\leq e^{\frac{\ln 4}{g}} \leq  4^{1/g}$ from \eqref{est:e^b}.

Using this, we now bound $\mathcal{N} ^{\ell+1}$:
\Be\begin{split}
&|	\mathcal{N} ^{\ell+1}(t,x,v)| 
\\
& \leq  4^{1/g} \int^t_{t-\tB^{\ell+1} (t,x,v)}
e^{- \lambda_\infty s}  
\sup_{s \in [0, t] } \| e^{\lambda_\infty s} \varrho^\ell (s) \|_\infty
e^{- \frac{\beta}{4} |v|^2} e^{- \frac{\beta g}{4} x_3} \| w_\beta \nabla_v h \|_\infty \dd s \\
& \leq  4^{1/g}   e^{-\lambda_\infty t} 
\sup_{s \in [0, t] } \| e^{\lambda_\infty s} \varrho^\ell (s) \|_\infty  \| w_\beta \nabla_v h \|_\infty \\
& \ \ \times 
\underline{\tB^{\ell+1} (t,x,v)
e^{\lambda_\infty  {\tB^{\ell+1} (t,x,v)}}
e^{- \frac{\beta}{4} |v|^2} e^{- \frac{\beta g}{4} x_3} } .\notag
\end{split}\Ee
Then using \eqref{est:tB} for $\tB^{\ell+1} (t,x,v)$, we bound the above underlined term as 
\Be
\begin{split}\notag
&\tB^{\ell+1} (t,x,v)
e^{\lambda_\infty  {\tB^{\ell+1} (t,x,v)}}
e^{- \frac{\beta}{4} |v|^2} e^{- \frac{\beta g}{4} x_3}\\
& \leq  \frac{4}{g} \sqrt{|v_3|^2 + g x_3} 
e^{ \frac{4 \lambda_\infty}{g} (|v_3| + \sqrt{g x_3})}e^{- \frac{\beta}{4} |v|^2} e^{- \frac{\beta g}{4} x_3}\\
& \leq \frac{8}{g \beta^{1/2}} e^{- \frac{\beta}{8} |v|^2 + \frac{4 \lambda_\infty}{g} |v|  } 
e^{- \frac{\beta g}{8} x_3 + \frac{4 \lambda_\infty}{g} \sqrt{g x_3}} 
e^{- \frac{\beta}{8} |v|^2} e^{- \frac{\beta g}{8} x_3}\\
& \leq  \frac{8}{g \beta^{1/2}} e^{ \frac{16^2 \lambda_\infty^2}{ 4 g^2 \beta}}
e^{- \frac{\beta}{8} |v|^2} e^{- \frac{\beta g}{8} x_3} .
\end{split}
\Ee
Therefore we get 
\Be\label{est:N1}
\begin{split}
& |	\mathcal{N} ^{\ell+1}(t,x,v)| \\
& \leq  \frac{8}{g \beta^{1/2}}   
4^{1/g}
e^{ \frac{16^2 \lambda_\infty^2}{ 4 g^2 \beta}}  	\sup_{s \in [0, t] } \| e^{\lambda_\infty s} \varrho^\ell (s) \|_\infty  \| w_\beta \nabla_v h \|_\infty
e^{-\lambda_\infty t}  e^{- \frac{\beta}{8} |v|^2} e^{- \frac{\beta g}{8} x_3}
\end{split}\Ee
Finally we use \eqref{est:D^-1 Db} to conclude \eqref{est:Nell}. \hide
\Be\label{est:N1}
\begin{split}
& |	\mathcal{N} ^{\ell+1}(t,x,v)| \\
& \leq  \frac{8}{g \beta^{1/2}}   e^{
\frac{16^2  C^2 (1+ \frac{1}{\delta \beta g})^2}{2 g^2 \delta^2 \beta } \| \w^\ell_{\frac{\delta \beta}{8}} f^\ell \|_\infty^2	
}    e^{ \frac{16^2 \lambda_\infty^2}{ 4 g^2 \beta}}  	\sup_{s \in [0, t] } \| e^{\lambda_\infty s} \varrho^\ell (s) \|_\infty  \| w_\beta \nabla_v h \|_\infty
e^{-\lambda_\infty t}  e^{- \frac{\beta}{8} |v|^2} e^{- \frac{\beta g}{8} x_3}
\end{split}\Ee

\Be\begin{split}\label{est:N}
&	\frac{e^{-\lambda_\infty t}}{\lambda_\infty} e^{\frac{16^2 \beta}{2 g^2} \| \Delta_0^{-1} (\nabla \cdot b^\ell) \|_\infty^2}    \| w_\beta \nabla_v h \|_\infty  
e^{- \frac{\beta}{8} |v|^2 }e^{- \frac{\beta}{8}  gx_3 }
e^{ - \frac{\beta}{8} 
\left(
|v|^2- \frac{64\lambda_\infty}{g \beta} |v|
\right)
}
e^{ - \frac{\beta}{8} 
\left(
g x_3- \frac{64 \lambda_\infty}{g \beta} \sqrt{g x_3}
\right)
}\\
&  \leq \frac{e^{-\lambda_\infty t}}{\lambda_\infty} e^{\frac{32^2 \beta}{4 g^2} \| \Delta_0^{-1} (\nabla \cdot b^\ell) \|_\infty^2}  e^{ \frac{16^2 \lambda_\infty^2}{ 2 g^2 \beta } }   \| w_\beta \nabla_v h \|_\infty  
\\
&	  \leq \frac{e^{-\lambda_\infty t}}{\lambda_\infty} e^{
\frac{   C^2 32^2 \beta  (1+ \frac{1}{\delta \beta g} )^2}{ 4  \delta^2 \beta^2 g^2 }  \| \w^\ell_{\frac{\delta \beta}{8}} f^\ell   \|^2_{L^\infty_{t,x}}
}  e^{ \frac{16^2 \lambda_\infty^2}{ 2 g^2 \beta } }   \| w_\beta \nabla_v h \|_\infty  .
\end{split}\Ee

\unhide\end{proof}\unhide

\hide
\begin{proposition}\label{prop:decay_N}
Choose $p>3$. Assume there exist $0< \lambda_\infty < \lambda_p<\infty$ and $\mathfrak{C}_{\rho, p}, \mathfrak{C}_{\rho, \infty}  \in (0 , \infty)$ such that 
\begin{align}
\sup_{t \geq 0 }\| e^{\lambda_p t} \rho  (t)\|_{L^p(\O)}< \mathfrak{C}_{\rho, p} ,\\
\sup_{t \geq 0 }\| e^{\lambda_\infty t} \rho  (t)\|_{L^\infty(\O)}< \mathfrak{C}_{\rho, \infty} .
\end{align}
and $ \| w_h \nabla_v h \|_{L^\infty_{x,v}}< \mathfrak{C}_h < \infty$. Then 
\Be\begin{split}
&\Big\| \mathcal{N}[\rho_f, h ,Z^F] (t, \cdot ) \Big\|_{L^p_x} \\
&\leq
\end{split}\Ee
and 
\Be
\sup_{t\geq 0}
\Big\|	e^{\lambda_\infty t} \mathcal{N}[\rho_f, h ,Z^F] (t,\cdot)\Big\|_{L^\infty_x} \leq    \frac{\mathfrak{C}_h}{g  \beta^{2}} e^{\frac{32 \lambda _\infty^2  }{  g^2\beta}} \mathfrak{C}_{\rho, \infty}  e^{-\lambda_\infty t}. 
\Ee
\end{proposition}

\begin{proof}

\Be
\begin{split}\label{consistence} 
&\Big\|	 \mathcal{N} ^{\ell+1}(t,\cdot)\Big\|_{L^\infty_x}  \\
& \leq  \Big\| 
\tB^{\ell+1} (t,x,v) 
e^{\frac{16^2 \beta }{2 g^2}
\| \Delta_0^{-1} (\nabla_x \cdot b^\ell) \|_{L^\infty_{t,x}}^2	
}
e^{- \frac{\beta}{4}|v|^2 } e^{- \frac{\beta g}{4} x_3}  \| w \nabla_v h \|_\infty 
\\
& \ \ \ \ \  \times 
\sup_{s \in [ t- \tB^{\ell+1} (t,x,v),t]}| \nabla_x \Delta_0^{-1} \varrho^\ell  (s) | 
\dd v
\Big\|_{L^\infty_x}
\\
&\leq \mathfrak{C}_h \Big\| \int_{\R^3}  \frac{4}{g}|v_{\b,3}^F(x,v)| 
e^{- \frac{\beta}{2} \Big(1 - 	  \frac{16 }{g}	\big(2+\frac{8\| \nabla_x \phi_F \|_\infty }{g} \big) - \frac{32 \| \nabla_x \phi_F \|_\infty^2}{g^2}
\Big) |v_{\b,3}^F (t,x,v)|^2} e^{- \frac{\beta}{2} 
\Big(1 - 
\frac{16}{g} 
\| \nabla_x \phi_f \|_\infty 
\Big)
|v_\parallel|^2} 
\\
& \ \ \ \ \  \times 
\mathfrak{C}_{\rho, \infty}  e^{-\lambda_\infty\Big(t- 4\frac{|v_{\b,3}^F(t,x,v)|}{g}\Big)}  
\dd v\Big\|_{L^\infty_x}
\\
& \leq \mathfrak{C}_h \mathfrak{C}_{\rho, \infty}   e^{-\lambda_\infty t} \int_{\R^3} e^{\lambda_\infty  \frac{ 4 |v_{\b,3}^F (t,x,v)|}{g}}  \frac{1}{\sqrt{\beta}}  \sqrt{\beta} \frac{ 4| v_{\b,3}^F(t,x,v) |}{g} e^{- \frac{\beta}{3} |v_{\b,3}^F (t,x,v)|^2 } e^{- \frac{\beta}{4} |v_\parallel|^2}  \dd v\\
&\leq  \mathfrak{C}_h\mathfrak{C}_{\rho, \infty}  
e^{-\lambda_\infty t}\frac{1}{g\beta^{3/2}} 
\int_{\R}
e^{- \frac{\beta}{4} \{|v^F _{\b,3}(t,x,v)|^2 - \frac{16 \lambda_\infty}{g \beta} |v^F_{\b,3}(t,x,v)| + \frac{ 16 \lambda_\infty^2 }{g^2\beta^2} \}}
e^{\frac{4 \lambda _\infty^2 }{g^2\beta}} \dd v_3\\
& \leq \frac{\mathfrak{C}_h \mathfrak{C}_{\rho, \infty}  }{g  \beta^{3/2}} e^{\frac{32 \lambda _\infty^2  }{  g^2\beta}} e^{-\lambda_\infty t}
\int_{\R} e^{ - \frac{\beta}{8} |v_{\b, 3}^F (t,x,v)|  ^2} \dd v_3\\
& \leq \frac{\mathfrak{C}_h \mathfrak{C}_{\rho, \infty}  }{g  \beta^{3/2}} e^{\frac{32 \lambda _\infty^2  }{  g^2\beta}} e^{-\lambda_\infty t}
\int_{\R} e^{ - \frac{\beta}{8} |v_3|  ^2} \dd v_3\\
& \leq \frac{\mathfrak{C}_h \mathfrak{C}_{\rho, \infty}  }{g  \beta^{2}} e^{\frac{32 \lambda _\infty^2  }{  g^2\beta}} e^{-\lambda_\infty t}
\end{split}
\Ee

\Be
\begin{split}\label{consistence}
&\Big\|	 \mathcal{N}[\rho_f, h ,Z^F] (t,\cdot)\Big\|_{L^\infty_x}  \\
& \leq  \Big\|
\int_{\R^3} 
\tB^{\ell+1} (t,x,v) \mathfrak{C}_h  e^{- \frac{\beta}{2} \Big(1 - 	  \frac{16 }{g}	\big(2+\frac{8\| \nabla_x \phi_F \|_\infty }{g} \big) - \frac{32 \| \nabla_x \phi_F \|_\infty^2}{g^2}
\Big) |v_{\b,3}^F (t,x,v)|^2} e^{- \frac{\beta}{2} 
\Big(1 - 
\frac{16}{g} 
\| \nabla_x \phi_f \|_\infty 
\Big)
|v_\parallel|^2}  \\
& \ \ \ \ \  \times 
\sup_{s \in [ t- t_\b^{F} (t,x,v),t]}| \nabla_x \Delta_0^{-1} \rho_f (s) | 
\dd v
\Big\|_{L^\infty_x}
\\
&\leq \mathfrak{C}_h \Big\| \int_{\R^3}  \frac{4}{g}|v_{\b,3}^F(x,v)| 
e^{- \frac{\beta}{2} \Big(1 - 	  \frac{16 }{g}	\big(2+\frac{8\| \nabla_x \phi_F \|_\infty }{g} \big) - \frac{32 \| \nabla_x \phi_F \|_\infty^2}{g^2}
\Big) |v_{\b,3}^F (t,x,v)|^2} e^{- \frac{\beta}{2} 
\Big(1 - 
\frac{16}{g} 
\| \nabla_x \phi_f \|_\infty 
\Big)
|v_\parallel|^2} 
\\
& \ \ \ \ \  \times 
\mathfrak{C}_{\rho, \infty}  e^{-\lambda_\infty\Big(t- 4\frac{|v_{\b,3}^F(t,x,v)|}{g}\Big)}  
\dd v\Big\|_{L^\infty_x}
\\
& \leq \mathfrak{C}_h \mathfrak{C}_{\rho, \infty}   e^{-\lambda_\infty t} \int_{\R^3} e^{\lambda_\infty  \frac{ 4 |v_{\b,3}^F (t,x,v)|}{g}}  \frac{1}{\sqrt{\beta}}  \sqrt{\beta} \frac{ 4| v_{\b,3}^F(t,x,v) |}{g} e^{- \frac{\beta}{3} |v_{\b,3}^F (t,x,v)|^2 } e^{- \frac{\beta}{4} |v_\parallel|^2}  \dd v\\
&\leq  \mathfrak{C}_h\mathfrak{C}_{\rho, \infty}  
e^{-\lambda_\infty t}\frac{1}{g\beta^{3/2}} 
\int_{\R}
e^{- \frac{\beta}{4} \{|v^F _{\b,3}(t,x,v)|^2 - \frac{16 \lambda_\infty}{g \beta} |v^F_{\b,3}(t,x,v)| + \frac{ 16 \lambda_\infty^2 }{g^2\beta^2} \}}
e^{\frac{4 \lambda _\infty^2 }{g^2\beta}} \dd v_3\\
& \leq \frac{\mathfrak{C}_h \mathfrak{C}_{\rho, \infty}  }{g  \beta^{3/2}} e^{\frac{32 \lambda _\infty^2  }{  g^2\beta}} e^{-\lambda_\infty t}
\int_{\R} e^{ - \frac{\beta}{8} |v_{\b, 3}^F (t,x,v)|  ^2} \dd v_3\\
& \leq \frac{\mathfrak{C}_h \mathfrak{C}_{\rho, \infty}  }{g  \beta^{3/2}} e^{\frac{32 \lambda _\infty^2  }{  g^2\beta}} e^{-\lambda_\infty t}
\int_{\R} e^{ - \frac{\beta}{8} |v_3|  ^2} \dd v_3\\
& \leq \frac{\mathfrak{C}_h \mathfrak{C}_{\rho, \infty}  }{g  \beta^{2}} e^{\frac{32 \lambda _\infty^2  }{  g^2\beta}} e^{-\lambda_\infty t}
\end{split}
\Ee 
where we have used Lemma \ref{lem:elliptic_est:C^1}
\Be
\|\nabla_x \Delta_0^{-1} \rho_f (s,\cdot ) \|_{L^\infty_x} \lesssim 
\| \rho_f (s,\cdot) \|_{L^\infty_x} ,
\Ee 
\Be
\sqrt{\beta}|v_{\b,3}^F| e^{- \frac{ \beta}{3} |v_{\b,3}^F|^2 } \leq e^{- \frac{\beta}{4} |v_{\b,3}^F|^2 } ,
\Ee 
\Be
- \frac{16 \lambda_\infty}{g \beta} |v_{\b,3}^F (t,x,v)| \geq  - \frac{1}{2}|v_{\b,3}^F (t,x,v)| ^2 - \frac{128 \lambda_\infty^2}{g^2 \beta^2},
\Ee
and $|v^\b _{F,3} (t,x,v)| \geq |v_3|$.

\hide

Hence 
\Be\begin{split}
& \int_{\R} e^{ - \frac{\beta}{2} ( |v_{F, 3}^\b (t,x,v)| - \frac{aC}{g \beta })^2} \dd v_3
\int_{\R^2}  e^{- \frac{\beta}{2} |v_\parallel|^2}\dd v _\parallel
\\
&\leq  \Big\{\int_{|v_3| \geq  \frac{aC}{g \beta }}e^{- \frac{\beta}{2} (|v_3| - \frac{aC}{g \beta})^2} \dd v_3 + \int_{|v_3| \leq \frac{aC}{g \beta }} \dd v_3\Big\} 	\int_{\R^2}  e^{- \frac{\beta}{2} |v_\parallel|^2}\dd v _\parallel\\
& \leq \Big\{ \int_{\tau\geq 0} e^{- \frac{\beta}{2} \tau^2} \dd \tau + \frac{a C}{g \beta }\Big\} 	\int_{\R^2}  e^{- \frac{\beta}{2} |v_\parallel|^2}\dd v _\parallel\\
& \leq \Big\{ \frac{1}{\sqrt{\beta}} + \frac{a C}{g \beta }\Big\} \frac{1}{\beta}.
\end{split}\Ee

Back to \eqref{consistence}:
\Be
\Big\|	 \mathcal{N}[\rho_f, h ,Z^F] (t,\cdot)\Big\|_{L^\infty_x}  \lesssim e^{-a_\infty t} \frac{1}{g \sqrt{\beta}} e^{\frac{a^2 C^2}{2 g^2\beta}}  \Big\{ \frac{1}{\sqrt{\beta}} + \frac{a C}{g \beta }\Big\} \frac{1}{\beta}
\Ee\unhide

Choosing $g \gg_\beta 1$ we might expect to have an exponential decay. 

Using the control of $e^{\frac{\beta}{2} |v|^2} e^{ \frac{\beta g}{2} x_3} |\nabla_v h_\pm (x,v)| \leq \mathfrak{C}_h \| w_G^{\pm} \nabla_{x_\parallel, v} G_\pm \|_{L^\infty_{\gamma_-}}$ in \eqref{est:hk_v}, 
\Be\begin{split}
&	\left| \int^t_{{\max\{0, t- \tb^{\ell+1,\pm } (t,x,v) \}}   } \nabla_x\Delta^{-1}_0 \rho ^\ell  (s, \X^{\ell+1} _\pm (s;t,x,v)) \cdot    \nabla_v h_\pm 
( \mathcal{Z}_\pm ^{\ell+1}  (s;t,x,v) ) 
\dd s   \right|\\
& \leq   \int^{t}_{ t- \tb^{\ell+1,\pm } (t,x,v)}\mathfrak{C}_{\rho, \infty} e^{-\lambda_\infty s}
\mathfrak{C}_h \| w_G^{\pm} \nabla_{x_\parallel, v} G_\pm \|_{L^\infty_{\gamma_-}}
e^{-\frac{\beta}{4} | v |^2} e^{- \frac{\beta g}{4} x_3}
\\
& \leq  \mathfrak{C}_{\rho, \infty}  e^{- \lambda_\infty t}   \tb^{\ell+1,\pm }  e^{ \lambda_\infty  \tb^{\ell+1,\pm }  }\mathfrak{C}_h \| w_G^{\pm} \nabla_{x_\parallel, v} G_\pm \|_{L^\infty_{\gamma_-}}
e^{-\frac{\beta}{4} | v |^2} e^{- \frac{\beta g}{4} x_3}\\
& \leq  \mathfrak{C}_{\rho, \infty} \mathfrak{C}_h \| w_G^{\pm} \nabla_{x_\parallel, v} G_\pm \|_{L^\infty_{\gamma_-}}   e^{- \lambda_\infty t} 
\frac{4   }{g}  \sqrt{|v_3|^2 + g x_3}  e^{ \frac{4 \lambda_\infty }{g}  \sqrt{|v_3|^2 + g x_3}  }e^{-\frac{\beta}{4} | v |^2} e^{- \frac{\beta g}{4} x_3}\\
& \leq \left(	\mathfrak{C}_h \| w_G^{\pm} \nabla_{x_\parallel, v} G_\pm \|_{L^\infty_{\gamma_-}} 
\frac{16 e^{-1/2}}{g \beta^{1/2} }  
e^{\frac{32\lambda_\infty^2}{ g^2 \beta }} 
\right)\mathfrak{C}_{\rho, \infty} e^{- \lambda_\infty t} ,
\end{split}
\Ee
where we have used  
\Be\notag
\begin{split}
&	\frac{4   }{g}  \sqrt{|v_3|^2 + g x_3}  e^{ \frac{4 \lambda_\infty }{g}  \sqrt{|v_3|^2 + g x_3}  }e^{-\frac{\beta}{4} | v |^2} e^{- \frac{\beta g}{4} x_3} \\
& \leq 	\frac{4}{g}  \sqrt{|v|^2 +   g x_3} 
e^{- \frac{\beta}{4} (|v|^2 +   g x_3)} 
e^{ \frac{4 \lambda_\infty}{g}\sqrt{|v|^2 +   g x_3}  }\\
&= 	\frac{4}{g \beta^{1/2}}  \sqrt{ \beta (|v|^2 + gx_3)}
e^{- \frac{1}{4}\sqrt{ \beta (|v|^2 + gx_3)}^2 }	e^{ \frac{4 \lambda_\infty}{g}\sqrt{|v|^2 +   g x_3}  }\\
& \leq \frac{4}{g \beta^{1/2}} \frac{2}{e^{1/2}} 	e^{- \frac{ \beta}{8}\sqrt{  |v|^2 + gx_3}^2 }e^{ \frac{4 \lambda_\infty}{g}\sqrt{|v|^2 +   g x_3}  }\\
& = \frac{4}{g \beta^{1/2}} \frac{2}{e^{1/2}} 
e^{- \frac{\beta}{8} \left(
\sqrt{  |v|^2 + gx_3} - \frac{16 \lambda_\infty}{ g \beta }
\right)^2} e^{\frac{32\lambda_\infty^2}{ g^2 \beta }}.
\end{split}
\Ee

\hide\Be
\begin{split}
e^{- \frac{\beta}{4} |v|^2} e^{ \frac{4 \lambda_\infty}{g} |v|} &= e^{- \frac{\beta}{4} \left( |v| - \frac{8 \lambda_\infty}{ g \beta}\right)^2} 
e^{ \frac{16 \lambda_\infty^2}{   g^2 \beta}}	, \\
e^{- \frac{\beta}{4}  g x_3} e^{ \frac{4 \lambda_\infty}{g} \sqrt{gx_3}} &= e^{- \frac{\beta}{4} \left(  \sqrt{gx_3} - \frac{8\lambda_\infty}{ g \beta}\right)^2} 
e^{ \frac{16  \lambda_\infty^2}{  g^2 \beta}}	.
\end{split}
\Ee\unhide

\end{proof}

\unhide


\hide
\begin{proposition}There exists a family of solutions $f^\ell$ to such that 
\Be
\sup_\ell	\| w^{\ell} f^\ell \|_{L^\infty (\R_+ \times \O \times \R^3)} , \ \ \ \sup_\ell \| \nabla_x \phi_{f^\ell} (t,x) \|_{L^p (\R_+ \times \O  )} < \infty    \text{ for any} \ p< \infty. 
\Ee

\Be\label{choice:L}
L\geq    
\frac{4C}{\beta^{3/2}} 	e^{ \frac{16 \lambda_\infty^2}{\beta g^2}}	  \|w_{0 }   f_{0  } \|_{L^\infty_{x,v}} .
\Ee

\Be\label{condition:g}
\begin{split}
\max\left\{1,	\frac{\lambda_\infty^2}{ {\beta}},
8^3   e^{\frac{16^2}{4}}      C \| w_\beta \nabla_v h \|_\infty \frac{1}{\beta^{5/2}} 
\right\}\leq g^2 .
\end{split}
\Ee

\Be
\sup_{t \geq 0 }e^{  \lambda_\infty t}	\| \varrho(t) \|_{\infty} \lesssim 1. 
\Ee

\Be
e^{\lambda_\infty t}	\| e^{ \frac{\beta}{8}  (|v|^2+  g x_3)} f(t) \|_{L^\infty (\O \times \R^3)}
\lesssim \| w_0 f_0 \|_\infty + \| w_\beta \nabla_v h \|_\infty
\Ee
\end{proposition}

\begin{proof}
\it{Step 1.} Taking $v$-integration to \eqref{form:f^ell} and using \eqref{est:Iell}-\eqref{est:Nell}, we derive that 
\Be\begin{split}\label{est:varrho}
& e^{  \lambda_\infty t}	|\varrho^{\ell+1} (t,x)|  \leq  e^{  \lambda_\infty t}	 \left\{\int_{\R^3} | \mathcal{I}^{\ell+1} (t,x,v)| \dd v +  \int_{\R^3} |\mathcal{N}^{\ell+1} (t,x,v) |\dd v\right\} \\
& \leq  \frac{C}{\beta^{3/2}}  \left\{ 2  e^{-  \frac{\beta}{4}  g  x_3 } 
e^{ \frac{16 \lambda_\infty^2}{\beta g^2}}	  \|w_{0 }   f_{0  } \|_{L^\infty_{x,v}}  \right.\\
&\left. \ \ \ \ \ \ \  \ \ \ \ + e^{- \frac{\beta }{8} g x_3}
\frac{8}{g \beta^{1/2}} 
\frac{8 \cdot 4^{1/g}}{g \beta^{1/2}}  
e^{ \frac{16^2 \lambda_\infty^2}{ 4 g^2 \beta}}  	\sup_{s \in [0, t] } \| e^{\lambda_\infty s} \varrho^\ell (s) \|_\infty  \| w_\beta \nabla_v h \|_\infty\right\}.
\end{split}	\Ee

Suppose 
\Be
\max_{\ell  \leq  k}	\sup_{t\in [0, \infty) } \| e^{\lambda_\infty t} \varrho^\ell (t) \|_\infty \leq L.
\Ee
Then using \eqref{est:varrho} and our choice of $g$ in \eqref{condition:g}, we derive that 
\Be
\begin{split}
\sup_{t \in [0, \infty) }	e^{  \lambda_\infty t}	 \|  \varrho^{k+1} (t)\|_{\infty}  
& \leq   
\frac{2C}{\beta^{3/2}} 	e^{ \frac{16 \lambda_\infty^2}{\beta g^2}}	  \|w_{0 }   f_{0  } \|_{L^\infty_{x,v}}   +\left( 
\frac{ C 8^2 4^{1/g} }{g^2 \beta^{5/2} }  
e^{ \frac{16^2 \lambda_\infty^2}{ 4 g^2 \beta}}   \| w_\beta \nabla_v h \|_\infty \right)L \\
&\leq L  .
\end{split}
\Ee

\end{proof}
\unhide

\hide
In order to prove \eqref{h_v}, we need to have 
\Be
1) \ \ \ \nabla_v h(x,v) \lesssim \frac{1}{w_{\pm}(x,v) = w(\frac{|v|^2}{2} + \phi_h(x) + g m_\pm x_3)} 
\Ee

Note that 
\Be
w_h(x,v) = w\Big( \frac{|v|^2}{2} + \phi_h (x) + g m_\pm x_3\Big)
\Ee
\Be
w_F(x,v) = w\Big( \frac{|v|^2}{2} + \phi_f (x)+ \phi_h (x) + g m_\pm x_3\Big)
\Ee

We expect a bound for $\nabla_v h(X_F(s;t,x,v), V_F(s;t,x,v)) \lesssim \frac{1}{w_h (X_F(s;t,x,v), V_F(s;t,x,v))}$.

and 2) for large $v$, $|V(s;t,x,v)| \gtrsim  |v|$:
\Be
\big||V(s;t,x,v)| - |v|\big| \leq |t-s| \{ \|\nabla_x \phi_F \|_\infty + gm_\pm \}
\leq t^\b\{ \|\nabla_x \phi_F \|_\infty + gm_\pm \} \lesssim C \frac{|v_3|}{g}\{ \|\nabla_x \phi_F \|_\infty + gm_\pm \} 
\Ee
\Be
|V(s;t,x,v)|= |v| + O(1) C \frac{|v_3|}{g}\{ \|\nabla_x \phi_F \|_\infty + gm_\pm \} 
\Ee
\unhide

\end{document}